\documentclass[a4paper,oneside]{amsbook}

\usepackage{amsmath,amsfonts,amsthm}
\usepackage{amssymb,latexsym}
\usepackage{url}
\usepackage{enumerate}
\usepackage[all]{xy}
\usepackage{setspace}

\theoremstyle{plain}
\newtheorem{thm}{Theorem}[chapter]
\newtheorem{lem}[thm]{Lemma}
\newtheorem{prop}[thm]{Proposition}
\newtheorem*{thm2}{Theorem}
\newtheorem{claim}[thm]{Claim}
\newtheorem*{prop2}{Proposition}

\theoremstyle{definition}
\newtheorem{defin}[thm]{Definition}

\theoremstyle{remark}
\newtheorem{rem}[thm]{Remark}

\numberwithin{equation}{chapter}
\numberwithin{section}{chapter}

\DeclareMathOperator{\im}{Im}
\DeclareMathOperator{\gal}{Gal}

\DeclareMathOperator{\Ad}{Ad}
\DeclareMathOperator{\normed}{N}
\DeclareMathOperator{\SL}{SL}
\DeclareMathOperator{\GL}{GL}
\DeclareMathOperator{\Su}{SU}
\newcommand{\SU}{\ensuremath{\Su( 2, 1 )}}
\newcommand{\wtilde}[1]{\widetilde{#1}}
\newcommand{\ol}{\overline}
\newcommand{\Qp}{\ensuremath{\mathbb{Q}_{p}}}
\newcommand{\Zp}{\ensuremath{\mathbb{Z}_{p}}}
\newcommand{\lra}{\longrightarrow}
\newcommand{\thet}{\ensuremath{\theta_{0}}}
\newcommand{\ha}[1]{\ensuremath{h_{\alpha}\left( #1 \right)}}
\newcommand{\norm}[1]{\ensuremath{\normed\left( #1 \right)}}
\newcommand{\xa}[2]{\ensuremath{x_{\alpha}\left( #1,#2 \right)}}
\newcommand{\xam}[2]{\ensuremath{x_{- \alpha}\left( #1,#2 \right)}}
\newcommand{\wa}[2]{\ensuremath{w_{\alpha}\left( #1,#2 \right)}}
\newcommand{\del}[2]{\ensuremath{\delta_{#1}\left( #2 \right)}}
\newcommand{\txa}[2]{\ensuremath{\widetilde{x}_{\alpha}\left( #1,#2 \right)}}
\newcommand{\twa}[2]{\ensuremath{\widetilde{w}_{\alpha}\left( #1,#2 \right)}}
\newcommand{\txam}[2]{\ensuremath{\widetilde{x}_{- \alpha}\left( #1,#2 \right)}}
\newcommand{\twam}[2]{\ensuremath{\widetilde{w}_{- \alpha}\left( #1,#2 \right)}}
\newcommand{\sigu}[2]{\ensuremath{\sigma_{u} \left( #1,#2 \right)}}
\newcommand{\delh}[1]{\ensuremath{\delta\left( h_{\alpha}\left( #1 \right) \right)}}
\newcommand{\ba}[2]{\ensuremath{b_{\alpha}\left( #1,#2 \right)}}
\newcommand{\hilkn}[2]{\ensuremath{\left( #1,#2 \right)_{k, n}}}
\DeclareMathOperator{\homo}{Hom}
\newcommand{\hilk}[2]{\ensuremath{\left( #1,#2 \right)_{k, 2}}}
\newcommand{\sig}[2]{\ensuremath{\sigma\left( #1,#2 \right)}}

\newcommand{\hilbk}[2]{\ensuremath{\left( #1,#2 \right)_{K, 2}}}
\newcommand{\Ok}{\ensuremath{\mathcal{O}_{k}}}
\newcommand{\OK}{\ensuremath{\mathcal{O}_{K}}}
\DeclareMathOperator{\Tr}{Tr}
\newcommand{\trace}[1]{\ensuremath{\Tr\left( #1 \right)}}
\newcommand{\HC}{\ensuremath{\mathcal{H}_{\mathbb{C}}}}
\newcommand{\Sig}[2]{\ensuremath{\Sigma\left( #1,#2 \right)}}
\newcommand{\ess}[2]{\ensuremath{S\left( #1,#2 \right)}}
\newcommand{\spl}[2]{\ensuremath{\left( #1,#2 \right)}}

\newcommand{\Xg}[1]{\ensuremath{X( \gamma_{#1} )}}
\newcommand{\ug}[2]{\ensuremath{u\left( #1, #2  \right)}}
\newcommand{\kup}[1]{\ensuremath{\kappa_{\mathfrak{p}}\left( #1 \right)}}
\newcommand{\lp}{\ensuremath{l_{\mathfrak{p}}}}
\newcommand{\Lp}{\ensuremath{L_{\mathfrak{p}}}}
\newcommand{\Olp}{\ensuremath{\mathcal{O}_{l_{\mathfrak{p}}}}}
\newcommand{\OLp}{\ensuremath{\mathcal{O}_{L_{\mathfrak{p}}}}}
\DeclareMathOperator{\Res}{Res}
\DeclareMathOperator{\ord}{ord}
\DeclareMathOperator{\Spec}{Spec}
\newcommand{\hilbkn}[2]{\ensuremath{\left( #1,#2 \right)_{K, n}}}
\newcommand{\kxam}[1]{\ensuremath{\rho\left( #1 \right)}}

\begin{document}

\title{\Large Declaration}
\author{\begin{flushleft} \normalsize I, Lina Jalal, confirm that the work presented in this thesis is my own. Where information has been derived from other sources, I confirm that this has been indicated in the thesis. \end{flushleft} \begin{flushright} \normalsize \textsc{Signed} $ \phantom{\qquad \qquad} $ \end{flushright}}
\date{}

\begin{center}
\thispagestyle{empty}
 

\textsc{\LARGE Department of Mathematics, UCL}\\[3.2cm]


{ \huge \bfseries A calculation of a half-integral weight multiplier system on \SU{} }\\[1.5cm]

\LARGE

Lina Jalal\\[1.5cm]

\textsc{\Large A thesis submitted for the degree of Doctor of Philosophy}\\[3cm]
 

\begin{minipage}{0.4\textwidth}
\begin{center} \Large
\emph{Supervisor:} \\
Dr. Richard~\textsc{Hill}
\end{center}
\end{minipage}
 
\vfill
 
{\large December 2010}
 
\end{center}

\begin{abstract}
In this thesis, we construct a half-integral weight multiplier system on the group $ \SU $.  In order to do so, we first find a formula for a 2-cocycle representing the double cover of $ \SU( k ) $, where $ k $ is a local field.  For each non-archimedean local field $ k $, we describe how the cocycle splits on a compact open subgroup.  The multiplier system is then expressed in terms of the product of the local splittings at each prime.
\end{abstract}

\maketitle

\newpage

\begin{center}
\thispagestyle{empty}
{\itshape
For Baba, Mama, Nuha, `Alia, `Aida and Adam.  You know why.
\\[1.5cm]
And for Auntie Mah, for making it all possible in the first place all those years ago, and now.
\\[2.5cm]
Of course, my ultimate thanks must go to my supervisor, Dr. Richard Hill.  Without him, this thesis would have been only a poor shadow of what it is, and his knowledge, insight and support during my PhD, as well as my undergraduate years, is much valued and much appreciated.
\\[1.5cm]
Also, I will never forget what UCL has done for me.  I came as an undergraduate wanting to experience higher education in a foreign land and I will always appreciate the generosity in awarding me both the Overseas Research Student Award and the Graduate School Research Scholarship, which has been vital for my continued stay in London, one of the greatest cities in the world.  Thank you for giving me an experience I will cherish my whole life.
}
\end{center}

\tableofcontents


\chapter*{Introduction} \label{c:intro}
Modular forms of half-integral weight have been known to exist for some time, and standard examples may be found in the form of theta functions and the Dedekind eta function.  But although modular forms of half-integral weight have been found on groups such as special linear groups, symplectic groups and orthogonal groups, it is not possible to obtain modular forms of half-integral weight on the group $ \SU{} $ by restriction, even though they are known to exist.

This thesis is concerned with finding a half-integral weight multiplier system on \SU{}.  With this half-integral weight multiplier system in place, it would be possible to write down a modular form of half-integral weight on \SU{}.

\section*{Modular forms on $ \SL_{2} $}

We first recall the definition of a modular form on the group $ \SL_{2} $.
The \emph{modular group} is defined to be
	\[ \SL_{2}( \mathbb{Z} ) = \left\{ \begin{pmatrix} a & b \\ c & d \end{pmatrix} \colon a, b, c, d \in \mathbb{Z}, a d - b c = 1 \right\}. \]
This group acts on the \emph{upper half-plane}
	\[ \mathcal{H} = \{ \tau \in \mathbb{C} \colon \im( \tau ) > 0 \} \]
by fractional linear transformations, i.e.
	\[ \gamma = \begin{pmatrix} a & b \\ c & d \end{pmatrix} \colon \tau \to \gamma( \tau ) = \frac{a \tau + b}{c \tau + d}. \]
Let $k$ be a positive integer.
A \emph{modular form of weight $k$}
is a holomorphic function $ f \colon \mathcal{H} \to \mathbb{C} $
such that for each $\gamma\in \SL_{2}(\mathbb{Z})$,
	\[ f( \gamma ( \tau ) ) = ( c \tau + d )^{k} f( \tau ), \]
and such that $ f $ is holomorphic at the cusp $ \infty $.

Now let $k/2$ be a half-integer and let $\Gamma \subset \SL_{2}(\mathbb{Z})$ be a subgroup of finite index.
By a weight $k / 2$ multiplier system, we shall mean a continuous function
	\[ j \colon \Gamma \times \mathcal{H} \to \mathbb{C}, \]
such that for $ \gamma $, $ \gamma' \in \Gamma $, $ \tau \in \mathcal{H} $,
	\[ j( \gamma \gamma', \tau ) = j( \gamma, \gamma'( \tau ) ) \cdot j( \gamma', \tau ), \]
and
	\[ j( \gamma, \tau )^{2} = ( c \tau + d )^{k}, \quad \gamma = \begin{pmatrix} a & b \\ c & d \end{pmatrix}. \]
A modular form of weight $k / 2$ on $ \Gamma$ is defined to be a function $ f \colon \mathcal{H} \to \mathbb{C} $ such that for $\gamma\in\Gamma$ and $\tau\in\mathcal{H}$, we have
	\[ f( \gamma( \tau ) ) = j( \gamma, \tau ) f( \tau ), \]
and such that $ f $ is holomorphic on $ \mathcal{H} $ and the cusps of $ \Gamma $.
An example of a half-integral weight modular form on $ \SL_{2} $ is the \emph{Dedekind eta function}
	\[ \eta( z ) = q^{1 / 24} \prod_{1}^{\infty} ( 1 - q^{n} ), \]
where $ q = e^{2 \pi i z} $, and it can be shown that for $ \gamma = \begin{pmatrix} a & b \\ c & d \end{pmatrix} \in \SL_{2}( \mathbb{Z} ) $,
	\[ \eta\left( \frac{a z + b}{c z + d} \right) = \epsilon( \gamma ) ( c z + d )^{1 / 2} \eta( z ), \]
where $ \epsilon( \gamma ) $ is a $ 24 $th root of unity (see Section~1.3 of \cite{bump98}).
When one restricts to elements $\gamma$ in the commutator subgroup
$\Gamma=[\SL_{2}(\mathbb{Z}),\SL_{2}(\mathbb{Z})]$ we have
$\epsilon(\gamma)=\pm 1$.

\section*{Modular forms on \SU{}}

We can define similar notions for the group \SU{}.
To describe this group, let $ \thet = \sqrt{- d}$, where $d\ge 2$ is a square-free natural number (fixed once and for all).
We shall write the non-trivial Galois automorphism of $\mathbb{Q}(\thet)$ by
	\[ \ol{a + b \thet} = a - b \thet. \]
We define $\SU$ as an algebraic group over $\mathbb{Q}$ as follows:
for a commutative $ \mathbb{Q} $-algebra $ A $ we let
	\[ \SU( A ) = \{ \nu \in \SL_{3}( A \otimes_{\mathbb{Q}} \mathbb{Q}( \thet ) ) \colon \nu^{t} J' \ol{\nu} = J' \}. \]
In this definition, the matrix $J'$ is given by
	\[ J' = \begin{pmatrix} 0 & 0 & 1 \\ 0 & 1 & 0 \\ 1 & 0 & 0 \end{pmatrix}. \]
The notation $\nu^{t}$ denotes the transpose of $\nu$, and $\ol\nu$ denoted the image of $\nu$
under conjugation in $\mathbb{Q}( \thet )$.
We shall also on occasion regard $\SU$ as a group scheme over $\mathbb{Z}$,
 defined (for a commutative ring $A$) by
	\[ \SU( A ) = \{ \nu \in \SL_{3}( A \otimes_{\mathbb{Z}} \mathcal{O}_{\mathbb{Q}( \thet )} ) \colon \nu^{t} J' \ol{\nu} = J' \}. \]
Here $\mathcal{O}_{\mathbb{Q}(\thet)}$ denotes the ring of algebraic integers in $\mathbb{Q}(\thet)$.

Consider the Hermitian form on the vector space $V=\mathbb{C}^{3}$ defined by
	\[ \langle u, v \rangle = u^{t} J' \ol{v}. \]
The group $\SU{}(\mathbb{R})$ acts on $ V $, and hence on $ X = \mathbb{P}^{2}( \mathbb{C} ) $ in an obvious way.
Furthermore $ \SU{}( \mathbb{R} ) $ preserves the subsets
\begin{gather*}
	X^{-} = \{ [ v ] \in \mathbb{P}^{2}( \mathbb{C} ) \colon \langle v, v \rangle < 0 \}, \\
	V^{-} = \{ v \in \mathbb{C}^{3} \colon \langle v, v \rangle < 0 \},
\end{gather*}
where $ [ v ] $ denotes the image in projective space of a vector $v$.
Hence $ V^{-} $ is the preimage of $ X^{-} $ in $ V \backslash 0 $.

Let $k$ be a positive integer. We define a weight $k$ modular form on an arithmetic subgroup $ \Gamma \subset \SU( \mathbb{Q} ) $ as follows.
Let $ F \colon V^{-} \to \mathbb{C} $ be a holomorphic function such that
\begin{gather*}
	F( \gamma v ) = F( v ), \quad \text{for all } \gamma \in \Gamma, \\
	F( \lambda v ) = \lambda^{- k} F( v ), \quad \text{for all } \lambda \in \mathbb{C}.
\end{gather*}
It turns out that no growth conditions at the cusps are required.

We shall now explain how this definition is related to the definition for $\SL_{2}(\mathbb{Z})$.
Any point of $ X^{-} $ has a unique representative in $V^{-}$ of the form $ \begin{pmatrix} \tau_{1} \\ \tau_{2} \\ 1 \end{pmatrix} $, where $ \norm{\tau_{2}} + \trace{\tau_{1}} < 0 $.
Here we are using the notation $ \norm{x} = x \ol{x} $ and $ \trace{x} = x + \ol{x} $ for a complex number $ x $.
We let \HC{} be the set of all such pairs $ \begin{pmatrix} \tau_{1} \\ \tau_{2} \end{pmatrix} $, so we have a bijection $ X^{-} \cong \HC $.
Given a modular form $ F $ in the sense just described, we define a function $ f $ on $\HC$ by
	\[ f \left( \begin{pmatrix} \tau_{1} \\ \tau_{2} \end{pmatrix} \right)
	=
	F \left( \begin{pmatrix} \tau_{1} \\ \tau_{2} \\ 1 \end{pmatrix} \right).
	 \]
The action of $ \SU( \mathbb{R} ) $ on $ X^{-} $ gives us an action on $ \HC $, which we will now examine.
Let
	\[ g = \begin{pmatrix} g_{11} & g_{12} & g_{13} \\ g_{21} & g_{22} & g_{23} \\ g_{31} & g_{32} & g_{33} \end{pmatrix} \in \SU( \mathbb{R} ). \]
We decompose this matrix into blocks as follows:
	\[ A = \begin{pmatrix} g_{11} & g_{12} \\ g_{21} & g_{22} \end{pmatrix},\ B = \begin{pmatrix} g_{13} \\ g_{23} \end{pmatrix},\ C = \begin{pmatrix} g_{31} & g_{32} \end{pmatrix},\ D = g_{33}. \]
This implies that
	\[ g = \begin{pmatrix} A & B \\ C & D \end{pmatrix}. \]
The action of $ \SU( \mathbb{R} ) $ on \HC{} is described by
	\[ g (\tau) = \frac{A \tau + B}{C \tau + D},
	\qquad
	\tau = \begin{pmatrix} \tau_{1} \\ \tau_{2} \end{pmatrix} \in \HC.
	\]
Now let $f$ be as above, and suppose that $g$ is in the arithmetic group $\Gamma$.
Then we have:
\begin{align*}
	f( g ( \tau ) )
	&= F \left( \begin{pmatrix} \frac{A \tau + B}{C \tau + D} \\ 1 \end{pmatrix} \right) \\
	&= ( C \tau + D )^{k} F \left( \begin{pmatrix} A \tau + B \\ C \tau + D \end{pmatrix} \right) \\
	&= ( C \tau + D )^{k} F\left( g \begin{pmatrix} \tau \\ 1 \end{pmatrix} \right) \\
	&= ( C \tau + D )^{k} F\left( \begin{pmatrix} \tau \\ 1 \end{pmatrix} \right) \\
	&= ( C \tau + D )^{k} f( \tau ).
\end{align*}

We can thus similarly define a weight $ k / 2 $ multiplier system on \SU{} as a holomorphic function $ j \colon \Gamma \times \HC \to \mathbb{C} $ such that for $ \gamma $, $ \gamma' \in \Gamma $, $ \tau \in \HC $, and $ \gamma $ defined in the same way as $ g $ above:
	\[ j( \gamma \gamma', \tau ) = j( \gamma, \gamma' ( \tau ) ) \cdot j( \gamma', \tau ), \]
and
	\[ j( \gamma, \tau )^{2} = ( C \tau + D )^{k}, \quad \gamma = \begin{pmatrix} A & B \\ C & D \end{pmatrix}. \]

We may also define half-integral weight modular forms on $\SU$ entirely analogous to the case of $\SL_{2}$.
However, no example of a half-integral weight modular form on $\SU$ has ever been found.
There are no standard examples such as theta series. The other standard way of writing down a half-integral weight form would be to write down an Eisenstein series. However, this cannot be done without knowing the multiplier system in advance, and no example of a half-integral weight multiplier system has previously been found (although they were known to exist; see \cite{deligne96}). The aim of this thesis is to give a half-integral weight multiplier system on \SU.

\section*{Strategy for constructing the multiplier system}

Let $\mathbb{A}$ denote the ad\`{e}le ring of $ \mathbb{Q} $ and write $\mu_{2}$ for the group $\{1,-1\}$.
The group $\SU(\mathbb{A})$ has a canonical double cover, called the ``metaplectic cover'':
	\[ 1 \lra \mu_{2} \lra \wtilde{\SU}( \mathbb{A} ) \lra \SU( \mathbb{A} ) \lra 1. \]
This is a central extension of topological groups. It is a ``cover'' in the following sense:
there is a neighbourhood $U$ of the identity in $\SU( \mathbb{A} )$, such that the restriction preimage of
$U$ is topologically a product $U \times\mu_{2}$.
The word ``metaplectic'' means that the extension splits over the rational points $\SU(\mathbb{Q})$.
In fact, every reductive group over a number field has a canonical mateplectic cover, with kernel the
roots of unity in the field (see \cite{deligne96}). 

For a place $p$ of $\mathbb{Q}$ we shall write $\wtilde{\SU}( \mathbb{Q}_{p} )$ for the pre-image of
$\SU(\mathbb{Q}_{p})$ in $\wtilde{\SU}( \mathbb{A} )$. This means that we have local extensions:
	\[ 1 \lra \mu_{2} \lra \wtilde{\SU}( \mathbb{Q}_{p} ) \lra \SU( \mathbb{Q}_{p} ) \lra 1. \]
Our first aim is to describe a 2-cocycle $\sigma_{p}$ on $\SU(\mathbb{Q}_{p})$ corresponding to this
extension.
In fact, our cocycle will be expressed in terms of Hilbert symbols $(-,-)_{\Qp,2}$.
This has the consequence (by the quadratic reciprocity law) that for $g,g'\in \SU(\mathbb{Q})$
we have
$$
	\prod_{p} \sigma_{p}(g,g') = 1.
$$
This product formula reflects the fact that our extension splits on the rational points.

One cannot define a cocycle $\sigma_{\mathbb{A}}$ on $\SU(\mathbb{A})$
 to be simply the product of the local cocycles, since this product will usually have infinite support.
However, we can do something rather similar. For each finite prime $p$, there is a compact open subgroup
$\Gamma_{p} \subset \SU(\Qp)$ on which the extension splits, and for almost all $p$ we may take
$\Gamma_{p}=\SU(\Zp)$. This means that there is a function $\kappa_{p}\colon\Gamma_{p}\to\mu_{2}$,
such that for $g,g'\in \Gamma_{p}$ we have
$$
	\sigma_{p}(g,g')
	=
	\partial\kappa_{p}(g,g')
	=
	\frac{\kappa_{p}(gg')}{\kappa_{p}(g)\kappa_{p}(g')}.
$$
The functions $\kappa_{p}\colon\Gamma_{p}\to\mu_{2}$ are called ``local Kubota symbols''.
If we extend $\kappa_{p}$ in some arbitrary way to $\SU(\Qp)$, then we may now form the product
$$
	\sigma_{\mathbb{A}}(g,g')
	=
	\sigma_{\infty}(g_{\infty},g'_{\infty})
	\prod_{p \  \mathrm{finite}}
	\frac{\sigma_{p}(g_{p},g'_{p})}{\partial\kappa_{p}(g_{p},g'_{p})},
	\qquad
	g,g' \in \SU(\mathbb{A}).
$$
This product does have finite support, and is a 2-cocycle representing the full metaplectic extension.
The second main aim of the thesis is to calculate the local Kubota symbols.

Now consider the following congruence subgroup:
$$
	\Gamma
	=
	\SU(\mathbb{Q}) \cap \Big( \SU(\mathbb{R}) \times \prod_{p}\Gamma_{p}\Big).
$$
We define a map $\kappa\colon\Gamma\to \mu_{2}$ (called the ``global Kubota symbol'') by
$$
	\kappa(\gamma)
	=
	\prod_{p < \infty} \kappa_{p}(\gamma).
$$
(This product always converges: see Chapter~\ref{c:global}.)  From the formulae above, we immediately have
$$
	\sigma_{\infty}(\gamma,\gamma')
	=
	\frac{\kappa(\gamma)\kappa(\gamma')}{\kappa(\gamma\gamma')}.
$$

Our next step is to examine the cocycle $\sigma_{\infty}$ more closely.
It turns out that there is another way of constructing the group $ \wtilde{\SU}( \mathbb{R} ) $.
Let
	\[ \wtilde{\SU}( \mathbb{R} )
	 =
	 \left\{ \left( g, \phi \colon \HC \to \mathbb{C}^{\times} \right) \colon g \in \SU( \mathbb{R} ) \right\}, \]
where $ \phi $ is continuous, and for every
	\[ g = \begin{pmatrix} A & B \\ C & D \end{pmatrix} \in \SU( \mathbb{R} ) \]
(as defined earlier), we have $ \phi( \tau )^{2} = C \tau + D $.
Multiplication in this group is defined by
$$
	(g,\phi)(g',\phi')
	=
	(gg', (\phi\circ g')\phi').
$$
There is an obvious homomorphism $(g,\phi) \mapsto g$, which makes this group a double cover
of $\SU(\mathbb{R})$.
We prove in Chapter~\ref{c:real} that this is the unique connected double cover of $\SU(\mathbb{R})$ and
is isomorphic to the local factor at infinity of the metaplectic group.
The final aim of the thesis is to describe explicitely a section $g \mapsto (g,\phi_{g})$,
which corresponds to the 2-cocycle $\sigma_{\infty}$.
In more elementary terms this means
$$
	\sigma_{\infty}(g,g')
	=
	\frac{\phi_{g}(g' ( \tau ) ) \phi_{g'}(\tau)}{\phi_{gg'}(\tau)},
	\qquad
	g,g'\in\SU(\mathbb{R}),\ \tau \in \HC.
$$
If we now define for $\gamma\in \Gamma$,
$$
	j(\gamma,\tau)
	=
	\kappa(\gamma) \phi_{\gamma}(\tau),
$$
then the formulae above show that $j(\gamma,\tau)$ is a multiplier system of weight $1/2$.

In fact we shall work almost entirely in a more general setting than was described above.
We shall replace the rational numbers by an arbitrary number field $l$ and $\mathbb{Q}(\thet)$ by an arbitrary quadratic extension $L/l$.
In the case that $l$ is totally complex, the cocycles $\sigma_{p}$ for complex places $p$ are all trivial,
and hence the global Kubota symbol $\kappa$ is a group homomorphism.

The plan of the thesis is divided into four parts.  In order to calculate the local Kubota symbol, we will first need to find an explicit formula in terms of quadratic Hilbert symbols for the 2-cocycle representing $ \wtilde{\SU}( \Qp ) $, for $ p $ finite.  Deodhar worked on the computation of the fundamental group of quasi-split groups in \cite{deodhar78}.  We will be extending the methods described in this paper to find the explicit formula of the 2-cocycle that we need.  Hence, the first part of the thesis is concerned with establishing some important facts and results for later use.  In the second part, we will give an explicit formula for the 2-cocycle representing $ \wtilde{\SU}( \Qp ) $.  The third part concerns the calculation of the local Kubota symbol for every finite prime $ p $.  In the fourth part, we will look at some calculations of the global Kubota symbol, and find the section for $ \SU( \mathbb{R} ) $.  We will then establish what the half-integral weight multiplier system is.

\section*{Summary of the results of the thesis}

We shall fix once and for all a number field $l$ and a quadratic extension
$L=l(\thet)$. We define our group over $l$ by
\[
	\SU( - )
	=
	\{ \nu \in \SL_{3}( - \otimes_{l} L ) \colon \nu^{t} J' \overline{\nu} = J' \}.
\]
Let $k=l_{\mathfrak{p}}$ be a local completion of $l$, either archimedean or non-archimedean
and let $K=k\otimes_{l}L$. Thus $K$ is either a quadratic extension of $k$ or a sum of two copies of $k$.
Recall that we have a double cover $\wtilde{\SU}(k)$ of the group $\SU(k)$.
We begin by specifying a section $\delta'\colon\SU(k) \to \wtilde{\SU}(k)$.
This section defines a 2-cocycle $\sigma$ corresponding to the cover:
$$
	\sigma(g,h)
	=
	\delta'(g)\delta'(h)\delta'(gh)^{-1},
	\quad
	g,h\in\SU(k).
$$
Our first results are an expression for $\sigma$ in terms of quadratic Hilbert symbols on $k$.
In the case that $K/k$ is a field extension, our result completely describes $\sigma$.
In the split case, we obtain expressions which are valid on (a) the maximal torus and
(b) the subgroup $\SU(l)$.
This will be enough for our purposes.

\subsection*{The cocycle on the torus}
Let $k$ be a local field and let $K=k(\thet)$ be either a quadratic extension of $k$
or a sum of two copies of $k$.
As before, we write $\lambda\mapsto\ol\lambda$ for the non-trivial Galois automorphism when $K$ is a field.
When $K= k \oplus k$, this notation will mean
$$
	\ol{(x,y)}
	=
	(y,x).
$$
We shall also use the following notation in either case:
$$
	\trace{\lambda}= \lambda +\ol{\lambda},
	\quad
	\norm{\lambda}
	=
	\lambda\ol{\lambda}.
$$
We shall always assume that $\trace{\thet}=0$.
The symbol $(-,-)_{k,2}$ will be the quadratic Hilbert symbol on $k$.
When $K$ is a field, we shall write $(-,-)_{K,2}$ for the quadratic Hilbert symbol on $K$.
In the case that $K=k\oplus k$, this symbol will be defined as follows:
$$
	\hilbk{\spl{x}{y}}{\spl{x'}{y'}}
	=
	\hilk{x}{x'} \cdot \hilk{y}{y'}.
$$

Before describing the cocycle $\sigma$ in general, we first study its restriction to the following maximal torus
$$
	T(k)
	=
	\{\ha{\lambda} \colon
	\lambda\in K^{\times}
	\},
$$
where
$$
	\ha{\lambda}
	=
	\begin{pmatrix}
		\lambda & 0&0\\
		0& \ol{\lambda} / \lambda & 0 \\
		0&0& \ol{\lambda}^{- 1}
	\end{pmatrix}.
$$

\begin{thm2}
For $\lambda,\mu\in K^{\times}$ we have
\begin{multline*}
	\sig{\ha{\lambda}}{\ha{\mu}} \\
	=
	\begin{cases}
		\displaystyle \hilk{\lambda}{\mu},
		&\hbox{if }\lambda,\mu\in k^{\times};\\
		\displaystyle \hilk{\trace{\lambda \thet}}{\mu},
		&\hbox{if }\lambda\notin k^{\times},\mu\in k^{\times};\\
		\displaystyle \hilbk{\lambda}{\mu} \cdot \hilk{\lambda}{\trace{\mu \thet}},
		&\hbox{if }\lambda\in k^{\times},\mu\notin k^{\times};\\
		\displaystyle \hilbk{\lambda}{\mu}
		\cdot \hilk{\lambda \mu}{-\trace{\lambda \thet}},
		&\hbox{if }\lambda\not\in k^{\times},\mu\notin k^{\times},\lambda\mu\in k^{\times};\\
		\displaystyle \hilk{\norm{\lambda}}{\norm{\mu}} \cdot
		\hilk{\trace{\lambda\thet} \norm{\mu}}{\trace{\mu\thet}} \\
		\cdot \hilk{\trace{\lambda\mu\thet}}{-\trace{\lambda\thet}\norm{\mu}\trace{\mu\thet}}, & \hbox{otherwise.}
	\end{cases}
\end{multline*}
\end{thm2}

Along the way, we also find the following formula for the commutator of the cocycle on $T(k)$:
$$
	\frac{\sig{\ha{\lambda}}{\ha{\mu}}}{\sig{\ha{\mu}}{\ha{\lambda}}}
	=
	(\lambda,\ol{\mu})_{K,2}.
$$

\subsection*{The cocycle on the whole group}
Let $N$ be the following unipotent subgroup
of $\SU$:
$$
	N(k)
	=
	\{ \xa{r}{m} \in \SU( k ) \colon
	r,m\in K \hbox{ and } \norm{r}+\trace{m} = 0
	\},
$$
where
$$
	\xa{r}{m} := \begin{pmatrix} 1 & r & m \\ 0 & 1 & - \ol{r} \\ 0 & 0 & 1 \end{pmatrix}.
$$
The section $\delta'\colon\SU(k)\to\wtilde{\SU}(k)$ is chosen in such a way that for $g\in \SU(k)$ and
$n\in N(k)$ we always have
$$
	\delta'(gn) = \delta'(g)\delta'(n)
	\quad\hbox{and}\quad
	\delta'(ng) = \delta'(n)\delta'(g).
$$
As a consequence, our cocycle $\sigma$ satisfies the following:
$$
	\sigma(g,n)=\sigma(n,g)=1,
	\quad
	n\in N(k), \ g \in \SU( k ).
$$
Combining this property with the Bruhat decomposition, we are able to calculate
$\sigma$ on the bigger group $\SU(k)$, at least when $K$ is a field.
For the moment we assume that $K/k$ is a field extension
rather than a sum of two copies of $k$.
We shall discuss how the results must be modified in the split case later.

In order to describe the cocycle, we first introduce some notation.
For $\lambda,\mu\in K^{\times}$ we define
\[
	\ug{\lambda}{\mu}
	=
	\begin{cases}
		\hilk{\lambda}{- \mu}, &\text{if $ \lambda $, $ \mu \in k^{\times} $;} \\
		\hilk{\norm{\lambda}}{- \norm{\mu}}, &\text{otherwise.}
	\end{cases}
\]
We also define a function $\delta_{2} \colon K^{\times} \to k^{\times} \thet$ by
$$
	\del{2}{\lambda}
	=
	\begin{cases}
		\displaystyle \frac{1}{\ol{\lambda} - \lambda}, 
		&
		\hbox{if }\lambda \notin k^{\times};\\
		\thet,
		&
		\hbox{if }\lambda \in k^{\times}.
	\end{cases}
$$
Given an element
\[
	\gamma
	=
	\begin{pmatrix} * & * & * \\ * & * & * \\ g & h & j \end{pmatrix} \in \SU( k ),
 \]
we define $\Xg{}\in K^{\times}$ by
\[
	\Xg{}
	=
	\begin{cases}
		(\ol{g} \thet )^{- 1}, &\text{if $ g \neq 0 $;} \\
		\ol{j}^{- 1}, &\text{if $ g = 0 $.}
	\end{cases}
\]

We prove the following in Chapter~\ref{c:univ}:

\begin{thm2}
Let $ \gamma_{i} \in \SU( k ) $, where $ i = 1 $, $ 2 $, $ 3 $, with $ \gamma_{3} = \gamma_{1} \gamma_{2} $ and
	\[ \gamma_{i} = \begin{pmatrix} * & * & * \\ * & * & * \\ g_{i} & h_{i} & j_{i} \end{pmatrix}. \]
If $ \Xg{3} / ( \Xg{1} \Xg{2} ) \in k^{\times} $, then we have
\[
	\sig{\gamma_{1}}{\gamma_{2}} =
	\begin{aligned}[t]
			&\ug{\frac{\Xg{3}}{\Xg{2}}}{\Xg{1} \Xg{2}} \cdot \hilk{\frac{\del{2}{\Xg{3}}}{\del{2}{\Xg{2}}}}{- \frac{\norm{\Xg{2}} \del{2}{\Xg{2}}}{\del{2}{\Xg{1}}}} \\
			&\phantom{\ } \cdot \hilk{\frac{\Xg{3}}{\Xg{1} \Xg{2}}}{\frac{\del{2}{\Xg{1}} \del{2}{\Xg{2}}}{\del{2}{\Xg{3}} \thet}}.
	\end{aligned} \]
If on the other hand $ \Xg{3} / ( \Xg{1} \Xg{2} ) \notin k^{\times} $, then we let
\[
	r
	=
	r( \gamma_{1}, \gamma_{2} ) = \frac{h_{2} g_{3} - h_{3} g_{2}}{g_{1} \ol{g_{2}}}.
\]
And we have
\[ \sig{\gamma_{1}}{\gamma_{2}} =
		\begin{aligned}[t]
			&\hilk{- \del{2}{- \frac{r \Xg{3}}{\Xg{1} \Xg{2}}} \thet^{- 1}}{\norm{- \frac{r \Xg{3}}{\Xg{1} \Xg{2}}}} \\
			&\phantom{\ } \cdot \hilk{\norm{r}}{\frac{\del{2}{r}}{\thet}} \cdot \ug{\Xg{1}}{\frac{\Xg{3}}{\Xg{2}}} \cdot \ug{\frac{\Xg{3}}{\Xg{2}}}{\Xg{3}} \\
			&\phantom{\ } \cdot \bigg( \frac{\del{2}{\Xg{3} / ( \Xg{1} \Xg{2} )}}{\del{2}{\Xg{3} / \Xg{2}}}, \\
			&\phantom{\del{2}{\Xg{3}}} - \frac{\norm{\Xg{3} / ( \Xg{1} \Xg{2} )} \del{2}{\Xg{3} / ( \Xg{1} \Xg{2} )}}{\del{2}{\Xg{1}}} \bigg)_{k, 2} \\
			&\phantom{\ } \cdot \hilk{\frac{\del{2}{\Xg{2}}}{\del{2}{\Xg{3}}}}{- \frac{\norm{\Xg{2}} \del{2}{\Xg{2}}}{\del{2}{\Xg{3} / \Xg{2}}}}.
		\end{aligned}
\]
\end{thm2}

In fact we obtain a more general theorem describing a cocycle corresponding to an $n$-fold cover
of $\SU(k)$, where $k$ contains a primitive $n$-th root of unity;
however this cocycle is a little more complicated and is not required for our
main aim, which is to produce a half-integral weight multiplier system.

Fortunately, our formula for the local Kubota symbols will be rather simpler than the formula for
$\sigma$. Nevertheless, we require the formula for $\sigma$ in order to calculate the Kubota symbol.

\subsection*{The split case}
In the case $K=k\oplus k$ the theorem above does not completely describe the cocycle $\sigma$.
This is because there are numbers in $K$ which are neither zero nor invertible, and so there are a number
of extra cases to consider.
One can see why this happens from a different point of view: the group $\SU$ has rank $1$ over $l$,
and so there are two cells in the Bruhat decomposition of $\SU(l)$.
However if $K$ is split, then $\SU(k)$ is isomorphic to $\SL_{3}(k)$, which has 6 cells in
its Bruhat decomposition (one for each element of the Weyl group, which in this case is $S_{3}$).
There are therefore four Bruhat cells in $\SU(k)$ which contain no elements of $\SU(l)$. In fact only the biggest
and the smallest Bruhat cells of $\SU(k)$ contain elements of $\SU(l)$.
Our formula for $\sig{\gamma_{1}}{\gamma_{2}}$ is valid whenever $\gamma_{1} $, $ \gamma_{2} $, $ \gamma_{3} = \gamma_{1} \gamma_{2} $
are in one of these two cells.
We may ignore these extra cells since we are interested in the restriction of $\sigma$ to $\SU(l)$.

\subsection*{The level of the multiplier system}

Let $k$ be a non-archimedean local field,
 and assume again that $K/k$ is either a quadratic extension of local fields, or that
 $K$ is a sum of two copies of $k$.
Before we can calculate the local Kubota symbols and the multiplier system,
we must first determine the compact open subgroups $\hat\Gamma_{\mathfrak{p}}$
on which each local extension splits.
These compact open subgroups determine the arithmetic subgroup on which the multiplier system will be defined.
Our result is the following:

\begin{thm2}
	\begin{itemize}
		\item
		If $K/k$ is unramified or split and $k$ has odd residue characteristic
		then the cocycle $\sigma$ splits on $\SU(\Ok)$.
		\item
		Suppose $K/k$ is a ramified field extension and $k$ has odd residue characteristic.
		Let $\mathfrak{P}$ be a prime in $K$.
		Then the cocycle $\sigma$ splits on the subgroup
		$$
			\SU(\Ok,\mathfrak{P})
			=
			\{g \in \SU(\Ok) \colon g \equiv I_{3}\bmod \mathfrak{P} \}.
		$$
		\item
		If $k$ has even residue characteristic and $K=k\oplus k$ then the cocycle splits
		on the subgroup
		$$
			\SU(\Ok,4)
			=
			\{g \in \SU(\Ok) \colon g \equiv I_{3}\bmod 4 \}.
		$$
	\end{itemize}
\end{thm2}

Note that if $L/l$ is a quadratic extension of number fields in which every even prime splits,
then the theorem determines the compact open subgroups
$\hat\Gamma_{\mathfrak{p}} \subset \SU(l_{\mathfrak{p}})$
at all primes $\mathfrak{p}$.

\subsection*{The local Kubota symbol}
Let $ \mathfrak{p} $ be the maximal ideal of $ k $.  Recall that the local Kubota symbol $ \kappa_{\mathfrak{p}} $ is a map $\hat\Gamma_{\mathfrak{p}}\to \mu_{2}$,
satisfying the following:
$$
	\sigma(g,h)=\frac{\kup{g}\kup{h}}{\kup{gh}}.
$$
This condition does not always determine the local Kubota symbol, since we may always
multiply by a character of $\hat\Gamma_{\mathfrak{p}}$.
We therefore let $\Gamma_{\mathfrak{p}}$ be the intersection of the kernels
of the homomorphisms  $\hat\Gamma_{\mathfrak{p}}\to\mu_{2}$.
The restriction of the Kubota symbol to $\Gamma_{\mathfrak{p}}$ is unique,
and we shall only calculate this restriction.
Fortunately we have not lost much, since for all odd primes we have
$\Gamma_{\mathfrak{p}}=\hat\Gamma_{\mathfrak{p}}$.
Suppose $\mathfrak{p}$ is even.
If we assume (as is the case for split primes) that the cocycle splits at level $4$,
then we will have
$$
	\Gamma_{\mathfrak{p}}
	=
	\SU(\mathcal{O}_{k},8) = \{g \in \SU(\Ok) \colon g \equiv I_{3}\bmod 8 \}.
$$
We will also assume that when $ K / k $ is ramified, $ \thet $ is a prime element of $ K $.

Recall that $N$ denotes a unipotent subgroup of $\SU$ described above.
Our first observation is the following:

\begin{prop2}
For any $n \in \Gamma_{\mathfrak{p}} \cap N(k)$ we have
$\kup{n} = 1$.
More generally, for any $g\in\Gamma_{\mathfrak{p}}$ we have
$$
	\kup{ng} =\kup{gn} = \kup{g}.
$$
\end{prop2}

The relation $\kup{ng} = \kup{g}$ implies that $\kup{g}$ is determined by
the bottom row of the matrix $g$. The other relation shows that $\kup{g}$
is unchanged by certain column operations.
To describe our next result we need a little more notation.
Let
\begin{gather*}
	\xa{r}{m} := \begin{pmatrix} 1 & r & m \\ 0 & 1 & - \ol{r} \\ 0 & 0 & 1 \end{pmatrix},
	\quad
	\xam{r}{m} := \begin{pmatrix} 1 & 0 & 0 \\ r & 1 & 0 \\ m & - \ol{r} & 1 \end{pmatrix}
\end{gather*}
be elements of $ \SU( k ) $.
This entails $ r $, $ m \in K $ and $ \trace{m} = - \norm{r} $.
For a number $\lambda\in K^{\times}$ we shall write
$$
	\kxam{\lambda}
	=
	\begin{cases}
		\displaystyle \hilk{-\trace{\lambda}}{\norm{\lambda\thet}},
		& \hbox{if }\trace{\lambda} \ne 0;\\
		\displaystyle 1,
		& \hbox{otherwise.}
	\end{cases}
$$

\begin{prop2}
Let $ \xam{s_{1}}{n_{1}} \in \Gamma_{\mathfrak{p}}$ with $s_{1} $, $ n_{1} \in L$.  Then we have
$$
	\kup{\xam{s_{1}}{n_{1}}}
	=
	\kxam{s_{1}} \cdot \kxam{-\frac{s_{1}\thet}{n_{1}}}.
$$
(If $n_{1}=0$ then $s_{1}$ must also be zero, and $\kup{\xam{s_{1}}{n_{1}}}=1$.)
\end{prop2}


We next consider elements of the maximal torus $T(k)$.
Our results for such elements are as follows:

\begin{prop2}
If $ \mathfrak{p}$ is odd and unramified (either inert or split), then for $ \lambda \in \OK^{\times} $,
 we have
\[
	\kup{\ha{\lambda}}
	=
	\begin{cases} \displaystyle \hilk{a}{b}, &\text{if $ \lambda = a + b \thet $, $ a $, $ b \neq 0 $ and $ b \notin \Ok^{\times} $;} \\
	\displaystyle 1, & \text{otherwise.} \end{cases}
\]
If $\mathfrak{p}$ is even or ramified in $K$,
 then for $ \ha{\lambda} \in T(k)\cap \Gamma_{\mathfrak{p}}$ we have
\[
	\kup{\ha{\lambda}} = 1.
\]
\end{prop2}

We next obtain an expression for the Kubota symbol, expressed in terms of the special cases
already described. The following theorem is proven in Section~\ref{s:other}:

\begin{thm2}
Let
\[
	\gamma
	=
	\begin{pmatrix} * & * & * \\ * & * & * \\ g & h & j \end{pmatrix}
	\in
	\Gamma_{\mathfrak{p}}.
\]
Then
\[
	\kup{\gamma} \\
	=
	\begin{cases}
			\displaystyle \kup{\ha{\ol{j}^{-1}}}, &\text{if $ g = 0 $;} \\
			\displaystyle \kup{\ha{( \ol{g} \thet )^{- 1}}}, &\text{if $ g \in \OK^{\times} $;} \\
			\begin{aligned}[b]
				&\kup{\ha{\ol{j}^{- 1}}} \cdot \kup{\xam{- \frac{\ol{h}}{\ol{j}}}{\frac{g}{j}}} \\
				&\phantom{\ } \cdot \sig{\ha{{\ol{j}^{- 1}}}}{\ha{\frac{\ol{j}}{\ol{g} \thet}}},
			\end{aligned}
			&\text{if $ g \neq 0 $, $ g \notin \OK^{\times} $ and $j\in \OK^{\times}$.}
	\end{cases}
\]
\end{thm2}

Again, note that if $\mathfrak{p}$ is split, then we have not covered all possibilities
since it is possible for neither $g$ nor $j$ to be a unit in this case (and only in this case).
However, note that if $g$ is not a unit then there is always an element
$\xa{s_{1}}{n_{1}}\in N \cap \Gamma_{\mathfrak{p}}$ such that
$$
	\begin{pmatrix}g&h&j\end{pmatrix} \cdot
	\xa{s_{1}}{n_{1}}
	=
	\begin{pmatrix}g&h'&j'\end{pmatrix},
$$
where $j'$ is a unit, and we will always have
$$
	\kup{\gamma}
	=
	\kup{\gamma \cdot \xa{s_{1}}{n_{1}} }.
$$
We may apply the theorem to calculate $\kup{\gamma  \cdot \xa{s_{1}}{n_{1}} }$.

\subsection*{The section over the real points, and the half-integral weight multiplier system}

As was described above, we shall choose for each element $g\in \SU(\mathbb{R})$ a continuous
square root $\phi_{g}(\tau)$ of the function $\tau\mapsto C\tau+D$, satisfying the condition
$$
	\phi_{g}(g' ( \tau ) )\phi_{g'}(\tau) = \sigma_{\infty}(g,g')\phi_{gg'}(\tau).
$$
In fact there is only one such choice, since any two choices would differ by a homomorphism
$\SU(\mathbb{R})\to\mu_{2}$ and $\SU(\mathbb{R})$ is generated by commutators.
In Chapter~\ref{c:real} we determine the signs of these square roots.
Our result is:

\begin{thm2}
Let $n_{1} $, $ n_{2} \in N(\mathbb{R})$, $h=\ha{\lambda}\in T(\mathbb{R})$ and let
$w = \begin{pmatrix} 0&0&i \\ 0&1&0 \\ i & 0&0 \end{pmatrix}$.
The assignment $g\mapsto \phi_{g}$ defined above is given by:
\begin{itemize}
	\item $ \phi_{h \cdot n_{1}}( \tau ) = \ol{\lambda}^{- 1 / 2} $, where $ \arg\left( \ol{\lambda}^{- 1 / 2} \right) \in ( - \pi / 2, \pi / 2 ] $;
	\item $ \arg( \phi_{w}( \tau ) ) \in ( - \pi / 2, 0 ) $;
	\item $ \phi_{n_{2} \cdot w \cdot h \cdot n_{1}}( \tau ) = \phi_{w}( ( h \cdot n_{1} ) ( \tau ) ) \phi_{h \cdot n_{1}}(\tau ) $.
\end{itemize}
\end{thm2}

In particular, this means that we always have $\arg(\phi_{g}(\tau))\in (-\pi/2,\pi/2]$.
As a consequence, we have the following:

\begin{thm2}
Suppose $\mathbb{Q}(\thet)$ is a quadratic extension in which the prime $2$ splits.
Define, for $\gamma\in \SU(\mathbb{Z},8\thet)$,
	\[ j( \gamma, \tau ) =  \phi_{\gamma}( \tau ) \prod_{p < \infty} \kappa_{p}(\gamma ). \]
Then $j(\gamma,\tau)$ is a multiplier system of weight $1/2$.
\end{thm2}

\section*{Verification of the results}

The thesis contains rather a lot of calculations, and it would be useful to know that the results
are genuinely correct, rather than perhaps being out by a sign here and there.
To give some evidence of this, we can look at the restriction of the global Kubota symbol to some subgroups
to check that it has the expected properties.

\subsection*{Restriction to the torus}

We examine first the restriction of the Kubota symbol to $\Gamma\cap T(l)$,
where $\Gamma$ is the level $8\thet$ principal congruence subgroup described above.
This intersection consists of elements $\ha{\lambda}$ where $\lambda$ is a unit in
$\mathcal{O}_{L}$ and is congruent to $1$ modulo $8\thet$.

Recall that for elements $a,b\in\mathcal{O}_{l}$ with $b$ coprime to $2a$,
the quadratic Legendre symbol is defined by
$$
	\left(\frac{a}{b}\right)_{l,2}
	=
	\prod_{\mathfrak{p} | b} (a,b)_{\lp,2}.
$$
Our results imply the following:

\begin{prop2}
	Let $\lambda=a+b\thet$ be a unit in $L$ congruent to $1$ modulo $8\thet$.
	Then we have
	$$
		\kappa(\ha{\lambda})
		=
		\begin{cases}
			\displaystyle \left(\frac{b}{a}\right)_{l,2} \cdot\prod_{\mathfrak{p}|\infty} ( a, b )_{\lp, 2}, &\hbox{if }b \ne 0;\\
			\displaystyle 1, &\hbox{otherwise.}
		\end{cases}
	$$
\end{prop2}

In the case that $l$ is totally complex, this implies that the map
$a+b\thet \mapsto \left(\frac{b}{a}\right)_{l,2}$ is a group homomorphism.
This is indeed the case, and can be verified directly using the quadratic reciprocity law in $k$.

\subsection*{Restriction to $\SL_{2}$}

The group $\SL_{2}$ embeds into $\SU$ as follows:
$$
	\begin{pmatrix} a & b \\ c & d\end{pmatrix}
	\mapsto
	\begin{pmatrix}
		a&0& b \thet\\
		0&1&0\\
		c/\thet&0&d
	\end{pmatrix}.
$$
We may therefore examine the restriction of the Kubota symbol to $\SL_{2}(\mathcal{O}_{l},8\thet^{2})$.
Our results imply the following:
$$
	\kappa\left( \begin{pmatrix} a & b \\ c & d\end{pmatrix} \right)
	=
	\begin{cases}
		\displaystyle \left(\frac{c}{d}\right)_{l,2}, & c\ne 0;\\
		\displaystyle 1, & c=0.
	\end{cases}
$$
When $l$ is totally complex, our results imply that this map is a homomorphism.
Again, this turns out to be true, as was shown by Kubota (see \cite{kubota69}).


\part{Preliminaries} \label{p:prel}


\chapter{The group \SU} \label{c:su}

In this chapter, we will outline the definition of the group \SU{} that we will use along with some of its subgroups, the ad\`{e}le group, the Bruhat decomposition of \SU{} and the Iwahori factorisation.

\section{The structure of \SU} \label{s:sustruc}

Let $k$ be an arbitrary field of characteristic zero, and $ K / k $ is a quadratic extension where $ K = k( \thet ) $, $ \thet = \sqrt{- d} $, $ d \in k^{\times} $.  Then the Galois group $ \gal( K / k ) $ has two elements, and the non-trivial element may be described by
	\[ \ol{a + b \thet} \mapsto a - b \thet. \]

To describe \SU{}, suppose $ A $ is a $ k $-algebra.  Then
	\[ \SU( A ) = \{ \nu \in \SL_{3}( A \otimes_{k} K ) \colon \nu^{t} J \overline{\nu} = J \}, \]
where
	\[ J = \begin{pmatrix}  1 &  0 &   0  \\
                            0 &  1 &   0  \\
                            0 &  0 & - 1  \end{pmatrix}, \]
and $ \nu^{t} $ denotes the transpose of a matrix $ \nu $.  This is the ``usual'' definition of \SU{}, but there are other definitions which are isomorphic to the above.  In fact, we will work with another definition, where $ J $ is replaced by
	\[ J' = \begin{pmatrix} 0 & 0 & 1  \\
                            0 & 1 & 0  \\
                            1 & 0 & 0  \end{pmatrix}, \]
as it is a more convenient presentation of \SU{}.  (It is possible to show that if
	\[ V = \begin{pmatrix} 1 & 0 & - 1 \\ 1 & 1 & - 1 \\ 0 & - 1 & 1 \end{pmatrix}, \]
then $ J = V^{t} J' \ol{V} $.)

Let $ G = \SU $.  We consider 
\begin{equation} \label{e:gk}
	G( k ) = \{ \nu \in \SL_{3}( k \otimes_{k} K ) \colon \nu^{t} J' \ol{\nu} = J' \}.
\end{equation}
Let $ S $ be a maximal $ k $-split torus of $ G $, with $ T $ a maximal torus of  $ G $ containing $ S $.  By Section~2.5 of \cite{deodhar78}, we may choose these as follows:
	\[ S( k ) := \left\{ \begin{pmatrix} t & 0 & 0 \\ 0 & 1 & 0 \\ 0 & 0 & t^{-1} \end{pmatrix} \colon t \in k^{\times} \right\} \]
as a maximal $ k $-split torus of $ G( k ) $, and
	\[ T( k ) := \left\{ \begin{pmatrix} \lambda & 0 & 0 \\ 0 & \ol{\lambda} / \lambda & 0 \\ 0 & 0 & \ol{\lambda}^{-1} \end{pmatrix} \colon \lambda \in K^{\times} \right\} \]
as a maximal torus of $ G( k ) $.  As $ T( k ) \cong K^{\times} $, we will denote an element in $ T( k ) $ by
	\[ \ha{\lambda} = \begin{pmatrix} \lambda & 0 & 0 \\ 0 & \ol{\lambda} / \lambda & 0 \\ 0 & 0 & \ol{\lambda}^{-1} \end{pmatrix}, \]
where $ \lambda \in K^{\times} $, and $ \alpha $ is defined below.

The root system of $ G $ with respect to $ S $, $ \Phi $, consists of 4 roots with one simple root which we will call $ \alpha $.  Thus, $ \Phi = \{ \alpha, 2 \alpha, - \alpha, - 2 \alpha \} $ and $ \alpha $ may be described by
	\[ \alpha( \ha{t} ) = t, \]
where $ t \in k^{\times} $.  This implies that if we let $ \mathfrak{g} $ be the Lie algebra of $ G $, and for  $ \beta \in \Phi $, let
	\[ \mathfrak{g}_{\beta} = \{ X \in \mathfrak{g} \colon ( \Ad s )( X ) = \beta( s ) \cdot X \  \forall s \in S \} \]
be the corresponding root space, then
\begin{gather*}
	\mathfrak{g}_{0} = \left\{ \begin{pmatrix} a & 0 & 0 \\ 0 & \ol{a} - a & 0 \\ 0 & 0 & - \ol{a} \end{pmatrix} \colon a \in K \right\}, \\
	\mathfrak{g}_{\alpha} = \left\{ \begin{pmatrix} 0 & b & 0 \\ 0 & 0 & - \ol{b} \\ 0 & 0 & 0 \end{pmatrix} \colon b \in K \right\}, \\
	\mathfrak{g}_{2 \alpha} = \left\{ \begin{pmatrix} 0 & 0 & t \thet \\ 0 & 0 & 0 \\ 0 & 0 & 0 \end{pmatrix} \colon t \in k \right\}, \\
	\mathfrak{g}_{- \alpha} = \left\{ \begin{pmatrix} 0 & 0 & 0 \\ b & 0 & 0 \\ 0 & - \ol{b} & 0 \end{pmatrix} \colon b \in K \right\}, \\
	\mathfrak{g}_{- 2 \alpha} = \left\{ \begin{pmatrix} 0 & 0 & 0 \\ 0 & 0 & 0 \\ t \thet & 0 & 0 \end{pmatrix} \colon t \in k \right\}.
\end{gather*}
Thus, the root space decomposition of $ \mathfrak{g} $ is
	\[ \mathfrak{g} = \mathfrak{g}_{0} \oplus \bigoplus_{\beta \in \Phi} \mathfrak{g}_{\beta}. \]

\section{The Bruhat decomposition of \SU} \label{s:bruh}

Recall that $ G = \SU $.  We shall use the following system of positive roots: $ \Phi^{+} = \{ \alpha, 2 \alpha \} $.  Let $ N $, denoted as $ U^{+} $ in \cite{deodhar78}, be the unipotent algebraic subgroup of $ G $ whose Lie algebra is $ \bigoplus_{\beta \in \Phi^{+}} \mathfrak{g}_{\beta} $ (and similarly $ \ol{N} $, denoted by $ U^{-} $ in \cite{deodhar78}, is the unipotent algebraic subgroup of $ G $ whose Lie algebra is $ \bigoplus_{- \beta \in \Phi^{+}} \mathfrak{g}_{\beta} $).  Hence, since $ G( k ) = \SU( k ) $, $ N( k ) $ and $ \ol{N}( k ) $ may be explicitly described, i.e.
\begin{gather*}
	N( k ) = \left\{ \begin{pmatrix} 1 & r & m \\ 0 & 1 & - \ol{r} \\ 0 & 0 & 1 \end{pmatrix} \colon ( r, m ) \in K \times K, \; \trace{m} = - \norm{r} \right\}, \\
	\ol{N}( k ) = \left\{ \begin{pmatrix} 1 & 0 & 0 \\ r & 1 & 0 \\ m & - \ol{r} & 1 \end{pmatrix} \colon ( r, m ) \in K \times K, \; \trace{m} = - \norm{r} \right\},
\end{gather*}
where $ \norm{s} = s \ol{s} $ and $ \trace{s} = s + \ol{s} $ are the norm and trace of $ s \in K $ over $ k $.  We will let
	\[ \xa{r}{m} := \begin{pmatrix} 1 & r & m \\ 0 & 1 & - \ol{r} \\ 0 & 0 & 1 \end{pmatrix}, \quad \xam{r}{m} := \begin{pmatrix} 1 & 0 & 0 \\ r & 1 & 0 \\ m & - \ol{r} & 1 \end{pmatrix}, \]
where $ r $, $ m \in K $ and $ \trace{m} = - \norm{r} $.

By Proposition~2.7 of \cite{deodhar78}, if we define
\begin{equation} \label{e:warm}
	\wa{r}{m} = \xa{r}{m} \cdot \xam{\frac{\ol{r}}{\ol{m}}}{\frac{1}{\ol{m}}} \cdot \xa{r \cdot \frac{\ol{m}}{m}}{m},
\end{equation}
then
	\[ \wa{r}{m} \cdot N( k ) \cdot \wa{r}{m}^{-1} = \ol{N}( k ). \]
Thus by the above definitions for \xa{r}{m} and \xam{r}{m},
	\[ \wa{r}{m} = \begin{pmatrix} 0 & 0 & m \\ 0 & - \ol{m} / m & 0 \\ \ol{m}^{-1} & 0 & 0 \end{pmatrix}. \]
(This would imply for any $ m \in K $ such that for $ r $, $ r' \in K $, $ \trace{m} = - \norm{r} = - \norm{r'} $, $ \wa{r}{m} = \wa{r'}{m} $.)

If $ N_{G}( S ) $ is the normaliser of $ S $ in $ G $, and $ Z_{G}( S ) $ is the centraliser of $ S $ in $ G $, we define $ W_{0} = N_{G}( S ) / Z_{G}( S ) $ as the Weyl group in $ G $.  Thus, we may choose (as we need this to define the section for the 2-cocycle in Section~\ref{s:sec})
	\[ W = \left\{ 1, \wa{0}{\thet} \right\} \]
as a complete set of representatives for the Weyl group in $ G( k ) $ (this is the same Weyl group chosen by Deodhar in Section~2.21 in \cite{deodhar78}).  Thus, by V.21.29 of \cite{borel91}, the Bruhat decomposition may be described as
	\[ G( k ) = N( k ) \cdot T( k ) \sqcup N( k ) \cdot T( k ) \cdot \wa{0}{\thet} \cdot N( k ). \]

This implies that a matrix in $ G( k ) $ is either upper triangular or has a non-zero $ ( 3, 1 ) $-entry.  It can be easily shown that for $ a $, $ b $, $ c $, $ d $, $ e \in K $ such that
	\[ a \ol{c} + \ol{a} c = - \norm{b},\ e \ol{c} + \ol{e} c = - \norm{d}, \]
the Bruhat decomposition of any matrix of $ G( k ) $ with a non-zero $ ( 3, 1 ) $-entry (i.e. $ c \neq 0 $) may be described as
\begin{equation} \label{e:bruc}
	\begin{pmatrix}
		a & * & * \\ b & * & * \\ c & d & e
	\end{pmatrix} = \xa{- \frac{\ol{b}}{\ol{c}}}{\frac{a}{c}} \cdot \ha{\frac{1}{\ol{c} \thet}} \cdot \wa{0}{\thet} \cdot \xa{\frac{d}{c}}{\frac{e}{c}}.
\end{equation}
Otherwise, an upper triangular matrix (i.e. an element of the Borel subgroup of $ G $) will have the Bruhat decomposition
\begin{equation} \label{e:bruf}
	\begin{pmatrix}
		f & g & h \\ 0 & \ol{f} / f & - \ol{g} / f \\ 0 & 0 & \ol{f}^{-1}
	\end{pmatrix} = \ha{f} \cdot \xa{\frac{g}{f}}{\frac{h}{f}} = \xa{\frac{g f}{\ol{f}}}{h \ol{f}} \cdot \ha{f}
\end{equation}
where $ h $, $ f $, $ g \in K $ with $ h \ol{f} + \ol{h} f = - \norm{g} $.

\section{The Iwahori factorisation} \label{s:iwahori}

In this section, we shall assume that $ k $ is a non-archimedean local field.  We shall write $ \Ok $ for the valuation ring in $ k $.  We shall use the notation
\begin{equation} \label{e:gok}
	G( \Ok ) = \{ \nu \in \SL_{3}( \Ok \otimes_{\Ok} \OK ) \colon \nu^{t} J' \ol{\nu} = J' \},
\end{equation}
where $ J' $ is as defined in Section~\ref{s:sustruc}.  Let $ \mathfrak{a} $ be an ideal of \OK, and let
\begin{gather*}
	G( \Ok )_{0}( \mathfrak{a} ) = \left\{ \begin{pmatrix} a & b & c \\ d & e & f \\ g & h & j \end{pmatrix} \in G( \Ok ) \colon d \equiv g \equiv h \equiv 0 \pod{\mathfrak{a}} \right\}, \\
	G( \Ok )_{1}( \mathfrak{a} ) = \left\{ \begin{pmatrix} a & b & c \\ d & e & f \\ g & h & j \end{pmatrix} \in G( \Ok ) \colon \begin{aligned} &b \equiv c \equiv d \equiv f \equiv g \equiv h \equiv 0 \pod{\mathfrak{a}}, \\ &a \equiv e \equiv j \equiv 1 \pod{\mathfrak{a}} \end{aligned} \right\}.
\end{gather*}

We may define
\begin{gather*}
	T( \Ok ) = T( k ) \cap G( \Ok ), \\
	N( \Ok ) = N( k ) \cap G( \Ok ), \\
	\ol{N}( \Ok ) = \ol{N}( k ) \cap G( \Ok ). \\
\end{gather*}
Let us also define
\begin{gather*}
	T( \mathfrak{a} ) = \left\{ \begin{pmatrix} \lambda & 0 & 0 \\ 0 & \ol{\lambda} / \lambda & 0 \\ 0 & 0 & \ol{\lambda}^{- 1} \end{pmatrix} \in T( \Ok ) \colon \lambda \equiv 1 \pod{\mathfrak{a}} \right\}, \\ 
	N( \mathfrak{a} ) = \left\{ \begin{pmatrix} 1 & r & m \\ 0 & 1 & - \ol{r} \\ 0 & 0 & 1 \end{pmatrix} \in N( \Ok ) \colon r \equiv m \equiv 0 \pod{\mathfrak{a}} \right\}, \\
	\ol{N}( \mathfrak{a} ) = \left\{ \begin{pmatrix} 1 & 0 & 0 \\ r & 1 & 0 \\ m & - \ol{r} & 1 \end{pmatrix} \in \ol{N}( \Ok ) \colon r \equiv m \equiv 0 \pod{\mathfrak{a}} \right\}.
\end{gather*}

\begin{prop} \label{p:iwahori}
We have the Iwahori factorisations
\begin{gather*}
	G( \Ok )_{0}( \mathfrak{a} ) = N( \Ok ) \cdot T( \Ok ) \cdot \ol{N}( \mathfrak{a} ), \\
	G( \Ok )_{1}( \mathfrak{a} ) = N( \mathfrak{a} ) \cdot T( \mathfrak{a} ) \cdot \ol{N}( \mathfrak{a} ).
\end{gather*}
\end{prop}
\begin{proof}
Let
	\[ \begin{pmatrix} a & b & c \\ d & e & f \\ g & h & j \end{pmatrix} \in G( \Ok )_{0}( \mathfrak{a} ) \]
(resp. $ G( \Ok )_{1}( \mathfrak{a} ) $), where $ g \neq 0 $.  Consider the Hermitian form $ \langle - , - \rangle $ defined by
	\[ \langle u, v \rangle = u^{t} \begin{pmatrix} 0 & 0 & 1 \\ 0 & 1 & 0 \\ 1 & 0 & 0 \end{pmatrix} \ol{v}, \]
where $ u $, $ v \in K^{3} $.  Since
	\[ \left\langle \begin{pmatrix} c \\ f \\ j \end{pmatrix}, \begin{pmatrix} c \\ f \\ j \end{pmatrix} \right\rangle = 0, \] 
this implies that
	\[ \trace{\frac{c}{j}} = - \norm{\frac{f}{j}}. \]
Thus, $ \xa{\ol{f} / \ol{j}}{\ol{c} / \ol{j}} \in N( \Ok ) $ (resp. $ N ( \mathfrak{a} ) $), and hence
	\[ \xa{\frac{\ol{f}}{\ol{j}}}{\frac{\ol{c}}{\ol{j}}} \cdot \begin{pmatrix} a & b & c \\ d & e & f \\ g & h & j \end{pmatrix} = \begin{pmatrix} a + d \ol{f} / \ol{j} + \ol{c} g / \ol{j} & b + e \ol{f} / \ol{j} + \ol{c} h / \ol{j} & 0 \\ d - f g / j & e - f h / j & 0 \\ g & h & j \end{pmatrix}. \]
But
	\[ \left\langle \begin{pmatrix} a + d \ol{f} / \ol{j} + \ol{c} g / \ol{j} \\ b + e \ol{f} / \ol{j} + \ol{c} h / \ol{j} \\ 0 \end{pmatrix}, \begin{pmatrix} a + d \ol{f} / \ol{j} + \ol{c} g / \ol{j} \\ b + e \ol{f} / \ol{j} + \ol{c} h / \ol{j} \\ 0 \end{pmatrix} \right\rangle = 0, \]
which implies that $ b + e \ol{f} / \ol{j} + \ol{c} h / \ol{j} = 0 $, i.e.
	\[ \xa{\frac{\ol{f}}{\ol{j}}}{\frac{\ol{c}}{\ol{j}}} \cdot \begin{pmatrix} a & b & c \\ d & e & f \\ g & h & j \end{pmatrix} = \begin{pmatrix} a + d \ol{f} / \ol{j} + \ol{c} g / \ol{j} & 0 & 0 \\ d - f g / j & e - f h / j & 0 \\ g & h & j \end{pmatrix}. \]
Since $ j \in \Ok^{\times} $ and $ \ha{\ol{j}} \in T( \Ok ) $ (resp. $ T( \mathfrak{a} ) $), we have
	\[ \ha{\ol{j}} \cdot \xa{\frac{\ol{f}}{\ol{j}}}{\frac{\ol{c}}{\ol{j}}} \cdot \begin{pmatrix} a & b & c \\ d & e & f \\ g & h & j \end{pmatrix} = \begin{pmatrix} a \ol{j} + d \ol{f} + \ol{c} g & 0 & 0 \\ d j / \ol{j} - f g / \ol{j} & e j / \ol{j} - f h / \ol{j} & 0 \\ g / j & h / j & 1 \end{pmatrix}. \]
But
	\[ \left\langle \begin{pmatrix} a \\ d \\ g \end{pmatrix}, \begin{pmatrix} c \\ f \\ g \end{pmatrix} \right\rangle = 1, \]
i.e. $ a \ol{j} + d \ol{f} + \ol{c} g = 1 $.  This implies that $ e j / \ol{j} - f h / \ol{j} = 1 $, so that
	\[ \ha{\ol{j}} \cdot \xa{\frac{\ol{f}}{\ol{j}}}{\frac{\ol{c}}{\ol{j}}} \cdot \begin{pmatrix} a & b & c \\ d & e & f \\ g & h & j \end{pmatrix} = \begin{pmatrix} 1 & 0 & 0 \\ d j / \ol{j} - f g / \ol{j} & 1 & 0 \\ g / j & h / j & 1 \end{pmatrix}. \]
As
	\[ \left\langle \begin{pmatrix} 1 \\ d j / \ol{j} - f g / \ol{j} \\ g / j \end{pmatrix}, \begin{pmatrix} 0 \\ 1 \\ h / j \end{pmatrix} \right\rangle = 0, \]
this implies that $ d j / \ol{j} - f g / \ol{j} = - \ol{h} / \ol{j} $.  Thus,
	\[ \ha{\ol{j}} \cdot \xa{\frac{\ol{f}}{\ol{j}}}{\frac{\ol{c}}{\ol{j}}} \cdot \begin{pmatrix} a & b & c \\ d & e & f \\ g & h & j \end{pmatrix} = \begin{pmatrix} 1 & 0 & 0 \\ - \ol{h} / \ol{j} & 1 & 0 \\ g / j & h / j & 1 \end{pmatrix} = \xam{- \frac{\ol{h}}{\ol{j}}}{\frac{g}{j}}. \]
Note that $ \xam{- \ol{h} / \ol{j}}{g / j} \in \ol{N}( \Ok ) $ (resp. $ \ol{N}( \mathfrak{a} ) $), so we have in the end
	\[ \begin{pmatrix} a & b & c \\ d & e & f \\ g & h & j \end{pmatrix} = \xa{- \frac{\ol{f}}{\ol{j}}}{\frac{c}{j}} \cdot \ha{\ol{j}^{- 1}} \cdot \xam{- \frac{\ol{h}}{\ol{j}}}{\frac{g}{j}}. \]

If $ g = 0 $, then we have, by \eqref{e:bruf},
	\[ \begin{pmatrix} \ol{j}^{- 1} & - \ol{f} / j & c \\ 0 & \ol{j} / j & f \\ 0 & 0 & j \end{pmatrix} = \xa{- \frac{\ol{f}}{\ol{j}}}{\frac{c}{j}} \cdot \ha{\ol{j}^{- 1}}, \]
where $ \xa{b j / \ol{j}}{c / j} \in N( \Ok ) $ (resp. $ N( \mathfrak{a} ) $) and $ \ha{\ol{j}^{- 1}} \in T( \Ok ) $ (resp. $ T( \mathfrak{a} ) $).

This completes the proof of the Iwahori factorisation.
\end{proof}

\section{The ad\`{e}le group of \SU} \label{s:adele}

\subsection{Some notation for the local field} \label{ss:local}

Using the notation of Section~V.1 of \cite{neukirch99}, we first let $ v_{k} $ be the discrete valuation normalised by $ v_{k}( k^{\times} ) = \mathbb{Z} $, for any local field $ k $.  This implies that the valuation ring may be described by
	\[ \mathcal{O}_{k} = \{ b \in k \colon v_{k}( b ) \geq 0 \}, \]
with maximal ideal
	\[ \mathfrak{p} = \mathfrak{p}_{k} = \{ b \in k \colon v_{k}( b ) > 0 \}. \]
By defining $ q = | \mathcal{O}_{k} / \mathfrak{p}_{k} | $, we have the normalised $ \mathfrak{p} $-adic absolute value (multiplicative valuation)
	\[ | b |_{\mathfrak{p}} = q^{- v_{k}( b )}, \]
where $ b \in k $.  This implies that
	\[ \mathcal{O}_{k} = \{ b \in k \colon | b |_{\mathfrak{p}} \leq 1 \}, \]
and
	\[ \mathfrak{p}_{k} = \{ b \in k \colon | b |_{\mathfrak{p}} < 1 \}. \]
We define a \emph{prime element} of $ k $, $ \pi = \pi_{k} $, such that $ v_{k}( \pi ) = 1 $.  This implies that $ \mathfrak{p}_{k} = \pi \mathcal{O}_{k} $.

We also have the following lemma from Section~11 of \cite{cassels67}:

\begin{lem} \label{l:extval}
Let $ k $ be complete with respect to the normalised valuation $ | \cdot | $ and let $ K $ be an extension of $ k $ of degree $ [ K : k ] = N < \infty $.  Then the normalised valuation $ \| \cdot \| $ of $ K $ which is equivalent to the unique extension of $ | \cdot | $ to $ K $ is given by the formula
	\[ \| b \| = | \normed_{K / k}( b ) |, \]
where $ b \in K $ and $ \normed_{K / k} $ is the norm of an element of $ K $ over $ k $.
\end{lem}

Lemma~\ref{l:extval} will be useful as we can work out the normalised valuation of $ K $ by only knowing what the normalised valuation of $ k $ is.

\subsection{The ad\`{e}le ring} \label{ss:adelering}

Now let $ l $ be any global field.  Recall (see Section VI.1 of \cite{neukirch99}) that an ad\`{e}le is a family
	\[ a = ( a_{\mathfrak{p}} ), \]
of elements $ a_{\mathfrak{p}} \in \lp $, where $ \mathfrak{p} $ runs through all the primes of $ l $, $ \lp $ is the completion of $ l $ with respect to $ \mathfrak{p} $, and $ a_{\mathfrak{p}} $ is integral in $ \lp $ for almost all $ \mathfrak{p} $.  (Note that $ \mathfrak{p} $ is thus the maximal ideal of $ \Olp $, as defined in the previous subsection.)  The ad\`{e}les form a ring
	\[ \mathbb{A}_{l} = \sideset{}{'}{\prod}_{\mathfrak{p}} \lp, \]
where $ \mathbb{A}_{l} $ is the \emph{restricted product} of the $ \lp $ with respect to the subrings $ \Olp \subseteq \lp $.  Addition and multiplication are defined componentwise.  Let $ \mathbb{A}_{l_{f}} $ denote the finite component of $ \mathbb{A}_{l} $, i.e. let
	\[ \mathbb{A}_{{l}_{f}} = \sideset{}{'}{\prod}_{\text{$ \mathfrak{p} $ finite}} \lp, \]
and we similarly define the infinite component of $ \mathbb{A}_{l} $ as
	\[ \mathbb{A}_{{l}_{\infty}} = \sideset{}{'}{\prod}_{\text{$ \mathfrak{p} $ infinite}} \lp. \]

As stated in Section~10 of \cite{cassels67}, there is a natural mapping
\begin{align*}
	l &\to \mathbb{A}_{l} \\
	b &\mapsto ( b ),
\end{align*}
i.e. an injective map of $ l $ into $ \mathbb{A}_{l} $ since $ b \in \Olp $ for almost all $ \mathfrak{p} $ and the map of $ l $ into any $ \lp $ is an injection.  The image of $ l $ under this map is the ring of principal ad\`{e}les, and we can identify $ l $ with this ring.  Hence $ l $ is a subring of $ \mathbb{A}_{l} $.

\subsection{A description of the ad\`{e}le group $ \SU( \mathbb{A}_{l} ) $} \label{ss:descad}

Now let $ L $ be a quadratic extension of our global field $ l $.  We are interested in
	\[ G( \mathbb{A}_{l} ) = \{ \nu \in \SL_{3}( \mathbb{A}_{l} \otimes_{l} L ) \colon \nu^{t} J' \ol{\nu} = J' \}, \]
where $ J' $ is as defined in Section~\ref{s:sustruc}.  We want to calculate the global Kubota symbol on an arithmetic subgroup of $ G( \mathbb{A}_{{l}_{f}} ) $.  In order to do so, we will need to calculate the local Kubota symbol on a compact open subgroup of
	\[ G( \lp ) = \{ \nu \in \SL_{3}( \lp \otimes_{l} L ) \colon \nu^{t} J' \ol{\nu} = J' \}, \]
where $ \mathfrak{p} $ is finite.  Thus, we want to calculate the Kubota symbol on a subgroup of
	\[ G( \Olp ) = \{ \nu \in \SL_{3}( \Olp \otimes_{\mathcal{O}_{l}} \mathcal{O}_{L} ) \colon \nu^{t} J' \ol{\nu} = J' \}. \]

We should first describe $ \Lp := \lp \otimes_{l} L $ in order to understand $ G( \lp ) $.  Let $ \thet = \sqrt{- d} $, where $ d \in \mathcal{O}_{l} $ such that $ - d $ is not a square in $ l $ and $ L = l( \thet ) $.  Then the theorem in Section~10 of \cite{cassels67} states that there are at most 2 extensions of the valuation $ | \cdot |_{\mathfrak{p}} $ to $ L $, and if $ \mathfrak{P} $ is a prime above $ \mathfrak{p} $ in $ L $ (written as $ \mathfrak{P} \mid \mathfrak{p} $), we have
	\[ \Lp = \lp \otimes_{l} L = \bigoplus_{\mathfrak{P} \mid \mathfrak{p}} L_{\mathfrak{P}}, \]
where $ L_{\mathfrak{P}} $ denotes the completion of $ L $ with respect to $ \mathfrak{P} $.  This implies that for a finite prime $ \mathfrak{p} $,
	\[ \Lp = \begin{cases} \lp( \thet ), &\text{if $ \mathfrak{p} \mathcal{O}_{L} $ does not split in $ \mathcal{O}_{L} $;} \\ \lp \oplus \lp, &\text{if $ \mathfrak{p} \mathcal{O}_{L} $ splits in $ \mathcal{O}_{L} $.} \end{cases} \]

Thus for every finite prime $ \mathfrak{p} $, we have to consider if the extension $ \Lp / \lp $ is non-split (hence unramified or ramified) or split.  Also, note that the first lemma in Section~14 of \cite{cassels67} states that
	\[ \mathbb{A}_{l} \otimes_{l} L = \mathbb{A}_{L}, \]
in both an algebraic and a topological sense, and $ l \otimes_{l} L = L \subset \mathbb{A}_{l} \otimes_{l} L $, where $ l \subset \mathbb{A}_{l} $, is mapped identically on to $ L \subset \mathbb{A}_{L} $.  This implies that
	\[ G( \mathbb{A}_{l} ) = \left\{ \nu \in \SL_{3}\left( \mathbb{A}_{L} \right) \colon \nu^{t} J' \ol{\nu} = J' \right\}, \]
and hence for a finite prime $ \mathfrak{p} $,
	\[ G( \Olp ) = \{ \nu \in \SL_{3}( \OLp ) \colon \nu^{t} J' \ol{\nu} = J'\}. \]

Note that we will be putting $ k = \lp $ and $ K = \Lp $ as defined in Section~\ref{s:sustruc} in Part~\ref{p:kubota}, hence we get the same definition for $ G( \lp ) $ and $ G( \Olp ) $ using \eqref{e:gk} and \eqref{e:gok}.  We will also assume that when $ \Lp / \lp $ is ramified, $ \thet $ is a prime element of $ \Lp $.



\chapter{With reference to Deodhar's paper} \label{c:deodhar} 

As stated in the Introduction, \cite{deodhar78} is heavily relied upon when calculating the 2-cocycle of the universal central extension of $ G( k ) $.  We need to establish a few more facts from \cite{deodhar78} in order to show that our result will be valid.

\section{Some properties of $ \SU( k ) $} \label{s:suprop}

As before, we let $ k $ be an arbitrary field of characteristic zero and $ K = k( \thet ) $ a quadratic extension of $ k $.  Proposition~2.11 of \cite{deodhar78} states the following:

\begin{prop} \label{p:deltas}
There exists a well-defined function $ \delta = ( \delta_{1}, \delta_{2} ) \colon K^{\times} \to L \times L_{2} $, where
	\[ L = \{ m \in K^{\times} \colon \trace{m} = - \norm{r} \text{ for some } r \in K \} \]
and
	\[ L_{2} = \{q \thet \colon q \in k^{\times} \} \subset L \]
as follows:
\begin{enumerate}[\upshape (i)]

	\item If $ \lambda \in k^{\times} $, then
		\[ \del{1}{\lambda} = \lambda \thet, \quad \del{2}{\lambda} = \thet. \]
		
	\item If $ \lambda = a + b \thet $, $ b \neq 0 $, then
		\[ \del{1}{\lambda} = -\frac{1}{2} - \frac{a}{2 b \thet}, \quad \del{2}{\lambda} = - \frac{1}{2 b \thet}. \]
		
\end{enumerate}
\end{prop}

(Note that $ L $ has also been defined as a field, but there should be no overlap in notation as the definition of $ L $ as used in this section will not occur elsewhere.)

As a consequence, $ \lambda = \del{1}{\lambda} / \del{2}{\lambda} $.  Also, for any $ m \in \del{1}{K^{\times}} $ with $ m \notin k^{\times} $ and $ m' \in \del{2}{K^{\times}} $,
	\[ \del{1}{\frac{m}{m'}} = m, \quad \del{2}{\frac{m}{m'}} = m'. \]

Proposition~\ref{p:deltas} implies that for $ \lambda \in K^{\times} $, we have
	\[ \ha{\lambda} = \wa{y( \lambda)}{\del{1}{\lambda}} \cdot \wa{0}{\del{2}{\lambda}}^{- 1}, \]
where
\begin{equation} \label{e:ylam}
	y( \lambda ) =	\begin{cases}
						0, &\text{if $ \lambda \in k^{\times} $;} \\
						1, &\text{otherwise.}
					\end{cases}
\end{equation}
Thus by \eqref{e:warm} and the above, all elements of $ T( k ) $ are generated by elements of $ N( k ) $ and $ \ol{N}( k ) $.  This in turn implies that by the Bruhat decomposition for $ G( k ) $ (see \eqref{e:bruc} and \eqref{e:bruf}), every element of $ \SU( k ) $ is generated by $ N( k ) $ and $ \ol{N}( k ) $.  This verifies the theorem obtained from Sections~1.2 and~2.3 of \cite{deodhar78}:

\begin{thm} \label{t:quasi}
If a group $ G $ is quasi-split, then $ G( k ) $ is generated by the unipotent elements in $ G( k ) $ which belong to the radical of a parabolic subgroup $ P $ defined over $ k $.
\end{thm}

We also have the following definition:

\begin{defin} \label{d:perfect}
A perfect group $ G $ is a group which is its own commutator subgroup, i.e. if we express the commutator of $ g $, $ h \in G $ as
	\[ [ g, h ] = g \cdot h \cdot g^{-1} \cdot h^{-1}, \]
then
	\[ G = [ G, G ]. \]
\end{defin}

Recall that we have set $ G = \SU{} $.  As stated in Section~1.1 of \cite{deodhar78}, a necessary and sufficient condition for an abstract group to have a universal central extension is that the group is perfect.  As $ N( k ) $ and $ \ol{N}( k ) $ consist of commutators, and $ G( k ) $ is generated by $ N( k ) $ and $ \ol{N}( k ) $, this implies that $ G( k ) $ is perfect.  Thus there is a universal central extension of $ G( k ) $, which may be expressed by
	\[ 1 \lra \pi_{1} \lra \wtilde{G} \xrightarrow{\ \pi\ } G( k ) \lra 1, \]
where $ \pi_{1} $ denotes the kernel of $ \pi \colon \wtilde{G} \to G( k ) $.  (Note that $ \pi_{1} $ is also known as the Schur multiplier or the fundamental group of $ G( k ) $, and it is central in $ \wtilde{G} $.)


By Lemma~1.10 of \cite{deodhar78}, there is a unique lift of $ N( k ) $ to $ \wtilde{G} $.  This implies that we can write for $ \xa{r}{m} \in N( k ) $ the corresponding element in $ \wtilde{G} $ by \txa{r}{m}, and if we define
	\[ \wtilde{N}( k ) = \{ \txa{r}{m} \colon \xa{r}{m} \in N( k ) \}, \]
then $ \pi \colon \wtilde{N}( k ) \to N( k ) $ is an isomorphism.  Similarly $ \pi \colon \wtilde{\ol{N}}( k ) \to \ol{N}( k ) $ is also an isomorphism, and we may define the corresponding element of $ \xam{r}{m} $ in $ \wtilde{G} $ as $ \txam{r}{m} $.  Thus we may define, similar to \eqref{e:warm}, the element
	\[ \twa{r}{m} = \txa{r}{m} \cdot \txam{\frac{\ol{r}}{\ol{m}}}{\frac{m}{\ol{m}}} \cdot \txa{r \cdot \frac{\ol{m}}{m}}{m}. \]

Furthermore, Proposition~2.9 of \cite{deodhar78} lists a few relations relevant to our discussion.  We list the most important relations for reference here.

\begin{prop} \label{p:wa}
Let
\begin{equation} \label{e:defA1}
	A = \{ ( r, m ) \in K \times K \colon ( r, m ) \neq ( 0, 0 ), \trace{m} = - \norm{r} \},
\end{equation} 
and define $ f $, $ g \colon A \to K \times K $ by
	\[ f( r, m ) = \left( \frac{\ol{r}}{\ol{m}}, \frac{1}{\ol{m}} \right), \quad g( r, m ) = \left( \frac{\ol{r}}{m}, \frac{1}{\ol{m}} \right). \]
The following hold in $ \wtilde{G} $, and hence in $ G( k ) $ too:
\begin{enumerate}[\upshape (1)]

	\item $ \displaystyle \twa{r}{m} = \txa{r}{m} \cdot \txam{\frac{\ol{r}}{\ol{m}}}{\frac{1}{\ol{m}}} \cdot \txa{r \cdot \frac{\ol{m}}{m}}{m}, $
	\item[] $ \twa{r}{m}^{-1} = \twa{- r}{\ol{m}} $,
	\item[] $ \displaystyle \twa{r}{m} = \twa{r \cdot \frac{\ol{m}}{m}}{m} = \wtilde{w}_{\alpha}( ( g \circ f )^{i} ( r, m ) ) \qquad \forall i $
	\item[]	$ \displaystyle \phantom{\twa{r}{m}} = \twam{\frac{\ol{r}}{\ol{m}}}{\frac{1}{\ol{m}}} = \wtilde{w}_{- \alpha}( ( f \circ g )^{i} ( f( r, m ) ) ) \qquad \forall i. $

	\item $ \displaystyle \twa{r}{m} \cdot \txa{r'}{m'} \cdot \twa{r}{m}^{-1} = \txam{\frac{\ol{r'} \cdot \ol{m}}{m^{2}}}{\frac{m'}{\norm{m}}} $.
	
	\item $ \displaystyle \twa{r}{m} \cdot \txam{r'}{m'} \cdot \twa{r}{m}^{-1} = \txa{\frac{\ol{r'} \cdot m^{2}}{\ol{m}}}{m' \cdot \norm{m}} $.
	
	\item $ \displaystyle \twa{r}{m} \cdot \twa{r'}{m'} \cdot \twa{r}{m}^{-1} = \twam{\frac{\ol{r'} \cdot \ol{m}}{m^{2}}}{\frac{m'}{\norm{m}}} $ \\
		\[ \phantom{\qquad} = \twa{\frac{r' \cdot m^{2}}{\ol{m'} \cdot \ol{m}}}{\frac{\norm{m}}{\ol{m'}}} = \twa{\frac{r' \cdot m^{2}}{m' \cdot \ol{m}}}{\frac{\norm{m}}{\ol{m'}}}. \]
	
	\item $ \displaystyle [ \twa{p_{1}}{l_{1}} \cdot \twa{p_{1}'}{l_{1}'} ] \cdot \twa{r}{m} \cdot \twa{r'}{m'} \cdot [ \twa{p_{1}}{l_{1}} \cdot \twa{p_{1}'}{l_{1}'} ]^{-1} $ \\
		\[ \phantom{\qquad} = \twa{\frac{r l_{1}' l_{1}^{2}}{(\ol{l_{1}'})^{2} \ol{l_{1}}}}{\frac{m \norm{l_{1}}}{\norm{l_{1}'}}} \cdot \twa{\frac{r' l_{1}' l_{1}^{2}}{(\ol{l_{1}'})^{2} \ol{l_{1}}}}{\frac{m' \norm{l_{1}}}{\norm{l_{1}'}}}. \]
		
\end{enumerate}
\end{prop}

Note that we will use the above definition of $ A $ when we are finding an equation for the 2-cocycle of the universal central extension of $ G( k ) $.  We can also define $ A $ as
\begin{equation} \label{e:defA2}
	A = \{ ( z, - \norm{z} / 2 + t \thet ) \in K \times K \colon t \in k, ( z, - \norm{z} / 2 + t \thet ) \neq ( 0, 0 ) \}.
\end{equation}
This definition will be useful later in Part~\ref{p:kubota}.

\section{The universal topological central extension} \label{s:2coctop}

In this section we assume that $ k $ is a non-archimedean local field.  We shall regard $ G( k ) $ as a locally compact topological group, in which the topology is given by the norm on $ k $.  We shall write $ H_{m}^{i} $ for measurable cohomology as defined by Calvin Moore in \cite{moore68}, and we shall write $ H^{i} $ for continuous cohomology.

In Deodhar's paper, $ G( k ) $ is first regarded as an abstract group when constructing the universal central extension. It is subsequently regarded as a topological group, with the topology given by the norm on $k$.
It was shown in that paper that the universal topological covering group for $ G( k ) $ exists and that the topological fundamental group is a quotient of $ \mu( k ) $, where $ \mu( k ) $ is the group of roots of unity of $ k $.


We first recall what a topological central extension is.  This is a central extension of topological groups
	\[ 1 \lra K \lra \wtilde{G}^{top} \lra G \lra 1, \]
where $ \wtilde{G}^{top} $ is the universal topological covering group of $ G $, and in which $ K $ is discrete and there is a neighbourhood $ U $ of the identity in $ G $, such that the projection $ \wtilde{U} \to U $ is topologically isomorphic to $ K \times U \to U $.

We now turn to a paper by Prasad and Raghunathan which was published in two parts, \cite{prarag184} and \cite{prarag284}.  In 10.3 of \cite{prarag284}, we have the following proposition:

\begin{prop} \label{p:prarag10.3}
Let $ \mathcal{G} $ be a locally compact, second countable topological group.  Assume that $ \mathcal{G} = [ \mathcal{G}, \mathcal{G} ] $, and $ H^{2}_{m}( \mathcal{G}, \mathbb{R} / \mathbb{Z} ) $ is a finite group.  Then $ \mathcal{G} $ admits a universal topological covering and its topological fundamental group is isomorphic to the dual of $ H^{2}_{m}( \mathcal{G}, \mathbb{R} / \mathbb{Z} ) $.
\end{prop}

We first note that $ G( k ) $ is a locally compact second countable group, and we have already noted that $ G( k ) = [ G( k ), G( k ) ] $.  In order to use the above proposition, we must show $ H^{2}_{m}( G( k ), \mathbb{R} / \mathbb{Z} ) $ is a finite group, and that the dual of this group is equal to the fundamental group $ \pi_{1} = \mu( k ) $.

Theorem~1 of \cite{wigner73} states the following:

\begin{thm} \label{t:wigner}
Let $ G $ be a topological group, and $ A $ be a $ G $-module.  If $ G $ is a locally compact, $ \sigma $-compact, zero-dimensional, then $ H^{i}_{m}( G, A ) \cong H^{i}( G, A ) $.
\end{thm}

In other words, the cohomology groups based on continuous cochains and the cohomology groups based on measurable cochains coincide for our group $ G( k ) $.  This implies that $ H^{2}_{m}( G( k ), \mathbb{R} / \mathbb{Z} ) \cong H^{2}( G( k ), \mathbb{R} / \mathbb{Z} ) $.  In 5.10 and 5.11 of \cite{prarag184}, it was established that if $ \mathcal{G} $ is an absolutely simple, simply connected group defined and quasi-split over $ F $ (where $ F $ is a non-archimedean local field) and $ \mathcal{H} $ is the $ F $-subgroup of $ \mathcal{G} $, $ F $-isomorphic to $ \SL_{2} $ and determined by a long root, then the following theorem holds:

\begin{thm} \label{t:prarag5.11}
The restriction $ H^{2}( \mathcal{G}( F ), \mathbb{R} / \mathbb{Z} ) \to H^{2}( \mathcal{H}( F ), \mathbb{R} / \mathbb{Z} ) $ is an isomorphism.  Hence, $ H^{2}( \mathcal{G}( F ), \mathbb{R} / \mathbb{Z} ) $ is isomorphic to $ \hat{\mu}( F ) = \homo( \mu( F ), \mathbb{R} / \mathbb{Z} ) $, the Pontrjagin dual.
\end{thm}

Theorem~\ref{t:prarag5.11} shows that $ H^{2}( G( k ), \mathbb{R} / \mathbb{Z} ) = \hat{\mu}( k ) $, i.e. $ H^{2}( G( k ), \mathbb{R} / \mathbb{Z} ) $ is finite.  Thus using Theorem~\ref{t:wigner}, by Proposition~\ref{p:prarag10.3} $ H^{2}_{m}( G( k ), \mathbb{R} / \mathbb{Z} ) $ is a finite group and thus $ G( k ) $ admits a universal covering and its topological fundamental group is isomorphic to the dual of $ H^{2}_{m}( G( k ), \mathbb{R} / \mathbb{Z} ) = \hat{\mu}( k ) $, i.e. the topological fundamental group is isomorphic to $ \mu( k ) $.




We should also mention the cases when $ k = \mathbb{R} $ and $ k = \mathbb{C} $.  In these cases, $ G( k ) $ is a connected Lie group, and its universal cover $ \wtilde{G} $, in the topological sense, has the structure of a Lie group and is the universal topological central extension.  If we denote the topological fundamental group of $ G( k ) $ as $ \pi_{1}^{top} $, then when $ k = \mathbb{R} $ we have $ \pi_{1}^{top} = \mathbb{Z} $ and when $ k = \mathbb{C} $ we have $ \pi_{1}^{top} = 1 $.

\section{A section for $ \pi \colon \wtilde{\SU} \to \SU( k ) $} \label{s:sec}

Again let $ k $ be a local field and let $\wtilde{G}$ be the universal topological central extension
of $G(k)$.
We can now define the section $ \delta \colon G( k ) \to \wtilde{G} $ as in Section~2.21 of \cite{deodhar78}.  This section will give rise to a 2-cocycle $ \sigma_{u} $ defined by
	\[ \sigu{g}{h} = \delta( g ) \cdot \delta( h ) \cdot \delta( g \cdot h )^{-1}, \]
where $ g $, $ h \in G( k ) $.

By the Bruhat decomposition (see Section~\ref{s:bruh}), any element of $ G( k ) $ can be uniquely written in the form $ u h w v $, where $ h \in T( k ) $, $ w \in W $ (recall that $ W $ is the set of representatives of the Weyl group of $ G( k ) $) and $ u $, $ v \in N( k ) $.  We will choose, as in \cite{deodhar78}, \wa{0}{\thet} as the representative of the non-trivial element of $ W $.  So we may define $ \delta( \wa{0}{\thet} ) = \twa{0}{\thet} $, and since $ \pi \colon \wtilde{N}( k ) \to N( k ) $ is an isomorphism, we may choose $ \delta( \xa{r}{m} ) = \txa{r}{m} $ (see Section~\ref{s:suprop}).

As for $ h \in T( k ) $, suppose that $ h = \ha{\lambda} $.  Then we define
\begin{equation} \label{e:delh}
	\delh{\lambda} = \twa{y( \lambda )}{\del{1}{\lambda}} \cdot \twa{0}{\del{2}{\lambda}}^{- 1},
\end{equation} 
where $ y( \lambda ) $ is defined as in \eqref{e:ylam}.

Thus, for $ u h w v \in G( k ) $, we may define $ \delta( u h w v ) = \delta( u ) \cdot \delta( h ) \cdot \delta( w ) \cdot \delta( v ) $, and hence $ \delta $ is a section for $ \pi $.

Our next aim will be to express the 2-cocycle $ \sigma_{u} $ in terms of $ K_{2} $ symbols.  Deodhar performed this calculation on the split torus, but we shall need to extend his formula to the whole group.

For $ \lambda $, $ \mu \in K^{\times} $, let us define
\begin{equation} \label{e:ba}
	\ba{\lambda}{\mu} = \delh{\lambda} \cdot \delh{\mu} \cdot \delh{\lambda \mu}^{-1},
\end{equation}
i.e. $ \ba{\lambda}{\mu} = \sigu{\ha{\lambda}}{\ha{\mu}} $.  It is established in Section~2.30 of \cite{deodhar78} that 
\begin{align}
	\ba{s}{t} \cdot \ba{st}{r} &= \ba{s}{tr} \cdot \ba{t}{r}, & \forall s, t, r \in k^{\times}; \label{e:prop1e} \\
	\ba{1}{1} &= 1; & \notag \\
	\ba{s}{t} &= \ba{t^{-1}}{s}, & \forall s, t \in k^{\times}; \label{e:prop2e} \\
	\ba{s}{t} &= \ba{s}{-s t}, & \forall s, t \in k^{\times}; \label{e:prop3e} \\
	\ba{s}{t} &= \ba{s}{( 1 - s ) t}, & \forall s, t \in k^{\times},\; s \neq 1. \label{e:prop4e}
\end{align}
Indeed these relations all hold in the universal central extension, and so they are also true in the universal topological central extension.
Note that the section $ \delta $ is continuous on the split torus $ S( k ) $, and thus the restriction of $ b_{\alpha} $ to $ k^{\times} \times k^{\times} $ is continuous.
It is also established that $ \pi_{1} $ (as opposed to $\pi_{1}^{top}$)
 is given in terms of generators and relations as
	\[ \pi_{1} = \left\langle \{ \ba{s}{t} \colon s, t \in k^{\times} \} \mid \eqref{e:prop1e} \text{ - } \eqref{e:prop4e} \right\rangle. \]
Since $ \ba{\lambda}{\mu} \in \pi_{1}^{top} $ for any $ \lambda $, $ \mu \in K^{\times} $ (by the proof of Lemma~2.12 of \cite{deodhar78}), this implies that \ba{\lambda}{\mu} can be written in terms of these \ba{s}{t}'s.

Thus, we would establish what $ \sigma_{u} $ is on $ T( k ) $.  Unfortunately in \cite{deodhar78}, Deodhar only proves that the \ba{s}{t}'s ($ s $, $ t \in k^{\times} $) satisfy the relations above, and not what \ba{\lambda}{\mu} is explicitly for all values of $ \lambda $, $ \mu \in K^{\times} $.  But we can adapt the methods used in \cite{deodhar78} to achieve this goal, and by using the section $ \delta $, we will be able to express what $ \sigma_{u} $ is on the whole of $ G( k ) $ in terms of \ba{s}{t}'s, where $ s $, $ t \in k^{\times} $.

\section{$ k $ as a local field} \label{s:klocal}

Let $ k $ be a non-archimedean local field.  We first note that by Theorem~3.1 of \cite{moore68}, that \ba{s}{t}, for $ s $, $ t \in k^{\times} $, is bilinear, i.e. for $ s $, $ t $, $ r \in k^{\times} $,
\begin{align}
	\ba{st}{r} &= \ba{s}{r} \cdot \ba{t}{r}, \label{e:bastre} \\
	\ba{s}{tr} &= \ba{s}{t} \cdot \ba{s}{r}, \notag
\end{align}
since \ba{s}{t} is continuous (as $ k $ is a local field; see previous section).  A corollary of the abovementioned theorem is that the $ n $-th power Hilbert symbol \hilkn{-}{-} (where $ n $ is the number of roots of unity of $ k $) and its powers are the only functions which are continuous and satisfy the equations \eqref{e:prop1e} -- \eqref{e:prop4e} (a definition of the Hilbert symbol may be found in Chapter~\ref{c:propext}).  We have the following proposition:

\begin{prop} \label{p:baisom}
For any local field $ k $, if $ n $ denotes the number of roots of unity in $ k $ and $ \mu_{n} $ denotes the group of $ n $-th roots of unity of $ k $, then $ \pi_{1}^{top} \cong \mu( k ) = \mu_{n} $, and the isomorphism may be described by
\begin{align*}
	\Phi \colon \ \qquad \pi_{1}^{top} &\cong \mu_{n} \\
	\ba{s}{t} &\mapsto \hilkn{s}{t},
\end{align*}
where $ s $, $ t \in k^{\times} $ and $ \hilkn{-}{-} $ is the $ n $-th power Hilbert symbol.
\end{prop}

\begin{proof}
We have shown that there is an isomorphism of the form $ \ba{s}{t} \mapsto \hilkn{s}{t}^{r} $ for some $ r $.  However, the result of Prasad and Raghunathan shows that the order of the topological fundamental group is $ n $.  Hence we may take $ r = 1 $.
\end{proof}

Thus from Chapter~\ref{c:t2coc} onwards, when both $ s $, $ t \in k^{\times} $, we will use the Hilbert symbol $ \hilkn{s}{t} $ instead of $ \ba{s}{t} $.  This implies that when we apply $ \Phi $ to \eqref{e:bastre} and \eqref{e:prop2e} -- \eqref{e:prop4e}, we get
\begin{align}
	\hilkn{st}{r} &= \hilkn{s}{r} \cdot \hilkn{t}{r}, \label{e:bastr} \\
	\hilkn{s}{tr} &= \hilkn{s}{t} \cdot \hilkn{s}{r}, & \forall s, t \in k^{\times}; \notag \\
	\hilkn{s}{t} &= \hilkn{t^{-1}}{s}, & \forall s, t \in k^{\times}; \label{e:prop2} \\
	\hilkn{s}{t} &= \hilkn{s}{-s t}, & \forall s, t \in k^{\times}; \label{e:prop3} \\
	\hilkn{s}{t} &= \hilkn{s}{( 1 - s ) t}, & \forall s, t \in k^{\times},\; s \neq 1. \label{e:prop4}
\end{align}

\chapter{Some properties of quadratic extensions} \label{c:propext}

In this chapter, we gather some information to be used in later chapters, mostly to do with quadratic field extensions.

\section{The Hilbert symbol} \label{s:hilbert}

We define the $ n $-th power Hilbert symbol, where the group $ \mu_{n} $ of $ n $-th roots of unity is contained in a local field $ k $, with $ n $ a natural number relatively prime to the characteristic of $ k $.  This definition may be found in Section~V.3 of \cite{neukirch99}.  Letting $ K = k\left( \sqrt[n]{k^{\times}} \right) $ be the maximal abelian extension of exponent $ n $, it has been established that $ \gal( K / k ) \cong k^{\times} / ( k^{\times} )^{n} $ and $ \homo( \gal( K / k ), \mu_{n} ) \cong k^{\times} / ( k^{\times} )^{n}. $  The bilinear map
	\[ \gal( K / k ) \times \homo( \gal( K / k ), \mu_{n} ) \to \mu_{n}, \quad \spl{\gamma}{\chi} \mapsto \chi( \gamma ) \]
therefore defines a nondegenerate bilinear pairing
	\[ ( - , - )_{k, n} \colon k^{\times} / ( k^{\times} )^{n} \times k^{\times} / ( k^{\times} )^{n} \to \mu_{n}, \]
which we call the $ n $-th power Hilbert symbol.  Now using the notation introduced in Subsection~\ref{ss:local}, let $ p $ be the characteristic of the residue field $ \Ok / \mathfrak{p}_{k} $, and $ q = | \Ok / \mathfrak{p}_{k} | $.  We have the following proposition:

\begin{prop} \label{p:tame}
If $ n $ and $ p $ are relatively prime and $ a $, $ b \in k^{\times} $, then
	\[ ( a, b )_{k, n} \equiv \left( ( - 1 )^{v_{k}( a ) v_{k}( b )} \cdot \frac{b^{v_{k}( a )}}{a^{v_{k}( b )}} \right)^{( q - 1 ) / n} \pod{\mathfrak{p}_{k}}. \]
\end{prop}

A proof of the above proposition may be found in V.3.4 of \cite{neukirch99}.  When $ n $ and $ p $ are relatively prime, we call this the case of the \emph{tame Hilbert symbol}, as defined above.  Note that a consequence of the tame Hilbert symbol is that whenever $ a $, $ b \in \Ok^{\times} $, $ ( a, b )_{k, n} = 1 $.

We should also note that the Hilbert symbol obeys the \emph{product formula}, which is proved in Theorem~VI.8.1 of \cite{neukirch99}.  We state this theorem for later reference.

\begin{thm} \label{t:product}
Let $ l $ be a global field, $ \mathfrak{p} $ a prime of $ l $ with $ l_{\mathfrak{p}} $ the localisation of $ l $ at $ \mathfrak{p} $.  Also, let $ l $ contain the group $ \mu_{n} $ of $ n $-th roots of unity.  Then for $ a $, $ b \in l^{\times} $,
	\[ \prod_{\mathfrak{p}} ( a, b )_{\lp, n} = 1, \]
where $ \mathfrak{p} $ runs through all the primes of $ l $.
\end{thm}

In our case, we are more interested in quadratic Hilbert symbols (i.e. when $ n = 2 $).  It has been established that the quadratic Hilbert symbol has a more concrete meaning:
	\[ \hilk{a}{b} = 1 \iff aX^{2} + bY^{2} - Z^{2} = 0 \text{ has a non-trivial solution $ ( X, Y, Z ) $ in $ k^{3} $}. \]
(See Chapter~III of \cite{serre73}.)

Now using our usual definition of $ k $ and $ K $, we list the properties of quadratic Hilbert symbols here for convenient reference later, a proof of which may be found in Section~V.3 of \cite{neukirch99}.  For $ s $, $ t $, $ r \in k^{\times} $, $ m \in \mathbb{Z} $,
\begin{align}
	&\hilk{s}{t} = \hilk{s}{t}^{- 1} = \hilk{s^{- 1}}{t} = \hilk{s}{t^{- 1}} = \hilk{t}{s}; \label{e:prop2i} \\
	&\hilk{s}{t} = \hilk{s}{- s t} = \hilk{- s t}{t}; \label{e:prop3i} \\
	&\hilk{s}{t} = \hilk{s}{( 1 - s ) t} = \hilk{( 1 - t ) s}{t}, &\text{(for $ s \neq 1 $);} \label{e:prop4i} \\
	&\hilk{s t}{r} = \hilk{s}{r} \cdot \hilk{t}{r}, \label{e:bastri} \\
	&\hilk{s}{t r} = \hilk{s}{t} \cdot \hilk{s}{r}; \notag \\
	&\hilk{s}{t}^{m} = \hilk{s}{t^{m}} = \hilk{s^{m}}{t}; \label{e:prop5i} \\
	&\hilk{s}{t}^{2} = \hilk{s}{t^{2}} = \hilk{s^{2}}{t} = 1. \label{e:prop5ii}
\end{align}
(Note that the above is true for any field.) In addition, for all $ \lambda \in k^{\times} $, $ \mu \in K^{\times} $,
\begin{equation} \label{e:norm}
	\hilk{\lambda}{\norm{\mu}} = \hilbk{\lambda}{\mu} = \hilbk{\mu}{\lambda} = \hilk{\norm{\mu}}{\lambda}.
\end{equation}
(A statement and proof of this may be found in Chapter~2, Section~1, Theorem~2.14 on Page~101 of \cite{koch92}.  In fact,
\begin{gather}
	\hilkn{\lambda}{\norm{\mu}} = \hilbkn{\lambda}{\mu}, \label{e:norm1} \\
	\hilkn{\norm{\mu}}{\lambda} = \hilbkn{\mu}{\lambda} \label{e:norm2}
\end{gather}
by the same statement and proof.)

\section{Non-split and split quadratic extensions}

Recall that in Subsection~\ref{ss:descad}, we established that we will be calculating the local Kubota symbol on a compact open subgroup of
	\[ G( \lp ) = \{ \nu \in \SL_{3}( \Lp ) \colon \nu^{t} J' \ol{\nu} = J' \}, \]
where $ J' $ is as described in Section~\ref{s:sustruc}, $ l $ is a global field, $ \mathfrak{p} $ is a finite prime of $ l $, $ \lp $ is the completion of $ l $ with respect to $ \mathfrak{p} $, $ L = l( \thet ) $, $ \thet = \sqrt{- d} $, $ d \in \mathcal{O}_{l} $ such that $ - d $ is not a square in $ l $, and $ \Lp = \lp \otimes_{l} L $.

In our calculation of the Kubota symbol, we will often use the properties of the extension $ \Lp / \lp $.  We gather the statements of these properties in the next two subsections.

\subsection{The non-split case} \label{ss:nonsplit}

Consider any local field $ k $.  Let $ K $ be a finite extension of $ k $.  Then using the notation of Subsection~\ref{ss:local}, there is only one prime $ \mathfrak{p}_{K} $ above $ \mathfrak{p}_{k} $, i.e. $ \mathfrak{p}_{K} \mid \mathfrak{p}_{k} $ and
	\[ \mathfrak{p}_{k} \mathcal{O}_{K} = \mathfrak{p}_{K}^{e}, \]
where $ e = e( K / k ) \in \mathbb{N} $ is called the \emph{ramification index} of the extension $ K / k $.  If we define the \emph{residue class degree} of the extension $ K / k $ as
	\[ f = f( K / k ) = [ \mathcal{O}_{K} / \mathfrak{p}_{K} : \mathcal{O}_{k} / \mathfrak{p}_{k} ], \]
then it has been established in Proposition~3 of Section~5 of \cite{frohlich67} that
	\[ [ K : k ] = e( K / k ) f( K / k ). \]


An extension of local fields $ K / k $ is \emph{ramified} if $ e( K / k ) > 1 $ (and \emph{unramified} or \emph{inert} if $ e( K / k ) = 1 $), and is \emph{totally ramified} if $ e( K / k ) = [ K : k ] $.  

From Section~6 of \cite{frohlich67}, an \emph{Eisenstein polynomial} in $ k[ X ] $ is defined as a separable polynomial
	\[ E( X ) = X^{m} + b_{m - 1} X^{m - 1} + \cdots + b_{1} X + b_{0}, \]
with
	\[ v_{k}( b_{i} ) \geq 1 \text{ for } i = 1, \dots, m - 1, \text{ and } v_{k}( b_{0} ) = 1. \]
We have the following theorem (Theorem~1 in Section~6 of \cite{frohlich67}):

\begin{thm} \label{t:ram}
\begin{enumerate}[\upshape (i)]
	\item An Eisenstein polynomial $ E( X ) $ is irreducible.  If $ \Pi $ is a root of $ E( X ) $, then $ K = k[ \Pi ] $ is totally ramified and $ v_{K}( \Pi ) = 1 $.
	\item If $ K $ is totally ramified over $ k $ and $ v_{K}( \Pi ) = 1 $, then the minimal polynomial of $ \Pi $ over $ k $ is Eisenstein and
		\[ \mathcal{O}_{K} = \Ok [ \Pi ], \quad K = k[ \Pi ]. \]
\end{enumerate}
\end{thm}

The above theorem implies that since we have a prime element $ \pi = \pi_{\Lp} \in \OLp $ such that $ v_{\Lp}( \pi ) = 1 $, then if $ \Lp / \lp $ is totally ramified, 
	\[ \OLp = \Olp [ \pi ], \quad \Lp = \lp( \pi ). \]

We can find an analogue for the unramified case of the above theorem.  Let the image of an element $ a \in \Ok $ in the residue class field $ \mathcal{O}_{k} / \mathfrak{p}_{k} $ be denoted by $ \hat{a} $ and similarly let the image of an element $ f( X ) \in \Ok [ X ] $ in the polynomial ring $ ( \Ok / \mathfrak{p}_{k} ) [ X ] $ be denoted by $ \hat{f}( X ) $.  Proposition~1 of Section~7 of \cite{frohlich67} states the following:

\begin{prop} \label{p:unram}
\begin{enumerate}[\upshape (i)]
	\item Suppose $ K $ to be unramified over $ k $.  Then there exists an element $ c \in \mathcal{O}_{K} $ with $ \mathcal{O}_{K} / \mathfrak{p}_{K} = ( \Ok / \mathfrak{p}_{k} ) [ \hat{c} ] $.  If $ c $ is such an element and $ f( X ) $ is its minimal polynomial over $ k $, then $ \mathcal{O}_{K} = \Ok[ c ] $, $ K = k [ c ] $ and $ \hat{f}( X ) $ is irreducible in $ ( \Ok / \mathfrak{p}_{k} ) [ X ] $ and separable.
	\item Suppose $ f( X ) $ is a monic polynomial in $ \Ok[ X ] $, such that $ \hat{f}( X ) $ is irreducible in $ ( \Ok / \mathfrak{p}_{k} ) [ X ] $ and separable.  If $ c $ is a root of $ f( X ) $ then $ K = k[ c ] $ is unramified over $ k $ and $ \mathcal{O}_{K} / \mathfrak{p}_{K} = ( \Ok / \mathfrak{p}_{k} ) [ \hat{c} ] $.
\end{enumerate}
\end{prop}

We will be using the above to establish whether $ \Lp / \lp $ is ramified or unramified if $ \mathfrak{p} \OLp $ is not split.  We will also need to use Hensel's Lemma.  Appendix~C of \cite{cassels67} states this lemma, which we will write as a theorem.

\begin{thm}[Hensel's Lemma] \label{t:hensel}
Let $ k $ be a field complete with respect to the non-archimedean valuation $ | \cdot | $ and let
	\[ f( X ) \in \Ok[ X ]. \]
Let $ a_{0} \in \Ok $ be such that
	\[ | f( a_{0} ) | < | f'( a_{0} ) |^{2}, \]
where $ f'( X ) $ is the (formal) derivative of $ f( X ) $.  Then there is a solution of
	\[ f( a ) = 0, \quad | a - a_{0} | \leq \frac{| f( a_{0} ) |}{| f'( a_{0} ) |}. \]
\end{thm}

$ a_{0} $ is known as the \emph{approximate root} of the polynomial $ f( X ) $.  The above lemma will be useful when we want to apply Hensel's Lemma with respect to our field $ \Lp $, as we will know the normalised valuation of $ \lp $ but may not know directly the normalised valuation on $ \Lp $.

\subsection{The split case} \label{ss:split}

In the split case, we have $ \thet \in \Olp \subset \lp $, and $ \Lp = \lp \oplus \lp $.  We can think of $ \Lp $ as being isomorphic to $ \lp( \thet ) $ by the bijective map
\begin{align}
	\lp( \thet ) &\cong \lp \oplus \lp \notag \\
	a + b \thet &\mapsto \spl{a + b \thet}{a - b \thet} \label{e:splmap}
\end{align}
for any $ a $, $ b \in \lp $.  Hence, we identify $ \lp $ with $ \{ \spl{a}{a} \in \Lp \colon a \in \lp \} $.  Note that where no confusion can occur, we will denote an element $ \spl{a}{a} \in \Lp $ as $ a \in \lp $ and an element $ \spl{a + b \thet}{a - b \thet} \in \Lp $ as $ a + b \thet \in \Lp $.

Also, the norm and trace of $ \spl{a}{b} \in \Lp $ are defined respectively as
	\[ \norm{\spl{a}{b}} = a b, \quad \trace{\spl{a}{b}} = a + b. \]
Thus, for any element $ a \in \lp $, we can always choose an element in $ c \in \Lp $ such that $ \norm{c} = a $.  An obvious choice would be $ c = \spl{a}{1} $.


\part{The 2-cocycle of the universal central extension of $ \SU( k ) $} \label{p:su2coc}


\chapter{The 2-cocycle on $ T( k ) $} \label{c:t2coc}

Suppose that $k$ is a local field containing an $n$-th root of unity and let $ G = \SU $.
Recall that we have a homomorphism $\pi_{1}(G(k)) \to \mu_{n}$
given by $b_{\alpha}(s,t) \mapsto (s,t)_{k,n}$ for $s,t\in k^{\times}$.
This map gives rise to a central extension
$$
	1 \to \mu_{n} \to \wtilde G \to G(k) \to 1.
$$
We have seen that if $k$ is non-archimedean and $n$ is the number of roots of unity in $k$,
then this extension is the universal topological central extension of $G(k)$.
We have described a section $\delta \colon G(k)\to \wtilde{G}$. This section gives rise to
a 2-cocycle $\sigma_{u}$ on $G(k)$ with values in $\mu_{n}$.

Recall that we are actually interested in the 2-cocycle corresponding to the double cover
of $G(k)$, which represents a cohomology class $\sigma\in H^{2}(G(k),\mu_{2})$.
By the existence of a universal central extension we have (by Theorem~1.1 of \cite{moore68})
 for any trivial $G(k)$-module $A$,
	\[ H^{2}( G( k ), A ) = \homo ( \pi_{1}, A ) = \homo( \mu_{n}, A ).\]
This implies that
	\[ H^{2}( G( k ), \mu_{2} ) = \homo ( \mu_{n}, \mu_{2} ) = \mathbb{Z} / 2. \]
Thus, $ \sigma = \sigma_{u}^{n / 2} $.  We will use this fact in Chapter~\ref{c:univ} to describe
the cocycle $\sigma$.

%

We will first calculate the cocycle $ \sigma_{u} $ on $ \SU( k ) $.  This calculation is divided into 2 chapters.  This chapter introduces the notation we will use, cites a few results, and we will calculate the 2-cocycle on a maximal torus $ T( k ) $ of $ G( k ) $.  Recall that we write $ \hilkn{s}{t} $ instead of $ \ba{s}{t} $ when both $ s $, $ t \in k^{\times} $.  Also, recall the definitions of $ \delta_{1} $ and $ \delta_{2} $ as stated in Proposition~\ref{p:deltas}.  We will get the following theorem:

\begin{thm2}
For $ \lambda $, $ \mu \in K^{\times} $,
\begin{multline*}
	\sigu{\ha{\lambda}}{\ha{\mu}} \\
		=	\begin{cases}
				\displaystyle \hilkn{\lambda}{\mu}, &\text{if $ \lambda $, $ \mu \in k^{\times} $;} \\
				\displaystyle \hilkn{\mu}{- \del{2}{\lambda} / \thet}, &\text{if $ \lambda \notin k^{\times} $, $ \mu \in k^{\times} $;} \\
				\displaystyle \hilkn{\lambda}{\mu \ol{\del{1}{\mu}} / \thet}, &\text{if $ \lambda \in k^{\times} $, $ \mu \notin k^{\times} $;} \\
				\displaystyle \hilkn{- 1}{\norm{\lambda}} \cdot \hilkn{- \frac{\lambda \ol{\del{1}{\lambda}}}{\thet}}{\lambda \mu}, &\text{if $ \lambda $, $ \mu \notin k^{\times} $, $ \lambda \mu \in k^{\times} $;} \\
				\displaystyle
					\begin{aligned}[b]
						&\hilkn{- \frac{\norm{\del{1}{\mu}}}{\norm{\del{1}{\lambda}}}}{\frac{\norm{\lambda}}{q^{2}}} \cdot \hilkn{q}{\frac{\mu \ol{\del{1}{\mu}}}{\del{2}{\lambda}}} \\
						&\phantom{\ } \cdot \Sigma'( \lambda, \mu ),
					\end{aligned} &\text{otherwise,}
			\end{cases}
\end{multline*}
where, if $  \lambda = a + b \thet $, $ \mu = c + d \thet $, with $ a $, $ c \in k $, $ b $, $ d \in k^{\times} $,
	\[ q = a + \frac{b c}{d}, \]
and
	\[ \Sigma'( \lambda, \mu )
		=	\begin{cases}
				\displaystyle\hilkn{- \frac{\norm{\lambda}}{4 a^{2} \norm{\del{1}{\mu}}}}{\frac{( \norm{\lambda} )^{2} b^{4} \thet^{4}}{( ( a - q ) a - b^{2} \thet^{2} )^{4}}}, &\text{if $ \lambda \notin k^{\times} \thet $, $ a q \neq \norm{\lambda} $;} \\
			\displaystyle\hilkn{\frac{( \norm{\lambda} )^{2} b^{4}\thet^{4}}{( ( a - q ) a + b^{2} \thet^{2} )^{4}}}{- \frac{\norm{\lambda}}{4 a^{2} \norm{\del{1}{\mu}}}}, &\text{if $ \lambda \notin k^{\times} \thet $, $ a q = \norm{\lambda} $;} \\
				\displaystyle 1, &\text{if $ \lambda \in k^{\times} \thet $.}
			\end{cases} \]
\end{thm2}

It should be noted that the right-hand side is given in terms of $ n $-th power Hilbert symbols on $ k $.

The next chapter will use the results of this chapter to prove what the 2-cocycle is on the whole of $ G( k ) $. We will then get an expression for the 2-cocycle $ \sigma $ corresponding to the double cover of $ G( k ) $ using the relation above.

\section{Initial results} \label{s:initial}

We first note that for all $ s $, $ t \in k^{\times} $, by \eqref{e:bastr},
\begin{equation} \label{e:prop5}
	\hilkn{s}{t}^{m} = \hilkn{s^{m}}{t} = \hilkn{s}{t^{m}},
\end{equation}
for $ m \in \mathbb{Z} $.

For $ s \in k^{\times} $, since $ \hilkn{s}{- s^{- 1}} = 1 $, by \eqref{e:ba},
\begin{equation} \label{e:delinv}
	\delh{s}^{-1} = \delh{- \frac{1}{s}} \cdot \delh{-1}^{-1}.
\end{equation}

Also, by Proposition~\ref{p:wa}, for $ s $, $ t \in k^{\times} $,
\begin{multline*}
	[ \twa{0}{s \thet} \cdot \twa{0}{- \thet} ] \cdot \twa{0}{t \thet} \cdot \twa{0}{- \thet} \cdot [ \twa{0}{s \thet} \cdot \twa{0}{- \thet} ]^{-1} \\
	= \twa{0}{\frac{( t \thet ) ( - s^{2} \thet^{2} )}{- \thet^{2}}} \cdot \twa{0}{\frac{( - \thet ) ( - s^{2} \thet^{2} )}{- \thet^{2}}};
\end{multline*}
i.e. by \eqref{e:delh},
\begin{align*}
	\delh{s} &\cdot \delh{t} \cdot \delh{s}^{-1} \\
	&= \twa{0}{s^{2} t \thet} \cdot \twa{0}{- s^{2} \thet} \\
	&= \twa{0}{s^{2} t \thet} \cdot \left[ \twa{0}{\thet}^{- 1} \cdot \twa{0}{\thet} \right] \cdot \twa{0}{s^{2} \thet}^{-1} \\
	&= \delh{s^{2} t} \cdot \delh{s^{2}}^{-1}.
\end{align*}
Thus by \eqref{e:ba}, \eqref{e:prop2} and \eqref{e:prop5},
\begin{equation} \label{e:stcomm}
	\delh{s} \cdot \delh{t} \cdot \delh{s}^{-1} \cdot \delh{t}^{-1} = \hilkn{t}{s^{2}}^{- 1} = \hilkn{s}{t^{2}}.
\end{equation}
Also, since \hilkn{s}{t} is central in $ \wtilde{G} $ for all $ s $, $ t \in k^{\times} $, the above implies that
\begin{multline}  \label{e:stcomm2}
	\delh{s} \cdot \delh{t} \cdot \delh{s}^{-1} \cdot \delh{t}^{-1} \\
	= \delh{t}^{-1} \cdot \delh{s} \cdot \delh{t} \cdot \delh{s}^{-1}.
\end{multline}

\begin{rem} \label{r:dm1}
By \eqref{e:prop5}, $ \hilkn{- 1}{s^{2}} = 1 $ for all $ s \in k^{\times} $.  This implies, by \eqref{e:stcomm} and \eqref{e:stcomm2}, that \delh{- 1} commutes with \delh{s} for all $ s \in k^{\times} $.
\end{rem}

Hence by the above, we can now work out the ``easy'' cases, which are summarised in the following proposition based on Lemma~2.23 of \cite{deodhar78}:

\begin{prop} \label{p:2.23ii}
For $ q \in k^{\times} $, $ \lambda \notin k^{\times} $,
\begin{enumerate}[\upshape (i)]
	\item $ \displaystyle \ba{\lambda}{q} = \hilkn{q}{- \del{2}{\lambda} / \thet} $,
	\item $ \displaystyle \ba{q}{\lambda} = \hilkn{q}{\lambda \ol{\del{1}{\lambda}} / \thet} $.
\end{enumerate}
\end{prop}

\begin{proof}
We first prove \textnormal{(i)}.  We have by \eqref{e:ba} that
	\[ \ba{\lambda}{q} = \delh{\lambda} \cdot \delh{q} \cdot \delh{\lambda q}^{-1}, \]
and by \eqref{e:delh},
\begin{align*}
	\ba{\lambda}{q}
	&=	\begin{aligned}[t]
			&\twa{1}{\del{1}{\lambda}} \cdot \twa{0}{\del{2}{\lambda}}^{-1} \cdot \delh{q} \cdot \twa{0}{\del{2}{\lambda q}} \\
			&\phantom{\ } \cdot \twa{1}{\del{1}{\lambda q}}^{-1}
		\end{aligned} \\
	&=	\begin{aligned}[t]
			&\twa{1}{\del{1}{\lambda}} \cdot \twa{0}{\del{2}{\lambda}}^{-1} \cdot \delh{q} \cdot \twa{0}{\del{2}{\lambda q}} \\
			&\phantom{\ } \cdot \left[ \twa{0}{\thet}^{-1} \cdot \twa{0}{\thet} \right] \cdot \Big[ \twa{0}{q \del{2}{\lambda q}}^{-1} \\
			&\phantom{\ } \cdot \twa{0}{q \del{2}{\lambda q}} \Big] \cdot \twa{1}{\del{1}{\lambda q}}^{-1}
		\end{aligned} \\
	&=	\begin{aligned}[t]
			&\twa{1}{\del{1}{\lambda}} \cdot \twa{0}{\del{2}{\lambda}}^{-1} \cdot \delh{q} \cdot \delh{\del{2}{\lambda q} / \thet} \\
			&\phantom{\ } \cdot \delh{q \del{2}{\lambda q} / \thet}^{-1} \cdot \twa{0}{q \del{2}{\lambda q}} \cdot \twa{1}{\del{1}{\lambda q}}^{-1}.
		\end{aligned}
\end{align*}
Hence by \eqref{e:ba} again,
	\[ \ba{\lambda}{q}
	=	\begin{aligned}[t]
			&\hilkn{q}{\del{2}{\lambda q} / \thet} \cdot \twa{1}{\del{1}{\lambda}} \cdot \twa{0}{\del{2}{\lambda}}^{-1} \\
			&\phantom{\ } \cdot \twa{0}{q \del{2}{\lambda q}} \cdot \twa{1}{\del{1}{\lambda q}}^{-1}.
		\end{aligned}\]
But since $ q \in k^{\times} $ and $ \lambda \notin k^{\times} $, we have by Proposition~\ref{p:deltas} that
\begin{gather}
	\del{1}{\lambda q} = \del{1}{\lambda}, \label{e:del1q} \\
	\del{2}{\lambda q} = ( 1 / q ) \del{2}{\lambda}. \label{e:del2q}
\end{gather}
Thus,
\begin{align*}
	\ba{\lambda}{q}
	&=	\begin{aligned}[t]
			&\hilkn{q}{\del{2}{\lambda} / ( q \thet )} \cdot \twa{1}{\del{1}{\lambda}} \cdot \twa{0}{\del{2}{\lambda}}^{-1} \\
			&\phantom{\ } \cdot \twa{0}{\del{2}{\lambda}} \cdot \twa{1}{\del{1}{\lambda}}^{-1}
		\end{aligned} \\
	&= \hilkn{q}{- \del{2}{\lambda} / \thet}
\end{align*}
by \eqref{e:prop3}.  Similarly for \textnormal{(ii)}, by \eqref{e:ba},
	\[ \ba{q}{\lambda} = \delh{q} \cdot \delh{\lambda} \cdot \delh{\lambda q}^{-1}. \]
Thus,
	\[ \ba{q}{\lambda}
	=	\begin{aligned}[t]
			&\delh{q} \cdot \twa{1}{\del{1}{\lambda}} \cdot \twa{0}{\del{2}{\lambda}}^{-1} \cdot \twa{0}{\del{2}{\lambda q}} \\
			&\phantom{\ } \cdot \twa{1}{\del{1}{\lambda q}}^{-1}
		\end{aligned} \]
by \eqref{e:delh}.  By Proposition~\ref{p:wa},
	\[ \twa{1}{\del{1}{\lambda}} \cdot \twa{0}{\del{2}{\lambda}}^{- 1} \cdot \twa{1}{\del{1}{\lambda}}^{- 1} = \twa{0}{\norm{\del{1}{\lambda}} / \del{2}{\lambda}}. \]
With the above and \eqref{e:del2q}, the equation becomes
	\[ \ba{\lambda}{q}
=	\begin{aligned}[t]
			&\delh{q} \cdot \twa{0}{\frac{\norm{\del{1}{\lambda}}}{\del{2}{\lambda}}} \cdot \twa{1}{\del{1}{\lambda}} \\
			&\phantom{\ } \cdot \twa{0}{\del{2}{\lambda} / q} \cdot \twa{1}{\del{1}{\lambda}}^{-1}.
		\end{aligned} \]
Also by Proposition~\ref{p:wa},
\begin{align*}
	\twa{1}{\del{1}{\lambda}} \cdot \twa{0}{\del{2}{\lambda} / q} \cdot \twa{1}{\del{1}{\lambda}}^{-1}
	&= \twa{0}{\norm{\del{1}{\lambda}} q / \ol{\del{2}{\lambda}}} \\
	&= \twa{0}{\norm{\del{1}{\lambda}} q / \del{2}{\lambda}}^{- 1}.
\end{align*}
This implies by \eqref{e:delh} that
\begin{align*}
	\ba{\lambda}{q}
	&= \delh{q} \cdot \twa{0}{\frac{\norm{\del{1}{\lambda}}}{\del{2}{\lambda}}} \cdot \twa{0}{\frac{\norm{\del{1}{\lambda}} q}{\del{2}{\lambda}}}^{- 1} && \\
	&=	\begin{aligned}[t]
			&\delh{q} \cdot \twa{0}{\frac{\norm{\del{1}{\lambda}}}{\del{2}{\lambda}}} \cdot \left[ \twa{0}{\thet}^{- 1} \cdot \twa{0}{\thet} \right] \\
			&\phantom{\ } \cdot \twa{0}{\frac{\norm{\del{1}{\lambda}} q}{\del{2}{\lambda}}}^{- 1}
		\end{aligned} && \\
	&= \delh{q} \cdot \delh{\lambda \ol{\del{1}{\lambda}} / \thet} \cdot \delh{\lambda \ol{\del{1}{\lambda}} q / \thet}^{-1},
\end{align*}
since $ \lambda = \del{1}{\lambda} / \del{2}{\lambda} $.  Hence,
	\[ \ba{\lambda}{q} = \hilkn{q}{\lambda \ol{\del{1}{\lambda}} / \thet} \]
by \eqref{e:ba}.
\end{proof}

Thus we are left with finding an explicit expression for \ba{\lambda}{\mu} where $ \lambda $, $ \mu \notin k^{\times} $, in terms of \hilkn{s}{t}'s, with $ s $, $ t \in k^{\times} $.  But before proving the main theorem, we must first establish a few lemmas.

\section{Some useful lemmas} \label{s:useful}

Lemma~2.15 of \cite{deodhar78} states (with some change in notation) that

\begin{lem} \label{l:2.15}
For $ ( r, m ) $, $ ( 0, m' ) \in A $ (see \eqref{e:defA1}),
\begin{enumerate}[\upshape (i)]

	\item $ \twa{r}{m} \cdot \twa{0}{m'} = \twa{u}{v} \cdot \twa{u'}{v'} $, with
		\[ u = \frac{r m m'}{\ol{m} ( m + m' )}, \quad v = \frac{m m'}{m + m'}, \quad u' = \frac{r m'}{\ol{m} + \ol{m'}}, \quad v' = \frac{\norm{m'}}{\ol{m} + \ol{m'}}. \]
	
	\item $ \twa{0}{m'} \cdot \twa{r}{m} = \twa{u}{v} \cdot \twa{u''}{v''} $, with $ u $, $ v $ the same as in \textnormal{(i)} and
		\[ u'' = \frac{r \ol{m}}{\ol{m} + \ol{m'}}, \quad v'' = \frac{\norm{m}}{\ol{m} + \ol{m'}}. \]

\end{enumerate}
\end{lem}

Thus, using the above, we will prove the following:

\begin{lem} \label{l:2.16}
For a general $ ( s_{1}, n_{1} ) \in A $, $ t \in K^{\times} $,
\begin{multline*}
\twa{s_{1}}{n_{1}}^{-1} \cdot \twa{\frac{s_{1}}{t}}{\frac{n_{1}}{\norm{t}}} \\
=
	\begin{cases}
		\displaystyle \delh{- \frac{n_{1}}{\thet}} \cdot \delh{- \frac{n_{1}}{\norm{t} \thet}}^{-1}, & \text{if $ n_{1} \in k^{\times} \thet $;} \\
		\displaystyle	\begin{aligned}[b]
				&\delh{- \frac{\norm{n_{1}}}{\thet^{2}}} \cdot \delh{- \frac{\norm{n_{1}}}{t \thet^{2}}}^{-1} \\
				&\phantom{\quad} \cdot \delh{t},
			\end{aligned} &\text{if $ t \in k^{\times} $;} \\
		\displaystyle \hilkn{- \frac{n_{1}}{t \thet}}{\norm{t}} \cdot \delh{\norm{t}}, &\text{if $ n_{1} \in k^{\times} $, $ t \in k^{\times} \thet $.}
	\end{cases}
\end{multline*}
Otherwise for $ n_{1} = a + b \thet $, $ t = c + d \thet $, $ a $, $ d \in k^{\times} $, $ b $, $ c \in k $, if
	\[ q_{1} = c + \frac{b d \thet^{2}}{a} \neq 0, \]
then
\begin{multline*}
	\twa{s_{1}}{n_{1}}^{-1} \cdot \twa{\frac{s_{1}}{t}}{\frac{n_{1}}{\norm{t}}} \\
	= \hilkn{- \frac{a^{2} \norm{t}}{d^{2} \norm{n_{1}} \thet^{2}}}{q_{1}} \cdot \hilkn{- \frac{d \norm{n_{1}}}{a}}{\norm{t}} \cdot \delh{\norm{t}}.
\end{multline*}
If $ q_{1} = 0 $, then we let
	\[ s' = \frac{s_{1}}{t}, \quad n' = \frac{n_{1}}{\norm{t}} \]
so that
	\[ \twa{s_{1}}{n_{1}}^{-1} \cdot \twa{\frac{s_{1}}{t}}{\frac{n_{1}}{\norm{t}}} = \left[ \twa{s'}{n'}^{-1} \cdot \twa{s' t}{n' \norm{t}} \right]^{-1}, \]
and calculate $ \twa{s'}{n'}^{-1} \cdot \twa{s' t}{n' \norm{t}} $ instead, using the above results.
\end{lem}

\begin{proof}
Recall from Proposition~\ref{p:wa} that for any $ ( r, m ) \in A $ (where $ A $ is as in \eqref{e:defA1}),
\begin{equation} \label{e:wainv}
	\twa{r}{m}^{-1} = \twa{- r}{\ol{m}}.
\end{equation}
If $ n_{1} \in k^{\times} \thet $, then let $ n_{1} = b \thet $, where $ b \in k^{\times} $.  Then $ s_{1} = 0 $, and by \eqref{e:wainv},
\begin{align*}
	\twa{s_{1}}{n_{1}}^{-1} \cdot \twa{\frac{s_{1}}{t}}{\frac{n_{1}}{\norm{t}}}
	&= \twa{0}{b \thet}^{-1} \cdot \twa{0}{\frac{b \thet}{\norm{t}}} \\
	&= \twa{0}{- b \thet} \cdot \twa{0}{- \frac{b \thet}{\norm{t}}}^{-1}.
\end{align*}
Hence by \eqref{e:delh},
\begin{align*}
	&\twa{s_{1}}{n_{1}}^{-1} \cdot \twa{\frac{s_{1}}{t}}{\frac{n_{1}}{\norm{t}}} \\
	&= \twa{0}{- b \thet} \cdot \left[ \twa{0}{\thet}^{- 1} \cdot \twa{0}{\thet} \right] \cdot \twa{0}{- \frac{b \thet}{\norm{t}}}^{-1} \\
	&= \delh{- b} \cdot \delh{- \frac{b}{\norm{t}}}^{-1} \\
	&= \delh{- \frac{n_{1}}{\thet}} \cdot \delh{- \frac{n_{1}}{\norm{t} \thet}}^{-1}.
\end{align*}

So assume that $ n_{1} \notin k^{\times} \thet $, so that $ s_{1} \neq 0 $.  Let $ t \in k^{\times} $.  Then we know by \eqref{e:wainv} and applying Proposition~\ref{p:wa} twice for $ v $, $ v' \in k^{\times} $, that
\begin{align}
	[ \twa{0}{v \thet} & \cdot \twa{0}{v' \thet} ] \cdot \twa{s_{1}}{n_{1}} \cdot [ \twa{0}{v \thet} \cdot \twa{0}{v' \thet} ]^{-1} \label{e:snvv} \\
	&= \twa{\frac{s_{1} ( v' \thet ) ( v \thet )^{2}}{( - v' \thet )^{2} ( - v \thet )}}{\frac{n_{1} ( - v^{2} \thet^{2} )}{- v'^{2} \thet^{2}}} \notag \\
	&= \twa{- \frac{s_{1} v}{v'}}{\frac{n_{1} v^{2}}{v'^{2}}}. \notag
\end{align}
By Proposition~\ref{p:wa}, for $ t' \in k^{\times} $,
\begin{equation} \label{e:wa4k}
	\twa{0}{- t' \thet} \cdot \twa{s_{1}}{n_{1}} = \twa{s_{1}}{n_{1}} \cdot \twa{0}{\frac{\norm{n_{1}}}{t' \thet}}.
\end{equation}
This implies that if we let $ v = 1 $, $ v' = - t $, then \eqref{e:snvv} becomes
	\[ \twa{\frac{s_{1}}{t}}{\frac{n_{1}}{t^{2}}} = \left[ \twa{0}{\thet} \cdot \twa{0}{- t \thet} \right] \cdot \twa{s_{1}}{n_{1}} \cdot \left[ \twa{0}{\thet} \cdot \twa{0}{- t \thet} \right]^{-1}. \]
By using \eqref{e:wa4k} twice, we get
	\[ \twa{\frac{s_{1}}{t}}{\frac{n_{1}}{t^{2}}} = \twa{s_{1}}{n_{1}} \cdot \twa{0}{\frac{\norm{n_{1}}}{- \thet}} \cdot \twa{0}{\frac{\norm{n_{1}}}{t \thet}} \cdot \twa{0}{t \thet} \cdot \twa{0}{- \thet}. \]
Hence,
\begin{align*}
	\twa{\frac{s_{1}}{t}}{\frac{n_{1}}{t^{2}}}
	&=	\begin{aligned}[t]
			&\twa{s_{1}}{n_{1}} \cdot \twa{0}{\frac{\norm{n_{1}}}{- \thet}} \cdot \left[ \twa{0}{\thet}^{- 1} \cdot \twa{0}{\thet} \right] \\
			&\phantom{\ } \cdot \twa{0}{- \frac{\norm{n_{1}}}{t \thet}}^{- 1} \cdot \twa{0}{t \thet} \cdot \twa{0}{- \thet}
		\end{aligned} \\
	&= \twa{s_{1}}{n_{1}} \cdot \delh{- \frac{\norm{n_{1}}}{\thet^{2}}} \cdot \delh{- \frac{\norm{n_{1}}}{t \thet^{2}}}^{-1} \cdot \delh{t}
\end{align*}
by \eqref{e:wainv} and \eqref{e:delh}; i.e.
\begin{multline*}
	\twa{s_{1}}{n_{1}}^{-1} \cdot \twa{\frac{s_{1}}{t}}{\frac{n_{1}}{t^{2}}} \\
	= \delh{- \frac{\norm{n_{1}}}{\thet^{2}}} \cdot \delh{- \frac{\norm{n_{1}}}{t \thet^{2}}}^{-1} \cdot \delh{t}.
\end{multline*}
(Note that the above also applies to $ n_{1} \in k^{\times} \thet $: for if $ n_{1} = b \thet $, $ b \in  k^{\times} $, by using \eqref{e:ba} several times,
\begin{align*}
	&\twa{0}{b \thet}^{-1} \cdot \twa{0}{\frac{b \thet}{t^{2}}} \\
	&= \delh{- \frac{- b^{2} \thet^{2}}{\thet^{2}}} \cdot \delh{- \frac{- b^{2} \thet^{2}}{t \thet^{2}}}^{-1} \cdot \delh{t} \\
	&=	\begin{aligned}[t]
		&\delh{b^{2}} \cdot \delh{\frac{b^{2}}{t}}^{- 1} \cdot \left[ \delh{t}^{- 1} \cdot \delh{t} \right] \\
		&\phantom{\ } \cdot \delh{t} \cdot \left[ \delh{t^{2}}^{- 1} \cdot \delh{t^{2}} \right] \\
		&\phantom{\ } \cdot \left[ \delh{- \frac{b}{t^{2}}} \cdot \delh{- b}^{- 1} \cdot \delh{- b} \cdot \delh{- \frac{b}{t^{2}}}^{-1} \right]
	\end{aligned} \\
	&= \hilkn{t}{\frac{b^{2}}{t}}^{- 1} \cdot \hilkn{t}{t} \cdot \hilkn{t^{2}}{- \frac{b}{t^{2}}} \cdot \delh{- b} \cdot \delh{- \frac{b}{t^{2}}}^{-1}.
\end{align*}
Now it is a matter of using \eqref{e:bastr}, \eqref{e:prop3} and \eqref{e:prop5} to get
\begin{align*}
	&\twa{0}{b \thet}^{-1} \cdot \twa{0}{\frac{b \thet}{t^{2}}} \\
	&= \hilkn{t}{- b^{- 2}} \cdot \hilkn{t}{- 1} \cdot \hilkn{t^{2}}{b} \cdot \delh{- b} \cdot \delh{- \frac{b}{t^{2}}}^{-1}\\
	&= \hilkn{t^{- 2}}{b} \cdot \hilkn{t^{2}}{b} \cdot \delh{- b} \cdot \delh{- \frac{b}{t^{2}}}^{-1} \\
	&= \delh{- b} \cdot \delh{- \frac{b}{t^{2}}}^{-1}.
\end{align*}
Thus the two abovementioned cases coincide for $ n_{1} \in k^{\times} \thet $, $ t \in k^{\times} $.)

Let $ t \notin k^{\times} $, i.e. $ t = c + d \thet $ with $ c \in k $, $ d \in k^{\times} $.  Also, let $ n_{1} = a + b \thet $, with $ a \in k^{\times} $, $ b \in k $.

It is sufficient to find an explicit formula for
	\[ \twa{s_{1} q}{n_{1} q^{2}}^{-1} \cdot \twa{\frac{s_{1}}{t}}{\frac{n_{1}}{\norm{t}}}, \]
for some $ q \in k^{\times} $,
since we know from previous calculation that
\begin{multline} \label{e:nq2}
	\twa{s_{1}}{n_{1}}^{-1} \cdot \twa{s_{1} q}{n_{1} q^{2}} \\
	= \delh{- \frac{\norm{n_{1}}}{\thet^{2}}} \cdot \delh{- \frac{\norm{n_{1}} q}{\thet^{2}}}^{-1} \cdot \delh{\frac{1}{q}}.
\end{multline}

By the proof of Lemma~2.16 of \cite{deodhar78}, assume that we can choose $ q \in k^{\times} $ such that
	\[ \trace{( q t - 1 ) \frac{n_{1}}{\norm{t}}} = 0. \]
Then,
\begin{align*}
( q t - 1 ) n_{1} &= ( q ( c + d \thet) - 1 ) ( a + b \thet ) \\
&= ( ( q c - 1 ) + q d \thet ) ( a + b \thet ) \\
&= ( ( q c - 1 ) a + q b d \thet^{2} ) + ( q a d + ( q c - 1 ) b ) \thet.
\end{align*}
Hence, since $ ( q c - 1 ) a + q b d \thet^{2} = 0 $,
	\[ q^{-1} = c + \frac{b d \thet^{2}}{a}. \]
Let $ m = ( q \ol{t} - 1 ) \ol{n_{1}} / ( \norm{t} ) $, $ l_{1} = q ( q t - 1 ) n_{1} / t  $ and $ p_{1} = ( q t - 1 ) s_{1} / t $.  Then $ ( 0, m ) $, $ ( p_{1}, l_{1} ) \in A $, hence \twa{p_{1}}{l_{1}} and \twa{0}{m} are well-defined.  Using Lemma~\ref{l:2.15}, we know that
\begin{align*}
	\frac{p_{1} m}{\ol{l_{1}} + \ol{m}} &= \frac{( q t - 1 ) ( s_{1} / t ) m}{( q t - 1 ) m} = \frac{s_{1}}{t}, \\
	\frac{\norm{m}}{\ol{l_{1}} + \ol{m}} &= \frac{\norm{m}}{( q t - 1 ) m} = \frac{\ol{m}}{( q t - 1 )} = \frac{n_{1}}{\norm{t}}, \\
	\frac{p_{1} \ol{l_{1}}}{\ol{l_{1}} + \ol{m}} &= \frac{p_{1} m}{\ol{l_{1}} + \ol{m}} \cdot \frac{\ol{l_{1}}}{m} = \frac{s_{1}}{t} \cdot ( q t ) = s_{1} q, \\
	\frac{\norm{l_{1}}}{\ol{l_{1}} + \ol{m}} &= \frac{\norm{m}}{\ol{l_{1}} + \ol{m}} \cdot \frac{\norm{l_{1}}}{\norm{m}} = \frac{n_{1}}{\norm{t}} \cdot ( q t ) ( q \ol{t} ) = n_{1} q^{2},
\end{align*}
so
\begin{multline*}
	\twa{s_{1} q}{n_{1} q^{2}}^{-1} \cdot \twa{\frac{s_{1}}{t}}{\frac{n_{1}}{\norm{t}}} \\
	= \twa{p_{1}}{l_{1}}^{-1} \cdot \twa{0}{m}^{-1} \cdot \twa{p_{1}}{l_{1}} \cdot \twa{0}{m}
\end{multline*}
by Lemma~\ref{l:2.15}.  By Proposition~\ref{p:wa},
	\[ \twa{p_{1}}{l_{1}}^{-1} \cdot \twa{0}{m}^{-1} \cdot \twa{p_{1}}{l_{1}} = \twa{0}{\norm{l_{1}} / \ol{m}}^{-1}, \]
thus by \eqref{e:wainv},
\begin{align*}
	&\twa{s_{1} q}{n_{1} q^{2}}^{-1} \cdot \twa{\frac{s_{1}}{t}}{\frac{n_{1}}{\norm{t}}} \\
	&= \twa{0}{\frac{\norm{l_{1}}}{\ol{m}}}^{-1} \cdot \twa{0}{m} \\
	&= \twa{0}{q^{2} ( q \ol{t} - 1 ) \ol{n_{1}}}^{-1} \cdot \twa{0}{( q \ol{t} - 1 ) \frac{\ol{n_{1}}}{\norm{t}}} \\
	&= \twa{0}{q^{2} (q t - 1 ) n_{1}} \cdot \twa{0}{( q t - 1 ) \frac{n_{1}}{\norm{t}}}^{-1}.
\end{align*}
Also,
\begin{align*}
	&\twa{s_{1} q}{n_{1} q^{2}}^{-1} \cdot \twa{\frac{s_{1}}{t}}{\frac{n_{1}}{\norm{t}}} \\
	&= \twa{0}{q^{2} (q t - 1 ) n_{1}} \cdot \left[ \twa{0}{\thet}^{- 1} \cdot \twa{0}{\thet} \right] \cdot \twa{0}{( q t - 1 ) \frac{n_{1}}{\norm{t}}}^{-1} \\
	&= \delh{\frac{q^{2} ( q t - 1 ) n_{1}}{\thet}} \cdot \delh{ \frac{( q t - 1 ) n_{1}}{\norm{t} \thet}}^{-1}
\end{align*}
by \eqref{e:delh}.  This implies, using \eqref{e:nq2}, that
\begin{align*}
	&\twa{s_{1}}{n_{1}}^{-1} \cdot \twa{\frac{s_{1}}{t}}{\frac{n_{1}}{\norm{t}}} \\
	&= 	\begin{aligned}[t]
			&\left[ \delh{- \frac{\norm{n_{1}}}{\thet^{2}}} \cdot \delh{- \frac{\norm{n_{1}} q}{\thet^{2}}}^{- 1} \cdot \delh{\frac{1}{q}} \right] \\
			&\phantom{\ } \cdot \left[ \delh{\frac{q^{2} ( q t - 1 ) n_{1}}{\thet}} \cdot \delh{ \frac{( q t - 1 ) n_{1}}{\norm{t} \thet}}^{-1} \right].
		\end{aligned}
\end{align*}
Also, by \eqref{e:delinv},
	\[ \delh{- \frac{\norm{n_{1}} q}{\thet^{2}}}^{- 1} = \delh{\frac{\thet^{2}}{\norm{n_{1}} q}} \cdot \delh{- 1}^{-1}. \]
Hence,
\begin{multline*}
	\twa{s_{1}}{n_{1}}^{-1} \cdot \twa{\frac{s_{1}}{t}}{\frac{n_{1}}{\norm{t}}} \\
	=	\begin{aligned}[t]
			&\delh{- \frac{\norm{n_{1}}}{\thet^{2}}} \cdot \left[ \delh{\frac{\thet^{2}}{\norm{n_{1}} q}} \cdot \delh{-1}^{-1} \right] \cdot \delh{\frac{1}{q}} \\ 
			&\phantom{\ } \cdot \delh{\frac{q^{2} ( q t - 1 ) n_{1}}{\thet}} \cdot \delh{ \frac{( q t - 1 ) n_{1}}{\norm{t} \thet}}^{-1}.
		\end{aligned}
\end{multline*}
Since by Remark~\ref{r:dm1},
	\[ \delh{- 1}^{- 1} \cdot \delh{t'} = \delh{t'} \cdot \delh{- 1}^{- 1} \]
for all $ t' \in k^{\times} $, this implies that
\begin{multline*}
	\twa{s_{1}}{n_{1}}^{-1} \cdot \twa{\frac{s_{1}}{t}}{\frac{n_{1}}{\norm{t}}} \\
	=	\begin{aligned}[t]
			&\delh{- \frac{\norm{n_{1}}}{\thet^{2}}} \cdot \delh{\frac{\thet^{2}}{\norm{n_{1}} q}}\cdot \delh{\frac{1}{q}} \\
			&\phantom{\ } \cdot \delh{\frac{q^{2} ( q t - 1 ) n_{1}}{\thet}} \cdot \delh{-1}^{-1} \cdot \delh{ \frac{( q t - 1 ) n_{1}}{\norm{t} \thet}}^{-1}.
		\end{aligned}
\end{multline*}
Use \eqref{e:ba} several times to get
\begin{align*}
	&\twa{s_{1}}{n_{1}}^{-1} \cdot \twa{\frac{s_{1}}{t}}{\frac{n_{1}}{\norm{t}}} \\
	&=	\begin{aligned}[t]
			&\delh{- \frac{\norm{n_{1}}}{\thet^{2}}} \cdot \delh{\frac{\thet^{2}}{\norm{n_{1}} q}}\cdot \left[ \delh{- \frac{1}{q}}^{- 1} \cdot \delh{- \frac{1}{q}} \right] \\
			&\phantom{\ } \cdot \delh{\frac{1}{q}} \cdot \left[ \delh{- \frac{1}{q^{2}}}^{- 1} \cdot \delh{- \frac{1}{q^{2}}} \right] \\
			&\phantom{\ } \cdot \delh{\frac{q^{2} ( q t - 1 ) n_{1}}{\thet}} \cdot \Bigg[ \delh{- \frac{( q t - 1 ) n_{1}}{\thet}}^{- 1} \\
			&\phantom{\ } \cdot \delh{- \frac{( q t - 1 ) n_{1}}{\thet}} \Bigg] \cdot \delh{- 1}^{- 1} \cdot \Bigg[ \delh{\frac{( q t - 1 ) n_{1}}{\thet}}^{- 1} \\
			&\phantom{\ } \cdot \delh{\frac{( q t - 1 ) n_{1}}{\thet}} \Bigg] \cdot \delh{ \frac{( q t - 1 ) n_{1}}{\norm{t} \thet}}^{- 1} \\
			&\phantom{\ } \cdot \left[ \delh{\norm{t}}^{- 1} \cdot\delh{\norm{t}} \right]
		\end{aligned} \\
	&=	\begin{aligned}[t]
			&\hilkn{- \frac{\norm{n_{1}}}{\thet^{2}}}{\frac{\thet^{2}}{\norm{n_{1}} q}} \cdot \hilkn{- \frac{1}{q}}{\frac{1}{q}} \cdot \hilkn{- \frac{1}{q^{2}}}{\frac{q^{2} ( q t - 1 ) n_{1}}{\thet}} \\ &\phantom{\ } \cdot \hilkn{\frac{( q t - 1 ) n_{1}}{\thet}}{-1}^{- 1} \cdot \hilkn{\norm{t}}{\frac{( q t - 1 ) n_{1}}{\norm{t} \thet}}^{- 1} \cdot \delh{\norm{t}}.
		\end{aligned}
\end{align*}
By using \eqref{e:prop2}, \eqref{e:prop3}, \eqref{e:prop5} and \eqref{e:bastr} in the above,
\begin{multline*}
	\twa{s_{1}}{n_{1}}^{-1} \cdot \twa{\frac{s_{1}}{t}}{\frac{n_{1}}{\norm{t}}} \\
	= \hilkn{- \frac{\ol{n_{1}}}{( q t - 1 )^{2} n_{1}}}{\frac{1}{q}} \cdot \hilkn{- \frac{( q t - 1 ) n_{1}}{\thet}}{\norm{t}} \cdot \delh{\norm{t}}.
\end{multline*}
Since $ q t - 1 = \ol{n_{1}} d \thet / ( a c + b d \thet^{2} ) $,
\begin{align*}
	&\twa{s_{1}}{n_{1}}^{-1} \cdot \twa{\frac{s_{1}}{t}}{\frac{n_{1}}{\norm{t}}} \\
	&=	\begin{aligned}[t]
			&\hilkn{- \frac{\ol{n_{1}}}{( \ol{n_{1}} d \thet / ( a c + b d \thet^{2} ) )^{2} n_{1}}}{\frac{a c + b d \thet^{2}}{a}} \\
			&\phantom{\ } \cdot \hilkn{- \frac{( \ol{n_{1}} d \thet / ( a c + b d \thet^{2} ) ) n_{1}}{\thet}}{\norm{t}} \cdot \delh{\norm{t}}
		\end{aligned} \\
	&= \hilkn{- \frac{( a c + b d \thet^{2} )^{2}}{d^{2} \norm{n_{1}} \thet^{2}}}{\frac{a c + b d \thet^{2}}{a}} \cdot \hilkn{- \frac{d \norm{n_{1}}}{a c + b d \thet^{2}}}{\norm{t}} \cdot \delh{\norm{t}}.
\end{align*}
By \eqref{e:prop3}, $ \hilkn{s^{- 2}}{s} = 1 $ for all $ s \in k^{\times} $, hence by also using \eqref{e:bastr},
\begin{align*}
	&\twa{s_{1}}{n_{1}}^{-1} \cdot \twa{\frac{s_{1}}{t}}{\frac{n_{1}}{\norm{t}}} \\
	&=	\begin{aligned}[t]
			&\hilkn{- \frac{a^{2}}{d^{2} \norm{n_{1}} \thet^{2}}}{\frac{a c + b d \thet^{2}}{a}} \cdot \hilkn{\frac{a}{a c + b d \thet^{2}}}{\norm{t}} \\
			&\phantom{\ } \cdot \hilkn{- \frac{d \norm{n_{1}}}{a}}{\norm{t}} \cdot \delh{\norm{t}}.
		\end{aligned}
\end{align*}
Lastly, by \eqref{e:prop2} and \eqref{e:bastr},
\begin{multline*}
	\twa{s_{1}}{n_{1}}^{-1} \cdot \twa{\frac{s_{1}}{t}}{\frac{n_{1}}{\norm{t}}} \\
	= \hilkn{- \frac{a^{2} \norm{t}}{d^{2} \norm{n_{1}} \thet^{2}}}{\frac{a c + b d \thet^{2}}{a}} \cdot \hilkn{- \frac{d \norm{n_{1}}}{a}}{\norm{t}} \cdot \delh{\norm{t}}.
\end{multline*}

Thus, the above method applies for $ t \notin k^{\times} $, $ n_{1} \notin k^{\times} \thet $ such that
	\[ q_{1} := c + \frac{b d \thet^{2}}{a} \neq 0, \]
so that
\begin{equation} \label{e:l2.16a}
	\twa{s_{1}}{n_{1}}^{-1} \cdot \twa{\frac{s_{1}}{t}}{\frac{n_{1}}{\norm{t}}}
	=	\begin{aligned}[t]
			&\hilkn{- \frac{a^{2} \norm{t}}{d^{2} \norm{n_{1}} \thet^{2}}}{q_{1}} \cdot \hilkn{- \frac{d \norm{n_{1}}}{a}}{\norm{t}} \\
			&\phantom{\ } \cdot \delh{\norm{t}}.
		\end{aligned}
\end{equation}

If $ q_{1} = 0 $, for the majority of these cases, we can interchange $ ( s_{1}, n_{1} ) $ and $ ( s_{1} / t, n_{1} / ( \norm{t} ) ) $ to get the result, i.e. let
	\[ s' = \frac{s_{1}}{t}, \quad n' = \frac{n_{1}}{\norm{t}} \]
and apply the above to
	\[ \twa{s'}{n'}^{-1} \cdot \twa{s' t}{n' \norm{t}}. \]

The only exception is when $ c = b d \thet^{2} / a = 0 $, i.e. when $ c = 0 $, $ b = 0 $.  Hence, $ t = d \thet $ and $ n_{1} \in k^{\times} $, where $ d \in k^{\times} $.  Let $ t' = 1 + \thet $, and let $ s'' $, $ n'' \in K^{\times} $ be such that
	\[ \frac{s''}{s_{1} / t} = t', \quad \frac{n''}{n_{1} / ( \norm{t} )} = \norm{t'}; \]
i.e.
	\[ s'' = s_{1} \cdot \frac{t'}{t}, \quad n'' = n_{1} \cdot \frac{\norm{t'}}{\norm{t}}. \]
Since $ s'' / ( s_{1} / t ) $, $ s_{1} / s'' \notin k^{\times} $, $ k^{\times} \thet $, we will be able to work out $ \twa{s_{1}}{n_{1}}^{-1} \cdot \twa{s''}{n''} $ and $ \twa{s''}{n''}^{-1} \cdot \twa{s_{1} / t}{n_{1} / ( \norm{t} )} $ using \eqref{e:l2.16a}; then
\begin{multline*}
	\twa{s_{1}}{n_{1}}^{-1} \cdot \twa{\frac{s_{1}}{t}}{\frac{n_{1}}{\norm{t}}} \\
	= \left[ \twa{s_{1}}{n_{1}}^{-1} \cdot \twa{s''}{n''} \right] \cdot \left[ \twa{s''}{n''}^{-1} \cdot \twa{\frac{s_{1}}{t}}{\frac{n_{1}}{\norm{t}}} \right].
\end{multline*}
Let us calculate $ \twa{s_{1}}{n_{1}}^{-1} \cdot \twa{s''}{n''} $ first: since
	\[ \frac{s_{1}}{s''} = \frac{t}{t'} = \frac{d \thet}{1 + \thet} = \frac{- d \thet^{2} + d \thet}{\norm{t'}}, \quad n_{1} \in k^{\times}, \]
then in this case, using \eqref{e:l2.16a},
	\[ q_{1} = \frac{- d \thet^{2}}{1 - \thet^{2}} = - \frac{t \thet}{\norm{t'}}, \]
and therefore,
\begin{align*}
	&\twa{s_{1}}{n_{1}}^{-1} \cdot \twa{s''}{n''} \\
	&= \hilkn{- \frac{n_{1}^{2} \norm{t} / ( \norm{t'} )}{( d / ( \norm{t'} ) )^{2} n_{1}^{2} \thet^{2}}}{q_{1}} \cdot \hilkn{- \frac{( d / ( \norm{t'} ) ) n_{1}^{2}}{n_{1}}}{\frac{\norm{t}}{\norm{t'}}} \cdot \delh{\frac{\norm{t}}{\norm{t'}}} \\
	&= \hilkn{\norm{t'}}{- \frac{t \thet}{\norm{t'}}} \cdot \hilkn{- \frac{t n_{1}}{\norm{t'} \thet}}{\frac{\norm{t}}{\norm{t'}}} \cdot \delh{\frac{\norm{t}}{\norm{t'}}}.
\end{align*}
Thus,
\begin{multline} \label{e:l2.16b}
	\twa{s_{1}}{n_{1}}^{-1} \cdot \twa{s''}{n''} \\
	= \hilkn{\norm{t'}}{t \thet} \cdot \hilkn{- \frac{n_{1}}{t \thet}}{\frac{\norm{t}}{\norm{t'}}} \cdot \delh{\frac{\norm{t}}{\norm{t'}}}
\end{multline}
by \eqref{e:prop3}.  As for $ \twa{s''}{n''}^{- 1} \cdot \twa{s_{1} / t}{n_{1} / ( \norm{t} )} $,
	\[ \frac{s''}{s_{1} / t} = t' = 1 + \thet, \quad n'' = \frac{n_{1} \norm{t'}}{\norm{t}} \in k^{\times}, \]
which implies that by \eqref{e:l2.16a}, $ q_{1} = 1 $ so that
\begin{multline*}
	\twa{s''}{n''}^{-1} \cdot \twa{\frac{s_{1}}{t}}{\frac{n_{1}}{\norm{t}}} \\
	= \hilkn{- \frac{n''^{2} \norm{t'}}{1^{2} n''^{2} \thet^{2}}}{1} \cdot \hilkn{- \frac{1 \cdot n''^{2}}{n''}}{\norm{t'}} \delh{\norm{t'}},
\end{multline*}
i.e.
\begin{equation} \label{e:l2.16c}
	\twa{s''}{n''}^{-1} \cdot \twa{\frac{s_{1}}{t}}{\frac{n_{1}}{\norm{t}}} = \hilkn{- \frac{n_{1} \norm{t'}}{\norm{t}}}{\norm{t'}} \cdot \delh{\norm{t'}}.
\end{equation}
Therefore, by \eqref{e:l2.16b} and \eqref{e:l2.16c},
\begin{align*}
	&\twa{s_{1}}{n_{1}}^{-1} \cdot \twa{\frac{s_{1}}{t}}{\frac{n_{1}}{\norm{t}}} \\
	&= \left[ \twa{s_{1}}{n_{1}}^{-1} \cdot \twa{s''}{n''} \right] \cdot \left[ \twa{s''}{n''}^{-1} \cdot \twa{\frac{s_{1}}{t}}{\frac{n_{1}}{\norm{t}}} \right] \\
	&=	\begin{aligned}[t]
			&\left[ \hilkn{\norm{t'}}{t \thet} \cdot \hilkn{- \frac{n_{1}}{t \thet}}{\frac{\norm{t}}{\norm{t'}}} \cdot \delh{\frac{\norm{t}}{\norm{t'}}} \right] \\
			&\phantom{\ } \cdot \left[ \hilkn{- \frac{n_{1} \norm{t'}}{\norm{t}}}{\norm{t'}} \cdot \delh{\norm{t'}} \right];
		\end{aligned}
\end{align*}
and by \eqref{e:bastr} and \eqref{e:prop5},
\begin{multline*}
	\twa{s_{1}}{n_{1}}^{-1} \cdot \twa{\frac{s_{1}}{t}}{\frac{n_{1}}{\norm{t}}} \\
	=	\begin{aligned}[t]
			&\hilkn{\frac{1}{t \thet}}{\norm{t'}} \cdot \hilkn{- \frac{n_{1}}{t \thet}}{\norm{t}} \cdot \hilkn{- \frac{n_{1}}{t \thet}}{\norm{t'}}^{- 1} \\
			&\phantom{\ } \cdot \hilkn{- \frac{n_{1} \norm{t'}}{\norm{t}}}{\norm{t'}} \cdot \delh{\frac{\norm{t}}{\norm{t'}}} \cdot \delh{\norm{t'}} \\
			&\phantom{\ } \cdot \left[ \delh{\norm{t}}^{-1} \cdot \delh{\norm{t}} \right].
		\end{aligned}
\end{multline*}
Again by \eqref{e:bastr} as well as \eqref{e:ba},
\begin{multline*}
	\twa{s_{1}}{n_{1}}^{-1} \cdot \twa{\frac{s_{1}}{t}}{\frac{n_{1}}{\norm{t}}} \\
	= \hilkn{\frac{\norm{t'}}{\norm{t}}}{\norm{t'}} \cdot \hilkn{- \frac{n_{1}}{t \thet}}{\norm{t}} \cdot \hilkn{\frac{\norm{t}}{\norm{t'}}}{\norm{t'}} \cdot \delh{\norm{t}}.
\end{multline*}
Hence we have
	\[ \twa{s_{1}}{n_{1}}^{-1} \cdot \twa{\frac{s_{1}}{t}}{\frac{n_{1}}{\norm{t}}} =	\hilkn{- \frac{n_{1}}{t \thet}}{\norm{t}} \cdot \delh{\norm{t}} \]
by \eqref{e:bastr} again.  Thus the lemma is proved.
\end{proof}

\begin{lem} \label{l:2.18}
If $ ( r, m ) $, $ ( r', m ) \in A $, then
	\[ \frac{r'}{r} = \frac{\ol{t}}{t}, \]
for some (non-unique) $ t \in K^{\times} $.

Hence, if $ r' / r \neq 1 $ and $ m = a + b \thet $, $ t = c + d \thet $ such that $ a $, $ b $, $ c $, $ d \in k $, then
\begin{multline*}
	\twa{r}{m} \cdot \twa{r'}{m}^{-1} \\
	=	\begin{cases}
			\displaystyle \hilkn{- 1}{- \frac{\norm{m}}{\thet^{2}}}, &\text{if $ c = 0 $;} \\
			\displaystyle \hilkn{\frac{a^{2} \norm{t}}{d^{2} \norm{m} \thet^{2}}}{\frac{\left( a c - b d \thet^{2} \right)^{2}}{a^{2} \norm{t}}}, &\text{if $ c \neq 0 $, $ a c - b d \thet^{2} \neq 0 $;} \\
			\displaystyle \hilkn{\frac{\left( a c + b d \thet^{2} \right)^{2}}{a^{2} \norm{t}}}{\frac{a^{2} \norm{t}}{d^{2} \norm{m} \thet^{2}}}, &\text{if $ c \neq 0 $, $ a c - b d \thet^{2} = 0 $.}
		\end{cases}
\end{multline*}
\end{lem}

\begin{proof}
Since $ \trace{m} = - \norm{r} = - \norm{r'} $, it is necessary and sufficient that
	\[ \norm{\frac{r'}{r}} = 1. \]
Thus, the existence of $ t $ is Hilbert's Theorem 90.  A reference for this theorem is II.1.2 Proposition~1 of \cite{serre02}.

We may assume that $ r \neq r' $, otherwise there is nothing to prove.  So either $ r' / r = - 1 $ or $ r' / r = e + f \thet $, where $ e $, $ f \in k^{\times} $, therefore it is possible to choose
	\[ t =	\begin{cases}
				\thet, & \text{if $ r' / r = - 1 $;} \\
				- ( 1 + e ) + f \thet, & \text{otherwise.}
			\end{cases} \]
Note that these are not the only choices for $ t $.  The sequel is not dependent on the choice of $ t $, thus we may choose any $ t $ which satisfies $ r' / r = \ol{t} / t $.

Firstly, if $ r' / r = - 1 $ (so $ t \in k^{\times} \thet $), then by Lemma~\ref{l:2.16},  
\begin{align*}
	&\twa{r}{m} \cdot \twa{r'}{m}^{- 1} \\
	&= \twa{r}{m} \cdot \twa{\frac{r}{- 1}}{\frac{m}{( - 1 )^{2}}}^{- 1} \\
	&= \delh{- \frac{\norm{m}}{\thet^{2}}} \cdot \delh{- \frac{\norm{m}}{( - 1 ) \cdot \thet^{2}}}^{- 1} \cdot \delh{- 1}.
\end{align*}
And since by Remark~\ref{r:dm1}, \delh{- 1} commutes with \delh{s}, $ s \in k^{\times} $,
	\[ \twa{r}{m} \cdot \twa{r'}{m}^{- 1}= \delh{- 1} \cdot \delh{- \frac{\norm{m}}{\thet^{2}}} \cdot \delh{\frac{\norm{m}}{\thet^{2}}}^{- 1}. \]
This implies that by \eqref{e:ba},
	\[ \twa{r}{m} \cdot \twa{r'}{m}^{- 1} = \hilkn{- 1}{- \frac{\norm{m}}{\thet^{2}}}. \]

By Lemma~2.12 of \cite{deodhar78}, for arbitrary $ ( s_{1}, n_{1} ) $, $ ( p_{1}, l_{1} ) \in A $, $ t \in K^{\times} $,
	\[ \twa{s_{1}}{n_{1}} \cdot \twa{p_{1}}{l_{1}} \cdot \twa{\frac{p_{1}}{t}}{\frac{l_{1}}{\norm{t}}}^{- 1} \cdot \twa{\frac{s_{1}}{t}}{\frac{n_{1}}{\norm{t}}}^{- 1} \in \pi_{1}. \]
By choosing $ s_{1} = r $, $ n_{1} = m $, $ p_{1} = r \ol{t} $ and $ l_{1} = m \norm{t} $, by Lemma~2.18 of \cite{deodhar78}, the above is equal to $ \twa{r}{m} \cdot \twa{r'}{m}^{-1} $.

By Proposition~\ref{p:wa},
\begin{gather*}
	\twa{r}{m} \cdot \twa{r \ol{t}}{m \norm{t}} \cdot \twa{r}{m}^{- 1} = \twa{\frac{r \ol{t} \cdot m^{2}}{m \norm{t} \cdot \ol{m}}}{\frac{\norm{m}}{\ol{m} \norm{t}}}, \\
	\twa{r}{m} \cdot \twa{\frac{r \ol{t}}{t}}{m}^{-1} \cdot \twa{r}{m}^{- 1} = \twa{- \frac{r ( \ol{t} / t ) \cdot m^{2}}{ m \cdot \ol{m}}}{\frac{\norm{m}}{m}}.
\end{gather*}
Hence,
\begin{align*}
	&\twa{r}{m} \cdot \twa{r'}{m}^{-1} \\
	&= \twa{r}{m} \cdot \twa{r \ol{t}}{m \norm{t}} \cdot \twa{\frac{r \ol{t}}{t}}{m}^{-1} \cdot \twa{\frac{r}{t}}{\frac{m}{\norm{t}}}^{-1} \\
	&=	\begin{aligned}[t]
			&\twa{\frac{r \ol{t} \cdot m^{2}}{m \norm{t} \cdot \ol{m}}}{\frac{\norm{m}}{\ol{m} \norm{t}}} \cdot \twa{- \frac{r ( \ol{t} / t ) \cdot m^{2}}{m \cdot \ol{m}}}{\frac{\norm{m}}{m}} \cdot \twa{r}{m} \\
			&\phantom{\ } \cdot \twa{- \frac{r}{t}}{\frac{\ol{m}}{\norm{t}}}.
		\end{aligned}
\end{align*}
Thus by \eqref{e:wainv},
\begin{multline} \label{e:l2.18a}
	\twa{r}{m} \cdot \twa{r'}{m}^{-1} = \left[ \twa{- \frac{r m}{t \ol{m}}}{\frac{\ol{m}}{\norm{t}}}^{-1} \cdot \twa{- \frac{r \ol{t} m}{t \ol{m}}}{\ol{m}} \right] \\ \cdot \left[ \twa{- r}{\ol{m}}^{-1} \cdot \twa{- \frac{r}{t}}{\frac{\ol{m}}{\norm{t}}} \right].
\end{multline}

As in Lemma~\ref{l:2.16}, there are a few cases to consider.  We may assume that $ ( r, m ) \neq ( r', m ) $, otherwise there is nothing to prove; therefore $ m \notin k^{\times} \thet $.  Also, $ t \notin k^{\times} $, since $ r \neq r' $.  We have already looked at the case where $ t \in k^{\times} \thet $.  Thus, there is only one last case.

Let $ m = a + b \thet $ and  $ t = c + d \thet $, where $ a $, $ d \in k^{\times} $, $ c $, $ d \in k $.  We first consider 
	\[ \twa{- \frac{r m}{t \ol{m}}}{\frac{\ol{m}}{\norm{t}}}^{-1} \cdot \twa{- \frac{r \ol{t} m}{t \ol{m}}}{\ol{m}}. \]
By Lemma~\ref{l:2.16},
\begin{align*}
	&\twa{- \frac{r m}{t \ol{m}}}{\frac{\ol{m}}{\norm{t}}}^{-1} \cdot \twa{- \frac{r \ol{t} m}{t \ol{m}}}{\ol{m}} \\
	&=	\begin{aligned}[t]
			&\hilkn{- \frac{( a / ( \norm{t} ) )^{2} ( 1 / ( \norm{t} ) )}{( d / ( \norm{t} ) )^{2} ( \norm{m} / ( \norm{t} )^{2} ) \thet^{2}}}{\frac{c}{\norm{t}} + \frac{( - b / \norm{t} ) ( d / \norm{t} ) \thet^{2}}{a / \norm{t}}} \\
			&\phantom{\ } \cdot \hilkn{- \frac{( d / ( \norm{t} ) ) ( \norm{m} / ( \norm{t} )^{2} )}{a / ( \norm{t} )}}{\frac{1}{\norm{t}}} \cdot \delh{\frac{1}{\norm{t}}}
		\end{aligned} \\
	&= \hilkn{- \frac{a^{2} \norm{t}}{d^{2} \norm{m} \thet^{2}}}{\frac{a c - b d \thet^{2}}{a \norm{t}}} \cdot \hilkn{- \frac{d \norm{m}}{a ( \norm{t} )^{2}}}{\frac{1}{\norm{t}}} \cdot \delh{\frac{1}{\norm{t}}};
\end{align*}
note that this is only valid if $ a c - b d \thet^{2} \neq 0 $ (we will deal with the other case later in the proof).  Hence
\begin{multline} \label{e:l2.18b}
	\twa{- \frac{r m}{t \ol{m}}}{\frac{\ol{m}}{\norm{t}}}^{-1} \cdot \twa{- \frac{r \ol{t} m}{t \ol{m}}}{\ol{m}} \\
	= \hilkn{- \frac{a^{2} \norm{t}}{d^{2} \norm{m} \thet^{2}}}{\frac{a c - b d \thet^{2}}{a \norm{t}}} \cdot \hilkn{- \frac{d \norm{m}}{a}}{\frac{1}{\norm{t}}} \cdot \delh{\frac{1}{\norm{t}}}
\end{multline}
by \eqref{e:prop3}.  Similarly, for
	\[ \twa{- r}{\ol{m}}^{-1} \cdot \twa{- \frac{r}{t}}{\frac{\ol{m}}{\norm{t}}}, \]
Lemma~\ref{l:2.16} gives
\begin{multline} \label{e:l2.18c}
	\twa{- r}{\ol{m}}^{-1} \cdot \twa{- \frac{r}{t}}{\frac{\ol{m}}{\norm{t}}} \\
	= \hilkn{- \frac{a^{2} \norm{t}}{d^{2} \norm{m} \thet^{2}}}{c + \frac{( - b ) d \thet^{2}}{a}} \cdot \hilkn{- \frac{d \norm{m}}{a}}{\norm{t}} \cdot \delh{\norm{t}}.
\end{multline}

Using \eqref{e:l2.18a}, \eqref{e:l2.18b} and \eqref{e:l2.18c},
\begin{multline*}
	\twa{r}{m} \cdot \twa{r'}{m}^{- 1} \\
	=	\begin{aligned}[t]
			&\left[ \hilkn{- \frac{a^{2} \norm{t}}{d^{2} \norm{m} \thet^{2}}}{\frac{a c - b d \thet^{2}}{a \norm{t}}} \cdot \hilkn{- \frac{d \norm{m}}{a}}{\frac{1}{\norm{t}}} \cdot \delh{\frac{1}{\norm{t}}} \right] \\
			&\phantom{\ } \cdot \left[ \hilkn{- \frac{a^{2} \norm{t}}{d^{2} \norm{m} \thet^{2}}}{c + \frac{( - b ) d \thet^{2}}{a}} \cdot \hilkn{- \frac{d \norm{m}}{a}}{\norm{t}} \cdot \delh{\norm{t}} \right].
		\end{aligned}
\end{multline*}
We simplify the above using \eqref{e:bastr} and \eqref{e:ba} to get
	\[ \twa{r}{m} \cdot \twa{r'}{m}^{- 1}
		=	\begin{aligned}[t]
				&\hilkn{- \frac{a^{2} \norm{t}}{d^{2} \norm{m} \thet^{2}}}{\left( \frac{a c - b d \thet^{2}}{a \norm{t}} \right)^{2} \cdot \norm{t}} \\
				&\phantom{\ } \cdot \hilkn{- \frac{d \norm{m}}{a}}{1} \cdot \hilkn{\frac{1}{\norm{t}}}{\norm{t}}.
			\end{aligned} \]
Thus by \eqref{e:bastr} and \eqref{e:prop3},
\begin{multline*}
	\twa{r}{m} \cdot \twa{r'}{m}^{- 1} \\
	= \hilkn{\frac{a^{2} \norm{t}}{d^{2} \norm{m} \thet^{2}}}{\frac{\left( a c - b d \thet^{2} \right)^{2}}{a^{2} \norm{t}}} \cdot \hilkn{- 1}{\frac{\left( a c - b d \thet^{2} \right)^{2}}{a^{2} \norm{t}}} \cdot \hilkn{- 1}{\norm{t}}.
\end{multline*}
By using \eqref{e:bastr} and the fact that $ \hilkn{- 1}{s^{2}} = 1 $ for all $ s \in k^{\times} $ by \eqref{e:prop5}, we finally get
	\[ \twa{r}{m} \cdot \twa{r'}{m}^{- 1} = \hilkn{\frac{a^{2} \norm{t}}{d^{2} \norm{m} \thet^{2}}}{\frac{\left( a c - b d \thet^{2} \right)^{2}}{a^{2} \norm{t}}}. \]

If $ a c - b d \thet^{2} = 0 $, we may instead use the above for
	\[ \twa{r'}{m} \cdot \twa{r}{m}^{- 1}. \]
This would imply that since
	\[ \frac{r}{r'} = \frac{t}{\ol{t}} = \frac{\ol{t'}}{t'}, \]
where $ t' = \ol{t} = c - d \thet $, replacing $ t $ by $ t' $ would give
\begin{align*}
	\twa{r'}{m} \cdot \twa{r}{m}^{- 1}
	&= \hilkn{\frac{a^{2} \norm{t}}{( - d )^{2} \norm{m} \thet^{2}}}{\frac{( a c - b ( - d ) \thet^{2} )^{2}}{a^{2} \norm{t}}} \\
	&= \hilkn{\frac{a^{2} \norm{t}}{d^{2} \norm{m} \thet^{2}}}{\frac{( a c + b d \thet^{2} )^{2}}{a^{2} \norm{t}}}.
\end{align*}
Therefore,
\begin{align*}
	\twa{r}{m} \cdot \twa{r'}{m}^{- 1}
	&= \left[ \twa{r'}{m} \cdot \twa{r}{m}^{- 1} \right]^{- 1} \\
	&= \hilkn{\frac{a^{2} \norm{t}}{d^{2} \norm{m} \thet^{2}}}{\frac{( a c + b d \thet^{2} )^{2}}{a^{2} \norm{t}}}^{- 1} \\
	&= \hilkn{\frac{( a c + b d \thet^{2} )^{2}}{a^{2} \norm{t}}}{\frac{a^{2} \norm{t}}{d^{2} \norm{m} \thet^{2}}}
\end{align*}
by \eqref{e:prop2}.
\end{proof}

\section{A slightly more general result} \label{s:lammuk}

Now that we have established a few lemmas, we can work on \ba{\lambda}{\mu} for the cases where both $ \lambda $, $ \mu \notin k^{\times} $.

\begin{lem} \label{l:2.23i}
For all $ \lambda \notin k^{\times} $,
	\[ \ba{\lambda}{\lambda^{-1}} = \hilkn{- 1}{\norm{\lambda}}. \]
\end{lem}

\begin{proof}
We know by Proposition~\ref{p:deltas} that $ \lambda = \del{1}{\lambda} / \del{2}{\lambda} $.  Thus,
	\[ \lambda^{-1} = \frac{\del{2}{\lambda}}{\del{1}{\lambda}} = \frac{\ol{\del{1}{\lambda}}}{\norm{\lambda} \ol{\del{2}{\lambda}}}. \]
This implies that $ \del{1}{\lambda^{-1}} = \ol{\del{1}{\lambda}} $ and $ \del{2}{\lambda^{-1}} = \norm{\lambda} \ol{\del{2}{\lambda}} $.  Hence, by \eqref{e:ba} and \eqref{e:delh},
\begin{align*}
	\ba{\lambda}{\lambda^{-1}}
	&= \delh{\lambda} \cdot \delh{\lambda^{-1}} \\
	&=	\begin{aligned}[t]
			&\twa{1}{\del{1}{\lambda}} \cdot \twa{0}{\del{2}{\lambda}}^{-1} \cdot \twa{1}{\ol{\del{1}{\lambda}}} \\
			&\phantom{\ } \cdot \twa{0}{\norm{\lambda} \ol{\del{2}{\lambda}}}^{-1}.
		\end{aligned}
\end{align*}
By Proposition~\ref{p:wa},
	\[ \twa{0}{\del{2}{\lambda}}^{-1} \cdot \twa{1}{\ol{\del{1}{\lambda}}} = \twa{1}{\ol{\del{1}{\lambda}}} \cdot \twa{0}{\frac{\norm{\del{1}{\lambda}}}{\del{2}{\lambda}}}, \]
which implies that
\begin{align*}
	\ba{\lambda}{\lambda^{-1}}
	&=	\begin{aligned}[t]
			&\twa{1}{\del{1}{\lambda}} \cdot \twa{1}{\ol{\del{1}{\lambda}}} \cdot \twa{0}{\frac{\norm{\del{1}{\lambda}}}{\del{2}{\lambda}}} \\
			&\phantom{\ }\cdot \twa{0}{\norm{\lambda} \ol{\del{2}{\lambda}}}^{-1}
		\end{aligned} \\
	&= \twa{1}{\del{1}{\lambda}} \cdot \twa{-1}{\del{1}{\lambda}}^{-1}.
\end{align*}
So by Lemma~\ref{l:2.18},
	\[ \ba{\lambda}{\lambda^{-1}} = \hilkn{- 1}{- \frac{\norm{\del{1}{\lambda}}}{\thet^{2}}}. \]
Also, for $ s \in k^{\times} $, $ \hilkn{- 1}{s^{2}} = 1 $ by \eqref{e:prop5}, hence
	\[ \hilkn{- 1}{- \frac{\thet^{2}}{\norm{\del{2}{\lambda}}}} = 1, \]
which implies, using \eqref{e:bastr}, that
\begin{align*}
	\ba{\lambda}{\lambda^{-1}}
	&= \hilkn{- 1}{- \frac{\norm{\del{1}{\lambda}}}{\thet^{2}}} \cdot \hilkn{- 1}{- \frac{\thet^{2}}{\norm{\del{2}{\lambda}}}} \\
	&= \hilkn{- 1}{\norm{\lambda}}. \qedhere
\end{align*}
\end{proof}

\begin{lem} \label{l:2.23ii}
For all $ \lambda $, $ \mu \notin k^{\times} $, $ q_{1} $, $ q_{2} \in k^{\times} $,
	\[ \ba{\lambda}{\mu}
		= \ba{\lambda q_{1}}{\mu q_{2}} \cdot \hilkn{- \frac{\mu \ol{\del{1}{\mu}} q_{2}}{\del{2}{\lambda}}}{q_{1}} \cdot a( q_{1}, q_{2}, \lambda, \mu ), \]
where
	\[ a( q_{1}, q_{2}, \lambda, \mu )
		=	\begin{cases}
				\displaystyle \hilkn{\frac{\del{2}{\lambda \mu}}{\del{2}{\mu}}}{q_{1} q_{2}}, &\text{if $ \lambda \mu \notin k^{\times} $;} \\
				\displaystyle \hilkn{- \frac{\thet}{\lambda \mu \del{2}{\mu}}}{q_{1} q_{2}}, &\text{if $ \lambda \mu \in k^{\times} $.}
			\end{cases} \]
\end{lem}

\begin{proof}
We have, by \eqref{e:ba} and \eqref{e:delh},
\begin{align*}
	\ba{\lambda}{\mu}
	&= \delh{\lambda} \cdot \delh{\mu} \cdot \delh{\lambda \mu}^{- 1} \\
	&=	\begin{aligned}[t]
			&\twa{1}{\del{1}{\lambda}} \cdot \twa{0}{\del{2}{\lambda}}^{- 1} \cdot \twa{1}{\del{1}{\mu}} \cdot \twa{0}{\del{2}{\mu}}^{- 1} \\
			&\phantom{\ } \cdot \twa{0}{\del{2}{\lambda \mu}} \cdot \twa{y( \lambda \mu )}{\del{1}{\lambda \mu}}^{- 1},
		\end{aligned}
\end{align*}
where the function $ y $ is as defined in \eqref{e:ylam}.  Since by Proposition~\ref{p:wa},
	\[ \twa{1}{\del{1}{\mu}}^{- 1} \cdot \twa{0}{\del{2}{\lambda}}^{-1} \cdot \twa{1}{\del{1}{\mu}} = \twa{0}{\frac{\norm{\del{1}{\mu}}}{\del{2}{\lambda}}}, \]
the equation for \ba{\lambda}{\mu} becomes
\begin{equation} \label{e:2.23ii}
	\ba{\lambda}{\mu}
	=	\begin{aligned}[t]
			&\twa{1}{\del{1}{\lambda}} \cdot \twa{1}{\del{1}{\mu}} \cdot \twa{0}{\frac{\norm{\del{1}{\mu}}}{\del{2}{\lambda}}} \\
			&\phantom{\ } \cdot \twa{0}{\del{2}{\mu}}^{-1} \cdot \twa{0}{\del{2}{\lambda \mu}} \cdot \twa{y( \lambda \mu )}{\del{1}{\lambda \mu}}^{-1}.
		\end{aligned}
\end{equation}

Now, for any $ q_{1} $, $ q_{2} \in k^{\times} $, by \eqref{e:delh},
\begin{align*}
	&\begin{aligned}[t]
		&\twa{0}{\frac{\norm{\del{1}{\mu}}}{\del{2}{\lambda}}} \cdot \twa{0}{\del{2}{\mu}}^{-1} \cdot \twa{0}{\del{2}{\lambda \mu}} \\
		&\phantom{\ } \cdot \twa{0}{\frac{\del{2}{\lambda \mu}}{q_{1} q_{2}}}^{-1} \cdot \twa{0}{\frac{\del{2}{\mu}}{q_{2}}} \cdot \twa{0}{\frac{q_{1} \norm{\del{1}{\mu}}}{\del{2}{\lambda}}}^{-1}
	 \end{aligned} \\
	&=	\begin{aligned}[t]
			&\twa{0}{\frac{\norm{\del{1}{\mu}}}{\del{2}{\lambda}}} \cdot \left[ \twa{0}{\thet}^{- 1} \cdot \twa{0}{\thet} \right] \cdot \twa{0}{\del{2}{\mu}}^{-1} \\
			&\phantom{\ } \cdot \twa{0}{\del{2}{\lambda \mu}} \cdot \left[ \twa{0}{\thet}^{- 1} \cdot \twa{0}{\thet} \right] \cdot \twa{0}{\frac{\del{2}{\lambda \mu}}{q_{1} q_{2}}}^{-1} \\
			&\phantom{\ } \cdot \twa{0}{\frac{\del{2}{\mu}}{q_{2}}} \cdot \left[ \twa{0}{\thet}^{- 1} \cdot \twa{0}{\thet} \right] \cdot \twa{0}{\frac{q_{1} \norm{\del{1}{\mu}}}{\del{2}{\lambda}}}^{-1}
		\end{aligned} \\
	&=	\begin{aligned}[t]
			&\delh{\frac{\norm{\del{1}{\mu}}}{\del{2}{\lambda} \thet}} \cdot \delh{\frac{\del{2}{\mu}}{\thet}}^{-1} \cdot \delh{\frac{\del{2}{\lambda \mu}}{\thet}} \\
			&\phantom{\ } \cdot \delh{\frac{\del{2}{\lambda \mu}}{q_{1} q_{2} \thet}}^{-1} \cdot \delh{\frac{\del{2}{\mu}}{q_{2} \thet}} \cdot \delh{\frac{q_{1} \norm{\del{1}{\mu}}}{\del{2}{\lambda} \thet}}^{-1}.
		\end{aligned}
\end{align*}
We find by using \eqref{e:delinv} that
\begin{multline*}
	\begin{aligned}[t]
		&\twa{0}{\frac{\norm{\del{1}{\mu}}}{\del{2}{\lambda}}} \cdot \twa{0}{\del{2}{\mu}}^{-1} \cdot \twa{0}{\del{2}{\lambda \mu}} \\
		&\phantom{\ } \cdot \twa{0}{\frac{\del{2}{\lambda \mu}}{q_{1} q_{2}}}^{-1} \cdot \twa{0}{\frac{\del{2}{\mu}}{q_{2}}} \cdot \twa{0}{\frac{q_{1} \norm{\del{1}{\mu}}}{\del{2}{\lambda}}}^{-1}
	 \end{aligned} \\
	=	\begin{aligned}[t]
			&\delh{\frac{\norm{\del{1}{\mu}}}{\del{2}{\lambda} \thet}} \cdot \left[ \delh{- \frac{\thet}{\del{2}{\mu}}} \cdot \delh{- 1}^{- 1} \right] \\
			&\phantom{\ } \cdot \delh{\frac{\del{2}{\lambda \mu}}{\thet}} \cdot \left[ \delh{- \frac{q_{1} q_{2} \thet}{\del{2}{\lambda \mu}}} \cdot \delh{- 1}^{- 1} \right] \\
			&\phantom{\ } \cdot \delh{\frac{\del{2}{\mu}}{q_{2} \thet}} \cdot \delh{\frac{q_{1} \norm{\del{1}{\mu}}}{\del{2}{\lambda} \thet}}^{-1}.
		\end{aligned}
\end{multline*}
Since by Remark~\ref{r:dm1}, \delh{- 1} commutes with \delh{s}, $ s \in k^{\times} $,
\begin{multline*}
	\begin{aligned}[t]
		&\twa{0}{\frac{\norm{\del{1}{\mu}}}{\del{2}{\lambda}}} \cdot \twa{0}{\del{2}{\mu}}^{-1} \cdot \twa{0}{\del{2}{\lambda \mu}} \\
		&\phantom{\ } \cdot \twa{0}{\frac{\del{2}{\lambda \mu}}{q_{1} q_{2}}}^{-1} \cdot \twa{0}{\frac{\del{2}{\mu}}{q_{2}}} \cdot \twa{0}{\frac{q_{1} \norm{\del{1}{\mu}}}{\del{2}{\lambda}}}^{-1}
	 \end{aligned} \\
	=	\begin{aligned}[t]
			&\delh{\frac{\norm{\del{1}{\mu}}}{\del{2}{\lambda} \thet}} \cdot \delh{- \frac{\thet}{\del{2}{\mu}}} \cdot \left[ \delh{- \frac{\mu \ol{\del{1}{\mu}}}{\del{2}{\lambda}}}^{- 1} \right. \\
			&\left. \phantom{\Bigg|^{1}} \cdot \delh{- \frac{\mu \ol{\del{1}{\mu}}}{\del{2}{\lambda}}} \right] \cdot \delh{\frac{\del{2}{\lambda \mu}}{\thet}} \\
			&\left. \phantom{\Bigg|^{1}} \right. \cdot \left[ \delh{- \frac{\mu \ol{\del{1}{\mu}} \del{2}{\lambda \mu}}{\del{2}{\lambda} \thet}}^{- 1} \cdot \delh{- \frac{\mu \ol{\del{1}{\mu}} \del{2}{\lambda \mu}}{\del{2}{\lambda} \thet}} \right] \\
			&\left. \phantom{\Bigg|^{1}} \right. \cdot \delh{- \frac{q_{1} q_{2} \thet}{\del{2}{\lambda \mu}}} \cdot \left[ \delh{\frac{q_{1} q_{2} \mu \ol{\del{1}{\mu}}}{\del{2}{\lambda}}}^{- 1} \right. \\
			&\left. \phantom{\Bigg|^{1}} \cdot \delh{\frac{q_{1} q_{2} \mu \ol{\del{1}{\mu}}}{\del{2}{\lambda}}} \right] \cdot \delh{\frac{\del{2}{\mu}}{q_{2} \thet}} \\
			&\left. \phantom{\Bigg|^{1}} \right. \cdot \delh{\frac{q_{1} \norm{\del{1}{\mu}}}{\del{2}{\lambda} \thet}}^{-1} \cdot \delh{- 1}^{- 1} \cdot \delh{- 1}^{- 1};
		\end{aligned}
\end{multline*}
and so by \eqref{e:ba},
\begin{multline*}
	\begin{aligned}[t]
		&\twa{0}{\frac{\norm{\del{1}{\mu}}}{\del{2}{\lambda}}} \cdot \twa{0}{\del{2}{\mu}}^{-1} \cdot \twa{0}{\del{2}{\lambda \mu}} \\
		&\phantom{\ } \cdot \twa{0}{\frac{\del{2}{\lambda \mu}}{q_{1} q_{2}}}^{-1} \cdot \twa{0}{\frac{\del{2}{\mu}}{q_{2}}} \cdot \twa{0}{\frac{q_{1} \norm{\del{1}{\mu}}}{\del{2}{\lambda}}}^{-1}
	 \end{aligned} \\
	=	\begin{aligned}[t]
			&\hilkn{\frac{\norm{\del{1}{\mu}}}{\del{2}{\lambda} \thet}}{- \frac{\thet}{\del{2}{\mu}}} \cdot \hilkn{- \frac{\mu \ol{\del{1}{\mu}}}{\del{2}{\lambda}}}{\frac{\del{2}{\lambda \mu}}{\thet}} \\
			&\phantom{\ } \cdot \hilkn{- \frac{\mu \ol{\del{1}{\mu}} \del{2}{\lambda \mu}}{\del{2}{\lambda} \thet}}{- \frac{q_{1} q_{2} \thet}{\del{2}{\lambda \mu}}} \cdot \hilkn{\frac{q_{1} q_{2} \mu \ol{\del{1}{\mu}}}{\del{2}{\lambda}}}{\frac{\del{2}{\mu}}{q_{2} \thet}} \cdot \hilkn{- 1}{- 1}^{- 1}.
		\end{aligned}
\end{multline*}
Hence it is a matter of simplifying using \eqref{e:bastr}, \eqref{e:prop2} and \eqref{e:prop3} to get
\begin{multline*}
	\begin{aligned}[t]
		&\twa{0}{\frac{\norm{\del{1}{\mu}}}{\del{2}{\lambda}}} \cdot \twa{0}{\del{2}{\mu}}^{-1} \cdot \twa{0}{\del{2}{\lambda \mu}} \\
		&\phantom{\ } \cdot \twa{0}{\frac{\del{2}{\lambda \mu}}{q_{1} q_{2}}}^{-1} \cdot \twa{0}{\frac{\del{2}{\mu}}{q_{2}}} \cdot \twa{0}{\frac{q_{1} \norm{\del{1}{\mu}}}{\del{2}{\lambda}}}^{-1}
	 \end{aligned} \\
	= \hilkn{- \frac{\mu \ol{\del{1}{\mu}} q_{2}}{\del{2}{\lambda}}}{q_{1}} \cdot \hilkn{\frac{\del{2}{\lambda \mu}}{\del{2}{\mu}}}{q_{1} q_{2}}.
\end{multline*}
Therefore,
\begin{align*}
	&\twa{0}{\frac{\norm{\del{1}{\mu}}}{\del{2}{\lambda}}} \cdot \twa{0}{\del{2}{\mu}}^{-1} \cdot \twa{0}{\del{2}{\lambda \mu}} \\
	&=	\begin{aligned}[t]
			&\hilkn{- \frac{\mu \ol{\del{1}{\mu}} q_{2}}{\del{2}{\lambda}}}{q_{1}} \cdot \hilkn{\frac{\del{2}{\lambda \mu}}{\del{2}{\mu}}}{q_{1} q_{2}} \cdot \twa{0}{\frac{q_{1} \norm{\del{1}{\mu}}}{\del{2}{\lambda}}} \\
			&\phantom{\ } \cdot \twa{0}{\frac{\del{2}{\mu}}{q_{2}}}^{-1} \cdot \twa{0}{\frac{\del{2}{\lambda \mu}}{q_{1} q_{2}}}.
		\end{aligned}
\end{align*}

Replacing the above in \eqref{e:2.23ii} gives
	\[ \ba{\lambda}{\mu}
	=	\begin{aligned}[t]
			&\twa{1}{\del{1}{\lambda}} \cdot \twa{1}{\del{1}{\mu}} \cdot \Bigg[ \hilkn{- \frac{\mu \ol{\del{1}{\mu}} q_{2}}{\del{2}{\lambda}}}{q_{1}} \\
			&\phantom{\ } \cdot \hilkn{\frac{\del{2}{\lambda \mu}}{\del{2}{\mu}}}{q_{1} q_{2}} \cdot \twa{0}{\frac{q_{1} \norm{\del{1}{\mu}}}{\del{2}{\lambda}}} \cdot \twa{0}{\frac{\del{2}{\mu}}{q_{2}}}^{-1} \\
			&\phantom{\ } \cdot \twa{0}{\frac{\del{2}{\lambda \mu}}{q_{1} q_{2}}} \Bigg] \cdot \twa{y( \lambda \mu )}{\del{1}{\lambda \mu}}^{-1},
		\end{aligned} \]
and since by Proposition~\ref{p:wa},
	\[ \twa{1}{\del{1}{\mu}} \cdot \twa{0}{\frac{q_{1} \norm{\del{1}{\mu}}}{\del{2}{\lambda}}} = \twa{0}{\frac{\del{2}{\lambda}}{q_{1}}}^{- 1} \cdot \twa{1}{\del{1}{\mu}}, \]
using this in the equation above gives
	\[ \ba{\lambda}{\mu}
	=	\begin{aligned}[t]
			&\hilkn{- \frac{\mu \ol{\del{1}{\mu}} q_{2}}{\del{2}{\lambda}}}{q_{1}} \cdot \hilkn{\frac{\del{2}{\lambda \mu}}{\del{2}{\mu}}}{q_{1} q_{2}} \cdot \twa{1}{\del{1}{\lambda}} \\
			&\phantom{\ } \cdot \twa{0}{\frac{\del{2}{\lambda}}{q_{1}}}^{- 1} \cdot \twa{1}{\del{1}{\mu}} \cdot \twa{0}{\frac{\del{2}{\mu}}{q_{2}}}^{-1} \\
			&\phantom{\ } \cdot \twa{0}{\frac{\del{2}{\lambda \mu}}{q_{1} q_{2}}} \cdot \twa{y( \lambda \mu )}{\del{1}{\lambda \mu}}^{-1}.
		\end{aligned} \]
By \eqref{e:del1q} and \eqref{e:del2q},
	\[ \ba{\lambda}{\mu}
	=	\begin{aligned}[t]
			&\hilkn{- \frac{\mu \ol{\del{1}{\mu}} q_{2}}{\del{2}{\lambda}}}{q_{1}} \cdot \hilkn{\frac{\del{2}{\lambda \mu}}{\del{2}{\mu}}}{q_{1} q_{2}} \cdot \twa{1}{\del{1}{\lambda q_{1}}} \\
			&\phantom{\ } \cdot \twa{0}{\del{2}{\lambda q_{1}}}^{- 1} \cdot \twa{1}{\del{1}{\mu q_{2}}} \cdot \twa{0}{\del{2}{\mu q_{2}}}^{-1} \\
			&\phantom{\ } \cdot \twa{0}{\frac{\del{2}{\lambda \mu}}{q_{1} q_{2}}} \cdot \twa{y( \lambda \mu )}{\del{1}{\lambda \mu}}^{-1},
		\end{aligned} \]
and using \eqref{e:delh} on the right-hand side gives
	\[ \ba{\lambda}{\mu}
	=	\begin{aligned}[t]
			&\hilkn{- \frac{\mu \ol{\del{1}{\mu}} q_{2}}{\del{2}{\lambda}}}{q_{1}} \cdot \hilkn{\frac{\del{2}{\lambda \mu}}{\del{2}{\mu}}}{q_{1} q_{2}} \cdot \delh{\lambda q_{1}} \cdot \delh{\mu q_{2}} \\
			&\phantom{\ } \cdot \twa{0}{\frac{\del{2}{\lambda \mu}}{q_{1} q_{2}}} \cdot \twa{y( \lambda \mu )}{\del{1}{\lambda \mu}}^{-1}.
		\end{aligned} \]

By \eqref{e:del1q}, \eqref{e:del2q}, Proposition~\ref{p:deltas} and \eqref{e:delh},
\begin{align*}
	&\twa{0}{\frac{\del{2}{\lambda \mu}}{q_{1} q_{2}}} \cdot \twa{y( \lambda \mu )}{\del{1}{\lambda \mu}}^{-1} \\
	&=	\begin{cases}
			\displaystyle \twa{0}{\del{2}{\lambda \mu q_{1} q_{2}}} \cdot \twa{1}{\del{1}{\lambda \mu q_{1} q_{2}}}^{-1}, &\text{if $ \lambda \mu \notin k^{\times} $;} \\
			\displaystyle \twa{0}{\frac{\thet}{q_{1} q_{2}}} \cdot \twa{0}{\lambda \mu \thet}^{-1}, &\text{if $ \lambda \mu \in k^{\times} $,}
		\end{cases} \\
	&=	\begin{cases}
			\displaystyle \delh{\lambda \mu q_{1} q_{2}}^{-1}, &\text{if $ \lambda \mu \notin k^{\times} $;} \\
			\displaystyle \delh{\frac{1}{q_{1} q_{2}}} \cdot \delh{\lambda \mu}^{-1}, &\text{if $ \lambda \mu \in k^{\times} $.}
		\end{cases}
\end{align*}

So if $ \lambda \mu \notin k^{\times} $,
\begin{align*}
	\ba{\lambda}{\mu}
	&=	\begin{aligned}[t]
			&\hilkn{- \frac{\mu \ol{\del{1}{\mu}} q_{2}}{\del{2}{\lambda}}}{q_{1}} \cdot \hilkn{\frac{\del{2}{\lambda \mu}}{\del{2}{\mu}}}{q_{1} q_{2}} \cdot \delh{\lambda q_{1}} \cdot \delh{\mu q_{2}} \\
			&\phantom{\ } \cdot \delh{\lambda \mu q_{1} q_{2}}^{-1}
		\end{aligned} \\
	&=	\begin{aligned}[t]
			&\hilkn{- \frac{\mu \ol{\del{1}{\mu}} q_{2}}{\del{2}{\lambda}}}{q_{1}} \cdot \hilkn{\frac{\del{2}{\lambda \mu}}{\del{2}{\mu}}}{q_{1} q_{2}} \cdot \ba{\lambda q_{1}}{\mu q_{2}}
		\end{aligned}
\end{align*}
by \eqref{e:ba}.  If $ \lambda \mu \in k^{\times} $,
	\[ \ba{\lambda}{\mu}
	=	\begin{aligned}[t]
			&\hilkn{- \frac{\mu \ol{\del{1}{\mu}} q_{2}}{\del{2}{\lambda}}}{q_{1}} \cdot \hilkn{\frac{\del{2}{\lambda \mu}}{\del{2}{\mu}}}{q_{1} q_{2}} \cdot \delh{\lambda q_{1}} \cdot \delh{\mu q_{2}} \\
			&\phantom{\ } \cdot \left[ \delh{\lambda \mu q_{1} q_{2}}^{- 1} \cdot \delh{\lambda \mu q_{1} q_{2}} \right] \cdot \delh{\frac{1}{q_{1} q_{2}}} \\
			&\phantom{\ } \cdot \delh{\lambda \mu}^{-1},
		\end{aligned} \]
and since by Proposition~\ref{p:deltas}, $ \del{2}{\lambda \mu} = \thet $, this shows by using \eqref{e:ba} that
	\[ \ba{\lambda}{\mu}
	=	\begin{aligned}[t]
			&\hilkn{- \frac{\mu \ol{\del{1}{\mu}} q_{2}}{\del{2}{\lambda}}}{q_{1}} \cdot \hilkn{\frac{\thet}{\del{2}{\mu}}}{q_{1} q_{2}} \cdot \ba{\lambda q_{1}}{\mu q_{2}} \\
			&\phantom{\ } \cdot \hilkn{\lambda \mu q_{1} q_{2}}{\frac{1}{q_{1} q_{2}}}.
		\end{aligned} \]
Thus,
	\[ \ba{\lambda}{\mu}
	=	\hilkn{- \frac{\mu \ol{\del{1}{\mu}} q_{2}}{\del{2}{\lambda}}}{q_{1}} \cdot \hilkn{- \frac{\thet}{\lambda \mu \del{2}{\mu}}}{q_{1} q_{2}} \cdot \ba{\lambda q_{1}}{\mu q_{2}} \]
by \eqref{e:prop2}, \eqref{e:prop3} and \eqref{e:bastr}.  Therefore we have proven our result.
\end{proof}

\begin{prop} \label{p:bainv}
For $ \lambda \notin k^{\times} $, $ q \in k^{\times} $,
	\[ \ba{\lambda}{\lambda^{-1} q} = \hilkn{- 1}{\norm{\lambda}} \cdot \hilkn{- \frac{\lambda \ol{\del{1}{\lambda}}}{\thet}}{q}. \]
\end{prop}

\begin{proof}
This is easily proved using Lemmas~\ref{l:2.23i} and~\ref{l:2.23ii}.

By Lemma~\ref{l:2.23i},
	\[ \ba{\lambda}{\lambda^{-1}} = \hilkn{- 1}{\norm{\lambda}}, \]
and by Lemma~\ref{l:2.23ii},
	\[ \ba{\lambda}{\lambda^{-1}}
	=	\begin{aligned}[t]
			&\ba{\lambda \cdot 1}{\lambda^{-1} q} \cdot \hilkn{- \frac{\lambda^{- 1} \ol{\del{1}{\lambda^{- 1}}} q}{\del{2}{\lambda}}}{1} \\
			&\phantom{\ } \cdot \hilkn{- \frac{\thet}{\lambda \lambda^{- 1} \del{2}{\lambda^{-1}}}}{1 \cdot q}.
		\end{aligned} \]

Hence,
	\[ \ba{\lambda}{\lambda^{-1} q}
	= \hilkn{- 1}{\norm{\lambda}} \cdot \hilkn{- \frac{\thet}{\del{2}{\lambda^{-1}}}}{q}^{- 1}. \]
By the proof of Lemma~\ref{l:2.23i}, $ \del{2}{\lambda^{- 1}} = \norm{\lambda} \ol{\del{2}{\lambda}} $.  The right-hand side becomes
	\[ \ba{\lambda}{\lambda^{-1} q}
	= \hilkn{- 1}{\norm{\lambda}} \cdot \hilkn{- \frac{\thet}{\norm{\lambda} \ol{\del{2}{\lambda}}}}{q}^{- 1}, \]
and so
	\[ \ba{\lambda}{\lambda^{-1} q} = \hilkn{- 1}{\norm{\lambda}} \cdot \hilkn{- \frac{\lambda \ol{\del{1}{\lambda}}}{\thet}}{q} \]
by \eqref{e:prop5}, since $ \lambda = \del{1}{\lambda} / \del{2}{\lambda} $.
\end{proof}

\begin{rem} \label{r:bainv}
Proposition~\ref{p:bainv} shows that for $ \lambda $, $ \mu \notin k^{\times} $ such that $ \lambda \mu \in k^{\times} $,
	\[ \ba{\lambda}{\mu} = \hilkn{- 1}{\norm{\lambda}} \cdot \hilkn{- \frac{\lambda \ol{\del{1}{\lambda}}}{\thet}}{\lambda \mu}. \]
Therefore we may assume that $ \lambda $, $ \mu $, $ \lambda \mu \notin k^{\times} $ for the last section.
\end{rem}

\section{The most general case}

Before we can calculate $ \ba{\lambda}{\mu} $ for $ \lambda $, $ \mu $, $ \lambda \mu \notin k^{\times} $, we must first calculate the commutator of the 2-cocycle $ \sigma_{u} $ on $ T( k ) $.

Let us define what a commutator is.  Let $ \sigma $ be a 2-cocycle on an abelian group $ T $ with values in $ \mu_{n} $.  Then the commutator of $ \sigma $ is defined by
	\[ [ x, y ] = \sigma( x, y ) / \sigma( y, x ), \]
where $ x $, $ y \in T $.  The commutator of $ \sigma $ is both bimultiplicative, i.e. for $ x $, $ x' $, $ y $, $ y' \in T $,
	\[ [ x x', y ] = [ x, y ] [ x', y ] \text{ and } [ x, y y' ] = [ x, y ] [ x, y' ]; \]
and skew-symmetric, i.e.
	\[ [ x, x ] = 1. \]
These are the standard properties of the commutator of $ \sigma $.  The commutator of $ \sigma $ depends only on the cohomology class of $ \sigma $, and if $ T $ is a locally compact topological group and $ \sigma $ is measurable then the commutator of $ \sigma $ is continuous.

\begin{lem} \label{l:comm}
Let $ \lambda $, $ \mu \in K^{\times} $, and the commutator of the 2-cocycle $ \sigma_{u} $ on $ T( k ) $ be
	\[ [ \lambda, \mu ]_{\sigma_{u}} = \frac{\sigu{\ha{\lambda}}{\ha{\mu}}}{\sigu{\ha{\mu}}{\ha{\lambda}}} = \delh{\lambda} \cdot \delh{\mu} \cdot \delh{\lambda}^{-1} \cdot \delh{\mu}^{-1}. \]
Then
	\[ [ \lambda, \mu ]_{\sigma_{u}} = \hilbkn{\lambda}{\mu}^{2} \cdot \hilbkn{\lambda}{\ol{\mu}}^{- 1}. \]
\end{lem}

\begin{proof}
We first assume that $ K = k \oplus k $, i.e. $ K / k $ is a split extension.  In this case, $ \SU( k ) \cong \SL_{3}( k ) $.  Let $ \lambda = \spl{\lambda_{1}}{\lambda_{2}} $ and $ \mu = \spl{\mu_{1}}{\mu_{2}} $, where $ \lambda $, $ \mu \in K^{\times} $.  The isomorphism on $ T( k ) $ may be described by
\begin{align*}
	T( k ) &\cong \left\{ \begin{pmatrix} x & 0 & 0 \\ 0 & y & 0 \\ 0 & 0 & z \end{pmatrix} \in \SL_{3}( k ) \colon x, y, z \in k^{\times}, x y z = 1 \right\} \\
	\begin{pmatrix} \lambda & 0 & 0 \\ 0 & \ol{\lambda} / \lambda & 0 \\ 0 & 0 & \ol{\lambda}^{- 1} \end{pmatrix} &\mapsto \begin{pmatrix} \lambda_{1} & 0 & 0 \\ 0 & \lambda_{2} / \lambda_{1} & 0 \\ 0 & 0 & \lambda_{2}^{- 1} \end{pmatrix}.
\end{align*}

It has been established in Section~0.1 of \cite{kazhpat84} that the commutator of the 2-cocycle $ \sigma' \in H^{2}( \SL_{3}, \mu( k ) ) $ on the diagonal elements of $ \SL_{3}( k ) $ may be described as
	\[ \frac{\sigma'\left( \begin{pmatrix} h_{1} & 0 & 0 \\ 0 & h_{2} & 0 \\ 0 & 0 & h_{3} \end{pmatrix}, \begin{pmatrix} h_{1}' & 0 & 0 \\ 0 & h_{2}' & 0 \\ 0 & 0 & h_{3}' \end{pmatrix} \right)}{\sigma'\left( \begin{pmatrix} h_{1}' & 0 & 0 \\ 0 & h_{2}' & 0 \\ 0 & 0 & h_{3}' \end{pmatrix}, \begin{pmatrix} h_{1} & 0 & 0 \\ 0 & h_{2} & 0 \\ 0 & 0 & h_{3} \end{pmatrix} \right)} = \hilkn{h_{1}}{h_{1}'} \cdot \hilkn{h_{2}}{h_{2}'} \cdot \hilkn{h_{3}}{h_{3}'}. \]
Since $ \sigma' $ differs from $ \sigma_{u} $ by at most a 1-coboundary, this implies that we may use the above equation to calculate $ [ \lambda, \mu ]_{\sigma_{u}} $ in the case where $ K / k $ is a split extension.  Thus
	\[ [ \lambda, \mu ]_{\sigma_{u}} = \hilkn{\lambda_{1}}{\mu_{1}} \cdot \hilk{\lambda_{2} / \lambda_{1}}{\mu_{2} / \mu_{1}} \cdot \hilkn{\lambda_{2}^{- 1}}{\mu_{2}^{- 1}}. \]
By \eqref{e:bastr} and \eqref{e:prop5},
\begin{align*}
	[ \lambda, \mu ]_{\sigma_{u}}
	&=	\begin{aligned}[t]
			&\hilkn{\lambda_{1}}{\mu_{1}} \cdot \left[ \hilkn{\lambda_{2}}{\mu_{2}} \cdot \hilkn{\lambda_{1}}{\mu_{2}}^{- 1} \cdot \hilkn{\lambda_{2}}{\mu_{1}}^{- 1} \cdot \hilkn{\lambda_{1}}{\mu_{1}} \right] \\
			&\phantom{\ } \cdot \hilkn{\lambda_{2}}{\mu_{2}}
		\end{aligned} \\
	&= \left[ \hilkn{\lambda_{1}}{\mu_{1}} \cdot \hilkn{\lambda_{2}}{\mu_{2}} \right]^{2} \cdot \left[ \hilkn{\lambda_{1}}{\mu_{2}} \cdot \hilkn{\lambda_{2}}{\mu_{1}} \right]^{- 1}.
\end{align*}

Let us define (as in the Introduction)
	\[ \hilbkn{\lambda}{\mu} = \hilkn{\lambda_{1}}{\mu_{1}} \cdot \hilkn{\lambda_{2}}{\mu_{2}}. \]
Since $ \ol{\mu} = \spl{\mu_{2}}{\mu_{1}} $, this implies that
	\[ [ \lambda, \mu ]_{\sigma_{u}} = \hilbkn{\lambda}{\mu}^{2} \cdot \hilbkn{\lambda}{\ol{\mu}}^{- 1}. \]

We want to show that the above equation holds for all quadratic extensions $ K / k $.  In order to do so, we first assume that $ \pi $ is the prime element of $ K $ and that $ n $ is coprime to $ \pi $.  By the properties of the commutator of the 2-cocycle, the commutator of the 2-cocycle $ \sigma_{u} $ on the torus is a map
	\[ \sideset{}{^{2}}{\bigwedge} \left( K^{\times} / ( K^{\times} )^n \right) \to \mu_{n}. \]
When $ \pi $ does not divide $ n $, this means that
	\[ K^{\times} / ( K^{\times} )^n \cong \mathbb{Z} / n \oplus \mathbb{Z} / n, \]
and therefore
	\[ \sideset{}{^{2}}{\bigwedge} \left( K^{\times} / ( K^{\times} )^n \right) \cong \mathbb{Z} / n, \]
which is generated by the element $ a \wedge \pi $, where $ a $ is a generator for $ ( \OK / \pi )^{\times} / n $.  In particular, $ a $ is a unit. Thus, $ [ -, - ]_{\sigma_{u}} $ is determined by $ [ a, \pi ]_{\sigma_{u}} $, and we only need to check first that
	\[ [ a, \pi ]_{\sigma_{u}} = \hilbkn{a}{\pi}^{2} \cdot \hilbkn{a}{\ol{\pi}}^{- 1}, \]
	and second that the right hand side of this equation is bimultiplicative and skew-symmetric.
	We will demonstrate the bimultiplicativity and skew-symmetry at the end of this proof.

For $ \lambda $, $ \mu \in K^{\times} $, we have by \eqref{e:ba} that
	\[ [ \lambda, \mu ]_{\sigma_{u}} = \delh{\lambda} \cdot \delh{\mu} \cdot \delh{\lambda}^{-1} \cdot \delh{\mu}^{-1}. \]
By Proposition~\ref{p:wa},
\begin{align*}
	&\delh{\lambda} \cdot \delh{\mu} \cdot \delh{\lambda}^{-1} \\
	&=	\begin{aligned}[t]
			&\left[ \twa{y( \lambda )}{\del{1}{\lambda}} \cdot \twa{0}{\ol{\del{2}{\lambda}}} \right] \cdot \twa{y( \mu )}{\del{1}{\mu}} \cdot \twa{0}{\ol{\del{2}{\mu}}} \\
			&\phantom{\ } \cdot \left[ \twa{y( \lambda )}{\del{1}{\lambda}} \cdot \twa{0}{\ol{\del{2}{\lambda}}} \right]^{-1}
		\end{aligned} \\
	&= \twa{\frac{y( \mu ) \ol{\del{2}{\lambda}} \del{1}{\lambda}^{2}}{\del{2}{\lambda}^{2} \ol{\del{1}{\lambda}}}}{\frac{\del{1}{\mu} \norm{\del{1}{\lambda}}}{\norm{\del{2}{\lambda}}}} \cdot \twa{0}{\frac{\ol{\del{2}{\mu}} \norm{\del{1}{\lambda}}}{\norm{\del{2}{\lambda}}}}
\end{align*}
where $ y $ is defined as in \eqref{e:ylam}, and hence
\begin{multline*}
	 \delh{\lambda} \cdot \delh{\mu} \cdot \delh{\lambda}^{-1} \\
	 = \twa{\frac{y( \mu ) \lambda^{2}}{\ol{\lambda}}}{\norm{\lambda} \del{1}{\mu}} \cdot \twa{0}{\norm{\lambda} \del{2}{\mu}}^{- 1}.
\end{multline*}

Therefore, if $ \lambda \notin k^{\times} $, $ \mu \in k^{\times} $, then $ y( \mu ) = 0 $, hence by the above,
	\[ [ \lambda, \mu ]_{\sigma_{u}} = \twa{0}{\norm{\lambda} \del{1}{\mu}} \cdot \twa{0}{\norm{\lambda} \del{2}{\mu}}^{- 1} \cdot \delh{\mu}^{- 1}. \]
By Proposition~\ref{p:deltas}, $ \del{1}{\mu} = \mu \thet $ and $ \del{2}{\mu} = \thet $.  Hence by \eqref{e:delh},
\begin{align*}
	[ \lambda, \mu ]_{\sigma_{u}}
	&= \twa{0}{\norm{\lambda} \mu \thet} \cdot \twa{0}{\norm{\lambda} \thet}^{-1} \cdot \delh{\mu}^{-1} \\
	&=	\begin{aligned}[t]
			&\twa{0}{\norm{\lambda} \mu \thet} \cdot \left[ \twa{0}{\thet}^{- 1} \cdot \twa{0}{\thet} \right] \cdot \twa{0}{\norm{\lambda} \thet}^{-1} \\
			&\phantom{\ } \cdot \delh{\mu}^{-1}
		\end{aligned} \\
	&= \delh{\norm{\lambda} \mu} \cdot \delh{\norm{\lambda}}^{-1} \cdot \delh{\mu}^{-1}.
\end{align*}
This implies by \eqref{e:ba} that
	\[ [ \lambda, \mu ]_{\sigma_{u}} = \hilkn{\mu}{\norm{\lambda}}^{- 1}; \]
hence by \eqref{e:prop5} and \eqref{e:prop2},
	\[ [ \lambda, \mu ]_{\sigma_{u}} = \hilkn{\norm{\lambda}}{\mu}. \]

In the case where $ K / k $ is an unramified extension, we may take $ \pi \in k $.  So assuming that $ K / k $ is an unramified extension and that $ \pi $ is coprime to $ n $, we have
	\[ [ a, \pi ]_{\sigma_{u}} = \hilkn{\norm{a}}{\pi}. \]
Hence by the above and \eqref{e:norm2},
	\[ [ a, \pi ]_{\sigma_{u}} = \hilbkn{a}{\pi} = \hilbkn{a}{\pi}^{2} \cdot \hilbkn{a}{\pi}^{- 1} = \hilbkn{a}{\pi}^{2} \cdot \hilbkn{a}{\ol{\pi}}^{- 1}, \]
and hence for all unramified extensions $ K / k $ with $ \pi $ coprime to $ n $,
	\[ [ \lambda, \mu ]_{\sigma_{u}} = \hilbkn{\lambda}{\mu}^{2} \cdot \hilbkn{\lambda}{\ol{\mu}}^{- 1}. \]

We now use the product formula to show that the above is true for all quadratic extensions $ K / k $.  We will need to use the ad\`{e}le group (see Subsection~\ref{ss:descad}).  We define some notation.  Let $ L / l \supset \mu_{n} $ be a global quadratic extension.  Let $ \mathfrak{p} $ be a prime of $ l $, $ \lp $ the localisation of $ l $ at $ \mathfrak{p} $, and just as in Subsection~\ref{ss:descad}, let $ \Lp = \lp \otimes_{l} L $, i.e.
	\[ \Lp = \begin{cases} \lp( \thet ), &\text{if $ \mathfrak{p} \mathcal{O}_{L} $ does not split in $ \mathcal{O}_{L} $;} \\ \lp \oplus \lp, &\text{if $ \mathfrak{p} \mathcal{O}_{L} $ splits in $ \mathcal{O}_{L} $.} \end{cases} \]   
Now let $ \sigma_{\mathfrak{p}} \in H^{2}( G( \lp ), \mu_{n} ) $ such that $ \sigma_{\mathfrak{p}} = \sigma_{u} $ when $ k = \lp $ and $ K = \Lp $.  Then for almost all primes $ \mathfrak{p} $ of $ l $, we know that for $ \lambda $, $ \mu \in \Lp^{\times} $,
	\[ [ \lambda, \mu ]_{\sigma_{\mathfrak{p}}} = ( \lambda, \mu )_{\Lp, n}^{2} \cdot ( \lambda, \ol{\mu} )_{\Lp, n}^{- 1}. \]
We also know that since $ \prod_{\mathfrak{p}} \sigma_{\mathfrak{p}} $ splits on $ G( l ) $, if $ \lambda $, $ \mu \in L^{\times} $, then
	\[ \prod_{\mathfrak{p}} [ \lambda, \mu ]_{\sigma_{\mathfrak{p}}} = 1. \]
In addition, by the product formula (Theorem~\ref{t:product}), for $ \lambda $, $ \mu \in L^{\times} $,
	\[ \prod_{\mathfrak{p}} \left( ( \lambda, \mu )_{\Lp, n}^{2} \cdot ( \lambda, \ol{\mu} )_{\Lp, n}^{- 1} \right) = 1. \]

Now choose a prime $ \mathfrak{p} $ that is not unramified and coprime to $ n $, and also not split.  Let $ \lambda $, $ \mu \in \Lp^{\times} $, and choose $ \lambda' $, $ \mu' \in L^{\times} $ close to $ \lambda $ and $ \mu $ respectively such that
	\[ [ \lambda, \mu ]_{\sigma_{\mathfrak{p}}} = [ \lambda', \mu' ]_{\sigma_{\mathfrak{p}}}, \quad ( \lambda, \mu )_{\Lp, n}^{2} \cdot ( \lambda, \ol{\mu} )_{\Lp, n}^{- 1} = ( \lambda', \mu' )_{\Lp, n}^{2} \cdot ( \lambda', \ol{\mu'} )_{\Lp, n}^{- 1}; \]
and such that $ \lambda' $, $ \mu' $ are close to 1 in $ L_{\mathfrak{q}} $ for all other primes $ \mathfrak{q} $ which are not unramified and coprime to n and also not split, i.e.
	\[ [ \lambda', \mu' ]_{\sigma_{\mathfrak{q}}} = 1 = ( \lambda', \mu' )_{L_{\mathfrak{q}}, n}^{2} \cdot ( \lambda', \ol{\mu'} )_{L_{\mathfrak{q}}, n}^{- 1}. \]
Then by the product formula,
\begin{align*}
	( \lambda', \mu' )_{\Lp, n}^{2} \cdot ( \lambda', \ol{\mu'} )_{\Lp, n}^{- 1}
	&= \prod_{\mathfrak{v} \neq \mathfrak{p}} \left( ( \lambda', \mu' )_{L_{\mathfrak{v}}, n}^{2} \cdot ( \lambda', \ol{\mu'} )_{L_{\mathfrak{v}}, n}^{- 1} \right)^{- 1} \\
	&= \prod_{\mathfrak{q}} \left( ( \lambda', \mu' )_{L_{\mathfrak{q}}, n}^{2} \cdot ( \lambda', \ol{\mu'} )_{L_{\mathfrak{q}}, n}^{- 1} \right)^{- 1} \\
	&= \prod_{\mathfrak{q}} [ \lambda', \mu' ]_{\sigma_{\mathfrak{q}}}^{- 1} \\
	&= \prod_{\mathfrak{v} \neq \mathfrak{p}} [ \lambda', \mu' ]_{\sigma_{\mathfrak{v}}}^{- 1} \\
	&= [ \lambda', \mu' ]_{\sigma_{\mathfrak{p}}}.
\end{align*}

The above implies that for all local quadratic extensions $ K / k $, $ \lambda $, $ \mu \in K^{\times} $, we have
	\[ [ \lambda, \mu ]_{\sigma_{u}} = \hilbkn{\lambda}{\mu}^{2} \cdot \hilbkn{\lambda}{\ol{\mu}}^{- 1}, \]
but as stated earlier, we should show that the map $ \spl{\lambda}{\mu} \mapsto \hilbkn{\lambda}{\mu}^{2} \cdot \hilbkn{\lambda}{\ol{\mu}}^{- 1} $ is indeed a map $ \sideset{}{^{2}}{\bigwedge} \left( K^{\times} / ( K^{\times} )^n \right) \to \mu_{n} $.  Thus, we need to show that the map is bimultiplicative and skew-symmetric.

It is trivial to show that the map is bimultiplicative.  To show skew-symmetry, we first note that by \eqref{e:prop3},
	\[ \hilbkn{\lambda}{\lambda}^{2} = \hilbkn{\lambda}{- 1}^{2}; \]
and by \eqref{e:prop5},
	\[ \hilbkn{\lambda}{- 1}^{2} = \hilbkn{\lambda}{( - 1 )^{2}} = 1. \]
Now assume that $ \trace{\lambda} = 0 $.  Then $ \ol{\lambda} = - \lambda $, and by \eqref{e:prop3},
	\[ \hilbkn{\lambda}{\ol{\lambda}} = \hilbkn{\lambda}{- \lambda} = 1. \]
Hence in this case, $ \hilbkn{\lambda}{\lambda}^{2} \cdot \hilbkn{\lambda}{\ol{\lambda}}^{- 1} = 1 $.

As for the case where $ \trace{\lambda} \neq 0 $, let $ \lambda = s \lambda' $, where $ s \in k^{\times} $ and $ \trace{\lambda'} = 1 $.  Then $ \ol{\lambda'} = 1 - \lambda' $, and hence by \eqref{e:bastr},
	\[ \hilbkn{\lambda}{\ol{\lambda}} = \hilbkn{s \lambda'}{s( 1 - \lambda' )} = \hilbkn{s}{s} \cdot \hilbkn{s}{1 - \lambda'} \cdot \hilbkn{\lambda'}{s} \cdot \hilbkn{\lambda'}{1 - \lambda'}. \]
By \eqref{e:prop2}, \eqref{e:prop3} and \eqref{e:prop4},
	\[ \hilbkn{\lambda}{\ol{\lambda}} = \hilbkn{s}{- 1} \cdot \hilbkn{s}{\ol{\lambda'}} \cdot \hilbkn{s}{\lambda'^{- 1}} \cdot 1; \]
and by \eqref{e:bastr},
	\[ \hilbkn{\lambda}{\ol{\lambda}} = \hilbkn{s}{- \frac{\ol{\lambda'}}{\lambda'}}. \]
Hence by \eqref{e:norm1},
	\[ \hilbkn{\lambda}{\ol{\lambda}} = \hilkn{s}{\norm{- \frac{\ol{\lambda'}}{\lambda'}}} = 1. \]
Therefore $ \hilbkn{\lambda}{\lambda}^{2} \cdot \hilbkn{\lambda}{\ol{\lambda}}^{- 1} = 1 $ in this case as well, i.e. the map $ \spl{\lambda}{\mu} \mapsto \hilbkn{\lambda}{\mu}^{2} \cdot \hilbkn{\lambda}{\ol{\mu}}^{- 1} $ is skew-symmetric and hence the formula for the commutator on the torus is indeed $ [ \lambda, \mu ]_{\sigma_{u}} = \hilbkn{\lambda}{\mu}^{2} \cdot \hilbkn{\lambda}{\ol{\mu}}^{- 1} $.
\end{proof}

We are now ready to find the explicit formula for \ba{\lambda}{\mu}, with $ \lambda $, $ \mu $, $ \lambda \mu \notin k^{\times} $.

\begin{prop} \label{p:2.29}
For $ \lambda $, $ \mu \notin k^{\times} $ such that $ \lambda \mu \notin k^{\times} $,
	\[ \ba{\lambda}{\mu} = \hilkn{- \frac{\norm{\del{1}{\mu}}}{\norm{\del{1}{\lambda}}}}{\frac{\norm{\lambda}}{q^{2}}} \cdot \hilkn{q}{\frac{\mu \ol{\del{1}{\mu}}}{\del{2}{\lambda}}} \cdot \Sigma'( \lambda, \mu ), \]
where, if $  \lambda = a + b \thet $, $ \mu = c + d \thet $, with $ a $, $ c \in k $, $ b $, $ d \in k^{\times} $,
	\[ q = a + \frac{b c}{d}, \]
and
	\[ \Sigma'( \lambda, \mu )
		=	\begin{cases}
				\displaystyle\hilkn{- \frac{\norm{\lambda}}{4 a^{2} \norm{\del{1}{\mu}}}}{\frac{( \norm{\lambda} )^{2} b^{4} \thet^{4}}{( ( a - q ) a - b^{2} \thet^{2} )^{4}}}, &\text{if $ \lambda \notin k^{\times} \thet $, $ a q \neq \norm{\lambda} $;} \\
			\displaystyle\hilkn{\frac{( \norm{\lambda} )^{2} b^{4}\thet^{4}}{( ( a - q ) a + b^{2} \thet^{2} )^{4}}}{- \frac{\norm{\lambda}}{4 a^{2} \norm{\del{1}{\mu}}}}, &\text{if $ \lambda \notin k^{\times} \thet $, $ a q = \norm{\lambda} $;} \\
				\displaystyle 1, &\text{if $ \lambda \in k^{\times} \thet $.}
			\end{cases} \]
\end{prop}

\begin{proof}
This proof will follow closely the methods used in Section~2.29 of \cite{deodhar78}.  For some $ ( p_{i}, l_{i} ) $, $ ( r_{i}, m_{i} ) \in A $, $ t_{i} \in K^{\times} $, $ i = 1 $, $ 2 $, $ 3 $, define
	\[ g_{2}( ( p_{i}, l_{i} ), t_{i}, ( r_{i}, m_{i} ) ) = \xa{p_{i}}{l_{i}} \cdot \ha{t_{i}} \cdot \wa{0}{\thet} \cdot \xa{r_{i}}{m_{i}}, \]
and let
	\[ e_{i} = g_{2}( ( p_{i}, l_{i} ), t_{i}, ( r_{i}, m_{i} ) ). \]

We will consider the case when $ ( r_{j}, m_{j} ) \neq ( - p_{j + 1}, \ol{l}_{j + 1} ) $, $ j = 1 $, $ 2 $.  For any $ ( p', l' ) $, $ ( r', m' ) \in A $, let $ ( p', l' ) \circ ( r', m' ) $ be defined as the element $ ( s', n' ) \in A $ such that $ \xa{s'}{n'} = \xa{p'}{l'} \cdot \xa{r'}{m'} $, i.e. $ ( p', l' ) \circ ( r', m' ) = ( p' + r', l' + m' - p' \ol{r'} ) $.  (Clearly the operation is associative.)

As a first step, we let
	\[ ( r_{1}, m_{1} ) \circ ( p_{2}, l_{2} ) = ( s_{1}, n_{1} ). \]
Then,
	\[ \delta( e_{1} ) \cdot \delta( e_{2} )
	=	\begin{aligned}[t]
			&\txa{p_{1}}{l_{1}} \cdot \delh{t_{1}} \cdot \twa{0}{\thet} \cdot \txa{s_{1}}{n_{1}} \cdot \delh{t_{2}} \\
			&\phantom{\ } \cdot \twa{0}{\thet} \cdot \txa{r_{2}}{m_{2}}.
		\end{aligned} \]
By Proposition~\ref{p:wa},
	\[ \twa{0}{\thet} \cdot \txa{s_{1}}{n_{1}} \cdot \twa{0}{\thet}^{- 1} = \txam{- \frac{\ol{s_{1}}}{\thet}}{- \frac{n_{1}}{\thet^{2}}}. \]
This implies that
	\[ \delta( e_{1} ) \cdot \delta( e_{2} )
	=	\begin{aligned}[t]
			&\txa{p_{1}}{l_{1}} \cdot \delh{t_{1}} \cdot \txam{- \frac{\ol{s_{1}}}{\thet}}{- \frac{n_{1}}{\thet^{2}}} \cdot \twa{0}{\thet} \\
			&\phantom{\ } \cdot \delh{t_{2}} \cdot \twa{0}{\thet} \cdot \txa{r_{2}}{m_{2}}.
		\end{aligned} \]
By Proposition~\ref{p:wa} again,
	\[ \twa{- \frac{s_{1} \thet}{\ol{n_{1}}}}{- \frac{\thet^{2}}{\ol{n_{1}}}} = \txa{- \frac{s_{1} \thet}{\ol{n_{1}}}}{- \frac{\thet^{2}}{\ol{n_{1}}}} \cdot \txam{- \frac{\ol{s_{1}}}{\thet}}{- \frac{n_{1}}{\thet^{2}}} \cdot \txa{- \frac{s_{1} \thet}{n_{1}}}{- \frac{\thet^{2}}{\ol{n_{1}}}}. \]
Hence inserting the above into our equation,
	\[ \delta( e_{1} ) \cdot \delta( e_{2} )
	=	\begin{aligned}[t]
			&\txa{p_{1}}{l_{1}} \cdot \delh{t_{1}} \cdot \txa{\frac{s_{1} \thet}{\ol{n_{1}}}}{- \frac{\thet^{2}}{n_{1}}} \cdot \twa{- \frac{s_{1} \thet}{\ol{n_{1}}}}{- \frac{\thet^{2}}{\ol{n_{1}}}} \\
			&\phantom{\ } \cdot \txa{\frac{s_{1} \thet}{n_{1}}}{- \frac{\thet^{2}}{n_{1}}} \cdot \twa{0}{\thet} \cdot \delh{t_{2}} \cdot \twa{0}{\thet} \\
			&\phantom{\ } \cdot \txa{r_{2}}{m_{2}}.
		\end{aligned} \]

By Proposition~\ref{p:wa}, for arbitrary $ ( r, m ) $, $ ( r', m' ) $, $ ( p', l' ) \in A $,
\begin{align*}
	&\left[ \twa{r}{m} \cdot \twa{r'}{m'} \right] \cdot \xa{p'}{l'} \cdot \left[ \twa{r}{m} \cdot \twa{r'}{m'} \right]^{- 1} \\
	&= \twa{r}{m} \cdot \txam{\frac{\ol{p'} \cdot \ol{m'}}{( m' )^{2}}}{\frac{l'}{\norm{m'}}} \cdot \twa{r}{m}^{- 1} \\
	&= \txa{\frac{p' \cdot m'}{( \ol{m'} )^{2}} \cdot \frac{m^{2}}{\ol{m}}}{\frac{l'}{\norm{m'}} \cdot \norm{m}},
\end{align*}
therefore by the above, \eqref{e:delh} and \eqref{e:wainv},
\begin{align*}
	&\delh{t_{1}} \cdot \txa{\frac{s_{1} \thet}{\ol{n_{1}}}}{- \frac{\thet^{2}}{n_{1}}} \cdot \delh{t_{1}}^{- 1} \\
	&=	\begin{aligned}[t]
			&\left[ \twa{y( t_{1} )}{\del{1}{t_{1}}} \cdot \twa{0}{\del{2}{t_{1}}}^{- 1} \right] \cdot \txa{\frac{s_{1} \thet}{\ol{n_{1}}}}{- \frac{\thet^{2}}{n_{1}}} \\
			&\phantom{\ } \cdot \left[ \twa{y( t_{1} )}{\del{1}{t_{1}}} \cdot \twa{0}{\del{2}{t_{1}}}^{- 1} \right]^{- 1}
		\end{aligned} \\
	&=	\begin{aligned}[t]
			&\left[ \twa{y( t_{1} )}{\del{1}{t_{1}}} \cdot \twa{0}{\ol{\del{2}{t_{1}}}} \right] \cdot \txa{\frac{s_{1} \thet}{\ol{n_{1}}}}{- \frac{\thet^{2}}{n_{1}}} \\
			&\phantom{\ } \cdot \left[ \twa{y( t_{1} )}{\del{1}{t_{1}}} \cdot \twa{0}{\ol{\del{2}{t_{1}}}} \right]^{- 1}
		\end{aligned} \\
	&= \txa{\frac{s_{1} t_{1}^{2} \thet}{\ol{n_{1}} \ol{t_{1}}}}{- \frac{\norm{t_{1}} \thet^{2}}{n_{1}}},
\end{align*}
where $ y $ is as defined in \eqref{e:ylam}.  Also by the same method,
\begin{align*}
	&\begin{aligned}[t]
		&\left[ \twa{0}{\thet} \cdot \delh{t_{2}} \cdot \twa{0}{\thet} \right]^{- 1} \cdot \txa{\frac{s_{1} \thet}{n_{1}}}{- \frac{\thet^{2}}{n_{1}}} \\
		&\phantom{\ } \cdot \left[ \twa{0}{\thet} \cdot \delh{t_{2}} \cdot \twa{0}{\thet} \right]
	\end{aligned} \\
	&=	\begin{aligned}[t]
			&\left[ \twa{0}{- \thet} \cdot \twa{0}{\del{2}{t_{2}}} \cdot \twa{- y( t_{2} )}{\ol{\del{1}{t_{1}}}} \cdot \twa{0}{- \thet} \right] \\
			&\phantom{\ } \cdot \txa{\frac{s_{1} \thet}{n_{1}}}{- \frac{\thet^{2}}{n_{1}}} \cdot \Big[ \twa{0}{- \thet} \cdot \twa{0}{\del{2}{t_{2}}} \cdot \twa{- y( t_{2} )}{\ol{\del{1}{t_{1}}}} \\
			&\phantom{\ } \cdot \twa{0}{- \thet} \Big]^{- 1}
		\end{aligned} \\
	&=	\begin{aligned}[t]
			&\left[ \twa{0}{- \thet} \cdot \twa{0}{\del{2}{t_{2}}} \right] \cdot \txa{\frac{s_{1} \thet}{n_{1}} \cdot \frac{- \left( \ol{\del{1}{t_{2}}} \right)^{2}}{\del{1}{t_{2}} \thet}}{- \frac{\thet^{2}}{n_{1}} \cdot \frac{\norm{\del{1}{t_{2}}}}{- \thet^{2}}} \\
			&\phantom{\ } \cdot \left[ \twa{0}{- \thet} \cdot \twa{0}{\del{2}{t_{2}}} \right]^{- 1}
		\end{aligned} \\
	&= \txa{- \frac{s_{1} \left( \ol{\del{1}{t_{2}}} \right)^{2}}{n_{1} \del{1}{t_{2}}} \cdot \frac{\del{2}{t_{2}} \thet}{\left( \ol{\del{2}{t_{2}}} \right)^{2}}}{\frac{\norm{\del{1}{t_{2}}}}{n_{1}} \cdot \frac{- \thet^{2}}{\norm{\del{2}{t_{2}}}}} \\
	&= \txa{- \frac{s_{1} ( \ol{t_{2}} )^{2} \thet}{n_{1} t_{2}}}{- \frac{\norm{t_{2}} \thet^{2}}{n_{1}}}.
\end{align*}

Using these in the equation for $ \delta( e_{1} ) \cdot \delta( e_{2} ) $,
\begin{equation} \label{e:de1e2}
	\delta( e_{1} ) \cdot \delta( e_{2} )
		=	\begin{aligned}[t]
				&\txa{p_{1}}{l_{1}} \cdot \txa{\frac{s_{1} t_{1}^{2} \thet}{\ol{n_{1}} \ol{t_{1}}}}{- \frac{\norm{t_{1}} \thet^{2}}{n_{1}}} \cdot \delh{t_{1}} \\
				&\phantom{\ } \cdot \twa{- \frac{s_{1} \thet}{\ol{n_{1}}}}{- \frac{\thet^{2}}{\ol{n_{1}}}} \cdot \twa{0}{\thet} \cdot \delh{t_{2}} \cdot \twa{0}{\thet} \\
				&\phantom{\ } \cdot \txa{- \frac{s_{1} ( \ol{t_{2}} )^{2} \thet}{n_{1} t_{2}}}{- \frac{\norm{t_{2}} \thet^{2}}{n_{1}}} \cdot \txa{r_{2}}{m_{2}}.
			\end{aligned}
\end{equation}

Also, since the above applies also in the group $ G( k ) $ with the corresponding group elements, and
	\[ \wa{- \frac{s_{1} \thet}{\ol{n_{1}}}}{- \frac{\thet^{2}}{\ol{n_{1}}}} \cdot \wa{0}{\thet} = \ha{\frac{\thet}{\ol{n_{1}}}}, \]
this implies that
	\[ e_{1} e_{2}
	=	\begin{aligned}[t]
			&\xa{p_{1}}{l_{1}} \cdot \xa{\frac{s_{1} t_{1}^{2} \thet}{\ol{n_{1}} \ol{t_{1}}}}{- \frac{\norm{t_{1}} \thet^{2}}{n_{1}}} \cdot \ha{\frac{t_{1} t_{2} \thet}{\ol{n_{1}}}} \cdot \wa{0}{\thet} \\
			&\phantom{\ } \cdot \xa{- \frac{s_{1} ( \ol{t_{2}} )^{2} \thet}{n_{1} t_{2}}}{- \frac{\norm{t_{2}} \thet^{2}}{n_{1}}} \cdot \xa{r_{2}}{m_{2}};
		\end{aligned} \]
i.e.
\begin{equation} \label{e:e1e2}
	e_{1} e_{2}
	=	\begin{aligned}[t]
			g_{2} &\Bigg[ ( p_{1}, l_{1} ) \circ \left( \frac{s_{1} t_{1}^{2} \thet}{\ol{n_{1}} \ol{t_{1}}}, - \frac{\norm{t_{1}} \thet^{2}}{n_{1}} \right), \frac{t_{1} t_{2} \thet}{\ol{n}}, \\
			&\phantom{\qquad} \left( - \frac{s_{1} ( \ol{t_{2}} )^{2} \thet}{n_{1} t_{2}}, - \frac{\norm{t_{2}} \thet^{2}}{n_{1}} \right) \circ ( r_{2}, m_{2} ) \Bigg].
		\end{aligned}
\end{equation}

Let $ \lambda = a + b \thet $, with $ a \in k $, $ b \in k^{\times} $.  Define
	\[ n_{1} = \del{1}{\lambda} = - \frac{1}{2} - \frac{a}{2 b \thet}, \quad n' = \del{2}{\lambda} = - \frac{1}{2 b \thet}. \]
If $ \lambda \in k^{\times} \thet $, then $ n_{1} \in k^{\times} $.  As we will see later, this implies that there are two cases to consider.

Let $ t_{3} = \ol{n'} / \thet $, $ t_{1}^{-1} = t_{2} = n_{1} / \thet $.  Choose $ ( r_{1}, m_{1} ) $, $ ( p_{2}, r_{2} ) \in A $ such that
	\[ ( r_{1}, m_{1} ) \circ ( p_{2}, l_{2} ) = ( 1, n_{1} ), \]
which is always possible.  Also, choose $ ( r_{2}, m_{2} ) $ and $ ( p_{3}, l_{3} ) \in A $ such that
	\[ ( r_{2}, m_{2} ) \circ ( p_{3}, l_{3} ) = ( 0, n' ). \]
Then by \eqref{e:e1e2},
\begin{align*}
	e_{1} e_{2}
	&= g_{2} \left( ( p_{1}, l_{1} ) \circ \left( - \frac{\thet^{2}}{n_{1}^{2}}, \frac{\thet^{4}}{n_{1}^{2} \ol{n_{1}}} \right), \frac{\thet}{\ol{n_{1}}}, \left( - \frac{( \ol{n_{1}} )^{2}}{n_{1}^{2}}, \ol{n_{1}} \right) \circ ( r_{2}, m_{2} ) \right), \\
	e_{2} e_{3}
	&= g_{2} \left( ( p_{2}, l_{2} ) \circ \left( 0, \frac{\norm{n_{1}}}{n'} \right), \frac{n_{1}}{\thet}, \left( 0, \ol{n'} \right) \circ ( r_{3}, m_{3} ) \right).
\end{align*}

Let
	\[ \left[ \left( - \frac{( \ol{n_{1}} )^{2}}{n_{1}^{2}}, \ol{n_{1}} \right) \circ ( r_{2}, m_{2} ) \right] \circ ( p_{3}, l_{3} ) = \left( - \frac{( \ol{n_{1}} )^{2}}{n_{1}^{2}}, \ol{n_{1}} \right) \circ ( 0, n' ) = ( - \frac{( \ol{n_{1}} )^{2}}{n_{1}^{2}} , n'' ), \]
so that
	\[ n'' = \ol{n_{1}} + n' \neq 0; \]
and let
	\[ ( r_{1}, m_{1} ) \circ \left[ ( p_{2}, l_{2} ) \circ \left( 0, \frac{\norm{n_{1}}}{n'} \right) \right] = ( 1, n_{1} ) \circ \left( 0, \frac{\norm{n_{1}}}{n'} \right) = ( 1, n'''), \]
so that
	\[ n''' = n_{1} + \frac{\norm{n_{1}}}{n'} \neq 0. \]
This implies that by \eqref{e:e1e2},
\begin{align*}
	e_{1} e_{2} \cdot e_{3}
	&=	\begin{aligned}[t]
			g_{2} &\left[ ( p_{1}, l_{1} ) \circ \left( - \frac{\thet^{2}}{n_{1}^{2}}, \frac{\thet^{4}}{n_{1}^{2} \ol{n_{1}}} \right) \circ \left( \frac{\thet^{2}}{n_{1} \ol{n''}}, \frac{\thet^{4}}{\norm{n_{1}} n''} \right), \frac{\ol{n'} \thet}{\ol{n_{1}} \ol{n''}}, \right. \\
			&\left. \phantom{\qquad} \left( - \frac{( \ol{n_{1}} )^{2} n'}{n_{1}^{2} n''}, \frac{\norm{n'}}{n''} \right) \circ ( r_{3}, m_{3} ) \right],
		\end{aligned} \\
	e_{1} \cdot e_{2} e_{3}
	&=	\begin{aligned}[t]
			g_{2} &\left[ ( p_{1}, l_{1} ) \circ \left( - \frac{\ol{n_{1}} \thet^{2}}{n_{1}^{2} \ol{n'''}}, \frac{\thet^{4}}{\norm{n_{1}} n'''} \right), \frac{\thet}{\ol{n'''}}, \right. \\
			&\left. \phantom{\qquad} \left( - \frac{( \ol{n_{1}} )^{2}}{n''' n_{1}}, \frac{\norm{n_{1}}}{n'''} \right) \circ \left( 0, \ol{n'} \right) \circ ( r_{3}, m_{3} ) \right].
		\end{aligned}
\end{align*}

Since for all $ g $, $ g' \in G( k ) $,
	\[ \sigu{g}{g'} = \delta( g ) \cdot \delta( g' ) \cdot \delta( g g' )^{-1}, \]
we will work out \sigu{e_{1}}{e_{2}}, \sigu{e_{1} e_{2}}{e_{3}}, \sigu{e_{2}}{e_{3}} and \sigu{e_{1}}{e_{2} e_{3}} and use the 2-cocycle condition to get our result.

We first assume that $ \lambda \notin k^{\times} \thet $, so that $ n_{1} \notin k^{\times} $.  Firstly, by \eqref{e:de1e2},
	\[ \delta( e_{1} ) \cdot \delta( e_{2} ) 
		=	\begin{aligned}[t]
				&\txa{p_{1}}{l_{1}} \cdot \txa{- \frac{\thet^{2}}{n_{1}^{2}}}{\frac{\thet^{4}}{n_{1}^{2} \ol{n_{1}}}} \cdot \delh{\frac{\thet}{n_{1}}} \cdot \twa{- \frac{\thet}{\ol{n_{1}}}}{- \frac{\thet^{2}}{\ol{n_{1}}}} \\
				&\phantom{\ } \cdot \twa{0}{\thet} \cdot \delh{\frac{n_{1}}{\thet}} \cdot \twa{0}{\thet} \cdot \txa{- \frac{( \ol{n_{1}} )^{2}}{n_{1}^{2}}}{\ol{n_{1}}} \\
				&\phantom{\ } \cdot \txa{r_{2}}{m_{2}}. 
			\end{aligned} \]
Consider (using \eqref{e:wainv})
	\[ \twa{1}{n_{1}} \cdot \twa{- \frac{\thet}{\ol{n_{1}}}}{- \frac{\thet^{2}}{\ol{n_{1}}}}^{- 1} = \twa{- 1}{\ol{n_{1}}}^{- 1} \cdot \twa{\frac{\thet}{\ol{n_{1}}}}{- \frac{\thet^{2}}{n_{1}}}. \]
By Lemma~\ref{l:2.16},
\begin{align*}
	&\twa{1}{n_{1}} \cdot \twa{- \frac{\thet}{\ol{n_{1}}}}{- \frac{\thet^{2}}{\ol{n_{1}}}}^{- 1} \\
	&= \twa{- 1}{\ol{n_{1}}}^{- 1} \cdot \twa{\frac{\thet}{\ol{n_{1}}}}{- \frac{\thet^{2}}{n_{1}}} \\
	&=	\begin{aligned}[t]
			&\hilkn{- \frac{( - 1 / 2 )^{2} \norm{n_{1} / \thet}}{( 1 / ( 2 \thet^{2} ) )^{2} \norm{n_{1}} \thet^{2}}}{- \frac{a}{2 b \thet^{2}} + \frac{( a / ( 2 b \thet^{2} ) ) ( 1 / ( 2 \thet^{2} ) ) \thet^{2}}{- 1 / 2}} \\
			&\phantom{\ } \cdot \hilkn{- \frac{( 1 / ( 2 \thet^{2} ) ) \norm{n_{1}}}{- 1 / 2}}{\norm{\frac{n_{1}}{\thet}}} \cdot \delh{\norm{\frac{n_{1}}{\thet}}}
		\end{aligned} \\
	&= \hilkn{1}{- \frac{a}{b \thet^{2}}} \cdot \hilkn{\frac{\norm{n_{1}}}{\thet^{2}}}{- \frac{\norm{n_{1}}}{\thet^{2}}} \cdot \delh{- \frac{\norm{n_{1}}}{\thet^{2}}};
\end{align*}
and hence by \eqref{e:prop3},
	\[ \twa{1}{n_{1}} \cdot \twa{- \frac{\thet}{\ol{n_{1}}}}{- \frac{\thet^{2}}{\ol{n_{1}}}}^{- 1} = \delh{- \frac{\norm{n_{1}}}{\thet^{2}}}. \]
(Note that the calculation of
	\[ \twa{- \frac{\thet}{\ol{n_{1}}}}{- \frac{\thet^{2}}{\ol{n_{1}}}} \cdot \twa{1}{n_{1}}^{- 1} = \twa{\frac{\thet}{\ol{n_{1}}}}{- \frac{\thet^{2}}{n_{1}}}^{- 1} \cdot \twa{- 1}{\ol{n_{1}}} \]
is unsuitable as it gives $ q_{1} = 0 $ in Lemma~\ref{l:2.16}.)  Hence,
\begin{equation} \label{e:2.29a}
	\twa{- \frac{\thet}{\ol{n_{1}}}}{- \frac{\thet^{2}}{\ol{n_{1}}}} = \delh{- \frac{\norm{n_{1}}}{\thet^{2}}}^{- 1} \cdot \twa{1}{n_{1}},
\end{equation}
and this clearly applies to any $ n_{1} \in \del{1}{K^{\times}} $ with $ n_{1} \notin k^{\times} $.  By our previous calculations in this proof, we have
\begin{equation} \label{e:2.29b}
\sigu{e_{1}}{e_{2}}
	=	\begin{aligned}[t]
			&\txa{p_{1}}{l_{1}} \cdot \txa{- \frac{\thet^{2}}{n_{1}^{2}}}{\frac{\thet^{4}}{n_{1}^{2} \ol{n_{1}}}} \cdot \delh{\frac{\thet}{n_{1}}} \\
			&\phantom{\ } \cdot \twa{- \frac{\thet}{\ol{n_{1}}}}{- \frac{\thet^{2}}{\ol{n_{1}}}} \cdot \twa{0}{\thet} \cdot \delh{\frac{n_{1}}{\thet}}\\
			&\phantom{\ } \cdot \delh{\frac{\thet}{\ol{n_{1}}}}^{- 1} \cdot \txa{- \frac{\thet^{2}}{n_{1}^{2}}}{\frac{\thet^{4}}{n_{1}^{2} \ol{n_{1}}}}^{- 1} \cdot \txa{p_{1}}{l_{1}}^{- 1}.
		\end{aligned}
\end{equation}
Applying \eqref{e:2.29a} to the above,
	\[ \sigu{e_{1}}{e_{2}}
	=	\begin{aligned}[t]
			&\txa{p_{1}}{l_{1}} \cdot \txa{- \frac{\thet^{2}}{n_{1}^{2}}}{\frac{\thet^{4}}{n_{1}^{2} \ol{n_{1}}}} \cdot \delh{\frac{\thet}{n_{1}}} \\
			&\phantom{\ } \cdot \left[ \delh{- \frac{\norm{n_{1}}}{\thet^{2}}}^{- 1} \cdot \twa{1}{n_{1}} \right] \cdot \twa{0}{\thet} \cdot \delh{\frac{n_{1}}{\thet}} \\
			&\phantom{\ }\cdot \delh{\frac{\thet}{\ol{n_{1}}}}^{- 1} \cdot \txa{- \frac{\thet^{2}}{n_{1}^{2}}}{\frac{\thet^{4}}{n_{1}^{2} \ol{n_{1}}}}^{- 1} \cdot \txa{p_{1}}{l_{1}}^{- 1};
		\end{aligned} \]
and since $ \sigu{g}{g'} $ is central in $ \wtilde{G} $, the first two terms must cancel the last two terms, giving
	\[ \sigu{e_{1}}{e_{2}}
	=	\begin{aligned}[t]
			&\delh{\frac{\thet}{n_{1}}} \cdot \delh{- \frac{\norm{n_{1}}}{\thet^{2}}}^{- 1} \cdot \twa{1}{n_{1}} \cdot \twa{0}{\thet} \\
			&\phantom{\ } \cdot \delh{\frac{n_{1}}{\thet}} \cdot \delh{\frac{\thet}{\ol{n_{1}}}}^{- 1}.
		\end{aligned} \]
By \eqref{e:wainv},
	\[ \sigu{e_{1}}{e_{2}}
	=	\begin{aligned}[t]
			&\delh{\frac{\thet}{n_{1}}} \cdot \delh{- \frac{\norm{n_{1}}}{\thet^{2}}}^{- 1} \cdot \twa{1}{n_{1}} \cdot \twa{0}{- \thet}^{- 1} \\
			&\phantom{\ } \cdot \delh{\frac{n_{1}}{\thet}} \cdot \delh{\frac{\thet}{\ol{n_{1}}}}^{- 1},
		\end{aligned} \]
which implies by \eqref{e:delh} that
	\[ \sigu{e_{1}}{e_{2}}
	=	\begin{aligned}[t]
			&\delh{\frac{\thet}{n_{1}}} \cdot \delh{- \frac{\norm{n_{1}}}{\thet^{2}}}^{- 1} \cdot \delh{- \frac{n_{1}}{\thet}} \\
			&\phantom{\ } \cdot \delh{\frac{n_{1}}{\thet}} \cdot \delh{\frac{\thet}{\ol{n_{1}}}}^{- 1}.
		\end{aligned} \]
We have by \eqref{e:ba} that
	\[ \ba{- \frac{\norm{n_{1}}}{\thet^{2}}}{\frac{\thet}{\ol{n_{1}}}} = \delh{- \frac{\norm{n_{1}}}{\thet^{2}}} \cdot \delh{\frac{\thet}{\ol{n_{1}}}} \cdot \delh{- \frac{n_{1}}{\thet}}^{- 1}, \]
i.e.
	\[ \delh{- \frac{\norm{n_{1}}}{\thet^{2}}}^{- 1} \cdot \delh{- \frac{n_{1}}{\thet}} = \ba{- \frac{\norm{n_{1}}}{\thet^{2}}}{\frac{\thet}{\ol{n_{1}}}}^{- 1} \cdot \delh{\frac{\thet}{\ol{n_{1}}}}. \]
Replacing the above in the equation for \sigu{e_{1}}{e_{2}},
	\[ \sigu{e_{1}}{e_{2}}
	=	\begin{aligned}[t]
			&\delh{\frac{\thet}{n_{1}}} \cdot \left[ \ba{- \frac{\norm{n_{1}}}{\thet^{2}}}{\frac{\thet}{\ol{n_{1}}}}^{- 1} \cdot \delh{\frac{\thet}{\ol{n_{1}}}} \right] \\
			&\phantom{\ } \cdot \delh{\frac{n_{1}}{\thet}} \cdot \delh{\frac{\thet}{\ol{n_{1}}}}^{- 1}.
		\end{aligned} \]
Using the notation in Lemma~\ref{l:comm},
\begin{align*}
	\sigu{e_{1}}{e_{2}}
	&=	\begin{aligned}[t]
			&\ba{- \frac{\norm{n_{1}}}{\thet^{2}}}{\frac{\thet}{\ol{n_{1}}}}^{- 1} \cdot \delh{\frac{\thet}{n_{1}}} \cdot \Bigg[ \left[ \frac{\thet}{\ol{n_{1}}}, \frac{n_{1}}{\thet} \right]_{\sigma_{u}} \\
			&\phantom{\ } \cdot \delh{\frac{n_{1}}{\thet}} \cdot \delh{\frac{\thet}{\ol{n_{1}}}} \Bigg] \cdot \delh{\frac{\thet}{\ol{n_{1}}}}^{- 1}
		\end{aligned} \\
	&=	\begin{aligned}[t]
			&\ba{- \frac{\norm{n_{1}}}{\thet^{2}}}{\frac{\thet}{\ol{n_{1}}}}^{- 1} \cdot \left[ \frac{\thet}{\ol{n_{1}}}, \frac{n_{1}}{\thet} \right]_{\sigma_{u}} \cdot \delh{\frac{\thet}{n_{1}}} \\
			&\phantom{\ } \cdot \delh{\frac{n_{1}}{\thet}},
		\end{aligned}
\end{align*}
thus by \eqref{e:ba},
	\[ \sigu{e_{1}}{e_{2}} = \ba{- \frac{\norm{n_{1}}}{\thet^{2}}}{\frac{\thet}{\ol{n_{1}}}}^{- 1} \cdot \left[ \frac{\thet}{\ol{n_{1}}}, \frac{n_{1}}{\thet} \right]_{\sigma_{u}} \cdot \ba{\frac{\thet}{n_{1}}}{\frac{n_{1}}{\thet}}. \]
By Lemma~\ref{l:comm},
	\[ \left[ \frac{\thet}{\ol{n_{1}}}, \frac{n_{1}}{\thet} \right]_{\sigma_{u}} = \hilbkn{\frac{\thet}{\ol{n_{1}}}}{\frac{n_{1}}{\thet}}^{2} \cdot \hilbkn{\frac{\thet}{\ol{n_{1}}}}{- \frac{\ol{n_{1}}}{\thet}}^{- 1}. \]
Using the properties of Hilbert symbols, by \eqref{e:prop3},
	\[ \hilbkn{\frac{\thet}{\ol{n_{1}}}}{\frac{n_{1}}{\thet}} = \hilbkn{\frac{\thet}{\ol{n_{1}}}}{- \frac{n_{1}}{\thet} \cdot \frac{\ol{n_{1}}}{\thet}}; \]
and by \eqref{e:norm2} and \eqref{e:prop3},
	\[ \hilbkn{\frac{\thet}{\ol{n_{1}}}}{\frac{n_{1}}{\thet}} = \hilkn{\norm{\frac{\thet}{n_{1}}}}{\norm{\frac{n_{1}}{\thet}}} = \hilkn{\norm{\frac{\thet}{n_{1}}}}{- 1}.  \]
Therefore by \eqref{e:prop5},
	\[ \hilbkn{\frac{\thet}{\ol{n_{1}}}}{\frac{n_{1}}{\thet}}^{2} = \hilkn{\norm{\frac{\thet}{n_{1}}}}{( - 1 )^{2}} = 1. \]
Also by \eqref{e:prop3},
	\[ \hilbkn{\frac{\thet}{\ol{n_{1}}}}{- \frac{\ol{n_{1}}}{\thet}} = 1. \]
Hence, we have
	\[ \left[ \frac{\thet}{\ol{n_{1}}}, \frac{n_{1}}{\thet} \right]_{\sigma_{u}} = 1. \]
By Proposition~\ref{p:2.23ii},
\begin{align*}
	\ba{- \frac{\norm{n_{1}}}{\thet^{2}}}{\frac{\thet}{\ol{n_{1}}}}^{- 1}
	&= \hilkn{- \frac{\norm{n_{1}}}{\thet^{2}}}{\frac{( \thet / \ol{n_{1}} ) \ol{\del{1}{\thet / \ol{n_{1}}}}}{\thet}}^{- 1} \\
	&= \hilkn{- \frac{\norm{n_{1}}}{\thet^{2}}}{1}^{- 1} \\
	&= 1.
\end{align*}
Thus,
	\[ \sigu{e_{1}}{e_{2}} = \ba{\frac{\thet}{n_{1}}}{\frac{n_{1}}{\thet}}. \]

Using \eqref{e:de1e2} also gives
\begin{align*}
	\sigu{e_{1} e_{2}}{e_{3}}
	&= [ \delta( e_{1} e_{2} ) \cdot \delta( e_{3} ) ] \cdot \delta( e_{1} e_{2} \cdot e_{3} )^{-1} \\
	&=	\begin{aligned}[t]
			&\txa{p_{1}}{l_{1}} \cdot \txa{- \frac{\thet^{2}}{n_{1}^{2}}}{\frac{\thet^{4}}{n_{1}^{2} \ol{n_{1}}}} \cdot \txa{\frac{\thet^{2}}{n_{1} \ol{n''}}}{\frac{\thet^{4}}{\norm{n_{1}} n''}} \\
			&\phantom{\ } \cdot \delh{\frac{\thet}{\ol{n_{1}}}} \cdot \twa{\frac{( \ol{n_{1}} )^{2} \thet}{n_{1}^{2} \ol{n''}}}{- \frac{\thet^{2}}{\ol{n''}}} \cdot \twa{0}{\thet} \cdot \delh{\frac{\ol{n'}}{\thet}} \\
			&\phantom{\ } \cdot \delh{\frac{\ol{n'} \thet}{\ol{n_{1}} \ol{n''}}}^{-1} \cdot \bigg[ \txa{p_{1}}{l_{1}} \cdot \txa{- \frac{\thet^{2}}{n_{1}^{2}}}{\frac{\thet^{4}}{n_{1}^{2} \ol{n_{1}}}} \\
			&\phantom{\ } \cdot \txa{\frac{\thet^{2}}{n_{1} \ol{n''}}}{\frac{\thet^{4}}{\norm{n_{1}} n''}} \bigg]^{-1},
		\end{aligned}
\end{align*}
and since $ \sigu{e_{1} e_{2}}{e_{3}} $ is a central element of $ \wtilde{G} $,
\begin{equation} \label{e:2.29c}
	\sigu{e_{1} e_{2}}{e_{3}}
	=	\begin{aligned}[t]
			&\delh{\frac{\thet}{\ol{n_{1}}}} \cdot \twa{\frac{( \ol{n_{1}} )^{2} \thet}{n_{1}^{2} \ol{n''}}}{- \frac{\thet^{2}}{\ol{n''}}} \cdot \twa{0}{\thet} \\
			&\phantom{\ } \cdot \delh{\frac{\ol{n'}}{\thet}} \cdot \delh{\frac{\ol{n'} \thet}{\ol{n_{1}} \ol{n''}}}^{-1}.
		\end{aligned}
\end{equation}
Consider
\begin{align*}
	\twa{\frac{( \ol{n_{1}} )^{2} \thet}{n_{1}^{2} \ol{n''}}}{- \frac{\thet^{2}}{\ol{n''}}} &\cdot \twa{\frac{\thet}{\ol{n''}}}{- \frac{\thet^{2}}{\ol{n''}}}^{- 1} \\
	&=	\begin{aligned}[t]
			&\left[ \twa{\frac{( \ol{n_{1}} )^{2} \thet}{n_{1}^{2} \ol{n''}}}{- \frac{\thet^{2}}{\ol{n''}}} \cdot \twa{\frac{\ol{n_{1}} \thet}{n_{1} \ol{n''}}}{- \frac{\thet^{2}}{\ol{n''}}}^{- 1} \right] \\
			&\phantom{\ } \cdot \left[ \twa{\frac{\ol{n_{1}} \thet}{n_{1} \ol{n''}}}{- \frac{\thet^{2}}{\ol{n''}}} \cdot \twa{\frac{\thet}{\ol{n''}}}{- \frac{\thet^{2}}{\ol{n''}}}^{- 1} \right].
		\end{aligned}
\end{align*}
By Lemma~\ref{l:2.18}, it is obvious that
	\[ \twa{\frac{( \ol{n_{1}} )^{2} \thet}{n_{1}^{2} \ol{n''}}}{- \frac{\thet^{2}}{\ol{n''}}} \cdot \twa{\frac{\ol{n_{1}} \thet}{n_{1} \ol{n''}}}{- \frac{\thet^{2}}{\ol{n''}}}^{- 1} = \twa{\frac{\ol{n_{1}} \thet}{n_{1} \ol{n''}}}{- \frac{\thet^{2}}{\ol{n''}}} \cdot \twa{\frac{\thet}{\ol{n''}}}{- \frac{\thet^{2}}{\ol{n''}}}^{- 1}, \]
so this implies that
\begin{multline*}
	\twa{\frac{( \ol{n_{1}} )^{2} \thet}{n_{1}^{2} \ol{n''}}}{- \frac{\thet^{2}}{\ol{n''}}} \cdot \twa{\frac{\thet}{\ol{n''}}}{- \frac{\thet^{2}}{\ol{n''}}}^{- 1} \\
	= \left[ \twa{\frac{\ol{n_{1}} \thet}{n_{1} \ol{n''}}}{- \frac{\thet^{2}}{\ol{n''}}} \cdot \twa{\frac{\thet}{\ol{n''}}}{- \frac{\thet^{2}}{\ol{n''}}}^{- 1} \right]^{2}.
\end{multline*}
Hence, by Lemma~\ref{l:2.18},
\begin{align*}
	&\twa{\frac{\ol{n_{1}} \thet}{n_{1} \ol{n''}}}{- \frac{\thet^{2}}{\ol{n''}}} \cdot \twa{\frac{\thet}{\ol{n''}}}{- \frac{\thet^{2}}{\ol{n''}}}^{- 1} \\
	&=	\begin{aligned}[t]
			&\Bigg( \frac{\left( \thet^{2} / ( 2 \norm{n''} ) \right)^{2} \norm{\ol{n_{1}}}}{( a / \left( 2 b \thet^{2} \right) )^{2} \norm{- \thet^{2} / \ol{n''}} \thet^{2}}, \\
			&\phantom{\qquad} \frac{\left( \left( \thet^{2} / ( 2 \norm{n''} ) \right) ( - 1 / 2 ) - \left( - ( a - 1 ) / \left( 2 b \norm{n''} \right) \right) \left( a / \left( 2 b \thet^{2} \right) \right) \thet^{2} \right)^{2}}{\left( \thet^{2} / ( 2 \norm{n''} ) \right)^{2} \norm{\ol{n_{1}}}} \Bigg)_{k, n}
		\end{aligned} \\
	&= \hilkn{\frac{b^{2} \norm{n_{1}} \thet^{2}}{a^{2} \norm{n''}}}{\frac{( ( a - 1 ) a - b^{2} \thet^{2} )^{2}}{4 b^{4} \norm{n_{1}} \thet^{4}}}
\end{align*}
if $ ( a - 1 ) a - b^{2} \thet^{2} \neq 0 $ (we will deal with the case $ ( a - 1 ) a - b^{2} \thet^{2} = 0 $ later).  Since $ \lambda = n_{1} / n' $ and $ n' = - 1 / ( 2 b \thet ) $,
	\[ \twa{\frac{\ol{n_{1}} \thet}{n_{1} \ol{n''}}}{- \frac{\thet^{2}}{\ol{n''}}} \cdot \twa{\frac{\thet}{\ol{n''}}}{- \frac{\thet^{2}}{\ol{n''}}}^{- 1} = \hilkn{- \frac{\norm{\lambda}}{4 a^{2} \norm{n''}}}{- \frac{( ( a - 1 ) a - b^{2} \thet^{2} )^{2}}{\norm{\lambda} b^{2} \thet^{2}}}. \]
This implies that
\begin{align*}
	&\twa{\frac{( \ol{n_{1}} )^{2} \thet}{n_{1}^{2} \ol{n''}}}{- \frac{\thet^{2}}{\ol{n''}}} \cdot \twa{\frac{\thet}{\ol{n''}}}{- \frac{\thet^{2}}{\ol{n''}}}^{- 1} \\
	&= \hilkn{- \frac{\norm{\lambda}}{4 a^{2} \norm{n''}}}{- \frac{( ( a - 1 ) a - b^{2} \thet^{2} )^{2}}{\norm{\lambda} b^{2} \thet^{2}}}^{2} \\
	&= \hilkn{\frac{\norm{\lambda}}{4 a^{2} \norm{n''}}}{\frac{( ( a - 1 ) a - b^{2} \thet^{2} )^{4}}{( \norm{\lambda} )^{2} b^{4} \thet^{4}}}
\end{align*}
by \eqref{e:prop5}.

If $ ( a - 1 ) a - b^{2} \thet^{2} = 0 $, then by Lemma~\ref{l:2.18},
\begin{align*}
	&\twa{\frac{\ol{n_{1}} \thet}{n_{1} \ol{n''}}}{- \frac{\thet^{2}}{\ol{n''}}} \cdot \twa{\frac{\thet}{\ol{n''}}}{- \frac{\thet^{2}}{\ol{n''}}}^{- 1} \\
	&=	\begin{aligned}[t]
			&\Bigg( \frac{\left( \left( \thet^{2} / ( 2 \norm{n''} ) \right) ( - 1 / 2 ) + \left( - ( a - 1 ) / \left( 2 b \norm{n''} \right) \right) \left( a / \left( 2 b \thet^{2} \right) \right) \thet^{2} \right)^{2}}{\left( \thet^{2} / ( 2 \norm{n''} ) \right)^{2} \norm{\ol{n_{1}}}}, \\
			&\phantom{\qquad} \frac{\left( \thet^{2} / ( 2 \norm{n''} ) \right)^{2} \norm{\ol{n_{1}}}}{( a / \left( 2 b \thet^{2} \right) )^{2} \norm{- \thet^{2} / \ol{n''}} \thet^{2}} \Bigg)_{k, n}
		\end{aligned} \\
	&= \hilkn{\frac{( ( a - 1 ) a + b^{2} \thet^{2} )^{2}}{4 b^{4} \norm{n_{1}} \thet^{4}}}{\frac{b^{2} \norm{n_{1}} \thet^{2}}{a^{2} \norm{n''}}}.
\end{align*}
Since $ \lambda = n_{1} / n' $ and $ n' = - 1 / ( 2 b \thet ) $,
	\[ \twa{\frac{\ol{n_{1}} \thet}{n_{1} \ol{n''}}}{- \frac{\thet^{2}}{\ol{n''}}} \cdot \twa{\frac{\thet}{\ol{n''}}}{- \frac{\thet^{2}}{\ol{n''}}}^{- 1} = \hilkn{- \frac{( ( a - 1 ) a + b^{2} \thet^{2} )^{2}}{\norm{\lambda} b^{2}\thet^{2}}}{- \frac{\norm{\lambda}}{4 a^{2} \norm{n''}}}, \]
thus
\begin{align*}
	&\twa{\frac{( \ol{n_{1}} )^{2} \thet}{n_{1}^{2} \ol{n''}}}{- \frac{\thet^{2}}{\ol{n''}}} \cdot \twa{\frac{\thet}{\ol{n''}}}{- \frac{\thet^{2}}{\ol{n''}}}^{- 1} \\
	&= \hilkn{- \frac{( ( a - 1 ) a + b^{2} \thet^{2} )^{2}}{\norm{\lambda} b^{2}\thet^{2}}}{- \frac{\norm{\lambda}}{4 a^{2} \norm{n''}}}^{2} \\
	&= \hilkn{\frac{( ( a - 1 ) a + b^{2} \thet^{2} )^{4}}{( \norm{\lambda} )^{2} b^{4}\thet^{4}}}{- \frac{\norm{\lambda}}{4 a^{2} \norm{n''}}}
\end{align*}
by \eqref{e:prop5}.  Let $ c_{1} $ denote the function
	\[ c_{1}( \lambda )
	=	\begin{cases}
			\displaystyle \hilkn{- \frac{\norm{\lambda}}{4 a^{2} \norm{n''}}}{\frac{( ( a - 1 ) a - b^{2} \thet^{2} )^{4}}{( \norm{\lambda} )^{2} b^{4} \thet^{4}}}, &\text{if $ ( a - 1 ) a - b^{2} \thet^{2} \neq 0 $;} \\
			\displaystyle \hilkn{\frac{( ( a - 1 ) a + b^{2} \thet^{2} )^{4}}{( \norm{\lambda} )^{2} b^{4}\thet^{4}}}{- \frac{\norm{\lambda}}{4 a^{2} \norm{n''}}}, &\text{if $ ( a - 1 ) a - b^{2} \thet^{2} = 0 $.}
		\end{cases} \]

By \eqref{e:wainv} and Lemma~\ref{l:2.16},
\begin{align*}
	&\twa{1}{n''} \cdot \twa{\frac{\thet}{\ol{n''}}}{- \frac{\thet^{2}}{\ol{n''}}}^{- 1} \\
	&= \twa{- 1}{\ol{n''}}^{- 1} \cdot \twa{- \frac{\thet}{\ol{n''}}}{- \frac{\thet^{2}}{n''}} \\
	&=	\begin{aligned}[t]
			&\hilkn{- \frac{( - 1 / 2 )^{2} (- \norm{n''} / ( \thet^{2} ) )}{( - 1 / ( 2 \thet^{2} ) )^{2} \norm{n''} \thet^{2}}}{- \frac{a - 1}{2 b \thet^{2}} + \frac{(- ( a - 1 ) / ( 2 b \thet^{2} ) ) ( - 1 / ( 2 \thet^{2} ) ) \thet^{2}}{- 1 / 2}} \\
			&\phantom{\ } \cdot \hilkn{- \frac{( - 1 / ( 2 \thet^{2} ) ) \norm{n''}}{- 1 / 2}}{- \frac{\norm{n''}}{\thet^{2}}} \cdot \delh{- \frac{\norm{n''}}{\thet^{2}}}
		\end{aligned} \\
	&= \hilkn{1}{\frac{a - 1}{b \thet^{2}}} \cdot \hilkn{- \frac{\norm{n''}}{\thet^{2}}}{- \frac{\norm{n''}}{\thet^{2}}} \cdot \delh{- \frac{\norm{n''}}{\thet^{2}}}
\end{align*}
if $ a \neq 1 $.  (Note that
	\[ \twa{\frac{\thet}{\ol{n''}}}{- \frac{\thet^{2}}{\ol{n''}}} \cdot \twa{1}{n''}^{- 1} = \twa{- \frac{\thet}{\ol{n''}}}{- \frac{\thet^{2}}{n''}}^{- 1} \cdot \twa{- 1}{\ol{n''}} \]
gives $ q_{1} = 0 $ in Lemma~\ref{l:2.16}.)  This implies that by \eqref{e:prop3},
	\[ \twa{1}{n''} \cdot \twa{\frac{\thet}{\ol{n''}}}{- \frac{\thet^{2}}{\ol{n''}}}^{- 1} = \hilkn{- \frac{\norm{n''}}{\thet^{2}}}{- 1} \cdot \delh{- \frac{\norm{n''}}{\thet^{2}}}, \]
i.e.
\begin{equation} \label{e:2.29ci}
	\twa{\frac{\thet}{\ol{n''}}}{- \frac{\thet^{2}}{\ol{n''}}} = \hilkn{- \frac{\norm{n''}}{\thet^{2}}}{- 1} \cdot \delh{- \frac{\norm{n''}}{\thet^{2}}}^{- 1} \cdot \twa{1}{n''}
\end{equation}
(by \eqref{e:prop5}).

If  $ a = 1 $, then $ n'' = - 1 / 2 $.  This implies that by \eqref{e:wainv} and Lemma~\ref{l:2.16},
\begin{align*}
	\twa{\frac{\thet}{\ol{n''}}}{- \frac{\thet^{2}}{\ol{n''}}} \cdot \twa{1}{n''}^{- 1}
	&= \twa{- 2 \thet}{2 \thet^{2}} \cdot \twa{1}{- \frac{1}{2}}^{- 1} \\
	&= \twa{2 \thet}{2 \thet^{2}}^{- 1} \cdot \twa{- 1}{- \frac{1}{2}} \\
	&= \hilkn{- \frac{2 \thet^{2}}{( - 2 \thet ) \thet}}{\norm{- 2 \thet}} \cdot \delh{\norm{- 2 \thet}} \\
	&= \delh{\norm{- 2 \thet}},
\end{align*}
so that
	\[ \twa{\frac{\thet}{\ol{n''}}}{- \frac{\thet^{2}}{\ol{n''}}} = \delh{- 4 \thet^{2}} \cdot \twa{1}{- \frac{1}{2}}. \]
Thus, let $ c_{2} $ be the function
	\[ c_{2}( \lambda )
	=	\begin{cases}
			\displaystyle \hilkn{- \frac{\norm{n''}}{\thet^{2}}}{- 1} \cdot \delh{- \frac{\norm{n''}}{\thet^{2}}}^{- 1} \cdot \twa{1}{n''}, &\text{if $ a \neq 1 $;} \\
			\displaystyle \delh{- 4 \thet^{2}} \cdot \twa{1}{- \frac{1}{2}}, &\text{if $ a = 1 $.}
		\end{cases} \]
This implies that
\begin{align*}
	&\twa{\frac{( \ol{n_{1}} )^{2} \thet}{n_{1}^{2} \ol{n''}}}{- \frac{\thet^{2}}{\ol{n''}}} \\
	&= \left[ \twa{\frac{( \ol{n_{1}} )^{2} \thet}{n_{1}^{2} \ol{n''}}}{- \frac{\thet^{2}}{\ol{n''}}} \cdot \twa{\frac{\thet}{\ol{n''}}}{- \frac{\thet^{2}}{\ol{n''}}}^{- 1} \right] \cdot \twa{\frac{\thet}{\ol{n''}}}{- \frac{\thet^{2}}{\ol{n''}}} \\
	&= c_{1}( \lambda ) \cdot c_{2}( \lambda ).
\end{align*}
By the above, there are three cases to consider: $ ( a - 1 ) a - b^{2} \thet^{2} \neq 0 $, $ a \neq 1 $; $ ( a - 1 ) a - b^{2} \thet^{2} = 0 $, $ a \neq 1 $; and $ ( a - 1 ) a - b^{2} \thet^{2} \neq 0 $, $ a = 1 $.

We can consider the first two cases together, as they differ only by $ c_{1}( \lambda ) $ which is a central element in $ \wtilde{G} $.  Thus,
	\[ \sigu{e_{1} e_{2}}{e_{3}}
	=	\begin{aligned}[t]
			&\delh{\frac{\thet}{\ol{n_{1}}}} \cdot \Bigg[ c_{1}( \lambda ) \cdot \hilkn{- \frac{\norm{n''}}{\thet^{2}}}{- 1} \cdot \delh{- \frac{\norm{n''}}{\thet^{2}}}^{- 1} \\
			&\phantom{\ } \cdot \twa{1}{n''} \Bigg] \cdot \twa{0}{\thet} \cdot \delh{\frac{\ol{n'}}{\thet}} \cdot \delh{\frac{\ol{n'} \thet}{\ol{n_{1}} \ol{n''}}}^{-1}.
		\end{aligned} \]
Since by \eqref{e:wainv}, $ \twa{0}{\thet} = \twa{0}{- \thet}^{- 1} $, this implies by \eqref{e:delh} that
	\[ \sigu{e_{1} e_{2}}{e_{3}}
	=	\begin{aligned}[t]
			&c_{1}( \lambda ) \cdot \hilkn{- \frac{\norm{n''}}{\thet^{2}}}{- 1} \cdot \delh{\frac{\thet}{\ol{n_{1}}}} \cdot \delh{- \frac{\norm{n''}}{\thet^{2}}}^{- 1} \\
			&\phantom{\ } \cdot \delh{- \frac{n''}{\thet}} \cdot \delh{\frac{\ol{n'}}{\thet}} \cdot \delh{\frac{\ol{n'} \thet}{\ol{n_{1}} \ol{n''}}}^{-1}.
		\end{aligned} \]
By \eqref{e:ba},
	\[ \ba{- \frac{\norm{n''}}{\thet^{2}}}{\frac{\thet}{\ol{n''}}}^{- 1} \cdot \delh{\frac{\thet}{\ol{n''}}} = \delh{- \frac{\norm{n''}}{\thet^{2}}}^{- 1} \cdot \delh{- \frac{n''}{\thet}}, \]
which implies that
	\[ \sigu{e_{1} e_{2}}{e_{3}}
	=	\begin{aligned}[t]
			&c_{1}( \lambda ) \cdot \hilkn{- \frac{\norm{n''}}{\thet^{2}}}{- 1} \cdot \delh{\frac{\thet}{\ol{n_{1}}}} \cdot \Bigg[ \ba{- \frac{\norm{n''}}{\thet^{2}}}{\frac{\thet}{\ol{n''}}}^{- 1} \\
			&\phantom{\ } \cdot \delh{\frac{\thet}{\ol{n''}}} \Bigg] \cdot \delh{\frac{\ol{n'}}{\thet}} \cdot \delh{\frac{\ol{n'} \thet}{\ol{n_{1}} \ol{n''}}}^{-1}.
		\end{aligned} \]
By Lemma~\ref{l:comm},
	\[ \sigu{e_{1} e_{2}}{e_{3}}
	=	\begin{aligned}[t]
			&c_{1}( \lambda ) \cdot \hilkn{- \frac{\norm{n''}}{\thet^{2}}}{- 1} \cdot \ba{- \frac{\norm{n''}}{\thet^{2}}}{\frac{\thet}{\ol{n''}}}^{- 1} \cdot \delh{\frac{\thet}{\ol{n_{1}}}} \\
			&\phantom{\ } \cdot \left[ \left[ \frac{\thet}{\ol{n''}}, \frac{\ol{n'}}{\thet} \right]_{\sigma_{u}} \cdot \delh{\frac{\ol{n'}}{\thet}} \cdot \delh{\frac{\thet}{\ol{n''}}} \right]  \\
			&\phantom{\ } \cdot \delh{\frac{\ol{n'} \thet}{\ol{n_{1}} \ol{n''}}}^{-1}.
		\end{aligned} \]
Hence by \eqref{e:ba},
\begin{align*}
	\sigu{e_{1} e_{2}}{e_{3}}
	&=	\begin{aligned}[t]
			&c_{1}( \lambda ) \cdot \hilkn{- \frac{\norm{n''}}{\thet^{2}}}{- 1}  \cdot \ba{- \frac{\norm{n''}}{\thet^{2}}}{\frac{\thet}{\ol{n''}}}^{- 1} \cdot \left[ \frac{\thet}{\ol{n''}}, \frac{\ol{n'}}{\thet} \right]_{\sigma_{u}} \\
			&\phantom{\ } \cdot \delh{\frac{\thet}{\ol{n_{1}}}} \cdot \delh{\frac{\ol{n'}}{\thet}} \cdot \left[ \delh{\frac{1}{\ol{\lambda}}}^{- 1} \cdot \delh{\frac{1}{\ol{\lambda}}} \right] \\
			&\phantom{\ } \cdot \delh{\frac{\thet}{\ol{n''}}} \cdot \delh{\frac{\ol{n'} \thet}{\ol{n_{1}} \ol{n''}}}^{-1}
		\end{aligned} \\
	&=	\begin{aligned}[t]
			&c_{1}( \lambda ) \cdot \hilkn{- \frac{\norm{n''}}{\thet^{2}}}{- 1} \cdot \ba{- \frac{\norm{n''}}{\thet^{2}}}{\frac{\thet}{\ol{n''}}}^{- 1} \cdot \left[ \frac{\thet}{\ol{n''}}, \frac{\ol{n'}}{\thet} \right]_{\sigma_{u}} \\
			&\phantom{\ } \cdot \ba{\frac{\thet}{\ol{n_{1}}}}{\frac{\ol{n'}}{\thet}} \cdot \ba{\frac{1}{\ol{\lambda}}}{\frac{\thet}{\ol{n''}}}.
		\end{aligned}
\end{align*}
By Proposition~\ref{p:2.23ii} and \eqref{e:prop5},
\begin{gather*}
	\begin{aligned}
		\ba{- \frac{\norm{n''}}{\thet^{2}}}{\frac{\thet}{\ol{n''}}}^{- 1}
		&= \hilkn{- \frac{\norm{n''}}{\thet^{2}}}{\frac{( \thet / \ol{n''} ) \ol{\del{1}{\thet / \ol{n''}}}}{\thet}}^{- 1} \\
		&= \hilkn{- \frac{\norm{n''}}{\thet^{2}}}{1}^{- 1} \\
		&= 1,
	\end{aligned} \\
	\ba{\frac{\thet}{\ol{n_{1}}}}{\frac{\ol{n'}}{\thet}} = \hilkn{\frac{\ol{n'}}{\thet}}{- \frac{\del{2}{\thet / \ol{n_{1}}}}{\thet}} = \hilkn{\frac{\ol{n'}}{\thet}}{- \frac{\norm{n_{1}}}{\thet^{2}}}.
\end{gather*}
Also, Lemma~\ref{l:comm} shows that
	\[ \left[ \frac{\thet}{\ol{n''}}, \frac{\ol{n'}}{\thet} \right]_{\sigma_{u}} = \hilbkn{\frac{\thet}{\ol{n''}}}{\frac{\ol{n'}}{\thet}}^{2} \cdot \hilbkn{\frac{\thet}{\ol{n''}}}{- \frac{n'}{\thet}}^{- 1}. \]
Since $ n' / \thet \in k^{\times} $, by \eqref{e:bastr} and \eqref{e:norm2}, we have
	\[ \left[ \frac{\thet}{\ol{n''}}, \frac{\ol{n'}}{\thet} \right]_{\sigma_{u}} = \hilkn{- \frac{\thet^{2}}{\norm{n''}}}{\frac{\ol{n'}}{\thet}}. \]

Therefore by \eqref{e:prop2} and \eqref{e:bastr},
\begin{align*}
	\sigu{e_{1} e_{2}}{e_{3}}
	&=	\begin{aligned}[t]
			&c_{1}( \lambda ) \cdot \hilkn{- \frac{\norm{n''}}{\thet^{2}}}{- 1} \cdot 1 \cdot \hilkn{- \frac{\thet^{2}}{\norm{n''}}}{\frac{\ol{n'}}{\thet}} \\
			&\phantom{\ } \cdot \hilkn{\frac{\ol{n'}}{\thet}}{- \frac{\norm{n_{1}}}{\thet^{2}}} \cdot \ba{\frac{1}{\ol{\lambda}}}{\frac{\thet}{\ol{n''}}}
		\end{aligned} \\
	&= c_{1}( \lambda ) \cdot \hilkn{- \frac{\norm{n''}}{\thet^{2}}}{- 1} \cdot \hilkn{\frac{\ol{n'}}{\thet}}{\frac{\norm{n_{1}} \norm{n''}}{\thet^{4}}} \cdot \ba{\frac{1}{\ol{\lambda}}}{\frac{\thet}{\ol{n''}}}.
\end{align*}

As for the last case ($ ( a - 1 ) a - b^{2} \thet^{2} \neq 0 $, $ a = 1 $), recall that in this case $ n'' = - 1 / 2 $.  This implies that
	\[ \sigu{e_{1} e_{2}}{e_{3}}
	=	\begin{aligned}[t]
			&\delh{\frac{\thet}{\ol{n_{1}}}} \cdot \left[ c_{1}( \lambda ) \cdot \delh{- 4 \thet^{2}} \cdot \twa{1}{- \frac{1}{2}} \right] \cdot \twa{0}{\thet} \\
			&\phantom{\ } \cdot \delh{\frac{\ol{n'}}{\thet}} \cdot \delh{- \frac{2 \ol{n'} \thet}{\ol{n_{1}}}}^{-1}.
		\end{aligned} \]
Since by \eqref{e:wainv}, $ \twa{0}{\thet} = \twa{0}{- \thet}^{- 1} $, this implies by \eqref{e:delh} that
	\[ \sigu{e_{1} e_{2}}{e_{3}}
	=	\begin{aligned}[t]
			&c_{1}( \lambda ) \cdot \delh{\frac{\thet}{\ol{n_{1}}}} \cdot \delh{- 4 \thet^{2}} \cdot \delh{\frac{1}{2 \thet}} \\
			&\phantom{\ } \cdot \delh{\frac{\ol{n'}}{\thet}} \cdot \delh{- \frac{2 \ol{n'} \thet}{\ol{n_{1}}}}^{-1}.
		\end{aligned} \]
By \eqref{e:ba},
	\[ \ba{- 4 \thet^{2}}{\frac{1}{2 \thet}} \cdot \delh{- 2 \thet} = \delh{- 4 \thet^{2}} \cdot \delh{\frac{1}{2 \thet}}. \]
Hence, the equation we have becomes
	\[ \sigu{e_{1} e_{2}}{e_{3}}
	=	\begin{aligned}[t]
			&c_{1}( \lambda ) \cdot \delh{\frac{\thet}{\ol{n_{1}}}} \cdot \left[ \ba{- 4 \thet^{2}}{\frac{1}{2 \thet}} \cdot \delh{- 2 \thet} \right] \\
			&\phantom{\ } \cdot \delh{\frac{\ol{n'}}{\thet}} \cdot \delh{- \frac{2 \ol{n'} \thet}{\ol{n_{1}}}}^{-1}.
		\end{aligned} \]
Using Lemma~\ref{l:comm}, we get
	\[ \sigu{e_{1} e_{2}}{e_{3}}
	=	\begin{aligned}[t]
			&c_{1}( \lambda ) \cdot \ba{- 4 \thet^{2}}{\frac{1}{2 \thet}} \cdot \delh{\frac{\thet}{\ol{n_{1}}}} \cdot \Bigg[ \left[ - 2 \thet, \frac{\ol{n'}}{\thet} \right]_{\sigma_{u}} \\
			&\phantom{\ } \cdot \delh{\frac{\ol{n'}}{\thet}} \cdot \delh{- 2 \thet} \Bigg] \cdot \delh{- \frac{2 \ol{n'} \thet}{\ol{n_{1}}}}^{-1}.
		\end{aligned} \]
Thus by \eqref{e:ba},
\begin{align*}
	\sigu{e_{1} e_{2}}{e_{3}}
	&=	\begin{aligned}[t]
			&c_{1}( \lambda ) \cdot \ba{- 4 \thet^{2}}{\frac{1}{2 \thet}} \cdot \left[ - 2 \thet, \frac{\ol{n'}}{\thet} \right]_{\sigma_{u}} \cdot \delh{\frac{\thet}{\ol{n_{1}}}} \\
			&\phantom{\ } \cdot \delh{\frac{\ol{n'}}{\thet}} \cdot \left[ \delh{\frac{1}{\ol{\lambda}}}^{- 1} \cdot \delh{\frac{1}{\ol{\lambda}}} \right] \\
			&\phantom{\ } \cdot \delh{- 2 \thet} \cdot \delh{- \frac{2 \ol{n'} \thet}{\ol{n_{1}}}}^{-1}
		\end{aligned} \\
	&=	\begin{aligned}[t]
			&c_{1}( \lambda ) \cdot \ba{- 4 \thet^{2}}{\frac{1}{2 \thet}} \cdot \left[ - 2 \thet, \frac{\ol{n'}}{\thet} \right]_{\sigma_{u}} \cdot \ba{\frac{\thet}{\ol{n_{1}}}}{\frac{\ol{n'}}{\thet}} \\
			&\phantom{\ } \cdot \ba{\frac{1}{\ol{\lambda}}}{- 2 \thet}.
		\end{aligned}
\end{align*}
By Proposition~\ref{p:2.23ii},
\begin{align} \label{e:2.29f}
	\ba{- 4 \thet^{2}}{\frac{1}{2 \thet}} &= \hilkn{- 4 \thet^{2}}{\frac{( 1 / ( 2 \thet ) ) \ol{\del{1}{1 / ( 2 \thet )}}}{\thet}} \\ &= \hilkn{- 4 \thet^{2}}{- \frac{1}{4 \thet^{2}}}, \notag
\end{align}
and by \eqref{e:prop3},
	\[ \hilkn{- 4 \thet^{2}}{- \frac{1}{4 \thet^{2}}} = \hilkn{- 4 \thet^{2}}{- 1}. \]
Therefore by our previous calculations,
	\[ \sigu{e_{1} e_{2}}{e_{3}}
	=	\begin{aligned}[t]
			&c_{1}( \lambda ) \cdot \hilkn{- 4 \thet^{2}}{- 1} \cdot \hilkn{- 4 \thet^{2}}{\frac{\ol{n'}}{\thet}} \cdot \hilkn{\frac{\ol{n'}}{\thet}}{- \frac{\norm{n_{1}}}{\thet^{2}}} \\
			&\phantom{\ } \cdot \ba{\frac{1}{\ol{\lambda}}}{- 2 \thet},
		\end{aligned} \]
and by \eqref{e:prop2} and \eqref{e:bastr},
	\[ \sigu{e_{1} e_{2}}{e_{3}}
	= c_{1}( \lambda ) \cdot \hilkn{- 4 \thet^{2}}{- 1} \cdot \hilkn{\frac{\ol{n'}}{\thet}}{\frac{\norm{n_{1}}}{4 \thet^{4}}} \cdot \ba{\frac{1}{\ol{\lambda}}}{- 2 \thet}. \]
In fact,
\begin{multline*}
	c_{1}( \lambda ) \cdot \hilkn{- 4 \thet^{2}}{- 1} \cdot \hilkn{\frac{\ol{n'}}{\thet}}{\frac{\norm{n_{1}}}{4 \thet^{4}}} \cdot \ba{\frac{1}{\ol{\lambda}}}{- 2 \thet} \\
	= c_{1}( \lambda ) \cdot \hilkn{- \frac{\norm{n''}}{\thet^{2}}}{- 1} \cdot \hilkn{\frac{\ol{n'}}{\thet}}{\frac{\norm{n_{1}} \norm{n''}}{\thet^{4}}} \cdot \ba{\frac{1}{\ol{\lambda}}}{\frac{\thet}{\ol{n''}}};
\end{multline*}
i.e. we can use the same equation
\begin{equation} \label{e:2.29cii}
	\sigu{e_{1} e_{2}}{e_{3}}
	=	\begin{aligned}[t]
			&c_{1}( \lambda ) \cdot \hilkn{- \frac{\norm{n''}}{\thet^{2}}}{- 1} \cdot \hilkn{\frac{\ol{n'}}{\thet}}{\frac{\norm{n_{1}} \norm{n''}}{\thet^{4}}} \\
			&\phantom{\ } \cdot \ba{\frac{1}{\ol{\lambda}}}{\frac{\thet}{\ol{n''}}}
		\end{aligned}
\end{equation}
for all three cases.

Also, by similar methods to the above, we can show that
	\[ \sigu{e_{1}}{e_{2} e_{3}}
		=	\begin{aligned}[t]
				&\delh{\frac{\thet}{n_{1}}} \cdot \twa{- \frac{\thet}{\ol{n'''}}}{- \frac{\thet^{2}}{\ol{n'''}}} \cdot \twa{0}{\thet} \\
				&\phantom{\ } \cdot \delh{\frac{n_{1}}{\thet}} \cdot \delh{\frac{\thet}{\ol{n'''}}}^{-1}.
			\end{aligned} \]
We know that
	\[ n''' = n_{1} + \frac{\norm{n_{1}}}{n'} = \frac{a + b \thet}{- 2 b \thet} + \norm{\frac{a + b \thet}{- 2 b \thet}} \cdot \left( - \frac{1}{2 b \thet} \right)^{- 1} = - \frac{1}{2} + \frac{a ( a - 1 ) - b^{2} \thet^{2}}{2 b \thet}. \]
There are two cases to consider: $ a ( a - 1 ) - b^{2} \thet^{2} \neq 0 $; and $ a ( a - 1 ) - b^{2} \thet^{2} = 0 $.

When $ a ( a - 1 ) - b^{2} \thet^{2} \neq 0 $, we can use \eqref{e:2.29a} to get
	\[ \twa{- \frac{\thet}{\ol{n'''}}}{- \frac{\thet^{2}}{\ol{n'''}}} = \delh{- \frac{\norm{n'''}}{\thet^{2}}}^{- 1} \cdot \twa{1}{n'''}, \]
thus
	\[ \sigu{e_{1}}{e_{2} e_{3}}
	=	\begin{aligned}[t]
			&\delh{\frac{\thet}{n_{1}}} \cdot \left[ \delh{- \frac{\norm{n'''}}{\thet^{2}}}^{- 1} \cdot \twa{1}{n'''} \right] \cdot \twa{0}{\thet} \\
			&\phantom{\ } \cdot \delh{\frac{n_{1}}{\thet}} \cdot \delh{\frac{\thet}{\ol{n'''}}}^{-1}.
		\end{aligned} \]
Since by \eqref{e:wainv}, $ \twa{0}{\thet} = \twa{0}{- \thet}^{- 1} $, by \eqref{e:delh} the equation becomes
	\[ \sigu{e_{1}}{e_{2} e_{3}}
	=	\begin{aligned}[t]
			&\delh{\frac{\thet}{n_{1}}} \cdot \delh{- \frac{\norm{n'''}}{\thet^{2}}}^{- 1} \cdot \delh{- \frac{n'''}{\thet}} \\
			&\phantom{\ } \cdot \delh{\frac{n_{1}}{\thet}} \cdot \delh{\frac{\thet}{\ol{n'''}}}^{-1}.
		\end{aligned} \]
Using \eqref{e:ba},
	\[ \ba{- \frac{\norm{n'''}}{\thet^{2}}}{\frac{\thet}{\ol{n'''}}}^{- 1} \cdot \delh{\frac{\thet}{\ol{n'''}}} = \delh{- \frac{\norm{n'''}}{\thet^{2}}}^{- 1} \cdot \delh{- \frac{n'''}{\thet}}. \]
We thus get
	\[ \sigu{e_{1}}{e_{2} e_{3}}
	=	\begin{aligned}[t]
			&\delh{\frac{\thet}{n_{1}}} \cdot \left[ \ba{- \frac{\norm{n'''}}{\thet^{2}}}{\frac{\thet}{\ol{n'''}}}^{- 1} \cdot \delh{\frac{\thet}{\ol{n'''}}} \right] \\
			&\phantom{\ } \cdot \delh{\frac{n_{1}}{\thet}} \cdot \delh{\frac{\thet}{\ol{n'''}}}^{-1}.
		\end{aligned} \]
By Lemma \ref{l:comm},
	\[ \sigu{e_{1}}{e_{2} e_{3}}
	=	\begin{aligned}[t]
			&\ba{- \frac{\norm{n'''}}{\thet^{2}}}{\frac{\thet}{\ol{n'''}}}^{- 1} \cdot \delh{\frac{\thet}{n_{1}}} \cdot \Bigg[ \left[ \frac{\thet}{\ol{n'''}}, \frac{n_{1}}{\thet} \right]_{\sigma_{u}} \\
			&\phantom{\ } \cdot \delh{\frac{n_{1}}{\thet}} \cdot \delh{\frac{\thet}{\ol{n'''}}} \Bigg] \cdot \delh{\frac{\thet}{\ol{n'''}}}^{-1},
		\end{aligned} \]
and hence by \eqref{e:ba},
	\[ \sigu{e_{1}}{e_{2} e_{3}}
	=	\ba{- \frac{\norm{n'''}}{\thet^{2}}}{\frac{\thet}{\ol{n'''}}}^{- 1} \cdot \left[ \frac{\thet}{\ol{n'''}}, \frac{n_{1}}{\thet} \right]_{\sigma_{u}} \cdot \ba{\frac{\thet}{n_{1}}}{\frac{n_{1}}{\thet}}. \]
By Proposition~\ref{p:2.23ii},
\begin{align*}
	\ba{- \frac{\norm{n'''}}{\thet^{2}}}{\frac{\thet}{\ol{n'''}}}^{- 1}
	&= \hilkn{- \frac{\norm{n'''}}{\thet^{2}}}{\frac{( \thet / \ol{n'''} ) \ol{\del{1}{\thet / \ol{n'''}}}}{\thet}}^{- 1} \\
	&= \hilkn{- \frac{\norm{n'''}}{\thet^{2}}}{1}^{- 1} \\
	&= 1.
\end{align*}
Also, by Lemma~\ref{l:comm},
	\[ \left[ \frac{\thet}{\ol{n'''}}, \frac{n_{1}}{\thet} \right]_{\sigma_{u}} = \hilbkn{\frac{\thet}{\ol{n'''}}}{\frac{n_{1}}{\thet}}^{2} \cdot \hilbkn{\frac{\thet}{\ol{n'''}}}{- \frac{\ol{n_{1}}}{\thet}}^{- 1}. \]
Since $ n''' = n_{1} + \norm{n_{1}} / n' = n_{1} ( 1 - \ol{\lambda} ) $ and $ n_{1} = \lambda n' $, we have
	\[ \left[ \frac{\thet}{\ol{n'''}}, \frac{n_{1}}{\thet} \right]_{\sigma_{u}} = \hilbkn{- \frac{\thet}{\ol{\lambda} n' ( 1 - \lambda )}}{\frac{\lambda n'}{\thet}}^{2} \cdot \hilbkn{- \frac{\thet}{\ol{\lambda} n' ( 1 - \lambda )}}{\frac{\ol{\lambda} n'}{\thet}}^{- 1}. \]
By \eqref{e:bastr},
\begin{align*}
	\left[ \frac{\thet}{\ol{n'''}}, \frac{n_{1}}{\thet} \right]_{\sigma_{u}}
	&=	\begin{aligned}[t]
			&\left[ \hilbkn{- \frac{\thet}{\ol{\lambda} n' ( 1 - \lambda )}}{\lambda} \cdot \hilbkn{- \frac{\thet}{\ol{\lambda} n' ( 1 - \lambda )}}{\frac{n'}{\thet}} \right]^{2} \\
			&\phantom{\ } \cdot \left[ \hilbkn{- \frac{\thet}{\ol{\lambda} n' ( 1 - \lambda )}}{\ol{\lambda}} \cdot \hilbkn{- \frac{\thet}{\ol{\lambda} n' ( 1 - \lambda )}}{\frac{n'}{\thet}} \right]^{- 1}
		\end{aligned} \\
	&=	\begin{aligned}[t]
			&\hilbkn{- \frac{\thet}{\ol{\lambda} n' ( 1 - \lambda )}}{\lambda}^{2} \cdot \hilbkn{- \frac{\thet}{\ol{\lambda} n' ( 1 - \lambda )}}{\ol{\lambda}}^{- 1} \\
			&\phantom{\ } \cdot \hilbkn{- \frac{\thet}{\ol{\lambda} n' ( 1 - \lambda )}}{\frac{n'}{\thet}}.
		\end{aligned}
\end{align*}
By \eqref{e:prop3},
	\[ \left[ \frac{\thet}{\ol{n'''}}, \frac{n_{1}}{\thet} \right]_{\sigma_{u}}
	=	\begin{aligned}[t]
			&\hilbkn{\frac{\thet}{\norm{\lambda} n' ( 1 - \lambda )}}{\lambda}^{2} \cdot \hilbkn{\frac{\thet}{n' ( 1 - \lambda )}}{\ol{\lambda}}^{- 1} \\
			&\phantom{\ } \cdot \hilbkn{\frac{1}{\ol{\lambda} ( 1 - \lambda )}}{\frac{n'}{\thet}}.
		\end{aligned} \]
Using \eqref{e:prop4}, we have
	\[ \left[ \frac{\thet}{\ol{n'''}}, \frac{n_{1}}{\thet} \right]_{\sigma_{u}}
	=	\begin{aligned}[t]
			&\hilbkn{\frac{\thet}{\norm{\lambda} n'}}{\lambda}^{2} \cdot \hilbkn{\frac{\thet}{n' \norm{1 - \lambda}}}{\ol{\lambda}}^{- 1} \\
			&\phantom{\ } \cdot \hilbkn{\frac{1}{\ol{\lambda} ( 1 - \lambda )}}{\frac{n'}{\thet}}.
		\end{aligned} \]
By \eqref{e:norm1} and \eqref{e:norm2},
	\[ \left[ \frac{\thet}{\ol{n'''}}, \frac{n_{1}}{\thet} \right]_{\sigma_{u}}
	=	\begin{aligned}[t]
			&\hilkn{\frac{\thet}{\norm{\lambda} n'}}{\norm{\lambda}}^{2} \cdot \hilkn{\frac{\thet}{n' \norm{1 - \lambda}}}{\norm{\lambda}}^{- 1} \\
			&\phantom{\ } \cdot \hilkn{\frac{1}{\norm{\lambda} \norm{1 - \lambda}}}{\frac{n'}{\thet}}.
		\end{aligned} \]
Using \eqref{e:bastr} again,
\begin{align*}
	\left[ \frac{\thet}{\ol{n'''}}, \frac{n_{1}}{\thet} \right]_{\sigma_{u}}
	&=	\begin{aligned}[t]
			&\left[ \hilkn{\frac{1}{\norm{\lambda}}}{\norm{\lambda}} \cdot \hilkn{\frac{\thet}{n'}}{\norm{\lambda}} \right]^{2} \\
			&\phantom{\ } \cdot \left[ \hilkn{\frac{1}{\norm{1 - \lambda}}}{\norm{\lambda}} \cdot \hilkn{\frac{\thet}{n'}}{\norm{\lambda}} \right]^{- 1} \\
			&\phantom{\ } \cdot \hilkn{\frac{1}{\norm{\lambda} \norm{1 - \lambda}}}{\frac{n'}{\thet}}
		\end{aligned} \\
	&=	\begin{aligned}[t]
			&\hilkn{\frac{1}{\norm{\lambda}}}{\norm{\lambda}}^{2} \cdot \hilkn{\frac{1}{\norm{1 - \lambda}}}{\norm{\lambda}}^{- 1} \cdot \hilkn{\frac{\thet}{n'}}{\norm{\lambda}} \\
			&\phantom{\ } \cdot \hilkn{\frac{1}{\norm{\lambda} \norm{1 - \lambda}}}{\frac{n'}{\thet}}.
		\end{aligned}
\end{align*}
By \eqref{e:prop3},
	\[ \left[ \frac{\thet}{\ol{n'''}}, \frac{n_{1}}{\thet} \right]_{\sigma_{u}}
	=	\begin{aligned}[t]
			&\hilkn{- 1}{\norm{\lambda}}^{2} \cdot \hilkn{\frac{1}{\norm{1 - \lambda}}}{\norm{\lambda}}^{- 1} \cdot \hilkn{\frac{\thet}{n'}}{\norm{\lambda}} \\
			&\phantom{\ } \cdot \hilkn{\frac{1}{\norm{\lambda} \norm{1 - \lambda}}}{\frac{n'}{\thet}}.
		\end{aligned} \]
By \eqref{e:prop2} and \eqref{e:prop5},
	\[ \left[ \frac{\thet}{\ol{n'''}}, \frac{n_{1}}{\thet} \right]_{\sigma_{u}} = \hilkn{\norm{1 - \lambda}}{\norm{\lambda}} \cdot \hilkn{\frac{1}{\norm{\lambda}}}{\frac{\thet}{n'}} \cdot \hilkn{\norm{\lambda} \norm{1 - \lambda}}{\frac{\thet}{n'}}. \]
Hence by \eqref{e:bastr} again,
	\[ \left[ \frac{\thet}{\ol{n'''}}, \frac{n_{1}}{\thet} \right]_{\sigma_{u}} = \hilkn{\norm{1 - \lambda}}{\frac{\norm{\lambda} \thet}{n'}} = \hilkn{\frac{\norm{n'''}}{\norm{n}}}{\frac{\norm{\lambda} \thet}{n'}}. \]
Therefore,
	\[ \sigu{e_{1}}{e_{2} e_{3}} = \hilkn{\frac{\norm{n'''}}{\norm{n_{1}}}}{\frac{\norm{\lambda} \thet}{n'}} \cdot \ba{\frac{\thet}{n_{1}}}{\frac{n_{1}}{\thet}}. \]

As for the case $ a ( a - 1 ) - b^{2} \thet^{2} = 0 $, we have $ n''' = - 1 / 2 $.  This implies that
	\[ \sigu{e_{1}}{e_{2} e_{3}}
		=	\begin{aligned}[t]
				&\delh{\frac{\thet}{n_{1}}} \cdot \twa{2 \thet}{2 \thet^{2}} \cdot \twa{0}{\thet} \cdot \delh{\frac{n_{1}}{\thet}} \\
				&\phantom{\ } \cdot \delh{- 2 \thet}^{-1}.
			\end{aligned} \]
By \eqref{e:wainv} and Lemma~\ref{l:2.16},
\begin{align*}
	\twa{2 \thet}{2 \thet^{2}} \cdot \twa{1}{- \frac{1}{2}}^{- 1}
	&= \twa{- 2 \thet}{2 \thet^{2}}^{- 1} \cdot \twa{- 1}{- \frac{1}{2}} \\
	&= \hilkn{- \frac{2 \thet^{2}}{( 2 \thet ) \thet}}{\norm{2 \thet}} \cdot \delh{\norm{2 \thet}} \\
	&= \hilkn{- 1}{- 4 \thet^{2}} \cdot \delh{- 4 \thet^{2}},
\end{align*}
i.e.
\begin{equation} \label{e:2.29e}
	\twa{2 \thet}{2 \thet^{2}} = \hilkn{- 1}{- 4 \thet^{2}} \cdot \delh{- 4 \thet^{2}} \cdot \twa{1}{- \frac{1}{2}}.
\end{equation}
Thus,
	\[ \sigu{e_{1}}{e_{2} e_{3}}
		=	\begin{aligned}[t]
				&\delh{\frac{\thet}{n_{1}}} \cdot \left[ \hilkn{- 1}{- 4 \thet^{2}} \cdot \delh{- 4 \thet^{2}} \cdot \twa{1}{- \frac{1}{2}} \right] \\
				&\phantom{\ } \cdot \twa{0}{\thet} \cdot \delh{\frac{n_{1}}{\thet}} \cdot \delh{- 2 \thet}^{-1}.
			\end{aligned} \]
Since by \eqref{e:wainv}, $ \twa{0}{\thet} = \twa{0}{- \thet}^{- 1} $, utilising \eqref{e:delh} makes the above equation become
	\[ \sigu{e_{1}}{e_{2} e_{3}}
		=	\begin{aligned}[t]
				&\hilkn{- 1}{- 4 \thet^{2}} \cdot \delh{\frac{\thet}{n_{1}}} \cdot \delh{- 4 \thet^{2}} \cdot \delh{\frac{1}{2 \thet}} \\
				&\phantom{\ } \cdot \delh{\frac{n_{1}}{\thet}} \cdot \delh{- 2 \thet}^{-1}.
			\end{aligned} \]
By \eqref{e:ba},
	\[ \ba{- 4 \thet^{2}}{\frac{1}{2 \thet}} \cdot \delh{- 2 \thet} = \delh{- 4 \thet^{2}} \cdot \delh{\frac{1}{2 \thet}}, \]
hence
	\[ \sigu{e_{1}}{e_{2} e_{3}}
		=	\begin{aligned}[t]
				&\hilkn{- 1}{- 4 \thet^{2}} \cdot \delh{\frac{\thet}{n_{1}}} \cdot \left[ \ba{- 4 \thet^{2}}{\frac{1}{2 \thet}} \cdot \delh{- 2 \thet} \right] \\
				&\phantom{\ } \cdot \delh{\frac{n_{1}}{\thet}} \cdot \delh{- 2 \thet}^{-1}.
			\end{aligned} \]
We can use Lemma~\ref{l:comm} to get
	\[ \sigu{e_{1}}{e_{2} e_{3}}
		=	\begin{aligned}[t]
				&\hilkn{- 1}{- 4 \thet^{2}} \cdot \ba{- 4 \thet^{2}}{\frac{1}{2 \thet}} \cdot \delh{\frac{\thet}{n_{1}}} \cdot \Bigg[ \left[ - 2 \thet, \frac{n_{1}}{\thet} \right]_{\sigma_{u}} \\
				&\phantom{\ } \cdot \delh{\frac{n_{1}}{\thet}} \cdot \delh{- 2 \thet} \Bigg] \cdot \delh{- 2 \thet}^{-1}.
			\end{aligned} \]
Hence by \eqref{e:ba} and \eqref{e:2.29f},
	\[ \sigu{e_{1}}{e_{2} e_{3}}
		= \hilkn{- 1}{- 4 \thet^{2}} \cdot \hilkn{- 4 \thet^{2}}{- \frac{1}{4 \thet^{2}}} \cdot \left[ - 2 \thet, \frac{n_{1}}{\thet} \right]_{\sigma_{u}} \cdot \ba{\frac{\thet}{n_{1}}}{\frac{n_{1}}{\thet}}. \]
But since $ n''' = - 1 / 2 $, we have
	\[ \left[ - 2 \thet, \frac{n_{1}}{\thet} \right]_{\sigma_{u}} = \left[ \frac{\thet}{\ol{n'''}}, \frac{n_{1}}{\thet} \right]_{\sigma_{u}}. \]
This implies that we can use our previous result for $ n''' \neq - 1 / 2 $ so that
	\[ \left[ - 2 \thet, \frac{n_{1}}{\thet} \right]_{\sigma_{u}} = \hilkn{\frac{\norm{n'''}}{\norm{n_{1}}}}{\frac{\norm{\lambda} \thet}{n'}}. \]
In the end,
	\[ \sigu{e_{1}}{e_{2} e_{3}}
	=	\begin{aligned}[t]
			&\hilkn{- 1}{- 4 \thet^{2}} \cdot \hilkn{- 4 \thet^{2}}{- \frac{1}{4 \thet^{2}}} \cdot \hilkn{\frac{\norm{n'''}}{\norm{n_{1}}}}{\frac{\norm{\lambda} \thet}{n'}} \\
			&\phantom{\ } \cdot \ba{\frac{\thet}{n_{1}}}{\frac{n_{1}}{\thet}},
		\end{aligned} \]
which simplifies by \eqref{e:prop3} and \eqref{e:bastr} to
	\[ \sigu{e_{1}}{e_{2} e_{3}}
	= \hilkn{\frac{\norm{n'''}}{\norm{n_{1}}}}{\frac{\norm{\lambda} \thet}{n'}} \cdot \ba{\frac{\thet}{n_{1}}}{\frac{n_{1}}{\thet}}. \]
Thus, we may use the above equation for both cases.

Similarly for \sigu{e_{2}}{e_{3}},
	\[ \sigu{e_{2}}{e_{3}}
	=	\begin{aligned}[t]
			&\delh{\frac{n_{1}}{\thet}} \cdot \twa{0}{- \frac{\thet^{2}}{\ol{n'}}} \cdot \twa{0}{\thet} \cdot \delh{\frac{\ol{n'}}{\thet}} \\
			&\phantom{\ } \cdot \delh{\frac{n_{1}}{\thet}}^{-1}.
		\end{aligned} \]
By \eqref{e:delh},
\begin{align*}
	\sigu{e_{2}}{e_{3}}
	&=	\begin{aligned}[t]
			&\delh{\frac{n_{1}}{\thet}} \cdot \twa{0}{- \frac{\thet^{2}}{\ol{n'}}} \cdot \left[ \twa{0}{\thet}^{- 1} \cdot \twa{0}{\thet} \right] \\
			&\phantom{\ } \cdot \twa{0}{\thet} \cdot \delh{\frac{\ol{n'}}{\thet}} \cdot \delh{\frac{n_{1}}{\thet}}^{-1}
		\end{aligned} \\
	&=	\begin{aligned}[t]
			&\delh{\frac{n_{1}}{\thet}} \cdot \delh{- \frac{\thet}{\ol{n'}}} \cdot \delh{-1}^{-1} \cdot \delh{\frac{\ol{n'}}{\thet}} \\
			&\phantom{\ } \cdot \delh{\frac{n_{1}}{\thet}}^{-1}.
		\end{aligned}
\end{align*}
Thus using \eqref{e:delinv},
\begin{align*}
	\sigu{e_{2}}{e_{3}}
	&= \delh{\frac{n_{1}}{\thet}} \cdot \delh{\frac{\ol{n'}}{\thet}}^{-1} \cdot \delh{\frac{\ol{n'}}{\thet}} \cdot \delh{\frac{n_{1}}{\thet}}^{-1} \notag \\
	&= 1.
\end{align*}

Hence by the 2-cocycle condition, i.e.
	\[ \sigu{e_{1}}{e_{2}} \cdot \sigu{e_{1} e_{2}}{e_{3}} = \sigu{e_{1}}{e_{2} e_{3}} \cdot \sigu{e_{2}}{e_{3}}, \]
we have
\begin{multline*}
	\ba{\frac{\thet}{n_{1}}}{\frac{n_{1}}{\thet}} \cdot \Bigg[ c_{1}( \lambda ) \cdot \hilkn{- \frac{\norm{n''}}{\thet^{2}}}{- 1} \cdot \hilkn{\frac{\ol{n'}}{\thet}}{\frac{\norm{n_{1}} \norm{n''}}{\thet^{4}}} \\
	\cdot \ba{\frac{1}{\ol{\lambda}}}{\frac{\thet}{\ol{n''}}} \Bigg] = \left[ \hilkn{\frac{\norm{n'''}}{\norm{n_{1}}}}{\frac{\norm{\lambda} \thet}{n'}} \cdot \ba{\frac{\thet}{n_{1}}}{\frac{n_{1}}{\thet}} \right] \cdot 1.
\end{multline*}
Also, $ n''' = \lambda n'' $ and $ \lambda = n_{1} / n' $, hence we can rearrange the above as follows:
	\[ \ba{\frac{1}{\ol{\lambda}}}{\frac{\thet}{\ol{n''}}}
	=	\begin{aligned}[t]
			&\left[ c_{1}( \lambda ) \cdot \hilkn{- \frac{\norm{n''}}{\thet^{2}}}{- 1} \cdot \hilkn{\frac{\ol{n'}}{\thet}}{\frac{\norm{n_{1}} \norm{n''}}{\thet^{4}}} \right]^{- 1} \\
			&\phantom{\ } \cdot \hilkn{\frac{\norm{n''}}{\norm{n'}}}{\frac{\norm{\lambda} \thet}{n'}}.
		\end{aligned} \]
Thus by \eqref{e:prop5} and \eqref{e:prop2},
	\[ \ba{\frac{1}{\ol{\lambda}}}{\frac{\thet}{\ol{n''}}}
	=	\begin{aligned}[t]
			&c_{1}( \lambda )^{- 1} \cdot \hilkn{- \frac{\norm{n''}}{\thet^{2}}}{- 1} \cdot \hilkn{\frac{\norm{n_{1}} \norm{n''}}{\thet^{4}}}{\frac{\ol{n'}}{\thet}} \\
			&\phantom{\ } \cdot \hilkn{\frac{\norm{n''}}{\norm{n'}}}{\frac{\norm{\lambda} \thet}{n'}}.
		\end{aligned} \]
By \eqref{e:bastr},
\begin{align*}
	\ba{\frac{1}{\ol{\lambda}}}{\frac{\thet}{\ol{n''}}}
	&=	\begin{aligned}[t]
			&c_{1}( \lambda )^{- 1} \cdot \hilkn{- \frac{\norm{n''}}{\thet^{2}}}{- 1} \cdot \hilkn{- \frac{\norm{n''}}{\thet^{2}}}{\frac{\ol{n'}}{\thet}} \\
			&\phantom{\ } \cdot \hilkn{- \frac{\norm{n_{1}}}{\thet^{2}}}{\frac{\ol{n'}}{\thet}} \cdot \hilkn{- \frac{\norm{n''}}{\thet^{2}}}{\frac{\norm{\lambda} \thet}{n'}} \\
			&\phantom{\ } \cdot \hilkn{- \frac{\thet^{2}}{\norm{n'}}}{\norm{\lambda}} \cdot \hilkn{- \frac{\thet^{2}}{\norm{n'}}}{\frac{\thet}{n'}}
		\end{aligned} \\
	&=	\begin{aligned}[t]
			&c_{1}( \lambda )^{- 1} \cdot \hilkn{- \frac{\norm{n''}}{\thet^{2}}}{\norm{\lambda}} \cdot \hilkn{- \frac{\norm{n_{1}}}{\thet^{2}}}{\frac{\ol{n'}}{\thet}} \\
			&\phantom{\ } \cdot \hilkn{- \frac{\thet^{2}}{\norm{n'}}}{\norm{\lambda}} \cdot \hilkn{- \frac{\thet^{2}}{\norm{n'}}}{\frac{\thet}{n'}}
		\end{aligned} \\
	&=	\begin{aligned}[t]
			&c_{1}( \lambda )^{- 1} \cdot \hilkn{\frac{\norm{n''}}{\norm{n'}}}{\norm{\lambda}} \cdot \hilkn{- \frac{\norm{n_{1}}}{\thet^{2}}}{\frac{\ol{n'}}{\thet}} \\
			&\phantom{\ } \cdot \hilkn{- \frac{\thet^{2}}{\norm{n'}}}{\frac{\thet}{n'}}.
		\end{aligned}
\end{align*}
By \eqref{e:prop3},
\begin{gather*}
	\hilkn{- \frac{\norm{n_{1}}}{\thet^{2}}}{\frac{\ol{n'}}{\thet}} = \hilkn{\frac{\norm{n_{1}}}{\ol{n'} \thet}}{\frac{\ol{n'}}{\thet}} = \hilkn{\norm{\lambda}}{\frac{\ol{n'}}{\thet}}, \\
	\hilkn{- \frac{\thet^{2}}{\norm{n'}}}{\frac{\thet}{n'}} = \hilkn{- \frac{\thet}{n'}}{\frac{\thet}{n'}} = 1.
\end{gather*}
This implies that
	\[ \ba{\frac{1}{\ol{\lambda}}}{\frac{\thet}{\ol{n''}}} = c_{1}( \lambda )^{- 1} \cdot \hilkn{\frac{\norm{n''}}{\norm{n'}}}{\norm{\lambda}} \cdot \hilkn{\norm{\lambda}}{\frac{\ol{n'}}{\thet}}; \]
thus by \eqref{e:prop3} and \eqref{e:bastr},
\begin{align*}
	\ba{\frac{1}{\ol{\lambda}}}{\frac{\thet}{\ol{n''}}}
	&= c_{1}( \lambda )^{- 1} \cdot \hilkn{\frac{\norm{n''}}{\norm{n'}}}{\norm{\lambda}} \cdot \hilkn{\frac{\thet}{\ol{n'}}}{\norm{\lambda}} \\
	&= c_{1}( \lambda )^{- 1} \cdot \hilkn{\frac{\norm{n''} \thet}{( n' )^{3}}}{\norm{\lambda}}.
\end{align*}

We will now use Lemma~\ref{l:2.23ii} to get our desired result.  Firstly,
\begin{align*}
	&\ba{\lambda}{\frac{n''}{\thet}}
	= \ba{\frac{1}{\ol{\lambda}} \cdot \norm{\lambda}}{\frac{\thet}{\ol{n''}} \cdot \frac{\norm{n''}}{\thet^{2}}} \\
	&=	\begin{aligned}[t]
			&\ba{\lambda \cdot \frac{1}{\norm{\lambda}}}{\frac{n''}{\thet} \cdot \frac{\thet^{2}}{\norm{n''}}} \cdot \hilkn{- \frac{( n'' / \thet ) \ol{\del{1}{n'' / \thet}} ( \thet^{2} / \norm{n''} )}{n'}}{\frac{1}{\norm{\lambda}}} \\
			&\phantom{\ } \cdot \hilkn{\frac{\del{2}{\lambda ( n '' / \thet )}}{\del{2}{n'' / \thet}}}{\frac{1}{\norm{\lambda}} \cdot \frac{\thet^{2}}{\norm{n''}}},
		\end{aligned}
\end{align*}
since $ \lambda ( n'' / \thet ) \notin k^{\times} $.  Simplifying the above using the properties of Hilbert symbols, we get
	\[ \ba{\lambda}{\frac{n''}{\thet}} = \ba{\frac{1}{\ol{\lambda}}}{\frac{\thet}{\ol{n''}}} \cdot \hilkn{- \frac{\thet}{n'}}{\frac{1}{\norm{\lambda}}}. \]
Inserting the equation we have for \ba{1 / \ol{\lambda}}{\thet / \ol{n''}}, we have
	\[ \ba{\lambda}{\frac{n''}{\thet}} = c_{1}( \lambda )^{- 1} \cdot \hilkn{\frac{\norm{n''} \thet}{( n' )^{3}}}{\norm{\lambda}} \cdot \hilkn{- \frac{\thet}{n'}}{\frac{1}{\norm{\lambda}}}, \]
and by \eqref{e:prop2},
	\[ \ba{\lambda}{\frac{n''}{\thet}} = c_{1}( \lambda )^{- 1} \cdot \hilkn{\frac{\norm{n''} \thet}{( n' )^{3}}}{\norm{\lambda}} \cdot \hilkn{- \frac{n'}{\thet}}{\norm{\lambda}}. \]
Thus by \eqref{e:bastr},
	\[ \ba{\lambda}{\frac{n''}{\thet}} = c_{1}( \lambda )^{- 1} \cdot \hilkn{\frac{\norm{n''}}{\norm{n'}}}{\norm{\lambda}}, \]
and using \eqref{e:prop3} again,
	\[ \ba{\lambda}{\frac{n''}{\thet}} = c_{1}( \lambda )^{- 1} \cdot \hilkn{- \frac{\norm{n''}}{\norm{n_{1}}}}{\norm{\lambda}}. \]
Let $ q \in k^{\times} $ and replace $ \lambda $ with $ \lambda / q $.  This implies that by \eqref{e:del1q} and \eqref{e:del2q},
	\[ \ba{\frac{\lambda}{q}}{\frac{\ol{n_{1}} + q n'}{\thet}} = c_{1}\left( \frac{\lambda}{q} \right)^{- 1} \cdot \hilkn{- \frac{\norm{\ol{n_{1}} + q n'}}{\norm{n_{1}}}}{\frac{\norm{\lambda}}{q^{2}}}. \]
Therefore, choose $ q $, $ q' \in k^{\times} $ such that $ \mu = ( \ol{n_{1}} + q n' ) / ( q' \thet ) $, which is always possible due to Proposition~\ref{p:deltas}.  This implies by the same proposition that
	\[ \del{1}{\mu} = \ol{n_{1}} + q n', \quad \del{2}{\mu} = q' \thet. \]
Hence,
	\[ \ba{\frac{\lambda}{q}}{\frac{\del{1}{\mu}}{\thet}} = c_{1}\left( \frac{\lambda}{q} \right)^{- 1} \cdot \hilkn{- \frac{\norm{\del{1}{\mu}}}{\norm{n_{1}}}}{\frac{\norm{\lambda}}{q^{2}}}. \]
We know that by Lemma~\ref{l:2.23ii}, since $ \lambda \mu \notin k^{\times} $,
\begin{align*}
	\ba{\lambda}{\mu}
	&= \ba{\frac{\lambda}{q} \cdot q}{\frac{\del{1}{\mu}}{\thet} \cdot \frac{1}{q'}} \\
	&= \ba{\lambda \cdot \frac{1}{q}}{\mu \cdot q'} \cdot \hilkn{- \frac{\mu \ol{\del{1}{\mu}} q'}{n'}}{\frac{1}{q}} \cdot \hilkn{\frac{\del{2}{\lambda \mu}}{q' \thet}}{\frac{1}{q} \cdot q'}.
\end{align*}
It can be shown that $ \del{2}{\lambda \mu} = q' \thet / q $, thus with what we have so far the equation becomes
	\[ \ba{\lambda}{\mu} = c_{1}\left( \frac{\lambda}{q} \right)^{- 1} \cdot \hilkn{- \frac{\norm{\del{1}{\mu}}}{\norm{n_{1}}}}{\frac{\norm{\lambda}}{q^{2}}} \cdot \hilkn{- \frac{\mu \ol{\del{1}{\mu}} q'}{n'}}{\frac{1}{q}} \cdot \hilkn{\frac{1}{q}}{\frac{q'}{q}}. \]
By \eqref{e:bastr},
	\[ \ba{\lambda}{\mu}
	=	\begin{aligned}[t]
			&c_{1}\left( \frac{\lambda}{q} \right)^{- 1} \cdot \hilkn{- \frac{\norm{\del{1}{\mu}}}{\norm{n_{1}}}}{\frac{\norm{\lambda}}{q^{2}}} \cdot \Bigg[ \hilkn{\frac{\mu \ol{\del{1}{\mu}}}{n'}}{\frac{1}{q}} \\
			&\phantom{\ } \cdot \hilkn{- q'}{\frac{1}{q}} \Bigg] \cdot \hilkn{\frac{1}{q}}{\frac{q'}{q}};
		\end{aligned} \]
and by \eqref{e:prop2},
	\[ \ba{\lambda}{\mu}
	=	\begin{aligned}[t]
			&c_{1}\left( \frac{\lambda}{q} \right)^{- 1} \cdot \hilkn{- \frac{\norm{\del{1}{\mu}}}{\norm{n_{1}}}}{\frac{\norm{\lambda}}{q^{2}}} \cdot \hilkn{\frac{\mu \ol{\del{1}{\mu}}}{n'}}{\frac{1}{q}} \cdot \hilkn{\frac{1}{q}}{- \frac{1}{q'}} \\
			&\phantom{\ } \cdot \hilkn{\frac{1}{q}}{\frac{q'}{q}}.
		\end{aligned} \]
Thus by \eqref{e:bastr},
	\[ \ba{\lambda}{\mu} = c_{1}\left( \frac{\lambda}{q} \right)^{- 1} \cdot \hilkn{- \frac{\norm{\del{1}{\mu}}}{\norm{n_{1}}}}{\frac{\norm{\lambda}}{q^{2}}} \cdot \hilkn{\frac{\mu \ol{\del{1}{\mu}}}{n'}}{\frac{1}{q}} \cdot \hilkn{\frac{1}{q}}{- \frac{1}{q}}. \]
By \eqref{e:prop3},
	\[ \hilkn{\frac{1}{q}}{- \frac{1}{q}} = \hilkn{\frac{1}{q}}{1} = 1; \]
hence by also using \eqref{e:prop2}, we finally get
	\[ \ba{\lambda}{\mu} = c_{1}\left( \frac{\lambda}{q} \right)^{- 1} \cdot \hilkn{- \frac{\norm{\del{1}{\mu}}}{\norm{n_{1}}}}{\frac{\norm{\lambda}}{q^{2}}} \cdot \hilkn{q}{\frac{\mu \ol{\del{1}{\mu}}}{n'}}, \]
i.e.
\begin{equation} \label{e:2.29d}
	\ba{\lambda}{\mu} = c_{1}\left( \frac{\lambda}{q} \right)^{- 1} \cdot \hilkn{- \frac{\norm{\del{1}{\mu}}}{\norm{\del{1}{\lambda}}}}{\frac{\norm{\lambda}}{q^{2}}} \cdot \hilkn{q}{\frac{\mu \ol{\del{1}{\mu}}}{\del{2}{\lambda}}}.
\end{equation}
So if $ \mu = c + d \thet $, where $ c \in k $, $ d \in k^{\times} $, then $ q = a + ( b c ) / d $, and by \eqref{e:prop5},
	\[ c_{1}\left( \frac{\lambda}{q} \right)^{- 1}
	=	\begin{cases}
			\displaystyle \hilkn{- \frac{\norm{\lambda}}{4 a^{2} \norm{\del{1}{\mu}}}}{\frac{( \norm{\lambda} )^{2} b^{4} \thet^{4}}{( ( a - q ) a - b^{2} \thet^{2} )^{4}}}, &\text{if $ ( a - q ) a - b^{2} \thet^{2} \neq 0 $;} \\
			\displaystyle \hilkn{\frac{( \norm{\lambda} )^{2} b^{4}\thet^{4}}{( ( a - q ) a + b^{2} \thet^{2} )^{4}}}{- \frac{\norm{\lambda}}{4 a^{2} \norm{\del{1}{\mu}}}}, &\text{if $ ( a - q ) a - b^{2} \thet^{2} = 0 $.}
		\end{cases} \]

We are left with the case when $ \lambda \in k^{\times} \thet $.  This case is slightly different from the above because $ n_{1} = - 1 / 2 \in k^{\times} $, which alters \sigu{e_{1}}{e_{2}} and \sigu{e_{1} e_{2}}{e_{3}}.

Firstly, by \eqref{e:2.29b}, we know that
\begin{align*}
	\sigu{e_{1}}{e_{2}}
	&=	\begin{aligned}[t]
				&\delh{\frac{\thet}{n_{1}}} \cdot \twa{- \frac{\thet}{n_{1}}}{- \frac{\thet^{2}}{n_{1}}} \cdot \twa{0}{\thet} \cdot \delh{\frac{n_{1}}{\thet}} \\
				&\phantom{\ } \cdot \delh{\frac{\thet}{n_{1}}}^{-1}
		\end{aligned} \\
	&=	\begin{aligned}[t]
			&\delh{- 2 \thet} \cdot \twa{2 \thet}{2 \thet^{2}} \cdot \twa{0}{\thet} \cdot \delh{- \frac{1}{2 \thet}} \\
			&\phantom{\ } \cdot \delh{- 2 \thet}^{- 1}.
		\end{aligned}
\end{align*}
By \eqref{e:2.29e},
	\[ \sigu{e_{1}}{e_{2}} 
	=	\begin{aligned}[t]
			&\delh{- 2 \thet} \cdot \left[ \hilkn{- 1}{- 4 \thet^{2}} \cdot \delh{- 4 \thet^{2}} \cdot \twa{1}{- \frac{1}{2}} \right] \\
			&\phantom{\ } \cdot \twa{0}{\thet} \cdot \delh{- \frac{1}{2 \thet}} \cdot \delh{- 2 \thet}^{- 1}.
		\end{aligned} \]
Since by \eqref{e:wainv}, $ \twa{0}{\thet} = \twa{0}{- \thet}^{- 1} $, by \eqref{e:delh},
	\[ \sigu{e_{1}}{e_{2}} 
	=	\begin{aligned}[t]
			&\hilkn{- 1}{- 4 \thet^{2}} \cdot \delh{- 2 \thet} \cdot \delh{- 4 \thet^{2}} \cdot \delh{\frac{1}{2 \thet}} \\
			&\phantom{\ } \cdot \delh{- \frac{1}{2 \thet}} \cdot \delh{- 2 \thet}^{- 1}.
		\end{aligned} \]
Hence by \eqref{e:ba},
\begin{align*}
	\sigu{e_{1}}{e_{2}} 
	&=	\begin{aligned}[t]
			&\hilkn{- 1}{- 4 \thet^{2}} \cdot \delh{- 2 \thet} \cdot \delh{- 4 \thet^{2}} \cdot \delh{\frac{1}{2 \thet}} \\
			&\phantom{\ } \cdot \left[ \delh{- 2 \thet}^{- 1} \cdot \delh{- 2 \thet} \right] \cdot \delh{- \frac{1}{2 \thet}} \\
			&\phantom{\ } \cdot \delh{- 2 \thet}^{- 1}
		\end{aligned} \\
	&=	\begin{aligned}[t]
			&\hilkn{- 1}{- 4 \thet^{2}} \cdot \delh{- 2 \thet} \cdot \ba{- 4 \thet^{2}}{\frac{1}{2 \thet}} \cdot \ba{- 2 \thet}{- \frac{1}{2 \thet}} \\
			&\phantom{\ } \cdot \delh{- 2 \thet}^{- 1}
		\end{aligned} \\
	&= \hilkn{- 1}{- 4 \thet^{2}} \cdot \ba{- 4 \thet^{2}}{\frac{1}{2 \thet}} \cdot \ba{- 2 \thet}{- \frac{1}{2 \thet}}.
\end{align*}
By \eqref{e:2.29f},
	\[ \sigu{e_{1}}{e_{2}} = \hilkn{- 1}{- 4 \thet^{2}} \cdot \hilkn{- 4 \thet^{2}}{- \frac{1}{4 \thet^{2}}} \cdot \ba{- 2 \thet}{- \frac{1}{2 \thet}}. \]
Thus by \eqref{e:prop3} and \eqref{e:bastr},
\begin{align*}
	\sigu{e_{1}}{e_{2}}
	&= \hilkn{- 1}{- 4 \thet^{2}} \cdot \hilkn{- 1}{- \frac{1}{4 \thet^{2}}} \cdot \ba{- 2 \thet}{- \frac{1}{2 \thet}} \\
	&= \ba{- 2 \thet}{- \frac{1}{2 \thet}} \\
	&= \ba{\frac{\thet}{n_{1}}}{\frac{n_{1}}{\thet}},
\end{align*}
which coincides with the case when $ \lambda \notin k^{\times} \thet $.

As for \sigu{e_{1} e_{2}}{e_{3}}, by \eqref{e:2.29c}, we know that since $ n_{1} = \ol{n_{1}} $,
	\[ \sigu{e_{1} e_{2}}{e_{3}}
		=	\begin{aligned}[t]
				&\delh{\frac{\thet}{n_{1}}} \cdot \twa{\frac{\thet}{\ol{n''}}}{- \frac{\thet^{2}}{\ol{n''}}} \cdot \twa{0}{\thet} \cdot \delh{\frac{\ol{n'}}{\thet}} \\
				&\phantom{\ } \cdot \delh{\frac{\ol{n'} \thet}{n_{1} \ol{n''}}}^{- 1}.
		\end{aligned} \]
Since $ n'' \notin k^{\times} $, we can still use \eqref{e:2.29ci} in the above to obtain
	\[ \sigu{e_{1} e_{2}}{e_{3}}
		=	\begin{aligned}[t]
				&\delh{\frac{\thet}{n_{1}}} \cdot \Bigg[ \hilkn{- \frac{\norm{n''}}{\thet^{2}}}{- 1} \cdot \delh{- \frac{\norm{n''}}{\thet^{2}}}^{- 1} \\
				&\phantom{\ } \cdot \twa{1}{n''} \Bigg] \cdot \twa{0}{\thet} \cdot \delh{\frac{\ol{n'}}{\thet}} \cdot \delh{\frac{\ol{n'} \thet}{n_{1} \ol{n''}}}^{- 1},
			\end{aligned} \]
and in a similar fashion to what we have done previously in the first case when $ \lambda \notin k^{\times} \thet $, we can show that
	\[ \sigu{e_{1} e_{2}}{e_{3}} = \hilkn{- \frac{\norm{n''}}{\thet^{2}}}{- 1} \cdot \hilkn{\frac{\ol{n'}}{\thet}}{\frac{n_{1}^{2} \norm{n''}}{\thet^{4}}} \cdot \ba{\frac{1}{\ol{\lambda}}}{\frac{\thet}{\ol{n''}}}. \]
If we compare the above with \eqref{e:2.29cii}, we can see that they only differ by a factor $ c_{1}( \lambda ) $ which only exists if $ a \neq 0 $.  Consequently, since it can be checked that both \sigu{e_{1}}{e_{2} e_{3}} and \sigu{e_{2}}{e_{3}} remain unchanged, by same method we used above for $ \lambda \notin k^{\times} \thet $, this shows that for $ \lambda \in k^{\times} \thet $,
	\[ \sigu{\lambda}{\mu} = \hilkn{- \frac{\norm{\del{1}{\mu}}}{\norm{\del{1}{\lambda}}}}{\frac{\norm{\lambda}}{q^{2}}} \cdot \hilkn{q}{\frac{\mu \ol{\del{1}{\mu}}}{\del{2}{\lambda}}}, \]
where for $ \lambda = b \thet $, $ \mu = c + d \thet $, $ b $, $ c $, $ d \in k^{\times} $, $ q = b c / d $, i.e. the above only differs from \eqref{e:2.29d} by $ c_{1}( \lambda / q )^{- 1} $.
\end{proof}

Our results from Proposition~\ref{p:2.23ii}, Remark~\ref{r:bainv} and Proposition~\ref{p:2.29} can be summarised by the following theorem:

\begin{thm} \label{t:sigtk}
For $ \lambda $, $ \mu \in K^{\times} $,
\begin{multline*}
	\sigu{\ha{\lambda}}{\ha{\mu}} \\
		=	\begin{cases}
				\displaystyle \hilkn{\lambda}{\mu}, &\text{if $ \lambda $, $ \mu \in k^{\times} $;} \\
				\displaystyle \hilkn{\mu}{- \del{2}{\lambda} / \thet}, &\text{if $ \lambda \notin k^{\times} $, $ \mu \in k^{\times} $;} \\
				\displaystyle \hilkn{\lambda}{\mu \ol{\del{1}{\mu}} / \thet}, &\text{if $ \lambda \in k^{\times} $, $ \mu \notin k^{\times} $;} \\
				\displaystyle \hilkn{- 1}{\norm{\lambda}} \cdot \hilkn{- \frac{\lambda \ol{\del{1}{\lambda}}}{\thet}}{\lambda \mu}, &\text{if $ \lambda $, $ \mu \notin k^{\times} $, $ \lambda \mu \in k^{\times} $;} \\
				\displaystyle
					\begin{aligned}[b]
						&\hilkn{- \frac{\norm{\del{1}{\mu}}}{\norm{\del{1}{\lambda}}}}{\frac{\norm{\lambda}}{q^{2}}} \cdot \hilkn{q}{\frac{\mu \ol{\del{1}{\mu}}}{\del{2}{\lambda}}} \\
						&\phantom{\ } \cdot \Sigma'( \lambda, \mu ),
					\end{aligned} &\text{otherwise,}
			\end{cases}
\end{multline*}
where, if $  \lambda = a + b \thet $, $ \mu = c + d \thet $, with $ a $, $ c \in k $, $ b $, $ d \in k^{\times} $,
	\[ q = a + \frac{b c}{d}, \]
and
	\[ \Sigma'( \lambda, \mu )
		=	\begin{cases}
				\displaystyle\hilkn{- \frac{\norm{\lambda}}{4 a^{2} \norm{\del{1}{\mu}}}}{\frac{( \norm{\lambda} )^{2} b^{4} \thet^{4}}{( ( a - q ) a - b^{2} \thet^{2} )^{4}}}, &\text{if $ \lambda \notin k^{\times} \thet $, $ a q \neq \norm{\lambda} $;} \\
			\displaystyle\hilkn{\frac{( \norm{\lambda} )^{2} b^{4}\thet^{4}}{( ( a - q ) a + b^{2} \thet^{2} )^{4}}}{- \frac{\norm{\lambda}}{4 a^{2} \norm{\del{1}{\mu}}}}, &\text{if $ \lambda \notin k^{\times} \thet $, $ a q = \norm{\lambda} $;} \\
				\displaystyle 1, &\text{if $ \lambda \in k^{\times} \thet $.}
			\end{cases} \]
\end{thm}


\chapter{The 2-cocycle on the rest of $ \SU( k ) $} \label{c:rest}

Recall that we have set $ G = \SU $.  We have defined a section $ \delta \colon G( k ) \to \wtilde{G} $ in Section~\ref{s:sec}.  Hence, we can use this section and Theorem~\ref{t:sigtk} to find the universal 2-cocycle on $ G( k ) $.  Note that by \eqref{e:ba} and Theorem~\ref{t:sigtk}, we already know explicitly what $ \ba{-}{-} $ is on $ K^{\times} \times K^{\times} $ in terms of $ \hilkn{s}{t} $'s for $ s $, $ t \in k^{\times} $; so from this subsection onwards we will only use \hilkn{-}{-} as a function on $ k^{\times} \times k^{\times} $, and otherwise use $ \sigma_{u} $ as described by \eqref{e:ba}.

\section{The easy cases} \label{s:easysu}

To start with, let $ ( r, m ) $, $ ( r', m' ) \in A $, where $ A $ is defined as in \eqref{e:defA1}.  It is obvious that for $ g $, $ g' \in G( k ) $,
\begin{gather*}
	\sigu{\xa{r}{m} \cdot g}{g' \cdot \xa{r'}{m'}} = \sigu{g}{g'}, \\
	\sigu{g}{\xa{r}{m} \cdot g'} = \sigu{g \cdot \xa{r}{m}}{g'}, \\
	\sigu{g}{\xa{r}{m}} = \sigu{\xa{r}{m}}{g'} = 1.
\end{gather*}
Therefore, since for $ \lambda \in K^{\times} $,
	\[ \ha{\lambda} \cdot \xa{r}{m} = \xa{\frac{r \lambda^{2}}{\ol{\lambda}}}{m \norm{\lambda}} \cdot \ha{\lambda}, \]
the above implies that
\begin{gather*}
	\sigu{g}{\xa{r}{m} \cdot \ha{\lambda}} = \sigu{g}{\ha{\lambda} \cdot \xa{\frac{r \ol{\lambda}}{\lambda^{2}}}{\frac{m}{\norm{\lambda}}}} = \sigu{g}{\ha{\lambda}}, \\
	\sigu{\ha{\lambda} \cdot \xa{r}{m}}{g'} = \sigu{\xa{\frac{r \lambda^{2}} {\ol{\lambda}}}{m \norm{\lambda}} \cdot \ha{\lambda}}{g'} = \sigu{\ha{\lambda}}{g'}.
\end{gather*}

Also, if $ \mu \in k^{\times} $, by \eqref{e:delh},
	\[ \delh{\mu} = \twa{0}{\mu \thet} \cdot \twa{0}{\thet}^{- 1}. \]
Thus by \eqref{e:wainv},
\begin{align*}
	&\twa{0}{\thet} \cdot \delh{\mu} \cdot \twa{0}{\thet}^{-1} \\
	&= \twa{0}{\thet} \cdot \left[ \twa{0}{\mu \thet} \cdot \twa{0}{\thet}^{- 1} \right] \cdot \twa{0}{\thet}^{- 1} \\
	&= \left[ \twa{0}{- \mu \thet} \cdot \twa{0}{\thet}^{- 1} \right]^{- 1} \cdot \left[ \twa{0}{- \thet} \cdot \twa{0}{\thet}^{- 1} \right].
\end{align*}
Thus by \eqref{e:delh} again,
	\[ \twa{0}{\thet} \cdot \delh{\mu} \cdot \twa{0}{\thet}^{-1} = \delh{- \mu}^{- 1} \cdot \delh{- 1}; \]
and this implies by \eqref{e:delinv} that
\begin{equation} \label{e:muthet1}
	\twa{0}{\thet} \cdot \delh{\mu} \cdot \twa{0}{\thet}^{-1} = \delh{\frac{1}{\mu}}.
\end{equation}
As a result, for $ \lambda \in K^{\times} $,
\begin{align*}
	&\sigu{\ha{\lambda} \cdot \wa{0}{\thet}}{\ha{\mu} \cdot \wa{0}{\thet}} \\
	&=	\begin{aligned}[t]
			&\left[ \delh{\lambda} \cdot \twa{0}{\thet} \right] \cdot \left[ \delh{\mu} \cdot \twa{0}{\thet} \right] \\
			&\phantom{\ } \cdot \delta \left[ \ha{\lambda} \cdot \wa{0}{\thet} \cdot \ha{\mu} \cdot \wa{0}{\thet} \right]^{-1}
			\end{aligned} \\
	&= \delh{\lambda} \cdot \left[ \delh{\frac{1}{\mu}} \cdot \twa{0}{\thet} \right] \cdot \twa{0}{\thet} \cdot \delh{- \frac{\lambda}{\mu}}^{- 1}.
\end{align*}
(Note that for any $ \lambda $, $ \mu \in K^{\times} $, $ \ha{\lambda} \cdot \wa{0}{\thet} \cdot \ha{\mu} \cdot \wa{0}{\thet} = \ha{- \lambda / \ol{\mu}} $.)  By \eqref{e:delh}, we have
\begin{multline*}
	\sigu{\ha{\lambda} \cdot \wa{0}{\thet}}{\ha{\mu} \cdot \wa{0}{\thet}} \\ = \delh{\lambda} \cdot \delh{\frac{1}{\mu}} \cdot \delh{- 1}^{- 1} \cdot \delh{- \frac{\lambda}{\mu}}^{- 1}.
\end{multline*}
Thus by \eqref{e:ba},
\begin{align*}
	&\sigu{\ha{\lambda} \cdot \wa{0}{\thet}}{\ha{\mu} \cdot \wa{0}{\thet}} \\
	&=	\begin{aligned}[t]
			&\delh{\lambda} \cdot \delh{\frac{1}{\mu}} \cdot \left[ \delh{\frac{\lambda}{\mu}}^{- 1} \cdot \delh{\frac{\lambda}{\mu}} \right] \\
			&\phantom{\ } \cdot \delh{- 1}^{- 1} \cdot \delh{- \frac{\lambda}{\mu}}^{- 1}
		\end{aligned} \\
	&= \sigu{\ha{\lambda}}{\ha{\frac{1}{\mu}}} \cdot \sigu{\ha{- \frac{\lambda}{\mu}}}{\ha{- 1}}^{- 1}.
\end{align*}

If $ \mu \notin k^{\times} $, by \eqref{e:delh},
\begin{align*}
	\delh{\mu}^{- 1} &\cdot \twa{0}{\thet} \cdot \delh{\mu} \\
	&=	\begin{aligned}[t]
			&\left[ \twa{1}{\del{1}{\mu}} \cdot \twa{0}{\del{2}{\mu}}^{- 1} \right]^{- 1} \cdot \twa{0}{\thet} \\
			&\phantom{\ } \cdot \left[ \twa{1}{\del{1}{\mu}} \cdot \twa{0}{\del{2}{\mu}}^{- 1} \right].
		\end{aligned}
\end{align*}
Thus by \eqref{e:wainv},
\begin{multline*}
	\delh{\mu}^{- 1} \cdot \twa{0}{\thet} \cdot \delh{\mu} = \twa{0}{\del{2}{\mu}} \cdot \twa{- 1}{\ol{\del{1}{\mu}}} \cdot \twa{0}{\thet} \\ \cdot \twa{- 1}{\ol{\del{1}{\mu}}}^{- 1} \cdot \twa{0}{\del{2}{\mu}}^{- 1},
\end{multline*}
and by using Proposition~\ref{p:wa} twice,
\begin{align}
	\delh{\mu}^{- 1} \cdot \twa{0}{\thet} \cdot \delh{\mu}
	&=	\begin{aligned}[t]
			&\twa{0}{\del{2}{\mu}} \cdot \twa{0}{- \frac{\norm{\ol{\del{1}{\mu}}}}{\thet}} \\
			&\phantom{\ } \cdot \twa{0}{\del{2}{\mu}}^{- 1}
		\end{aligned} \label{e:muthet2} \\
	&= \twa{0}{\frac{\thet}{\norm{\mu}}}. \notag
\end{align}
Thus for $ \lambda \in K^{\times} $,
\begin{align*}
	&\sigu{\ha{\lambda} \cdot \wa{0}{\thet}}{\ha{\mu} \cdot \wa{0}{\thet}} \\
	&=	\begin{aligned}[t]
			&\delh{\lambda} \cdot \twa{0}{\thet} \cdot \delh{\mu} \cdot \twa{0}{\thet} \\
			&\phantom{\ } \cdot \delta[ \ha{\lambda} \cdot \wa{0}{\thet} \cdot \ha{\mu} \cdot \wa{0}{\thet} ]^{-1}
			\end{aligned} \\
	&= \delh{\lambda} \cdot \left[ \delh{\mu} \cdot \twa{0}{\frac{\thet}{\norm{\mu}}} \right] \cdot \twa{0}{\thet} \cdot \delh{- \frac{\lambda}{\ol{\mu}}}^{- 1}.
\end{align*}
Thus by \eqref{e:delh},
\begin{align*}
	&\sigu{\ha{\lambda} \cdot \wa{0}{\thet}}{\ha{\mu} \cdot \wa{0}{\thet}} \\
	&=	\begin{aligned}[t]
			&\delh{\lambda} \cdot \delh{\mu} \cdot \twa{0}{\frac{\thet}{\norm{\mu}}} \cdot \left[ \twa{0}{\thet}^{- 1} \cdot \twa{0}{\thet} \right] \cdot \twa{0}{\thet} \\
			&\phantom{\ } \cdot \delh{- \frac{\lambda}{\ol{\mu}}}^{- 1}
		\end{aligned} \\
	&= \delh{\lambda} \cdot \delh{\mu} \cdot \delh{\frac{1}{\norm{\mu}}} \cdot \delh{- 1}^{- 1} \cdot \delh{- \frac{\lambda}{\ol{\mu}}}^{- 1},
\end{align*}
hence by \eqref{e:ba},
\begin{align*}
	&\sigu{\ha{\lambda} \cdot \wa{0}{\thet}}{\ha{\mu} \cdot \wa{0}{\thet}} \\
	&=	\begin{aligned}[t]
			&\delh{\lambda} \cdot \delh{\mu} \cdot \left[ \delh{\lambda \mu}^{- 1} \cdot \delh{\lambda \mu} \right] \cdot \delh{\frac{1}{\norm{\mu}}} \\
			&\phantom{\ }\cdot \left[ \delh{\frac{\lambda}{\ol{\mu}}}^{- 1} \cdot \delh{\frac{\lambda}{\ol{\mu}}} \right] \cdot \delh{- 1}^{- 1} \cdot \delh{- \frac{\lambda}{\ol{\mu}}}^{- 1}
		\end{aligned} \\
	&= \sigu{\ha{\lambda}}{\ha{\mu}} \cdot \sigu{\ha{\lambda \mu}}{\ha{\frac{1}{\norm{\mu}}}} \cdot \sigu{\ha{- \frac{\lambda}{\ol{\mu}}}}{\ha{- 1}}^{- 1}.
\end{align*}

Similarly, since $ \ha{\lambda} \cdot \wa{0}{\thet} \cdot \ha{\mu} = \ha{\lambda / \ol{\mu}} \cdot \wa{0}{\thet} $,
\begin{multline*}
	\sigu{\ha{\lambda} \cdot \wa{0}{\thet}}{\ha{\mu}} \\
	= \delh{\lambda} \cdot \twa{0}{\thet} \cdot \delh{\mu} \cdot \left[ \delh{\frac{\lambda}{\ol{\mu}}} \cdot \twa{0}{\thet} \right] ^{- 1}.
\end{multline*}
Thus using \eqref{e:muthet1}, \eqref{e:muthet2} and \eqref{e:delh},
\begin{multline*}
	\sigu{\ha{\lambda} \cdot \wa{0}{\thet}}{\ha{\mu}} \\
	=	\begin{cases}
			\displaystyle \delh{\lambda} \cdot \delh{\frac{1}{\mu}} \cdot \delh{\frac{\lambda}{\mu}}^{- 1}, &\text{if $ \mu \in k^{\times} $;} \\
			\displaystyle \delh{\lambda} \cdot \delh{\mu} \cdot \delh{\frac{1}{\norm{\mu}}} \cdot \delh{\frac{\lambda}{\ol{\mu}}}^{- 1}, &\text{if $ \mu \notin k^{\times} $.}
		\end{cases}
\end{multline*}
Therefore by \eqref{e:ba},
\begin{multline*}
	\sigu{\ha{\lambda} \cdot \wa{0}{\thet}}{\ha{\mu}} \\
	=	\begin{cases}
			\displaystyle \sigu{\ha{\lambda}}{\ha{\frac{1}{\mu}}}, &\text{if $ \mu \in k^{\times} $;} \\
			\displaystyle \sigu{\ha{\lambda}}{\ha{\mu}} \cdot \sigu{\ha{\lambda \mu}}{\ha{\frac{1}{\norm{\mu}}}}, &\text{if $ \mu \notin k^{\times} $.}
		\end{cases}
\end{multline*}

Also, note that by the definition of our section,
	\[ \sigu{\ha{\lambda}}{\ha{\mu} \cdot \wa{0}{\thet}} = \sigu{\ha{\lambda}}{\ha{\mu}}. \]

\section{The difficult case} \label{s:diffsu}

We are now left with the most difficult case, i.e. the next proposition.

\begin{prop} \label{p:sigxa}
For $ \lambda $, $ \mu \in K^{\times} $, with $ ( r, m ) \in A $, where $ r = a + b \thet $, $ m = c + d \thet $ for $ a $, $ b $, $ c $, $ d \in k $,
\begin{multline*}
	\sigu{\ha{\lambda} \cdot \wa{0}{\thet}}{\xa{r}{m} \cdot \ha{\mu} \cdot \wa{0}{\thet}} \\
	= \Sigma_{u}( r, m ) \cdot \sigu{\ha{\lambda}}{\ha{\frac{\thet}{\ol{m}}}} \cdot \sigu{\ha{\frac{\lambda \thet}{\ol{m}}}}{\ha{\mu}},
\end{multline*}
where
\begin{multline*}
	\Sigma_{u}( r, m) \\
	=	\begin{cases}
			\displaystyle \hilkn{\frac{\thet}{\ol{m}}}{- 1}, &\text{if $ m \in k^{\times} \thet $;} \\
			\displaystyle \hilkn{- \frac{\norm{m}}{\thet^{2}}}{- \frac{r \thet}{\ol{m}}}, &\text{if $ \displaystyle - \frac{r \thet}{\ol{m}} \in k^{\times} $;} \\
			\displaystyle \hilkn{r}{- \frac{\thet^{2}}{m^{2}}}, &\text{if $ r $, $ m \in k^{\times} $;} \\
			\displaystyle \hilkn{\frac{a c + b d \thet^{2}}{c}}{\frac{- c^{2} \norm{r}}{b^{2} \norm{m} \thet^{2}}} \cdot \hilkn{\norm{r}}{\frac{- b \thet^{2}}{c}}, &\text{if $ b $, $ c \neq 0 $, $ \displaystyle - \frac{r \thet}{\ol{m}} \notin k^{\times} $;} \\
			\displaystyle \hilkn{r}{- \frac{\thet^{2}}{\norm{m}}}, &\text{otherwise.}
		\end{cases}
\end{multline*}
\end{prop}

\begin{proof}
Our proof will use similar methods used in Proposition~\ref{p:2.29}.  By \eqref{e:de1e2} and \eqref{e:e1e2},
\begin{multline*}
	\sigu{\ha{\lambda} \cdot \wa{0}{\thet}}{\xa{r}{m} \cdot \ha{\mu} \cdot \wa{0}{\thet}} \\
	=	\begin{aligned}[t]
			&\txa{\frac{r \lambda^{2} \thet}{\ol{m} \ol{\lambda}}}{- \frac{\norm{\lambda} \thet^{2}}{m}} \cdot \delh{\lambda} \cdot \twa{- \frac{r \thet}{\ol{m}}}{- \frac{\thet^{2}}{\ol{m}}} \cdot \twa{0}{\thet} \\
			&\phantom{\ } \cdot \delh{\mu} \cdot \delh{\frac{\lambda \mu \thet}{\ol{m}}}^{-1} \cdot \txa{\frac{r \lambda^{2} \thet}{\ol{m} \ol{\lambda}}}{- \frac{\norm{\lambda} \thet^{2}}{m}}^{- 1};
		\end{aligned}
\end{multline*}
hence simplifying the above gives
\begin{multline} \label{e:sigrm}
	\sigu{\ha{\lambda} \cdot \wa{0}{\thet}}{\xa{r}{m} \cdot \ha{\mu} \cdot \wa{0}{\thet}} \\
	= \delh{\lambda} \cdot \twa{- \frac{r \thet}{\ol{m}}}{- \frac{\thet^{2}}{\ol{m}}} \cdot \twa{0}{\thet} \cdot \delh{\mu} \cdot \delh{\frac{\lambda \mu \thet}{\ol{m}}}^{-1}.
\end{multline}
There are two cases to consider: when $ r = 0 $; and when $ r \neq 0 $.

When $ r = 0 $, this implies that $ m \in k^{\times} \thet $.  Thus, \eqref{e:sigrm} becomes
	\[ \twa{- \frac{r \thet}{\ol{m}}}{- \frac{\thet^{2}}{\ol{m}}} = \twa{0}{- \frac{\thet^{2}}{\ol{m}}}; \]
and by \eqref{e:delh},
\begin{align*}
	&\sigu{\ha{\lambda} \cdot \wa{0}{\thet}}{\xa{r}{m} \cdot \ha{\mu} \cdot \wa{0}{\thet}} \\
	&=	\begin{aligned}[t]
			&\delh{\lambda} \cdot \twa{0}{- \frac{\thet^{2}}{\ol{m}}} \cdot \left[ \twa{0}{\thet}^{- 1} \cdot \twa{0}{\thet} \right] \cdot \twa{0}{\thet} \\
			&\phantom{\ } \cdot \delh{\mu} \cdot \delh{\frac{\lambda \mu \thet}{\ol{m}}}^{-1}
		\end{aligned} \\
	&= \delh{\lambda} \cdot \delh{- \frac{\thet}{\ol{m}}} \cdot \delh{- 1}^{- 1} \cdot \delh{\mu} \cdot \delh{\frac{\lambda \mu \thet}{\ol{m}}}^{-1}.
\end{align*}
By \eqref{e:ba},
	\[ \hilkn{\frac{\thet}{\ol{m}}}{- 1}^{- 1} \cdot \delh{\frac{\thet}{\ol{m}}} = \delh{- \frac{\thet}{\ol{m}}} \cdot \delh{- 1}^{- 1}; \]
therefore the equation becomes
\begin{multline*}
	\sigu{\ha{\lambda} \cdot \wa{0}{\thet}}{\xa{r}{m} \cdot \ha{\mu} \cdot \wa{0}{\thet}} \\
	= \delh{\lambda} \cdot \left[ \hilkn{\frac{\thet}{\ol{m}}}{- 1}^{- 1} \cdot \delh{\frac{\thet}{\ol{m}}} \right] \cdot \delh{\mu} \cdot \delh{\frac{\lambda \mu \thet}{\ol{m}}}^{-1}.
\end{multline*}
By \eqref{e:ba} again,
\begin{align*}
	&\sigu{\ha{\lambda} \cdot \wa{0}{\thet}}{\xa{r}{m} \cdot \ha{\mu} \cdot \wa{0}{\thet}} \\
	&=	\begin{aligned}[t]
			&\hilkn{\frac{\thet}{\ol{m}}}{- 1}^{- 1} \cdot \delh{\lambda} \cdot \delh{\frac{\thet}{\ol{m}}} \cdot \Bigg[ \delh{\frac{\lambda \thet}{\ol{m}}}^{- 1} \\
			&\phantom{\ } \cdot \delh{\frac{\lambda \thet}{\ol{m}}} \Bigg] \cdot \delh{\mu} \cdot \delh{\frac{\lambda \mu \thet}{\ol{m}}}^{-1}
		\end{aligned} \\
	&= \hilkn{\frac{\thet}{\ol{m}}}{- 1} \cdot \sigu{\ha{\lambda}}{\ha{\frac{\thet}{\ol{m}}}} \cdot \sigu{\ha{\frac{\lambda \thet}{\ol{m}}}}{\ha{\mu}}.
\end{align*}
(Note that $ \hilkn{s}{- 1}^{- 1} = \hilkn{s}{- 1} $ for any $ s \in k^{\times} $ by \eqref{e:prop5}.)

Now let $ r \neq 0 $.  We want to consider
	\[ z'( r, m ) = \twa{- \frac{r \thet}{\ol{m}}}{- \frac{\thet^{2}}{\ol{m}}} \cdot \twa{1}{\frac{m}{\norm{r}}}^{- 1}. \]
The above implies that \eqref{e:sigrm} becomes
\begin{multline*}
	\sigu{\ha{\lambda} \cdot \wa{0}{\thet}}{\xa{r}{m} \cdot \ha{\mu} \cdot \wa{0}{\thet}} \\
	= \delh{\lambda} \cdot z'( r, m ) \cdot \twa{1}{\frac{m}{\norm{r}}} \cdot \twa{0}{\thet} \cdot \delh{\mu} \cdot \delh{\frac{\lambda \mu \thet}{\ol{m}}}^{-1};
\end{multline*}
and since by \eqref{e:wainv}, $ \twa{0}{\thet} = \twa{0}{- \thet}^{- 1} $, by \eqref{e:delh},
\begin{multline} \label{e:sigrm2}
	\sigu{\ha{\lambda} \cdot \wa{0}{\thet}}{\xa{r}{m} \cdot \ha{\mu} \cdot \wa{0}{\thet}} \\
	= \delh{\lambda} \cdot z'( r, m ) \cdot \delh{- \frac{m}{\norm{r} \thet}} \cdot \delh{\mu} \cdot \delh{\frac{\lambda \mu \thet}{\ol{m}}}^{-1}.
\end{multline}

We first note that
	\[ z'( r, m ) = \twa{\frac{r \thet}{\ol{m}}}{- \frac{\thet^{2}}{m}}^{- 1} \cdot \twa{- 1}{\frac{\ol{m}}{\norm{r}}} \]
by \eqref{e:wainv}.  Let
	\[ n_{1} = - \frac{\thet^{2}}{m}, \quad t = - \frac{r \thet}{\ol{m}}; \]
then
	\[ z'( r, m ) = \twa{\frac{r \thet}{\ol{m}}}{n_{1}}^{- 1} \cdot \twa{\frac{r \thet}{\ol{m} t}}{\frac{n_{1}}{\norm{t}}}. \]
By Lemma~\ref{l:2.16}, there are four cases to consider: $ t \in k^{\times} $; $ n_{1} \in k^{\times} $, $ t \in k^{\times} \thet $; $ n_{1} = a' + b' \thet $, $ t = c' + d' \thet $, $ a' $, $ d' \in k^{\times} $, $ b' $, $ c' \in k $, $ c' + b' d' \thet^{2} / a' \neq 0 $; and $ n_{1} = a' + b' \thet $, $ t = c' + d' \thet $, $ a' $, $ d' \in k^{\times} $, $ b' $, $ c' \in k $, $ c' + b' d' \thet^{2} / a' = 0 $.

Let $ t \in k^{\times} $, i.e. $ - r \thet / \ol{m} \in k^{\times} $.  By Lemma~\ref{l:2.16},
\begin{align*}
	z'( r, m )
	&= \delh{- \frac{\norm{n_{1}}}{\thet^{2}}} \cdot \delh{- \frac{\norm{n_{1}}}{t \thet^{2}}}^{- 1} \cdot \delh{t} \\
	&= \delh{- \frac{\thet^{2}}{\norm{m}}} \cdot \delh{\frac{\thet}{r m}}^{- 1} \cdot \delh{- \frac{r \thet}{\ol{m}}}.
\end{align*}
Using the above in \eqref{e:sigrm2},
\begin{multline*}
	\sigu{\ha{\lambda} \cdot \wa{0}{\thet}}{\xa{r}{m} \cdot \ha{\mu} \cdot \wa{0}{\thet}} \\
	=	\begin{aligned}[t]
			&\delh{\lambda} \cdot \left[ \delh{- \frac{\thet^{2}}{\norm{m}}} \cdot \delh{\frac{\thet}{r m}}^{- 1} \cdot \delh{- \frac{r \thet}{\ol{m}}} \right] \\
			&\phantom{\ } \cdot \delh{- \frac{m}{\norm{r} \thet}} \cdot \delh{\mu} \cdot \delh{\frac{\lambda \mu \thet}{\ol{m}}}^{-1}.
		\end{aligned}
\end{multline*}
By \eqref{e:ba},
	\[ \hilkn{\frac{\thet}{r m}}{- \frac{r^{2} m}{\ol{m}}}^{- 1} \cdot \delh{- \frac{r^{2} m}{\ol{m}}} = \delh{\frac{\thet}{r m}}^{- 1} \cdot \delh{- \frac{r \thet}{\ol{m}}}; \]
thus
\begin{multline*}
	\sigu{\ha{\lambda} \cdot \wa{0}{\thet}}{\xa{r}{m} \cdot \ha{\mu} \cdot \wa{0}{\thet}} \\
	=	\begin{aligned}[t]
			&\delh{\lambda} \cdot \delh{- \frac{\thet^{2}}{\norm{m}}} \cdot \left[ \hilkn{\frac{\thet}{r m}}{- \frac{r^{2} m}{\ol{m}}}^{- 1} \cdot \delh{- \frac{r^{2} m}{\ol{m}}} \right] \\
			&\phantom{\ } \cdot \delh{- \frac{m}{\norm{r} \thet}} \cdot \delh{\mu} \cdot \delh{\frac{\lambda \mu \thet}{\ol{m}}}^{-1}.
		\end{aligned}
\end{multline*}
By \eqref{e:prop5},
\begin{multline*}
	\sigu{\ha{\lambda} \cdot \wa{0}{\thet}}{\xa{r}{m} \cdot \ha{\mu} \cdot \wa{0}{\thet}} \\
	=	\begin{aligned}[t]
			&\hilkn{\frac{r m}{\thet}}{- \frac{r^{2} m}{\ol{m}}} \cdot \delh{\lambda} \cdot \delh{- \frac{\thet^{2}}{\norm{m}}} \cdot \delh{- \frac{r^{2} m}{\ol{m}}} \\
			&\phantom{\ } \cdot \delh{- \frac{m}{\norm{r} \thet}} \cdot \delh{\mu} \cdot \delh{\frac{\lambda \mu \thet}{\ol{m}}}^{-1}.
		\end{aligned}
\end{multline*}
Also, since $ t = - r \thet / \ol{m} \in k^{\times} $, $ t^{2} = - \norm{r} \thet^{2} / ( \norm{m} ) $, hence by \eqref{e:ba},
\begin{align*}
	&\sigu{\ha{\lambda} \cdot \wa{0}{\thet}}{\xa{r}{m} \cdot \ha{\mu} \cdot \wa{0}{\thet}} \\
	&=	\begin{aligned}[t]
			&\hilkn{\frac{r m}{\thet}}{- \frac{r^{2} m}{\ol{m}}} \cdot \delh{\lambda} \cdot \delh{- \frac{\thet^{2}}{\norm{m}}} \cdot \delh{- \frac{r^{2} m}{\ol{m}}} \\
			&\phantom{\ } \cdot \left[ \delh{- \frac{\norm{r} \thet^{2}}{\norm{m}}}^{- 1} \cdot \delh{- \frac{\norm{r} \thet^{2}}{\norm{m}}} \right] \cdot \delh{- \frac{m}{\norm{r} \thet}} \\
			&\phantom{\ } \cdot \left[ \delh{\frac{\thet}{\ol{m}}}^{- 1} \cdot \delh{\frac{\thet}{\ol{m}}} \right] \cdot \delh{\mu} \cdot \delh{\frac{\lambda \mu \thet}{\ol{m}}}^{-1}
		\end{aligned} \\
	&=	\begin{aligned}[t]
			&\hilkn{\frac{r m}{\thet}}{- \frac{r^{2} m}{\ol{m}}} \cdot \delh{\lambda} \cdot \hilkn{- \frac{\thet^{2}}{\norm{m}}}{- \frac{r^{2} m}{\ol{m}}} \\
			&\phantom{\ } \cdot \sigu{\ha{- \frac{\norm{r} \thet^{2}}{\norm{m}}}}{\ha{- \frac{m}{\norm{r} \thet}}} \cdot \delh{\frac{\thet}{\ol{m}}} \\
			&\phantom{\ } \cdot \left[ \delh{\frac{\lambda \thet}{\ol{m}}}^{- 1} \cdot \delh{\frac{\lambda \thet}{\ol{m}}} \right] \cdot \delh{\mu} \cdot \delh{\frac{\lambda \mu \thet}{\ol{m}}}^{-1}
		\end{aligned} \\
	&=	\begin{aligned}[t]
			&\hilkn{\frac{r m}{\thet}}{- \frac{r^{2} m}{\ol{m}}} \cdot \hilkn{- \frac{\thet^{2}}{\norm{m}}}{- \frac{r^{2} m}{\ol{m}}} \\
			&\phantom{\ } \cdot \sigu{\ha{- \frac{\norm{r} \thet^{2}}{\norm{m}}}}{\ha{- \frac{m}{\norm{r} \thet}}} \\
			&\phantom{\ } \cdot \sigu{\ha{\lambda}}{\ha{\frac{\thet}{\ol{m}}}} \cdot \sigu{\ha{\frac{\lambda \thet}{\ol{m}}}}{\ha{\mu}}.
		\end{aligned}
\end{align*}
By \eqref{e:bastr},
\begin{multline*}
	\sigu{\ha{\lambda} \cdot \wa{0}{\thet}}{\xa{r}{m} \cdot \ha{\mu} \cdot \wa{0}{\thet}} \\
	=	\begin{aligned}[t]
			&\hilkn{- \frac{r \thet}{\ol{m}}}{- \frac{r^{2} m}{\ol{m}}} \cdot \sigu{\ha{- \frac{\norm{r} \thet^{2}}{\norm{m}}}}{\ha{- \frac{m}{\norm{r} \thet}}} \\
			&\phantom{\ } \cdot \sigu{\ha{\lambda}}{\ha{\frac{\thet}{\ol{m}}}} \cdot \sigu{\ha{\frac{\lambda \thet}{\ol{m}}}}{\ha{\mu}}.
		\end{aligned}
\end{multline*}
By Theorem~\ref{t:sigtk}, \eqref{e:prop3} and \eqref{e:prop2},
\begin{multline} \label{e:sigrm3}
	\sigu{\ha{- \frac{\norm{r} \thet^{2}}{\norm{m}}}}{\ha{- \frac{m}{\norm{r} \thet}}} \\
	\begin{aligned}[t]
	&= \hilkn{- \frac{\norm{r} \thet^{2}}{\norm{m}}}{\frac{( - m / ( \norm{r} \thet ) ) \ol{\del{1}{- m / ( \norm{r} \thet )}}}{\thet}} \\
	&= \hilkn{- \frac{\norm{r} \thet^{2}}{\norm{m}}}{- \frac{\norm{m}}{( \norm{r} )^{2} \thet^{2}}} \\
	&= \hilkn{- \frac{\norm{r} \thet^{2}}{\norm{m}}}{- \frac{1}{\norm{r}}} \\
	&= \hilkn{- \frac{\thet^{2}}{\norm{m}}}{- \frac{1}{\norm{r}}} \\
	&= \hilkn{- \frac{1}{\norm{r}}}{- \frac{\norm{m}}{\thet^{2}}}.
	\end{aligned}
\end{multline}
Also, by \eqref{e:prop3},
	\[ \hilkn{- \frac{r \thet}{\ol{m}}}{- \frac{r^{2} m}{\ol{m}}} = \hilkn{- \frac{r \thet}{\ol{m}}}{- \frac{r^{2} m}{\ol{m}} \cdot \left( - \frac{\ol{m}}{r \thet} \right)^{2}} = \hilkn{- \frac{r \thet}{\ol{m}}}{- \frac{\norm{m}}{\thet^{2}}}. \]
Therefore by \eqref{e:bastr},
\begin{align*}
	&\hilkn{- \frac{r \thet}{\ol{m}}}{- \frac{r^{2} m}{\ol{m}}} \cdot \sigu{\ha{- \frac{\norm{r} \thet^{2}}{\norm{m}}}}{\ha{- \frac{m}{\norm{r} \thet}}} \\
	&= \hilkn{- \frac{r \thet}{\ol{m}}}{- \frac{\norm{m}}{\thet^{2}}} \cdot \hilkn{- \frac{1}{\norm{r}}}{- \frac{\norm{m}}{\thet^{2}}} \\
	&= \hilkn{\frac{\thet}{\ol{r} \ol{m}}}{- \frac{\norm{m}}{\thet^{2}}}.
\end{align*}
Since $ \thet / ( \ol{r} \ol{m} ) \in k^{\times} $, $ \thet / ( \ol{r} \ol{m} ) = - \thet / ( r m ) $, hence by \eqref{e:prop3},
\begin{align*}
	\hilkn{- \frac{r \thet}{\ol{m}}}{- \frac{r^{2} m}{\ol{m}}} &\cdot \sigu{\ha{- \frac{\norm{r} \thet^{2}}{\norm{m}}}}{\ha{- \frac{m}{\norm{r} \thet}}} \\
	&= \hilkn{\frac{\thet}{r m} \cdot \left( - \frac{\norm{m}}{\thet^{2}} \right)}{- \frac{\norm{m}}{\thet^{2}}} \\
	&= \hilkn{- \frac{\ol{m}}{r \thet}}{- \frac{\norm{m}}{\thet^{2}}}.
\end{align*}
Thus,
\begin{multline*}
	\hilkn{- \frac{r \thet}{\ol{m}}}{- \frac{r^{2} m}{\ol{m}}} \cdot \sigu{\ha{- \frac{\norm{r} \thet^{2}}{\norm{m}}}}{\ha{- \frac{m}{\norm{r} \thet}}} \\
	= \hilkn{- \frac{\norm{m}}{\thet^{2}}}{- \frac{r \thet}{\ol{m}}}
\end{multline*}
by \eqref{e:prop2}.  This implies that
\begin{multline*}
	\sigu{\ha{\lambda} \cdot \wa{0}{\thet}}{\xa{r}{m} \cdot \ha{\mu} \cdot \wa{0}{\thet}} \\
	= \hilkn{- \frac{\norm{m}}{\thet^{2}}}{- \frac{r \thet}{\ol{m}}} \cdot \sigu{\ha{\lambda}}{\ha{\frac{\thet}{\ol{m}}}} \cdot \sigu{\ha{\frac{\lambda \thet}{\ol{m}}}}{\ha{\mu}}.
\end{multline*}

If instead $ n_{1} \in k^{\times} $, $ t \in k^{\times} \thet $, i.e. $ - \thet^{2} / m \in k^{\times} $, $ - r \thet / \ol{m} \in k^{\times}{\thet} $, this implies that both $ r $, $ m \in k^{\times} $.  Also, by Lemma~\ref{l:2.16} and \eqref{e:prop2},
\begin{align*}
	z'( r, m)
	&= \hilkn{- \frac{n_{1}}{t \thet}}{\norm{t}} \cdot \delh{\norm{t}} \\
	&= \hilkn{- \frac{( - \thet^{2} / m )}{( - r \thet / m ) \thet}}{- \frac{r^{2} \thet^{2}}{m^{2}}} \cdot \delh{- \frac{r^{2} \thet^{2}}{m^{2}}} \\
	&= \hilkn{- \frac{1}{r}}{- \frac{r^{2} \thet^{2}}{m^{2}}} \cdot \delh{- \frac{r^{2} \thet^{2}}{m^{2}}} \\
	&= \hilkn{- \frac{r^{2} \thet^{2}}{m^{2}}}{- r} \cdot \delh{- \frac{r^{2} \thet^{2}}{m^{2}}}.
\end{align*}
Using the above in \eqref{e:sigrm2},
\begin{multline*}
	\sigu{\ha{\lambda} \cdot \wa{0}{\thet}}{\xa{r}{m} \cdot \ha{\mu} \cdot \wa{0}{\thet}} \\
	=	\begin{aligned}[t]
			&\delh{\lambda} \cdot \left[ \hilkn{- \frac{r^{2} \thet^{2}}{m^{2}}}{- r} \cdot \delh{- \frac{r^{2} \thet^{2}}{m^{2}}} \right] \cdot \delh{- \frac{m}{r^{2} \thet}} \\
			&\phantom{\ } \cdot \delh{\mu} \cdot \delh{\frac{\lambda \mu \thet}{m}}^{-1},
		\end{aligned}
\end{multline*}
and by \eqref{e:ba},
\begin{align*}
	&\sigu{\ha{\lambda} \cdot \wa{0}{\thet}}{\xa{r}{m} \cdot \ha{\mu} \cdot \wa{0}{\thet}} \\
	&=	 \begin{aligned}[t]
			&\hilkn{- \frac{r^{2} \thet^{2}}{m^{2}}}{- r} \cdot \delh{\lambda} \cdot \delh{- \frac{r^{2} \thet^{2}}{m^{2}}} \cdot \delh{- \frac{m}{r^{2} \thet}} \\
			&\phantom{\ } \cdot \left[ \delh{\frac{\thet}{m}}^{- 1} \cdot \delh{\frac{\thet}{m}} \right] \cdot \delh{\mu} \cdot \delh{\frac{\lambda \mu \thet}{m}}^{-1}
		\end{aligned} \\
	&=	\begin{aligned}[t]
			&\hilkn{- \frac{r^{2} \thet^{2}}{m^{2}}}{- r} \cdot \delh{\lambda} \cdot \sigu{\ha{- \frac{r^{2} \thet^{2}}{m^{2}}}}{\ha{- \frac{m}{r^{2} \thet}}} \\
			&\phantom{\ } \cdot \delh{\frac{\thet}{m}} \cdot \left[ \delh{\frac{\lambda \thet}{m}}^{- 1} \cdot \delh{\frac{\lambda \thet}{m}} \right] \cdot \delh{\mu} \\
			&\phantom{\ } \cdot \delh{\frac{\lambda \mu \thet}{m}}^{-1}
		\end{aligned} \\
	&=	\begin{aligned}[t]
			&\hilkn{- \frac{r^{2} \thet^{2}}{m^{2}}}{- r} \cdot \sigu{\ha{- \frac{r^{2} \thet^{2}}{m^{2}}}}{\ha{- \frac{m}{r^{2} \thet}}} \\
			&\phantom{\ } \cdot \sigu{\ha{\lambda}}{\ha{\frac{\thet}{m}}} \cdot \sigu{\ha{\frac{\lambda \thet}{m}}}{\ha{\mu}}.
		\end{aligned}
\end{align*}
By Theorem~\ref{t:sigtk} and \eqref{e:prop3}, we know that
\begin{align*}
	\sigu{\ha{- \frac{r^{2} \thet^{2}}{m^{2}}}}{\ha{- \frac{m}{r^{2} \thet}}}
	&= \hilkn{- \frac{r^{2} \thet^{2}}{m^{2}}}{\frac{( - m / ( r^{2} \thet ) ) \ol{\del{1}{- m / ( r^{2} \thet )}}}{\thet}} \\
	&= \hilkn{- \frac{r^{2} \thet^{2}}{m^{2}}}{- \frac{m^{2}}{r^{4} \thet^{2}}} \\
	&= \hilkn{- \frac{r^{2} \thet^{2}}{m^{2}}}{- \frac{1}{r^{2}}}.
\end{align*}
Hence by \eqref{e:bastr},
\begin{align*}
	&\sigu{\ha{\lambda} \cdot \wa{0}{\thet}}{\xa{r}{m} \cdot \ha{\mu} \cdot \wa{0}{\thet}} \\
	&=	\begin{aligned}[t]
			&\hilkn{- \frac{r^{2} \thet^{2}}{m^{2}}}{- r} \cdot \hilkn{- \frac{r^{2} \thet^{2}}{m^{2}}}{- \frac{1}{r^{2}}} \cdot \sigu{\ha{\lambda}}{\ha{\frac{\thet}{m}}} \\
			&\phantom{\ }\cdot \sigu{\ha{\frac{\lambda \thet}{m}}}{\ha{\mu}}
		\end{aligned} \\
	&= \hilkn{- \frac{r^{2} \thet^{2}}{m^{2}}}{\frac{1}{r}} \cdot \sigu{\ha{\lambda}}{\ha{\frac{\thet}{m}}} \cdot \sigu{\ha{\frac{\lambda \thet}{m}}}{\ha{\mu}}.
\end{align*}
But by \eqref{e:prop2},
	\[ \hilkn{- \frac{r^{2} \thet^{2}}{m^{2}}}{\frac{1}{r}} = \hilkn{r}{- \frac{r^{2} \thet^{2}}{m^{2}}}, \]
and by \eqref{e:prop3},
	\[ \hilkn{r}{- \frac{r^{2} \thet^{2}}{m^{2}}} = \hilkn{r}{- \frac{r^{2} \thet^{2}}{m^{2}} \cdot r^{- 2}} = \hilkn{r}{- \frac{\thet^{2}}{m^{2}}}. \]
Thus,
\begin{multline*}
	\sigu{\ha{\lambda} \cdot \wa{0}{\thet}}{\xa{r}{m} \cdot \ha{\mu} \cdot \wa{0}{\thet}} \\
	= \hilkn{r}{- \frac{\thet^{2}}{m^{2}}} \cdot \sigu{\ha{\lambda}}{\ha{\frac{\thet}{m}}} \cdot \sigu{\ha{\frac{\lambda \thet}{m}}}{\ha{\mu}}.
\end{multline*}

Otherwise, let $ r = a + b \thet $, $ m = c + d \thet $, where $ a $, $ b $, $ c $, $ d \in k $, with $ c \neq 0 $ since $ m \notin k^{\times} \thet $.  This implies that
	\[ n_{1} = - \frac{c \thet^{2}}{\norm{m}} + \frac{d \thet^{3}}{\norm{m}}, \quad t = - \frac{( a d + b c ) \thet^{2}}{\norm{m}} - \frac{( a c + b d \thet^{2} ) \thet}{\norm{m}}. \]
Then by Lemma~\ref{l:2.16}, the equation
\begin{align*}
	z'( r , m )
	&=	\begin{aligned}[t]
			&\bigg( - \frac{( - c \thet^{2} / \norm{m} )^{2} \norm{- r \thet / \ol{m}}}{( - ( a c + b d \thet^{2} ) / \norm{m} )^{2} \norm{- \thet^{2} / m} \thet^{2}}, \\
			&\phantom{\qquad \quad} - \frac{( a d + b c ) \thet^{2}}{\norm{m}} + \frac{( d \thet^{2} / \norm{m} ) ( - ( a c + b d \thet^{2} ) / \norm{m} ) \thet^{2}}{- c \thet^{2} / \norm{m}} \bigg)_{k, n} \\
			&\phantom{\ } \cdot \hilkn{- \frac{( - ( a c + b d \thet^{2} ) / \norm{m} ) \norm{- \thet^{2} / m}}{- c \thet^{2} / \norm{m}}}{\norm{- \frac{r \thet}{m}}} \\
			&\phantom{\ } \cdot \delh{\norm{- \frac{r \thet}{m}}}
		\end{aligned} \\
	&=	\begin{aligned}[t]
			&\hilkn{\frac{c^{2} \norm{r}}{( a c + b d \thet^{2} )^{2}}}{- \frac{b \thet^{2}}{c}} \cdot \hilkn{- \frac{( a c + b d \thet^{2} ) \thet^{2}}{c \norm{m}}}{- \frac{\norm{r} \thet^{2}}{\norm{m}}} \\
			&\phantom{\ } \cdot \delh{- \frac{\norm{r} \thet^{2}}{\norm{m}}}
		\end{aligned}
\end{align*}
is valid if we have $ b \neq 0 $.  Thus, assume $ b \neq 0 $.  By \eqref{e:bastr},
	\[ z'( r , m )
	=	\begin{aligned}[t]
			&\hilkn{\frac{c^{2}}{( a c + b d \thet^{2} )^{2}}}{- \frac{b \thet^{2}}{c}} \cdot \hilkn{\norm{r}}{- \frac{b \thet^{2}}{c}} \\
			&\phantom{\ } \cdot \hilkn{\frac{a c + b d \thet^{2}}{c}}{- \frac{\norm{r} \thet^{2}}{\norm{m}}} \cdot \hilkn{- \frac{\thet^{2}}{\norm{m}}}{- \frac{\thet^{2}}{\norm{m}}} \\
			&\phantom{\ } \cdot \hilkn{- \frac{\thet^{2}}{\norm{m}}}{\norm{r}} \cdot \delh{- \frac{\norm{r} \thet^{2}}{\norm{m}}};
		\end{aligned} \]
and by \eqref{e:prop2},
	\[ z'( r , m )
	=	\begin{aligned}[t]
			&\hilkn{\frac{( a c + b d \thet^{2} )^{2}}{c^{2}}}{- \frac{c}{b \thet^{2}}} \cdot \hilkn{\norm{r}}{- \frac{b \thet^{2}}{c}} \\
			&\phantom{\ } \cdot \hilkn{\frac{a c + b d \thet^{2}}{c}}{- \frac{\norm{r} \thet^{2}}{\norm{m}}} \cdot \hilkn{- \frac{\thet^{2}}{\norm{m}}}{- \frac{\thet^{2}}{\norm{m}}} \\
			&\phantom{\ } \cdot \hilkn{\norm{r}}{- \frac{\norm{m}}{\thet^{2}}} \cdot \delh{- \frac{\norm{r} \thet^{2}}{\norm{m}}}.
		\end{aligned} \]
By \eqref{e:prop5},
	\[ \hilkn{\frac{( a c + b d \thet^{2} )^{2}}{c^{2}}}{- \frac{c}{b \thet^{2}}} = \hilkn{\frac{a c + b d \thet^{2}}{c}}{\left( - \frac{c}{b \thet^{2}} \right)^{2}} = \hilkn{\frac{a c + b d \thet^{2}}{c}}{\frac{c^{2}}{b^{2} \thet^{4}}}; \]
and by \eqref{e:prop3},
	\[ \hilkn{- \frac{\thet^{2}}{\norm{m}}}{- \frac{\thet^{2}}{\norm{m}}} = \hilkn{- \frac{\thet^{2}}{\norm{m}}}{- 1}. \]
Therefore by \eqref{e:bastr},
	\[ z'( r , m )
	=	\begin{aligned}[t]
			&\hilkn{\frac{a c + b d \thet^{2}}{c}}{- \frac{c^{2} \norm{r}}{b^{2} \norm{m} \thet^{2}}} \cdot \hilkn{\norm{r}}{\frac{b \norm{m}}{c}} \\
			&\phantom{\ } \cdot \hilkn{- \frac{\thet^{2}}{\norm{m}}}{- 1} \cdot \delh{- \frac{\norm{r} \thet^{2}}{\norm{m}}}.
		\end{aligned} \]
Replacing the above in \eqref{e:sigrm2},
\begin{multline*}
	\sigu{\ha{\lambda} \cdot \wa{0}{\thet}}{\xa{r}{m} \cdot \ha{\mu} \cdot \wa{0}{\thet}} \\
	=	\begin{aligned}[t]
			&\delh{\lambda} \cdot \Bigg[ \hilkn{\frac{a c + b d \thet^{2}}{c}}{- \frac{c^{2} \norm{r}}{b^{2} \norm{m} \thet^{2}}} \cdot \hilkn{\norm{r}}{\frac{b \norm{m}}{c}} \\
			&\phantom{\ } \cdot \hilkn{- \frac{\thet^{2}}{\norm{m}}}{- 1} \cdot \delh{- \frac{\norm{r} \thet^{2}}{\norm{m}}} \Bigg] \cdot \delh{- \frac{m}{\norm{r} \thet}} \\
			&\phantom{\ } \cdot \delh{\mu} \cdot \delh{\frac{\lambda \mu \thet}{\ol{m}}}^{-1}.
		\end{aligned}
\end{multline*}
By \eqref{e:ba},
\begin{align*}
	&\sigu{\ha{\lambda} \cdot \wa{0}{\thet}}{\xa{r}{m} \cdot \ha{\mu} \cdot \wa{0}{\thet}} \\
	&=	\begin{aligned}[t]
			&\hilkn{\frac{a c + b d \thet^{2}}{c}}{- \frac{c^{2} \norm{r}}{b^{2} \norm{m} \thet^{2}}} \cdot \hilkn{\norm{r}}{\frac{b \norm{m}}{c}} \cdot \hilkn{- \frac{\thet^{2}}{\norm{m}}}{- 1}\\
			&\phantom{\ } \cdot \delh{\lambda} \cdot \delh{- \frac{\norm{r} \thet^{2}}{\norm{m}}} \cdot \delh{- \frac{m}{\norm{r} \thet}} \cdot \Bigg[ \delh{\frac{\thet}{\ol{m}}}^{- 1} \\
			&\phantom{\ } \cdot \delh{\frac{\thet}{\ol{m}}} \Bigg] \cdot \delh{\mu} \cdot \delh{\frac{\lambda \mu \thet}{\ol{m}}}^{-1}
		\end{aligned} \\
	&=	\begin{aligned}[t]
			&\hilkn{\frac{a c + b d \thet^{2}}{c}}{- \frac{c^{2} \norm{r}}{b^{2} \norm{m} \thet^{2}}} \cdot \hilkn{\norm{r}}{\frac{b \norm{m}}{c}} \cdot \hilkn{- \frac{\thet^{2}}{\norm{m}}}{- 1}\\
			&\phantom{\ } \cdot \delh{\lambda} \cdot \sigu{\ha{- \frac{\norm{r} \thet^{2}}{\norm{m}}}}{\ha{- \frac{m}{\norm{r} \thet}}} \cdot \delh{\frac{\thet}{\ol{m}}}\\
			&\phantom{\ } \cdot \left[ \delh{\frac{\lambda \thet}{\ol{m}}}^{- 1} \cdot \delh{\frac{\lambda \thet}{\ol{m}}} \right] \cdot \delh{\mu} \cdot \delh{\frac{\lambda \mu \thet}{\ol{m}}}^{-1}
		\end{aligned} \\
	&=	\begin{aligned}[t]
			&\hilkn{\frac{a c + b d \thet^{2}}{c}}{- \frac{c^{2} \norm{r}}{b^{2} \norm{m} \thet^{2}}} \cdot \hilkn{\norm{r}}{\frac{b \norm{m}}{c}} \cdot \hilkn{- \frac{\thet^{2}}{\norm{m}}}{- 1} \\
			&\phantom{\ } \cdot \sigu{\ha{- \frac{\norm{r} \thet^{2}}{\norm{m}}}}{\ha{- \frac{m}{\norm{r} \thet}}} \cdot \sigu{\ha{\lambda}}{\ha{\frac{\thet}{\ol{m}}}} \\
			&\phantom{\ } \cdot \sigu{\ha{\frac{\lambda \thet}{\ol{m}}}}{\ha{\mu}}.
		\end{aligned}
\end{align*}
By \eqref{e:sigrm3},
\begin{multline*}
	\sigu{\ha{\lambda} \cdot \wa{0}{\thet}}{\xa{r}{m} \cdot \ha{\mu} \cdot \wa{0}{\thet}} \\
	=	\begin{aligned}[t]
			&\hilkn{\frac{a c + b d \thet^{2}}{c}}{- \frac{c^{2} \norm{r}}{b^{2} \norm{m} \thet^{2}}} \cdot \hilkn{\norm{r}}{\frac{b \norm{m}}{c}} \cdot \hilkn{- \frac{\thet^{2}}{\norm{m}}}{- 1} \\
			&\phantom{\ } \cdot \hilkn{- \frac{1}{\norm{r}}}{- \frac{\norm{m}}{\thet^{2}}} \cdot \sigu{\ha{\lambda}}{\ha{\frac{\thet}{\ol{m}}}} \cdot \sigu{\ha{\frac{\lambda \thet}{\ol{m}}}}{\ha{\mu}}.
		\end{aligned}
\end{multline*}
But by \eqref{e:bastr},
	\[ \hilkn{- \frac{1}{\norm{r}}}{- \frac{\norm{m}}{\thet^{2}}} = \hilkn{- 1}{- \frac{\norm{m}}{\thet^{2}}} \cdot \hilkn{\frac{1}{\norm{r}}}{- \frac{\norm{m}}{\thet^{2}}}; \]
and by \eqref{e:prop2},
\begin{multline*}
	\hilkn{- 1}{- \frac{\norm{m}}{\thet^{2}}} \cdot \hilkn{\frac{1}{\norm{r}}}{- \frac{\norm{m}}{\thet^{2}}} \\
	= \hilkn{- \frac{\thet^{2}}{\norm{m}}}{- 1} \cdot \hilkn{\norm{r}}{- \frac{\thet^{2}}{\norm{m}}}.
\end{multline*}
This implies that
\begin{multline*}
	\sigu{\ha{\lambda} \cdot \wa{0}{\thet}}{\xa{r}{m} \cdot \ha{\mu} \cdot \wa{0}{\thet}} \\
	=	\begin{aligned}[t]
			&\hilkn{\frac{a c + b d \thet^{2}}{c}}{- \frac{c^{2} \norm{r}}{b^{2} \norm{m} \thet^{2}}} \cdot \hilkn{\norm{r}}{\frac{b \norm{m}}{c}} \cdot \hilkn{- \frac{\thet^{2}}{\norm{m}}}{- 1} \\
			&\phantom{\ } \cdot \left[ \hilkn{- \frac{\thet^{2}}{\norm{m}}}{- 1} \cdot \hilkn{\norm{r}}{- \frac{\thet^{2}}{\norm{m}}} \right] \cdot \sigu{\ha{\lambda}}{\ha{\frac{\thet}{\ol{m}}}} \\
			&\phantom{\ } \cdot \sigu{\ha{\frac{\lambda \thet}{\ol{m}}}}{\ha{\mu}}.
		\end{aligned}
\end{multline*}
By \eqref{e:bastr},
\begin{multline*}
	\sigu{\ha{\lambda} \cdot \wa{0}{\thet}}{\xa{r}{m} \cdot \ha{\mu} \cdot \wa{0}{\thet}} \\
	=	\begin{aligned}[t]
			&\hilkn{\frac{a c + b d \thet^{2}}{c}}{- \frac{c^{2} \norm{r}}{b^{2} \norm{m} \thet^{2}}} \cdot \hilkn{\norm{r}}{- \frac{b \thet^{2}}{c}} \cdot \hilkn{- \frac{\thet^{2}}{\norm{m}}}{- 1}^{2} \\
			&\phantom{\ } \cdot \sigu{\ha{\lambda}}{\ha{\frac{\thet}{\ol{m}}}} \cdot \sigu{\ha{\frac{\lambda \thet}{\ol{m}}}}{\ha{\mu}};
		\end{aligned}
\end{multline*}
thus by \eqref{e:prop5},
\begin{multline*}
	\sigu{\ha{\lambda} \cdot \wa{0}{\thet}}{\xa{r}{m} \cdot \ha{\mu} \cdot \wa{0}{\thet}} \\
	=	\begin{aligned}[t]
			&\hilkn{\frac{a c + b d \thet^{2}}{c}}{- \frac{c^{2} \norm{r}}{b^{2} \norm{m} \thet^{2}}} \cdot \hilkn{\norm{r}}{- \frac{b \thet^{2}}{c}} \cdot \sigu{\ha{\lambda}}{\ha{\frac{\thet}{\ol{m}}}} \\
			&\phantom{\ } \cdot \sigu{\ha{\frac{\lambda \thet}{\ol{m}}}}{\ha{\mu}}.
		\end{aligned}
\end{multline*}

If $ b = 0 $, then $ r \in k^{\times} $.  Consequently, $ d \neq 0 $, otherwise $ m \in k^{\times} $, which is a case we have already dealt with.  Consider
	\[ z'( r, m )^{- 1}
	= \twa{1}{\frac{m}{r^{2}}} \cdot \twa{- \frac{r \thet}{\ol{m}}}{- \frac{\thet^{2}}{\ol{m}}}^{- 1}. \]
By \eqref{e:wainv},
	\[ z'( r, m )^{- 1} = \twa{- 1}{\frac{\ol{m}}{r^{2}}}^{- 1} \cdot \twa{\frac{r \thet}{\ol{m}}}{- \frac{\thet^{2}}{m}}, \]
and thus by Lemma~\ref{l:2.16},
\begin{align*}
	z'( r, m )^{- 1}
	&=	\begin{aligned}[t]
			&\hilkn{- \frac{( c / r^{2} )^{2} ( - \norm{m} / ( r^{2} \thet^{2} ) )}{( - c / ( r \thet^{2} ) )^{2} ( \norm{m} / ( r ^{4} ) ) \thet^{2}}}{\frac{d}{r} + \frac{( - d / r^{2} ) ( - c / ( r \thet^{2} ) ) \thet^{2}}{c / r^{2}}} \\
			&\phantom{\ } \cdot \hilkn{- \frac{( - c / ( r \thet^{2} ) ) ( \norm{m} / ( r ^{4} ) )}{c / ( r^{2} )}}{- \frac{\norm{m}}{r^{2} \thet^{2}}} \cdot \delh{- \frac{\norm{m}}{r^{2} \thet^{2}}}
		\end{aligned} \\
	&= \hilkn{1}{\frac{2 d}{r}} \cdot \hilkn{\frac{\norm{m}}{r^{3} \thet^{2}}}{- \frac{\norm{m}}{r^{2} \thet^{2}}} \cdot \delh{- \frac{\norm{m}}{r^{2} \thet^{2}}},
\end{align*}
which is valid since $ d \neq 0 $. By using \eqref{e:prop3},
\begin{align*}
	z'( r, m )^{- 1}
	&= \hilkn{\frac{1}{r}}{- \frac{\norm{m}}{r^{2} \thet^{2}}} \cdot \delh{- \frac{\norm{m}}{r^{2} \thet^{2}}} \\
	&= \hilkn{\frac{1}{r}}{- \frac{\norm{m}}{\thet^{2}}} \cdot \delh{- \frac{\norm{m}}{r^{2} \thet^{2}}}.
\end{align*}
This implies that by \eqref{e:prop5},
\begin{align*}
	z'( r, m )
	&= \left[ \hilkn{\frac{1}{r}}{- \frac{\norm{m}}{\thet^{2}}} \cdot \delh{- \frac{\norm{m}}{r^{2} \thet^{2}}} \right]^{- 1} \\
	&= \hilkn{r}{- \frac{\norm{m}}{\thet^{2}}} \cdot \delh{- \frac{\norm{m}}{r^{2} \thet^{2}}}^{- 1}.
\end{align*}
Hence in \eqref{e:sigrm2},
\begin{multline*}
	\sigu{\ha{\lambda} \cdot \wa{0}{\thet}}{\xa{r}{m} \cdot \ha{\mu} \cdot \wa{0}{\thet}} \\
	=	\begin{aligned}[t]
			&\delh{\lambda} \cdot \left[ \hilkn{r}{- \frac{\norm{m}}{\thet^{2}}} \cdot \delh{- \frac{\norm{m}}{r^{2} \thet^{2}}}^{- 1} \right] \cdot \delh{- \frac{m}{r^{2} \thet}} \\
			&\phantom{\ } \cdot \delh{\mu} \cdot \delh{\frac{\lambda \mu \thet}{\ol{m}}}^{-1}.
		\end{aligned}
\end{multline*}
By \eqref{e:ba},
\begin{multline*}
	\sigu{\ha{- \frac{\norm{m}}{r^{2} \thet^{2}}}}{\ha{\frac{\thet}{\ol{m}}}}^{- 1} \cdot \delh{\frac{\thet}{\ol{m}}} \\ = \delh{- \frac{\norm{m}}{r^{2} \thet^{2}}}^{- 1} \cdot \delh{- \frac{m}{r^{2} \thet}}.
\end{multline*}
Therefore,
\begin{multline*}
	\sigu{\ha{\lambda} \cdot \wa{0}{\thet}}{\xa{r}{m} \cdot \ha{\mu} \cdot \wa{0}{\thet}} \\
	=	\begin{aligned}[t]
			&\hilkn{r}{- \frac{\norm{m}}{\thet^{2}}} \cdot \delh{\lambda} \cdot\Bigg[ \sigu{\ha{- \frac{\norm{m}}{r^{2} \thet^{2}}}}{\ha{\frac{\thet}{\ol{m}}}}^{- 1} \\
			&\phantom{\ } \cdot \delh{\frac{\thet}{\ol{m}}} \Bigg] \cdot \delh{\mu} \cdot \delh{\frac{\lambda \mu \thet}{\ol{m}}}^{-1}.
		\end{aligned}
\end{multline*}
Thus by \eqref{e:ba},
\begin{align*}
	&\sigu{\ha{\lambda} \cdot \wa{0}{\thet}}{\xa{r}{m} \cdot \ha{\mu} \cdot \wa{0}{\thet}} \\
	&=	\begin{aligned}[t]
			&\hilkn{r}{- \frac{\norm{m}}{\thet^{2}}} \cdot \sigu{\ha{- \frac{\norm{m}}{r^{2} \thet^{2}}}}{\ha{\frac{\thet}{\ol{m}}}}^{- 1} \cdot \delh{\lambda} \\
			&\phantom{\ } \cdot \delh{\frac{\thet}{\ol{m}}} \cdot \left[ \delh{\frac{\lambda \thet}{\ol{m}}}^{- 1} \cdot \delh{\frac{\lambda \thet}{\ol{m}}} \right] \cdot \delh{\mu} \\
			&\phantom{\ } \cdot \delh{\frac{\lambda \mu \thet}{\ol{m}}}^{-1}
		\end{aligned} \\
	&=	\begin{aligned}[t]
			&\hilkn{r}{- \frac{\norm{m}}{\thet^{2}}} \cdot \sigu{\ha{- \frac{\norm{m}}{r^{2} \thet^{2}}}}{\ha{\frac{\thet}{\ol{m}}}}^{- 1} \\
			&\phantom{\ } \cdot \sigu{\ha{\lambda}}{\ha{\frac{\thet}{\ol{m}}}} \cdot \sigu{\ha{\frac{\lambda \thet}{\ol{m}}}}{\ha{\mu}}.
		\end{aligned}
\end{align*}
But by Theorem~\ref{t:sigtk},
\begin{align*}
	\sigu{\ha{- \frac{\norm{m}}{r^{2} \thet^{2}}}}{\ha{\frac{\thet}{\ol{m}}}}^{- 1}
	&= \hilkn{- \frac{\norm{m}}{r^{2} \thet^{2}}}{\frac{( \thet / \ol{m} ) \ol{\del{1}{\thet / \ol{m}}}}{\thet}}^{- 1} \\
	&= \hilkn{- \frac{\norm{m}}{r^{2} \thet^{2}}}{\frac{\thet}{\ol{m}} \cdot \frac{\ol{m}}{r^{2}} \cdot \frac{1}{\thet}}^{- 1} \\
	&= \hilkn{- \frac{\norm{m}}{r^{2} \thet^{2}}}{\frac{1}{r^{2}}}^{- 1};
\end{align*}
and by \eqref{e:prop3},
	\[ \hilkn{- \frac{\norm{m}}{r^{2} \thet^{2}}}{\frac{1}{r^{2}}}^{- 1} = \hilkn{\frac{\norm{m}}{\thet^{2}}}{\frac{1}{r^{2}}}^{- 1}, \]
therefore by \eqref{e:prop5},
\begin{align*}
	\sigu{\ha{- \frac{\norm{m}}{r^{2} \thet^{2}}}}{\ha{\frac{\thet}{\ol{m}}}}^{- 1}
	&= \hilkn{\frac{\norm{m}}{\thet^{2}}}{\frac{1}{r^{2}}}^{- 1} \\
	&= \hilkn{\frac{\thet^{2}}{\norm{m}}}{\frac{1}{r^{2}}} \\
	&= \hilkn{\left( \frac{\thet^{2}}{\norm{m}} \right )^{2}}{\frac{1}{r}}.
\end{align*}
This implies that by \eqref{e:prop2},
	\[ \sigu{\ha{- \frac{\norm{m}}{r^{2} \thet^{2}}}}{\ha{\frac{\thet}{\ol{m}}}}^{- 1} = \hilkn{r}{\left( \frac{\thet^{2}}{\norm{m}} \right )^{2}}; \]
thus by \eqref{e:bastr},
\begin{align*}
	&\sigu{\ha{\lambda} \cdot \wa{0}{\thet}}{\xa{r}{m} \cdot \ha{\mu} \cdot \wa{0}{\thet}} \\
	&=	\begin{aligned}[t]
			&\hilkn{r}{- \frac{\norm{m}}{\thet^{2}}} \cdot \hilkn{r}{\left( \frac{\thet^{2}}{\norm{m}} \right )^{2}} \cdot \sigu{\ha{\lambda}}{\ha{\frac{\thet}{\ol{m}}}} \\
			&\phantom{\ } \cdot \sigu{\ha{\frac{\lambda \thet}{\ol{m}}}}{\ha{\mu}}
		\end{aligned} \\
	&= \hilkn{r}{- \frac{\thet^{2}}{\norm{m}}} \cdot \sigu{\ha{\lambda}}{\ha{\frac{\thet}{\ol{m}}}} \cdot \sigu{\ha{\frac{\lambda \thet}{\ol{m}}}}{\ha{\mu}}.
\end{align*}

Our result is proved.
\end{proof}


\chapter{The 2-cocycle of the double cover} \label{c:univ}

Recall that at the beginning of Chapter~\ref{c:t2coc}, we stated that
	\[ \sigma_{u}^{n / 2} = \sigma, \]
where $ \sigma \in H^{2}( G( k ), \mu_{2} ) $ corresponds to the non-trivial double cover and $ n $ is the number of roots of unity of $ k $.  Also, by Proposition~\ref{p:baisom}, we know that without loss of generality, we can replace \ba{s}{t} by \hilkn{s}{t} for all $ s $, $ t \in k^{\times} $.

We know that
	\[ \hilkn{s}{t}^{n / 2} = \hilk{s}{t}, \]
for all $ s $, $ t \in k^{\times} $ since $ n $ is always even for a local field $ k $ of characteristic zero.  Hence, by Proposition~\ref{p:baisom}, we have a homomorphism
\begin{align}
	\Psi \colon \ \qquad \pi_{1} &\to \mu_{2} \label{e:psi} \\
	\ba{s}{t} &\mapsto \hilk{s}{t}. \notag
\end{align}
This implies that if we let $ \Psi' $ be the map $ \hilkn{s}{t} \mapsto \hilk{s}{t} $ for all $ s $, $ t \in k^{\times} $, then we can state that $ \Psi'( \sigma_{u} ) = \sigma $.

Using the above as well as Theorem~\ref{t:sigtk} and Chapter~\ref{c:rest}, we have the following proposition:

\begin{prop} \label{p:sigma}
For $ \lambda $, $ \mu \in K^{\times} $,
\begin{multline*}
	\sig{\ha{\lambda}}{\ha{\mu}} \\
		=	\begin{cases}
				\displaystyle \hilk{\lambda}{\mu}, &\text{if $ \lambda $, $ \mu \in k^{\times} $;} \\
				\displaystyle \hilk{\mu}{- \del{2}{\lambda} / \thet}, &\text{if $ \lambda \notin k^{\times} $, $ \mu \in k^{\times} $;} \\
				\displaystyle \hilk{\lambda}{\mu \ol{\del{1}{\mu}} / \thet}, &\text{if $ \lambda \in k^{\times} $, $ \mu \notin k^{\times} $;} \\
				\displaystyle \hilk{- 1}{\norm{\lambda}} \cdot \hilk{- \frac{\lambda \ol{\del{1}{\lambda}}}{\thet}}{\lambda \mu}, &\text{if $ \lambda $, $ \mu \notin k^{\times} $, $ \lambda \mu \in k^{\times} $;} \\
				\displaystyle \hilk{\norm{\lambda}}{\norm{\mu}} \cdot \hilk{q}{\frac{\mu \ol{\del{1}{\mu}}}{\del{2}{\lambda}}}, &\text{otherwise,}
			\end{cases}
\end{multline*}
where, if $  \lambda = a + b \thet $, $ \mu = c + d \thet $, with $ a $, $ c \in k $, $ b $, $ d \in k^{\times} $,
	\[ q = a + \frac{b c}{d}. \]

Also, for $ ( r, m ) $, $ ( r', m' ) \in A $, $ g $, $ g' \in G $,
\begin{gather*}
	\sig{\xa{r}{m} \cdot g}{g' \cdot \xa{r}{m}} = \sig{g}{g'}, \\
	\sig{g}{\xa{r}{m} \cdot g'} = \sig{g \cdot \xa{r}{m}}{g'}, \\
	\sig{g}{\xa{r}{m}} = \sig{\xa{r}{m}}{g'} = 1;
\end{gather*}
and
\begin{enumerate}[\upshape (1)]
	\item
	$ \begin{aligned}[t]
		&\sig{\ha{\lambda} \cdot \wa{0}{\thet}}{\ha{\mu} \cdot \wa{0}{\thet}} \\
		&\qquad\qquad=	\displaystyle
					\begin{aligned}[t]
						&\sig{\ha{\lambda}}{\ha{\mu}} \cdot \sig{\ha{\lambda \mu}}{\ha{\frac{1}{\norm{\mu}}}} \\
						&\phantom{\ } \cdot \sig{\ha{- \frac{\lambda}{\ol{\mu}}}}{\ha{- 1}},
					\end{aligned}
	\end{aligned} $
	
	\item
	$ \begin{aligned}[t]
		&\sig{\ha{\lambda} \cdot \wa{0}{\thet}}{\ha{\mu}} \\
		&\qquad\qquad=	\displaystyle \sig{\ha{\lambda}}{\ha{\mu}} \cdot \sig{\ha{\lambda \mu}}{\ha{\frac{1}{\norm{\mu}}}},
	\end{aligned} $
	
	\item
	$ \sig{\ha{\lambda}}{\ha{\mu} \cdot \wa{0}{\thet}} = \sig{\ha{\lambda}}{\ha{\mu}} $,
	
	\item
	$ \begin{aligned}[t]
		&\sig{\ha{\lambda} \cdot \wa{0}{\thet}}{\xa{r}{m} \cdot \ha{\mu} \cdot \wa{0}{\thet}} \\
	&= \Sigma( r, m ) \cdot \sig{\ha{\lambda}}{\ha{\frac{\thet}{\ol{m}}}} \cdot \sig{\ha{\frac{\lambda \thet}{\ol{m}}}}{\ha{\mu}},
	\end{aligned} $ \\
where for $ r = a' + b' \thet $, $ m = c' + d' \thet $, $ a' $, $ b' $, $ c' $, $ d' \in k $,
\begin{multline*}
	\Sigma( r, m) \\
	=	\begin{cases}
			\displaystyle \hilk{\frac{\thet}{\ol{m}}}{- 1}, &\text{if $ m \in k^{\times} \thet $;} \\
			\displaystyle \hilk{- \frac{\norm{m}}{\thet^{2}}}{- \frac{r \thet}{\ol{m}}}, &\text{if $ \displaystyle - \frac{r \thet}{\ol{m}} \in k^{\times} $;} \\
			\displaystyle \hilk{r}{- \thet^{2}}, &\text{if $ r $, $ m \in k^{\times} $;} \\
			\displaystyle \hilk{\frac{a' c' + b' d' \thet^{2}}{c'}}{- \frac{\norm{r}}{\norm{m} \thet^{2}}} \cdot \hilk{\norm{r}}{- \frac{b' \thet^{2}}{c'}},
			&\begin{aligned}[t]
				&\text{if $ b' $, $ c' \neq 0 $,} \\
				&\text{$ - r \thet / \ol{m} \notin k^{\times} $;}
			\end{aligned} \\
			\displaystyle \hilk{r}{- \frac{\thet^{2}}{\norm{m}}}, &\text{otherwise.}
		\end{cases}
\end{multline*}
\end{enumerate}
\end{prop}

\begin{proof}
Note that we will use the properties of the Hilbert symbol \eqref{e:prop2i} -- \eqref{e:norm} as stated in Section~\ref{s:hilbert}.  We first apply $ \Psi' $ to Theorem~\ref{t:sigtk}.  The first four cases cannot be further simplified after applying $ \Psi' $; but when we look at the last case, i.e. when $ \lambda $, $ \mu $, $ \lambda \mu \notin k^{\times} $, we see that by \eqref{e:prop5ii},
	\[ \Psi'( \Sigma'( \lambda, \mu ) ) = 1. \]
Also,
	\[ \Psi'\left( \hilkn{- \frac{\norm{\del{1}{\mu}}}{\norm{\del{1}{\lambda}}}}{\frac{\norm{\lambda}}{q^{2}}} \right) = \hilk{- \frac{\norm{\del{1}{\mu}}}{\norm{\del{1}{\lambda}}}}{\frac{\norm{\lambda}}{q^{2}}}. \]
By \eqref{e:bastri},
	\[ \Psi'\left( \hilkn{- \frac{\norm{\del{1}{\mu}}}{\norm{\del{1}{\lambda}}}}{\frac{\norm{\lambda}}{q^{2}}} \right) = \hilk{- \frac{\norm{\del{1}{\mu}}}{\norm{\del{1}{\lambda}}}}{\norm{\lambda}} \cdot \hilk{- \frac{\norm{\del{1}{\mu}}}{\norm{\del{1}{\lambda}}}}{\frac{1}{q^{2}}}; \]
and since $ \norm{\del{2}{\lambda}} / \norm{\del{2}{\mu}} \in ( k^{\times} )^{2} $, by \eqref{e:prop5ii},
	\[ \Psi'\left( \hilkn{- \frac{\norm{\del{1}{\mu}}}{\norm{\del{1}{\lambda}}}}{\frac{\norm{\lambda}}{q^{2}}} \right) = \hilk{- \frac{\norm{\del{1}{\mu}}}{\norm{\del{1}{\lambda}}}}{\norm{\lambda}} \cdot \hilk{\frac{\norm{\del{2}{\lambda}}}{\norm{\del{2}{\mu}}}}{\norm{\lambda}}. \]
Hence by \eqref{e:bastri},
	\[ \Psi'\left( \hilkn{- \frac{\norm{\del{1}{\mu}}}{\norm{\del{1}{\lambda}}}}{\frac{\norm{\lambda}}{q^{2}}} \right) = \hilk{-\frac{\norm{\mu}}{\norm{\lambda}}}{\norm{\lambda}}; \]
thus using \eqref{e:prop3i} and \eqref{e:prop2i},
	\[ \Psi'\left( \hilkn{- \frac{\norm{\del{1}{\mu}}}{\norm{\del{1}{\lambda}}}}{\frac{\norm{\lambda}}{q^{2}}} \right) = \hilk{\norm{\mu}}{\norm{\lambda}} = \hilk{\norm{\lambda}}{\norm{\mu}}. \]

Consequently, for $ \lambda $, $ \mu $, $ \lambda \mu \notin k^{\times} $,
\begin{align*}
	&\Psi' \left( \sigu{\ha{\lambda}}{\ha{\mu}} \right) \\
	&= \sig{\ha{\lambda}}{\ha{\mu}} \\
	&= \Psi'\left( \hilkn{- \frac{\norm{\del{1}{\mu}}}{\norm{\del{1}{\lambda}}}}{\frac{\norm{\lambda}}{q^{2}}} \right) \cdot \Psi'\left( \hilkn{q}{\frac{\mu \ol{\del{1}{\mu}}}{\del{2}{\lambda}}} \right) \cdot \Psi'( \Sigma'( \lambda, \mu ) ) \\
	&= \hilk{\norm{\lambda}}{\norm{\mu}} \cdot \hilk{q}{\frac{\mu \ol{\del{1}{\mu}}}{\del{2}{\lambda}}}.
\end{align*}

Clearly for $ g $, $ g' \in G( k ) $,
\begin{gather*}
	\sig{\xa{r}{m} \cdot g}{g' \cdot \xa{r}{m}} = \sig{g}{g'}, \\
	\sig{g}{\xa{r}{m} \cdot g'} = \sig{g \cdot \xa{r}{m}}{g'}, \\
	\sig{g}{\xa{r}{m}} = \sig{\xa{r}{m}}{g'} = 1.
\end{gather*}

At the beginning of Chapter~\ref{c:rest}, we showed that
\begin{multline*}
	\sigu{\ha{\lambda} \cdot \wa{0}{\thet}}{\ha{\mu} \cdot \wa{0}{\thet}} \\
	=	\begin{cases}
			\displaystyle \sigu{\ha{\lambda}}{\ha{\frac{1}{\mu}}} \cdot \sigu{\ha{- \frac{\lambda}{\mu}}}{\ha{- 1}}^{- 1}, &\text{if $ \mu \in k^{\times} $;} \\
			\displaystyle
				\begin{aligned}[b]
					&\sigu{\ha{\lambda}}{\ha{\mu}} \cdot \sigu{\ha{\lambda \mu}}{\ha{\frac{1}{\norm{\mu}}}} \\
					&\phantom{\ } \cdot \sigu{\ha{- \frac{\lambda}{\ol{\mu}}}}{\ha{- 1}}^{- 1},
				\end{aligned} &\text{if $ \mu \notin k^{\times} $.}
		\end{cases}
\end{multline*}
By applying $ \Psi' $ to the above, we see that
\begin{multline*}
	\sig{\ha{\lambda} \cdot \wa{0}{\thet}}{\ha{\mu} \cdot \wa{0}{\thet}} \\
	=	\begin{cases}
			\displaystyle \sig{\ha{\lambda}}{\ha{\frac{1}{\mu}}} \cdot \sig{\ha{- \frac{\lambda}{\mu}}}{\ha{- 1}}^{- 1}, &\text{if $ \mu \in k^{\times} $;} \\
			\displaystyle
				\begin{aligned}[b]
					&\sig{\ha{\lambda}}{\ha{\mu}} \cdot \sig{\ha{\lambda \mu}}{\ha{\frac{1}{\norm{\mu}}}} \\
					&\phantom{\ } \cdot \sig{\ha{- \frac{\lambda}{\ol{\mu}}}}{\ha{- 1}}^{- 1},
				\end{aligned} &\text{if $ \mu \notin k^{\times} $.}
		\end{cases}
\end{multline*}
But for any $ s \in K^{\times} $, $ r \in k^{\times} $,
	\[ \sig{\ha{s}}{\ha{r}}
	=	\begin{cases}
			\hilk{s}{r}, &\text{if $ s \in k^{\times} $;} \\
			\hilk{r}{- \del{2}{s} / \thet}, &\text{if $ s \notin k^{\times} $.}
		\end{cases} \]
Also, as $ \sigma \in H^{2}( G( k ), \mu_{2} ) $, $ \sigma = \sigma^{- 1} $.  Hence by \eqref{e:prop2i},
\begin{align*}
	\sig{\ha{s}}{\ha{r}}
	&= \sig{\ha{s}}{\ha{r}}^{- 1} \\
	&=	\begin{cases}
			\hilk{s}{r^{- 1}}, &\text{if $ s \in k^{\times} $;} \\
			\hilk{r^{- 1}}{- \del{2}{s} / \thet}, &\text{if $ s \notin k^{\times} $}
		\end{cases} \\
	&= \sig{\ha{s}}{\ha{r^{- 1}}}.
\end{align*}
Thus, for $ \mu \in k^{\times} $,
	\[ \sig{\ha{\lambda}}{\ha{\frac{1}{\mu}}} = \sig{\ha{\lambda}}{\ha{\mu}}. \]
Also, using previous results in this proof, for $ \mu \in k^{\times} $,
	\[ \sig{\ha{\lambda \mu}}{\ha{\frac{1}{\norm{\mu}}}} = \sig{\ha{\lambda \mu}}{\ha{\frac{1}{\mu^{2}}}} = 1. \]
This implies that we can state that for any $ \lambda $, $ \mu \in K^{\times} $,
\begin{multline*}
	\sig{\ha{\lambda} \cdot \wa{0}{\thet}}{\ha{\mu} \cdot \wa{0}{\thet}} \\
	= \sig{\ha{\lambda}}{\ha{\mu}} \cdot \sig{\ha{\lambda \mu}}{\ha{\frac{1}{\norm{\mu}}}} \cdot \sig{\ha{- \frac{\lambda}{\ol{\mu}}}}{\ha{- 1}}.
\end{multline*}

By the same token, it possible to show that for any $ \lambda $, $ \mu \in K^{\times} $,
	\[ \sig{\ha{\lambda} \cdot \wa{0}{\thet}}{\ha{\mu}} = \sig{\ha{\lambda}}{\ha{\mu}} \cdot \sig{\ha{\lambda \mu}}{\ha{\frac{1}{\norm{\mu}}}}, \]
and we already know by definition that
\begin{align*}
	\Psi'( \sigu{\ha{\lambda}}{\ha{\mu} \cdot \wa{0}{\thet}} )
	&= \sig{\ha{\lambda}}{\ha{\mu} \cdot \wa{0}{\thet}} \\
	&= \Psi'( \sigu{\ha{\lambda}}{\ha{\mu}} ) \\
	&= \sig{\ha{\lambda}}{\ha{\mu}}.
\end{align*}

Thus, all we need to do now is to show explicitly what 
	\[ \Psi'( \sigu{\ha{\lambda} \cdot \wa{0}{\thet}}{\xa{r}{m} \cdot \ha{\mu} \cdot \wa{0}{\thet}} ) \]
is, for $ ( r, m ) \in A $, where we let $ r = a' + b' \thet $, $ m = c' + d' \thet $, $ a' $, $ b' $, $ c' $, $ d' \in k $.  By Proposition~\ref{p:sigxa}, we know that
\begin{align*}
	&\Psi'( \sigu{\ha{\lambda} \cdot \wa{0}{\thet}}{\xa{r}{m} \cdot \ha{\mu} \cdot \wa{0}{\thet}} ) \\
	&= \sig{\ha{\lambda} \cdot \wa{0}{\thet}}{\xa{r}{m} \cdot \ha{\mu} \cdot \wa{0}{\thet}} \\
	&= \Psi'( \Sigma_{u}( r, m) ) \cdot \sig{\ha{\lambda}}{\ha{\frac{\thet}{\ol{m}}}} \cdot \sig{\ha{\frac{\lambda \thet}{\ol{m}}}}{\ha{\mu}},
\end{align*}
where $ \Sigma_{u}( r, m ) $ is as described by the said proposition.  We will need to look at $ \Psi'( \Sigma_{u}( r, m ) ) $, specifically in the cases in which the Hilbert symbols can be further simplified, i.e. when $ r $, $ m \in k^{\times} $; and when $ b' $, $ c' \neq 0 $, $ - r \thet / \ol{m} \notin k^{\times} $.

When $ r $, $ m \in k^{\times} $,
	\[ \Psi'( \Sigma_{u}( r, m ) ) = \Psi'\left( \hilkn{r}{- \frac{\thet^{2}}{m^{2}}} \right) = \hilk{r}{- \frac{\thet^{2}}{m^{2}}}. \]
Using \eqref{e:bastri},
	\[ \Psi'( \Sigma_{u}( r, m ) ) = \hilk{r}{- \thet^{2}} \cdot \hilk{r}{\frac{1}{m^{2}}}; \]
hence,
	\[ \Psi'( \Sigma_{u}( r, m ) ) = \hilk{r}{- \thet^{2}} \]
by \eqref{e:prop5ii}.  Also, when $ b' $, $ c' \neq 0 $, $ - r \thet / \ol{m} \notin k^{\times} $,
\begin{align*}
	\Psi'( \Sigma_{u}( r, m ) )
	&= \Psi'\left( \hilkn{\frac{a' c' + b' d' \thet^{2}}{c'}}{- \frac{c'^{2} \norm{r}}{b'^{2} \norm{m} \thet^{2}}} \cdot \hilkn{\norm{r}}{- \frac{b' \thet^{2}}{c'}} \right) \\
	&= \hilk{\frac{a' c' + b' d' \thet^{2}}{c'}}{- \frac{c'^{2} \norm{r}}{b'^{2} \norm{m} \thet^{2}}} \cdot \hilk{\norm{r}}{- \frac{b' \thet^{2}}{c'}}.
\end{align*}
By \eqref{e:bastri},
	\[ \Psi'( \Sigma_{u}( r, m ) )
	=	\begin{aligned}[t]
			&\hilk{\frac{a' c' + b' d' \thet^{2}}{c'}}{- \frac{\norm{r}}{\norm{m} \thet^{2}}} \cdot \hilk{\frac{a' c' + b' d' \thet^{2}}{c'}}{\frac{c'^{2}}{b'^{2}}} \\
			&\phantom{\ } \cdot \hilk{\norm{r}}{- \frac{b' \thet^{2}}{c'}};
		\end{aligned} \]
thus
	\[ \Psi'( \Sigma_{u}( r, m ) ) = \hilk{\frac{a' c' + b' d' \thet^{2}}{c'}}{- \frac{\norm{r}}{\norm{m} \thet^{2}}} \cdot \hilk{\norm{r}}{- \frac{b' \thet^{2}}{c'}} \]
by \eqref{e:prop5ii}.  Thus our result is proved, letting $ \Sigma( r, m ) = \Psi'( \Sigma_{u}( r, m ) ) $.
\end{proof}

\begin{rem} \label{r:trace}
We should note that the formulae for $ \sig{\ha{\lambda}}{\ha{\mu}} $ may be written in terms of norms and traces.  This is because for $ \lambda \in K^{\times} $, $ \lambda \notin k^{\times} $,
	\[ \frac{\del{2}{\lambda}}{\thet} = \frac{1}{\left( \ol{\lambda} - \lambda \right) \thet} = - \frac{1}{\trace{\lambda \thet}}; \]
hence by Proposition~\ref{p:sigma}, for $ \lambda $, $ \mu \in K^{\times} $,
\begin{align*}
	&\sig{\ha{\lambda}}{\ha{\mu}} \\
	&	=	\begin{cases}
				\displaystyle \hilk{\lambda}{\mu}, &\text{if $ \lambda $, $ \mu \in k^{\times} $;} \\
				\displaystyle \hilk{\mu}{\frac{1}{\trace{\lambda \thet}}}, &\text{if $ \lambda \notin k^{\times} $, $ \mu \in k^{\times} $;} \\
				\displaystyle \hilk{\lambda}{\frac{\norm{\mu}}{\trace{\mu \thet}}}, &\text{if $ \lambda \in k^{\times} $, $ \mu \notin k^{\times} $;} \\
				\displaystyle \hilk{- 1}{\norm{\lambda}} \cdot \hilk{- \frac{\norm{\lambda}}{\trace{\lambda \thet}}}{\lambda \mu}, &\text{if $ \lambda $, $ \mu \notin k^{\times} $, $ \lambda \mu \in k^{\times} $;} \\
				\displaystyle \hilk{\norm{\lambda}}{\norm{\mu}} \cdot \hilk{\frac{\trace{\lambda \mu \thet}}{\trace{\mu \thet}}}{- \frac{\trace{\lambda \thet} \norm{\mu}}{\trace{\mu \thet}}}, &\text{otherwise.}
			\end{cases}
\end{align*}
By \eqref{e:bastri} and \eqref{e:prop2i},
\begin{align*}
	&\sig{\ha{\lambda}}{\ha{\mu}} \\
	&=	\begin{cases}
			\displaystyle \hilk{\lambda}{\mu}, &\text{if $ \lambda $, $ \mu \in k^{\times} $;} \\
			\displaystyle \hilk{\trace{\lambda \thet}}{\mu}, &\text{if $ \lambda \notin k^{\times} $, $ \mu \in k^{\times} $;} \\
			\displaystyle \hilk{\lambda}{\norm{\mu}} \cdot \hilk{\lambda}{\trace{\mu \thet}}, &\text{if $ \lambda \in k^{\times} $, $ \mu \notin k^{\times} $;} \\
			\displaystyle \hilk{\norm{\lambda}}{- \lambda \mu} \cdot \hilk{- \trace{\lambda \thet}}{\lambda \mu}, &\text{if $ \lambda $, $ \mu \notin k^{\times} $, $ \lambda \mu \in k^{\times} $;} \\
			\begin{aligned}[b]
				&\hilk{\norm{\lambda}}{\norm{\mu}} \cdot \hilk{- \frac{\trace{\lambda \thet}\norm{\mu}}{\trace{\mu \thet}}}{\trace{\mu \thet}} \\
				&\phantom{\ } \cdot \hilk{\trace{\lambda \mu \thet}}{- \frac{\trace{\lambda \thet} \norm{\mu}}{\trace{\mu \thet}}},
			\end{aligned} &\text{otherwise.}
		\end{cases}
\end{align*}

When $ \lambda \in k^{\times} $, $ \mu \notin k^{\times} $, we have by \eqref{e:norm} that
	\[ \hilk{\lambda}{\norm{\mu}} = \hilbk{\lambda}{\mu}. \]
Also, when $ \lambda $, $ \mu \notin k^{\times} $, $ \lambda \mu \in k^{\times} $, by \eqref{e:norm} and \eqref{e:prop3i},
	\[ \hilk{\norm{\lambda}}{- \lambda \mu} = \hilbk{\lambda}{- \lambda \mu} = \hilbk{\lambda}{\mu}. \]
Thus with the above, \eqref{e:prop5ii} and \eqref{e:bastri},
\begin{multline*}
	\sig{\ha{\lambda}}{\ha{\mu}} \\
	=	\begin{cases}
			\displaystyle \hilk{\lambda}{\mu}, &\text{if $ \lambda $, $ \mu \in k^{\times} $;} \\
			\displaystyle \hilk{\trace{\lambda \thet}}{\mu}, &\text{if $ \lambda \notin k^{\times} $, $ \mu \in k^{\times} $;} \\
			\displaystyle \hilbk{\lambda}{\mu} \cdot \hilk{\lambda}{\trace{\mu \thet}}, &\text{if $ \lambda \in k^{\times} $, $ \mu \notin k^{\times} $;} \\
			\displaystyle \hilbk{\lambda}{\mu} \cdot \hilk{- \trace{\lambda \thet}}{\lambda \mu}, &\text{if $ \lambda $, $ \mu \notin k^{\times} $, $ \lambda \mu \in k^{\times} $;} \\
			\begin{aligned}[b]
				&\hilk{\norm{\lambda}}{\norm{\mu}} \cdot \hilk{\trace{\lambda \thet}\norm{\mu}}{\trace{\mu \thet}} \\
				&\phantom{\ } \cdot \hilk{\trace{\lambda \mu \thet}}{- \trace{\lambda \thet} \norm{\mu} \trace{\mu \thet}},
			\end{aligned} &\text{otherwise,}
		\end{cases}
\end{multline*}
which is the formula given in the Introduction.
\end{rem}

Proposition~\ref{p:sigma} may seem complicated and unwieldy, but in fact, it can be further simplified.  Let $ \gamma_{i} \in G( k ) $, where $ i = 1 $, $ 2 $, $ 3 $, with $ \gamma_{3} = \gamma_{1} \gamma_{2} $ and
	\[ \gamma_{i} = \begin{pmatrix} * & * & * \\ * & * & * \\ g_{i} & h_{i} & j_{i} \end{pmatrix}. \]
Also, let
	\[ \Xg{i} = \begin{cases} ( \ol{g_{i}} \thet )^{- 1}, &\text{if $ g_{i} \neq 0 $;} \\ \ol{j_{i}}^{- 1}, &\text{if $ g_{i} = 0 $,} \end{cases} \]
and for $ \lambda $, $ \mu \in K^{\times} $, let
	\[ \ug{\lambda}{\mu} = \begin{cases} \hilk{\lambda}{- \mu}, &\text{if $ \lambda $, $ \mu \in k^{\times} $;} \\ \hilk{\norm{\lambda}}{- \norm{\mu}}, &\text{otherwise.} \end{cases} \]
Then $ \sigma $ can be expressed in terms of Hilbert symbols involving $ X( \gamma_{i} ) $ and $ u( \lambda, \mu ) $ as shown in the following theorem.  It should be noted that $ X( \gamma_{i} ) $ is analogous to Kubota's $ X( \gamma ) $ ($ \gamma \in \SL_{2}( k ) $) as defined in \cite{kubota67}, which was used in the formula for his 2-cocycle on $ \SL_{2} $.

\begin{thm} \label{t:sigma}
If $ \Xg{3} / ( \Xg{1} \Xg{2} ) \in k^{\times} $, then
	\[ \sig{\gamma_{1}}{\gamma_{2}} =
		\begin{aligned}[t]
			&\ug{\frac{\Xg{3}}{\Xg{2}}}{\Xg{1} \Xg{2}} \cdot \hilk{\frac{\del{2}{\Xg{3}}}{\del{2}{\Xg{2}}}}{- \frac{\norm{\Xg{2}} \del{2}{\Xg{2}}}{\del{2}{\Xg{1}}}} \\
			&\phantom{\ } \cdot \hilk{\frac{\Xg{3}}{\Xg{1} \Xg{2}}}{\frac{\del{2}{\Xg{1}} \del{2}{\Xg{2}}}{\del{2}{\Xg{3}} \thet}}.
		\end{aligned} \]

If $ \Xg{3} / ( \Xg{1} \Xg{2} ) \notin k^{\times} $, let
	\[ r = r( \gamma_{1}, \gamma_{2} ) = \frac{h_{2} g_{3} - h_{3} g_{2}}{g_{1} \ol{g_{2}}}. \]
Then we have
	\[ \sig{\gamma_{1}}{\gamma_{2}} =
		\begin{aligned}[t]
			&\hilk{- \del{2}{- \frac{r \Xg{3}}{\Xg{1} \Xg{2}}} \thet^{- 1}}{\norm{- \frac{r \Xg{3}}{\Xg{1} \Xg{2}}}} \\
			&\phantom{\ } \cdot \hilk{\norm{r}}{\frac{\del{2}{r}}{\thet}} \cdot \ug{\Xg{1}}{\frac{\Xg{3}}{\Xg{2}}} \cdot \ug{\frac{\Xg{3}}{\Xg{2}}}{\Xg{3}} \\
			&\phantom{\ } \cdot \bigg( \frac{\del{2}{\Xg{3} / ( \Xg{1} \Xg{2} )}}{\del{2}{\Xg{3} / \Xg{2}}}, \\
			&\phantom{\del{2}{\Xg{3}}} - \frac{\norm{\Xg{3} / ( \Xg{1} \Xg{2} )} \del{2}{\Xg{3} / ( \Xg{1} \Xg{2} )}}{\del{2}{\Xg{1}}} \bigg)_{k, 2} \\
			&\phantom{\ } \cdot \hilk{\frac{\del{2}{\Xg{2}}}{\del{2}{\Xg{3}}}}{- \frac{\norm{\Xg{2}} \del{2}{\Xg{2}}}{\del{2}{\Xg{3} / \Xg{2}}}}.
		\end{aligned} \]
\end{thm}

\begin{proof}
We first note that we want to use the matrix entries of $ \gamma_{i} $ in the formula for our 2-cocycle much in the same way as Kubota did (see \cite{kubota67}) in his theorem for the 2-cocycle on the group $ \SL_{2} $.  The matrix entry used depended on the Bruhat decomposition and which Bruhat cell the matrix belonged to; the choice made was considered using the bottom row of the matrix only.

In our case, we can do the same thing as there are only two Bruhat cells to consider (see Section~\ref{s:bruh}), just like the $ \SL_{2} $ case.  Looking at \eqref{e:bruc} and \eqref{e:bruf}, it is clear that the Bruhat cell a matrix in \SU{} belongs to depends on whether the $ ( 3, 1 ) $-entry is non-zero, and a choice can be made as follows.  With
	\[ \gamma_{i} = \begin{pmatrix} * & * & * \\ * & * & * \\ g_{i} & h_{i} & j_{i} \end{pmatrix} \in G( k ), \]
i = $ 1 $, $ 2 $, $ 3 $, and $ \gamma_{3} = \gamma_{1} \gamma_{2} $, let
	\[ \Xg{i} = \begin{cases} ( \ol{g_{i}} \thet )^{- 1}, &\text{if $ g_{i} \neq 0 $;} \\ \ol{j_{i}}^{- 1} &\text{if $ g_{i} = 0 $.} \end{cases} \]
Note that for any $ \xa{s_{1}}{n_{1}} \in N( k ) $, $ \gamma \in G( k ) $,
	\[ X( \xa{s_{1}}{n_{1}} \cdot \gamma ) = X( \gamma ) = X( \gamma \cdot \xa{s_{1}}{n_{1}} ). \]

We will have to prove this theorem in a few steps.  We first find a simplified formula for the 2-cocycle on the torus $ T( k ) $, written in terms of $ \Xg{i} $'s.  We then look for similar formulae for each of the cases (1) -- (3) of Proposition~\ref{p:sigma}.  We will also look at case (4) of Proposition~\ref{p:sigma}, finding a simplification of $ \Sigma( r, m ) $.  It will then be apparent how the formulae coincide.

We will often use the properties of Hilbert symbols \eqref{e:prop2i} -- \eqref{e:prop5ii} to simplify expressions in terms of Hilbert symbols.  We will use these properties mostly without mention throughout the rest of this proof.

Throughout this proof, let $ \lambda $, $ \mu \in K^{\times} $, $ \lambda' = a + b \thet $, $ \mu' = c + d \thet $, where $ a $, $ b $, $ c $, $ d \in k^{\times} $ and $ \lambda' \mu' \notin k^{\times} $.  By Proposition~\ref{p:sigma}, we have
	\[ \sig{\ha{\lambda'}}{\ha{\mu'}} = \hilk{\norm{\lambda'}}{\norm{\mu'}} \cdot \hilk{q}{\frac{\mu' \ol{\del{1}{\mu'}}}{\del{2}{\lambda'}}}, \]
where $ q = a + b c / d $.  This implies that
	\[ q = \frac{\del{2}{\mu'}}{\del{2}{\lambda' \mu'}}, \]
and
	\[ \sig{\ha{\lambda'}}{\ha{\mu'}} = \hilk{\norm{\lambda'}}{\norm{\mu'}} \cdot \hilk{\frac{\del{2}{\mu'}}{\del{2}{\lambda' \mu'}}}{- \frac{\norm{\mu'} \del{2}{\mu'}}{\del{2}{\lambda'}}}. \]
For any $ \lambda $, $ \mu \in K^{\times} $, let
	\[ F_{1}( \lambda, \mu ) = \hilk{\norm{\lambda}}{\norm{\mu}} \cdot \hilk{\frac{\del{2}{\mu}}{\del{2}{\lambda \mu}}}{- \frac{\norm{\mu} \del{2}{\mu}}{\del{2}{\lambda}}}. \]
Then we have
\begin{multline*}
	F_{1}( \lambda, \mu ) \\
	=	\begin{cases}
			\displaystyle \hilk{\lambda^{2}}{\mu^{2}} \cdot \hilk{1}{- \mu^{2}}, &\text{if $ \lambda $, $ \mu \in k^{\times} $;} \\
			\displaystyle \hilk{\norm{\lambda}}{\mu^{2}} \cdot \hilk{\frac{\thet}{\del{2}{\lambda \mu}}}{- \frac{\mu^{2} \thet}{\del{2}{\lambda}}}, &\text{if $ \lambda \notin k^{\times} $, $ \mu \in k^{\times} $;} \\
			\displaystyle \hilk{\lambda^{2}}{\norm{\mu}} \cdot \hilk{\frac{\del{2}{\mu}}{\del{2}{\lambda \mu}}}{- \frac{\norm{\mu} \del{2}{\mu}}{\thet}}, &\text{if $ \lambda \in k^{\times} $, $ \mu \notin k^{\times} $;} \\
			\displaystyle \hilk{\norm{\lambda}}{\norm{\mu}} \cdot \hilk{\frac{\del{2}{\mu}}{\thet}}{- \frac{\norm{\mu} \del{2}{\mu}}{\del{2}{\lambda}}}, &\text{if $ \lambda $, $ \mu \notin k^{\times} $, $ \lambda \mu \in k^{\times} $;} \\
			\displaystyle \hilk{\norm{\lambda}}{\norm{\mu}} \cdot \hilk{\frac{\del{2}{\mu}}{\del{2}{\lambda \mu}}}{- \frac{\norm{\mu} \del{2}{\mu}}{\del{2}{\lambda}}}, &\text{if $ \lambda $, $ \mu $, $ \lambda \mu \notin k^{\times} $.}
		\end{cases}
\end{multline*}
Using \eqref{e:del2q} and the properties of Hilbert symbols \eqref{e:prop2i} -- \eqref{e:prop5ii}, we can simplify $ F_{1}( \lambda, \mu ) $.  Only the case where $ \lambda $, $ \mu \notin k^{\times} $, $ \lambda \mu \in k^{\times} $ deserves some elaboration.  Noting that
	\[ \mu = \frac{\lambda \mu}{\lambda} = \frac{\lambda \mu}{\norm{\lambda}} \cdot \ol{\lambda}, \]
we have
	\[ F_{1}( \lambda, \mu )
	=	\begin{aligned}[t]
			&\hilk{\norm{\lambda}}{\frac{( \lambda \mu )^{2}}{\norm{\lambda}}} \\
			&\phantom{\ } \cdot \hilk{\frac{( \norm{\lambda} / ( \lambda \mu ) ) \del{2}{\ol{\lambda}}}{\thet}}{- \frac{\left( ( \lambda \mu )^{2} / \norm{\lambda} \right) ( \norm{\lambda} / ( \lambda \mu ) ) \del{2}{\ol{\lambda}}}{\del{2}{\lambda}}}.
		\end{aligned} \]
Also, note that $ \del{2}{\ol{\lambda}} = - \del{2}{\lambda} $.  The above can then be simplified using the properties of Hilbert symbols.  We will get 
\begin{multline*}
	F_{1}( \lambda, \mu ) \\
	=	\begin{cases}
			1, &\text{if $ \lambda $, $ \mu \in k^{\times} $;} \\
			\displaystyle \hilk{\mu}{- \frac{\del{2}{\lambda}}{\thet}}, &\text{if $ \lambda \notin k^{\times} $, $ \mu \in k^{\times} $;} \\
			\displaystyle \hilk{\lambda}{- \frac{\norm{\mu} \del{2}{\mu}}{\thet}}, &\text{if $ \lambda \in k^{\times} $, $ \mu \notin k^{\times} $;} \\
			\displaystyle \hilk{- 1}{\norm{\lambda}} \cdot \hilk{- \frac{\norm{\lambda} \del{2}{\lambda}}{\thet}}{\lambda \mu}, &\text{if $ \lambda $, $ \mu \notin k^{\times} $, $ \lambda \mu \in k^{\times} $;} \\
			\displaystyle \hilk{\norm{\lambda}}{\norm{\mu}} \cdot \hilk{\frac{\del{2}{\mu}}{\del{2}{\lambda \mu}}}{- \frac{\norm{\mu} \del{2}{\mu}}{\del{2}{\lambda}}}, &\text{if $ \lambda $, $ \mu $, $ \lambda \mu \notin k^{\times} $.}
		\end{cases}
\end{multline*}
By comparing $ F_{1}( \lambda, \mu ) $ with Proposition~\ref{p:sigma}, we find that
	\[ \sig{\ha{\lambda}}{\ha{\mu}} = \begin{cases} \hilk{\lambda}{\mu}, &\text{if $ \lambda $, $ \mu \in k^{\times} $;} \\ F_{1}( \lambda, \mu ), &\text{otherwise.} \end{cases} \]
In fact, if we let
	\[ u_{1}( \lambda, \mu ) = \begin{cases} \hilk{\lambda}{\mu}, &\text{if $ \lambda $, $ \mu \in k^{\times} $;} \\ \hilk{\norm{\lambda}}{\norm{\mu}}, &\text{otherwise,} \end{cases} \]
then
	\[ \sig{\ha{\lambda}}{\ha{\mu}} = u_{1}( \lambda, \mu ) \cdot \hilk{\frac{\del{2}{\mu}}{\del{2}{\lambda \mu}}}{- \frac{\norm{\mu} \del{2}{\mu}}{\del{2}{\lambda}}}. \]
Note that
	\[ X( \ha{\lambda} ) = \lambda , \quad X( \ha{\mu} ) = \mu, \quad X( \ha{\lambda} \cdot \ha{\mu} ) = \lambda \mu, \]
so we truly get a formula in terms of $ \Xg{i} $ on the torus.  Also, by Proposition~\ref{p:sigma}, $ \sig{\ha{\lambda}}{\ha{\mu}} = \sig{\ha{\lambda}}{\ha{\mu} \cdot \wa{0}{\thet}} $.  We note that
	\[ X( \ha{\mu} \cdot \wa{0}{\thet} ) = \mu, \quad X( \ha{\lambda} \cdot \ha{\mu} \cdot \wa{0}{\thet} ) = \lambda \mu, \]
i.e. for $ \gamma_{1} \in T( k ) $, $ \gamma_{2} \in T( k ) \cdot W $,
\begin{equation} \label{e:siglm}
	\sig{\gamma_{1}}{\gamma_{2}} = u_{1}( \Xg{1}, \Xg{2} ) \cdot \hilk{\frac{\del{2}{\Xg{2}}}{\del{2}{\Xg{3}}}}{- \frac{\norm{\Xg{2}} \del{2}{\Xg{2}}}{\del{2}{\Xg{1}}}}.
\end{equation}
Note that \eqref{e:siglm} can also be applied to any $ \gamma_{1} $, $ \gamma_{2} \in G( k ) $ such that $ \Xg{3} = \Xg{1} \Xg{2} $.  This is due to the Bruhat decomposition and Proposition~\ref{p:sigma}.

We now consider case (2) of Proposition~\ref{p:sigma}, where
	\[ \sig{\ha{\lambda} \cdot \wa{0}{\thet}}{\ha{\mu}} = \sig{\ha{\lambda}}{\ha{\mu}} \cdot \sig{\ha{\lambda \mu}}{\ha{\frac{1}{\norm{\mu}}}}. \]
Using Proposition~\ref{p:sigma} and the properties of Hilbert symbols \eqref{e:prop2i} -- \eqref{e:prop5ii}, we get
\begin{align*}
	&\sig{\ha{\lambda} \cdot \wa{0}{\thet}}{\ha{\mu}} \\
	&=	\begin{cases}
			\displaystyle \hilk{\lambda}{\mu}, &\text{if $ \lambda $, $ \mu \in k^{\times} $;} \\
			\displaystyle \hilk{\mu}{- \frac{\del{2}{\lambda}}{\thet}}, &\begin{aligned}[t]
				\text{if}\ &\text{$ \lambda \notin k^{\times} $,} \\
				&\text{$ \mu \in k^{\times} $;}
					\end{aligned} \\
			\displaystyle \hilk{\frac{\lambda}{\norm{\mu}}}{- \frac{\del{2}{\mu}}{\thet}}, &\begin{aligned}[t]
				\text{if}\ &\text{$ \lambda \in k^{\times} $,} \\
				&\text{$ \mu \notin k^{\times} $;}
				\end{aligned} \\
			\displaystyle \hilk{- 1}{\norm{\lambda}} \cdot \hilk{\frac{\del{2}{\lambda}}{\thet}}{\lambda \mu}, &\begin{aligned}[t]
				\text{if}\ &\text{$ \lambda $, $ \mu \notin k^{\times} $,} \\
				&\text{$ \lambda \mu \in k^{\times} $;}
				\end{aligned} \\
			\displaystyle \hilk{- \frac{\norm{\lambda} \del{2}{\lambda \mu}}{\thet}}{\norm{\mu}} \cdot \hilk{\frac{\del{2}{\mu}}{\del{2}{\lambda \mu}}}{- \frac{\norm{\mu} \del{2}{\mu}}{\del{2}{\lambda}}}, &\text{if $ \lambda $, $ \mu $, $ \lambda \mu \notin k^{\times} $.}
		\end{cases}
\end{align*}
We have
\begin{gather*}
	X( \ha{\lambda} \cdot \wa{0}{\thet} ) = \lambda, \quad X( \ha{\mu} ) = \mu, \\
	X( \ha{\lambda} \cdot \wa{0}{\thet} \cdot \ha{\mu} ) = \lambda / \ol{\mu}.
\end{gather*}
Just like the 2-cocycle on the torus, we will find a formula for $ \sig{\gamma_{1}}{\gamma_{2}} $, $ \gamma_{1} \in T( k ) \cdot \wa{0}{\thet} $, $ \gamma_{2} \in T( k ) $, in terms of $ \Xg{i} $'s.  By \eqref{e:del2q}, we have
	\[ \del{2}{\lambda' \mu'} = \del{2}{\frac{\lambda'}{\ol{\mu'}} \cdot \norm{\mu'}} = \frac{\del{2}{\lambda' / \ol{\mu'}}}{\norm{\mu'}}, \]
which implies that
\begin{multline*}
	\sig{\ha{\lambda'} \cdot \wa{0}{\thet}}{\ha{\mu'}} \\
	= \hilk{- \frac{\norm{\lambda'} \del{2}{\lambda' / \ol{\mu'}}}{\norm{\mu'} \thet}}{\norm{\mu'}} \cdot \hilk{\frac{\norm{\mu'} \del{2}{\mu'}}{\del{2}{\lambda' / \ol{\mu'}}}}{- \frac{\norm{\mu'} \del{2}{\mu'}}{\del{2}{\lambda'}}}.
\end{multline*}
The above can be simplified using the properties of Hilbert symbols, so that we get
\begin{multline*}
	\sig{\ha{\lambda'} \cdot \wa{0}{\thet}}{\ha{\mu'}} \\
	= \hilk{\frac{\norm{\lambda'} \del{2}{\lambda'} \del{2}{\mu'}}{\del{2}{\lambda' / \ol{\mu'}} \thet}}{\norm{\mu'}} \cdot \hilk{\frac{\del{2}{\mu'}}{\del{2}{\lambda' / \ol{\mu'}}}}{- \frac{\norm{\mu'} \del{2}{\mu'}}{\del{2}{\lambda'}}}.
\end{multline*}
Let
	\[ F_{2}( \lambda, \mu ) = \hilk{\frac{\norm{\lambda} \del{2}{\lambda} \del{2}{\mu}}{\del{2}{\lambda / \ol{\mu}} \thet}}{\norm{\mu}} \cdot \hilk{\frac{\del{2}{\mu}}{\del{2}{\lambda / \ol{\mu}}}}{- \frac{\norm{\mu} \del{2}{\mu}}{\del{2}{\lambda}}}. \]
It is possible to check that when $ \lambda $, $ \mu \in k^{\times} $, $ F_{2}( \lambda, \mu ) = 1 $, and that
	\[ \sig{\ha{\lambda} \cdot \wa{0}{\thet}}{\ha{\mu}}
	=	\begin{cases}
			\hilk{\lambda}{\mu}, &\text{if $ \lambda $, $ \mu \in k^{\times} $;} \\
			F_{2}( \lambda, \mu ), &\text{otherwise.}
		\end{cases} \]
Thus we can get a similar formula
\begin{multline*}
	\sig{\ha{\lambda} \cdot \wa{0}{\thet}}{\ha{\mu}} \\
	= u_{1}( \lambda, \mu ) \cdot \hilk{\frac{\del{2}{\mu}}{\del{2}{\lambda / \ol{\mu}}}}{- \frac{\norm{\mu} \del{2}{\mu}}{\del{2}{\lambda}}} \cdot \hilk{\frac{1}{\norm{\mu}}}{\frac{\del{2}{\lambda} \del{2}{\mu}}{\del{2}{\lambda / \ol{\mu}} \thet}},
\end{multline*}
i.e. for $ \gamma_{1} \in T( k ) \cdot \wa{0}{\thet} $, $ \gamma_{2} \in T( k ) $,
\begin{equation} \label{e:siglwm}
	\sig{\gamma_{1}}{\gamma_{2}}
	=	\begin{aligned}[t]
			&u_{1}( \Xg{1}, \Xg{2} ) \cdot \hilk{\frac{\del{2}{\Xg{2}}}{\del{2}{\Xg{3}}}}{- \frac{\norm{\Xg{2}} \del{2}{\Xg{2}}}{\del{2}{\Xg{1}}}} \\
			&\phantom{\ } \cdot \hilk{\frac{\Xg{3}}{\Xg{1} \Xg{2}}}{\frac{\del{2}{\Xg{1}} \del{2}{\Xg{2}}}{\del{2}{\Xg{3}} \thet}}.
		\end{aligned}
\end{equation}
In fact, \eqref{e:siglm} can be subsumed into \eqref{e:siglwm}, since for $ \gamma_{1} \in T( k ) $, $ \gamma_{2} \in T( k ) \cdot W $, we have $ \Xg{3} / \Xg{1} \Xg{2} = 1; $
hence using \eqref{e:siglwm} to calculate $ \sig{\gamma_{1}}{\gamma_{2}} $ will give the same answer as \eqref{e:siglm}.  Also, it can be checked that \eqref{e:siglwm} can be applied to any $ \gamma_{1} $, $ \gamma_{2} $ such that $ \Xg{3} / ( \Xg{1} \Xg{2} ) = 1 $ or $ 1 / \norm{\Xg{2}} $, using the Bruhat decomposition and Proposition~\ref{p:sigma}.

We also have, by Proposition~\ref{p:sigma},
\begin{multline*}
	\sig{\ha{\lambda} \cdot \wa{0}{\thet}}{\ha{\mu} \cdot \wa{0}{\thet}} \\
	= \sig{\ha{\lambda}}{\ha{\mu}} \cdot \sig{\ha{\lambda \mu}}{\ha{\frac{1}{\norm{\mu}}}} \cdot \sig{\ha{- \frac{\lambda}{\ol{\mu}}}}{\ha{- 1}},
\end{multline*}
i.e.
\begin{align*}
	&\sig{\ha{\lambda} \cdot \wa{0}{\thet}}{\ha{\mu} \cdot \wa{0}{\thet}} \\
	&=	\begin{cases}
			\displaystyle \hilk{- \lambda}{- \mu}, &\text{if $ \lambda $, $ \mu \in k^{\times} $;} \\
			\displaystyle \hilk{- \mu}{\frac{\del{2}{\lambda}}{\thet}}, &\begin{aligned}[t]
				\text{if}\ &\text{$ \lambda \notin k^{\times} $,} \\
				&\text{$ \mu \in k^{\times} $;}
					\end{aligned} \\
			\displaystyle \hilk{- \frac{\lambda}{\norm{\mu}}}{\frac{\del{2}{\mu}}{\thet}}, &\begin{aligned}[t]
				\text{if}\ &\text{$ \lambda \in k^{\times} $,} \\
				&\text{$ \mu \notin k^{\times} $;}
				\end{aligned} \\
			\displaystyle \hilk{- 1}{- 1} \cdot \hilk{- \frac{\del{2}{\lambda}}{\thet}}{\lambda \mu}, &\begin{aligned}[t]
				\text{if}\ &\text{$ \lambda $, $ \mu \notin k^{\times} $,} \\
				&\text{$ \lambda \mu \in k^{\times} $;}
				\end{aligned} \\
			\begin{aligned}[b]
				&\hilk{\frac{\norm{\lambda} \del{2}{\lambda} \del{2}{\mu}}{\del{2}{\lambda / \ol{\mu}} \thet}}{\norm{\mu}} \\
				&\phantom{\ } \cdot \hilk{\frac{\del{2}{\mu}}{\del{2}{\lambda / \ol{\mu}}}}{- \frac{\norm{\mu} \del{2}{\mu}}{\del{2}{\lambda}}} \cdot \hilk{- 1}{- \frac{\del{2}{- \lambda / \ol{\mu}}}{\thet}},
			\end{aligned} &\text{if $ \lambda $, $ \mu $, $ \lambda \mu \notin k^{\times} $.}
		\end{cases}
\end{align*}
Note that
\begin{gather*}
	X( \ha{\lambda} \cdot \wa{0}{\thet} ) = \lambda, \quad X( \ha{\mu} \cdot \wa{0}{\thet} ) = \mu, \\
	X( \ha{\lambda} \cdot \wa{0}{\thet} \cdot \ha{\mu} \cdot \wa{0}{\thet} ) = - \lambda / \ol{\mu}.
\end{gather*}
By \eqref{e:del2q},
	\[ \del{2}{\frac{\lambda'}{\ol{\mu'}}} = - \del{2}{- \frac{\lambda'}{\ol{\mu'}}}, \]
hence we get
\begin{align*}
	&\sig{\ha{\lambda'} \cdot \wa{0}{\thet}}{\ha{\mu'} \cdot \wa{0}{\thet}} \\
	&\qquad=	\begin{aligned}[t]
			&\hilk{\norm{\lambda'}}{\norm{\mu'}} \cdot \hilk{\frac{\del{2}{\mu'}}{\del{2}{- \lambda' / \ol{\mu'}}}}{- \frac{\norm{\mu'} \del{2}{\mu'}}{\del{2}{\lambda'}}} \\
			&\phantom{\ } \cdot \hilk{- \frac{1}{\norm{\mu'}}}{\frac{\del{2}{\lambda'} \del{2}{\mu'}}{\del{2}{- \lambda' / \ol{\mu'}} \thet}}.
		\end{aligned}
\end{align*}
Let
	\[ F_{3}( \lambda, \mu ) = \begin{aligned}[t]
			&\hilk{\norm{\lambda}}{\norm{\mu}} \cdot \hilk{\frac{\del{2}{\mu}}{\del{2}{- \lambda / \ol{\mu}}}}{- \frac{\norm{\mu} \del{2}{\mu}}{\del{2}{\lambda}}} \\
			&\phantom{\ } \cdot \hilk{- \frac{1}{\norm{\mu}}}{\frac{\del{2}{\lambda} \del{\mu}{t}}{\del{2}{- \lambda / \ol{\mu}} \thet}}.
		\end{aligned} \]
It can be checked that $ F_{3}( \lambda, \mu ) = 1 $ when $ \lambda $, $ \mu \in k^{\times} $, and that
	\[ \sig{\ha{\lambda} \cdot \wa{0}{\thet}}{\ha{\mu} \cdot \wa{0}{\thet}}
	=	\begin{cases}
			\hilk{- \lambda}{- \mu}, &\text{if $ \lambda $, $ \mu \in k^{\times} $;} \\
			F_{3}( \lambda, \mu ), &\text{otherwise.}
		\end{cases} \]
In other words, if we let
	\[ u_{2}( \lambda, \mu ) = \begin{cases} \hilk{- \lambda}{- \mu}, &\text{if $ \lambda $, $ \mu \in k^{\times} $;} \\
			\hilk{\norm{\lambda}}{\norm{\mu}}, &\text{otherwise,}
		\end{cases} \]
then
\begin{multline*}
	\sig{\ha{\lambda} \cdot \wa{0}{\thet}}{\ha{\mu} \cdot \wa{0}{\thet}} \\
	= u_{2}( \lambda, \mu ) \cdot \hilk{\frac{\del{2}{\mu}}{\del{2}{- \lambda / \ol{\mu}}}}{- \frac{\norm{\mu} \del{2}{\mu}}{\del{2}{\lambda}}} \cdot \hilk{- \frac{1}{\norm{\mu}}}{\frac{\del{2}{\lambda} \del{2}{\mu}}{\del{2}{- \lambda / \ol{\mu}} \thet}},
\end{multline*}
i.e. for $ \gamma_{1} $, $ \gamma_{2} \in T( k ) \cdot \wa{0}{\thet} $,
\begin{equation} \label{e:siglwmw}
	\sig{\gamma_{1}}{\gamma_{2}}
	=	\begin{aligned}[t]
			&u_{2}( \Xg{1}, \Xg{2} ) \cdot \hilk{\frac{\del{2}{\Xg{2}}}{\del{2}{\Xg{3}}}}{- \frac{\norm{\Xg{2}} \del{2}{\Xg{2}}}{\del{2}{\Xg{1}}}} \\
			&\phantom{\ } \cdot \hilk{\frac{\Xg{3}}{\Xg{1} \Xg{2}}}{\frac{\del{2}{\Xg{1}} \del{2}{\Xg{2}}}{\del{2}{\Xg{3}} \thet}}.
		\end{aligned}
\end{equation}
Again it can be checked that for any $ \gamma_{1} $, $ \gamma_{2} \in G( k ) $ such that $ \Xg{3} / ( \Xg{1} \Xg{2} ) = - 1 / \norm{\Xg{2}} $, \eqref{e:siglwmw} can be applied.

We now look at case (4) of Proposition~\ref{p:sigma}, i.e. where
\begin{multline*}
	\sig{\ha{\lambda} \cdot \wa{0}{\thet}}{\xa{r}{m} \cdot \ha{\mu} \cdot \wa{0}{\thet}} \\
	= \Sigma( r, m ) \cdot \sig{\ha{\lambda}}{\ha{\frac{\thet}{\ol{m}}}} \cdot \sig{\ha{\frac{\lambda \thet}{\ol{m}}}}{\ha{\mu}}.
\end{multline*}
We have
\begin{gather*}
	X( \ha{\lambda} \cdot \wa{0}{\thet} ) = \lambda, \quad
	X( \xa{r}{m} \cdot \ha{\mu} \cdot \wa{0}{\thet} ) = \mu, \\
	X( \ha{\lambda} \cdot \wa{0}{\thet} \cdot \xa{r}{m} \cdot \ha{\mu} \cdot \wa{0}{\thet} ) = \frac{\lambda \mu \thet}{\ol{m}}.
\end{gather*}
In all the previous cases, it is apparent that for any $ \gamma_{1} $, $ \gamma_{2} \in T( k ) \cdot W $, we have $ \Xg{3} / \Xg{1} \Xg{2} \in k^{\times} $.  We want to find out if a simplified formula exists and be similar to \eqref{e:siglm}, \eqref{e:siglwm} and \eqref{e:siglwmw} if
	\[ \frac{X( \ha{\lambda} \cdot \wa{0}{\thet} \cdot \xa{r}{m} \cdot \ha{\mu} \cdot \wa{0}{\thet} )}{X( \ha{\lambda} \cdot \wa{0}{\thet} ) X( \xa{r}{m} \cdot \ha{\mu} \cdot \wa{0}{\thet} )} = \frac{\thet}{\ol{m}} \in k^{\times}. \]
So assume $ \thet / \ol{m} \in k^{\times} $.  This implies that $ r = 0 $ and $ m \in k^{\times} \thet $.  By Proposition~\ref{p:sigma} and after simplifying, we have
\begin{align*}
	&\sig{\ha{\lambda} \cdot \wa{0}{\thet}}{\xa{0}{m} \cdot \ha{\mu} \cdot \wa{0}{\thet}} \\
	&=	\begin{cases}
			\displaystyle \hilk{\frac{\lambda \thet}{\ol{m}}}{- \lambda \mu}, &\text{if $ \lambda $, $ \mu \in k^{\times} $;} \\
			\displaystyle \hilk{\frac{\mu \thet}{\ol{m}}}{- \frac{\del{2}{\lambda} \ol{m}}{\thet^{2}}}, &\begin{aligned}[t]
				\text{if}\ &\text{$ \lambda \notin k^{\times} $,} \\
				&\text{$ \mu \in k^{\times} $;}
					\end{aligned} \\
			\displaystyle \hilk{\frac{\lambda \thet}{\ol{m}}}{\frac{\norm{\mu} \del{2}{\mu}}{\lambda \thet}}, &\begin{aligned}[t]
				\text{if}\ &\text{$ \lambda \in k^{\times} $,} \\
				&\text{$ \mu \notin k^{\times} $;}
				\end{aligned} \\
			\begin{aligned}[b]
				&\hilk{- 1}{\norm{\lambda}} \cdot \hilk{\frac{\norm{\lambda} \del{2}{\lambda} \ol{m}}{\thet^{2}}}{\frac{\lambda \mu \thet}{\ol{m}}} \\
				&\phantom{\ } \cdot \hilk{\frac{\thet}{\ol{m}}}{\frac{\del{2}{\lambda}}{\thet}},
			\end{aligned} &\begin{aligned}[t]
				\text{if}\ &\text{$ \lambda $, $ \mu \notin k^{\times} $,} \\
				&\text{$ \lambda \mu \in k^{\times} $;}
				\end{aligned} \\
			\begin{aligned}[b]
				&\hilk{\norm{\lambda}}{\norm{\mu}} \cdot \hilk{\frac{\del{2}{\mu}}{\del{2}{\lambda \mu \thet / \ol{m}}}}{- \frac{\norm{\mu} \del{2}{\mu}}{\del{2}{\lambda}}} \\
				&\phantom{\ } \cdot \hilk{\frac{\thet}{\ol{m}}}{\frac{\del{2}{\lambda} \del{2}{\mu}}{\del{2}{\lambda \mu \thet / \ol{m}} \thet}},
			\end{aligned} &\text{if $ \lambda $, $ \mu $, $ \lambda \mu \notin k^{\times} $.}
		\end{cases}
\end{align*}
For any $ \lambda $, $ \mu \in K^{\times} $, $ m \in k^{\times} \thet $, let
	\[ F_{4}\left( \lambda, \mu, \frac{\thet}{\ol{m}} \right) = \begin{aligned}[t]
				&\hilk{\norm{\lambda}}{\norm{\mu}} \cdot \hilk{\frac{\del{2}{\mu}}{\del{2}{\lambda \mu \thet / \ol{m}}}}{- \frac{\norm{\mu} \del{2}{\mu}}{\del{2}{\lambda}}} \\
				&\phantom{\ } \cdot \hilk{\frac{\thet}{\ol{m}}}{\frac{\del{2}{\lambda} \del{2}{\mu}}{\del{2}{\lambda \mu \thet / \ol{m}} \thet}}.
			\end{aligned} \]
Then we have $ F_{4}( \lambda, \mu, \thet / \ol{m} ) = 1 $ for $ \lambda $, $ \mu \in k^{\times} $, and
\begin{multline*}
	\sig{\ha{\lambda} \cdot \wa{0}{\thet}}{\xa{0}{m} \cdot \ha{\mu} \cdot \wa{0}{\thet}} \\
	=	\begin{cases}
			\displaystyle \hilk{\frac{\lambda \thet}{\ol{m}}}{- \lambda \mu}, &\text{if $ \lambda $, $ \mu \in k^{\times} $;} \\
			\displaystyle F_{4}\left( \lambda, \mu, \frac{\thet}{\ol{m}} \right), &\text{otherwise.}
		\end{cases}
\end{multline*}
So if
	\[ u_{3}\left( \lambda, \mu, \frac{\thet}{\ol{m}} \right) = \begin{cases}
			\displaystyle \hilk{\frac{\lambda \thet}{\ol{m}}}{- \lambda \mu}, &\text{if $ \lambda $, $ \mu \in k^{\times} $;} \\
			\displaystyle \hilk{\norm{\lambda}}{\norm{\mu}}, &\text{otherwise,}
		\end{cases} \]
then
\begin{multline*}
	\sig{\ha{\lambda} \cdot \wa{0}{\thet}}{\xa{0}{m} \cdot \ha{\mu} \cdot \wa{0}{\thet}} \\
	= u_{3}\left( \lambda, \mu, \frac{\thet}{\ol{m}} \right) \cdot \hilk{\frac{\del{2}{\mu}}{\del{2}{\lambda \mu \thet / \ol{m}}}}{- \frac{\norm{\mu} \del{2}{\mu}}{\del{2}{\lambda}}} \cdot \hilk{\frac{\thet}{\ol{m}}}{\frac{\del{2}{\lambda} \del{2}{\mu}}{\del{2}{\lambda \mu \thet / \ol{m}} \thet}}.
\end{multline*}
In other words, if $ \gamma_{1} \in T( k ) \cdot \wa{0}{\thet} $, $ \gamma_{2} \in \xa{0}{m} \cdot T( k ) \cdot \wa{0}{\thet} $, $ m \in k^{\times} \thet $, then
\begin{multline} \label{e:siglwxmw}
	\sig{\gamma_{1}}{\gamma_{2}} \\
	=	\begin{aligned}[t]
			&u_{3}\left( \Xg{1}, \Xg{2}, \frac{\Xg{3}}{\Xg{1} \Xg{2}} \right) \cdot \hilk{\frac{\del{2}{\Xg{2}}}{\del{2}{\Xg{3}}}}{- \frac{\norm{\Xg{2}} \del{2}{\Xg{2}}}{\del{2}{\Xg{1}}}} \\
			&\phantom{\ } \cdot \hilk{\frac{\Xg{3}}{\Xg{1} \Xg{2}}}{\frac{\del{2}{\Xg{1}} \del{2}{\Xg{2}}}{\del{2}{\Xg{3}} \thet}},
		\end{aligned}
\end{multline}
and we observe that \eqref{e:siglwxmw} may be applied to any two elements $ \gamma_{1} $, $ \gamma_{2} \in G( k ) $ such that $ \Xg{3} / ( \Xg{1} \Xg{2} ) = \thet / \ol{m} \in k^{\times} $. (Note that using the properties of Hilbert symbols, we can show that if $ \thet / \ol{m} = 1 $ or $ 1 / \norm{\mu} $, then $ u_{3}( \lambda, \mu, \thet / \ol{m} ) = u_{1}( \lambda, \mu ) $, and if $ \thet / \ol{m} = - 1 / \norm{\mu} $, then $ u_{3}( \lambda, \mu, \thet / \ol{m} ) = u_{2}( \lambda, \mu ) $.)

As we can see, \eqref{e:siglwm}, \eqref{e:siglwmw} and \eqref{e:siglwxmw} each differ from each other by a factor, i.e. if $ \gamma_{1} $, $ \gamma_{2} \in G( k ) $ such that $ \Xg{3} / \Xg{1} \Xg{2} \in k^{\times} $, then we want a function
\begin{align*}
	&u'( \Xg{1}, \Xg{2}, \Xg{3} ) \\
	&=	\begin{cases}
			\displaystyle \hilk{\frac{\Xg{3}}{\Xg{2}}}{- \Xg{1} \Xg{2}},
				&\begin{aligned}[t]
					\text{if}\ &\text{$ \Xg{1} $, $ \Xg{2} $, $ \Xg{3} \in k^{\times} $;} \\
				\end{aligned} \\
			\displaystyle \hilk{\norm{\Xg{1}}}{\norm{\Xg{2}}}, &\text{otherwise,}
		\end{cases}
\end{align*}
so that
	\[ \sig{\gamma_{1}}{\gamma_{2}}
	=	\begin{aligned}[t]
			&u'( \Xg{1}, \Xg{2}, \Xg{3} ) \cdot \hilk{\frac{\del{2}{\Xg{2}}}{\del{2}{\Xg{3}}}}{- \frac{\norm{\Xg{2}} \del{2}{\Xg{2}}}{\del{2}{\Xg{1}}}} \\
			&\phantom{\ } \cdot \hilk{\frac{\Xg{3}}{\Xg{1} \Xg{2}}}{\frac{\del{2}{\Xg{1}} \del{2}{\Xg{2}}}{\del{2}{\Xg{3}} \thet}}.
		\end{aligned} \]
But we have, using the properties of Hilbert symbols,
	\[ \hilk{\norm{\frac{\Xg{3}}{\Xg{2}}}}{- \norm{\Xg{1} \Xg{2}}} = \hilk{\norm{\Xg{1}}}{\norm{\Xg{2}}} \]
for any $ \Xg{1} $, $ \Xg{2} $, $ \Xg{3} \in K^{\times} $ such that $ \Xg{3} / ( \Xg{1} \Xg{2} ) \in k^{\times} $.  These remarks bring us to define for any $ s $, $ t \in K^{\times} $,
	\[ \ug{s}{t} = \begin{cases} \hilk{s}{- t}, &\text{if $ s $, $ t \in k^{\times} $;} \\ \hilk{\norm{s}}{- \norm{t}}, &\text{otherwise,} \end{cases} \]
so that we have
\begin{equation} \label{e:sigk}
	\sig{\gamma_{1}}{\gamma_{2}}
	=	\begin{aligned}[t]
			&\ug{\frac{\Xg{3}}{\Xg{2}}}{\Xg{1} \Xg{2}} \cdot \hilk{\frac{\del{2}{\Xg{2}}}{\del{2}{\Xg{3}}}}{- \frac{\norm{\Xg{2}} \del{2}{\Xg{2}}}{\del{2}{\Xg{1}}}} \\
			&\phantom{\ } \cdot \hilk{\frac{\Xg{3}}{\Xg{1} \Xg{2}}}{\frac{\del{2}{\Xg{1}} \del{2}{\Xg{2}}}{\del{2}{\Xg{3}} \thet}}.
		\end{aligned}
\end{equation}
\eqref{e:sigk} obviously is only defined when $ \Xg{3} / ( \Xg{1} \Xg{2} ) \in k^{\times} $, which implies that we need to find a different formula for the case when $ \Xg{3} / ( \Xg{1} \Xg{2} ) \notin k^{\times} $.

We still have yet to look at
	\[ \sig{\ha{\lambda} \cdot \wa{0}{\thet}}{\xa{r}{m} \cdot \ha{\mu} \cdot \wa{0}{\thet}}, \]
where $ r \neq 0 $.  What we first need to do is to find a formula for $ \Sigma( r, m ) $, where $ \Sigma( r, m ) $ is as defined in Proposition~\ref{p:sigma}.  As $ \Sigma( r, m ) $ depends on whether $ r $, $ m $ and $ - r \thet / \ol{m} $ lie in the subfield $ k $, we should find a formula which only involves these three values.

Recall that for $ \spl{r'}{m'} \in A $ (see \eqref{e:defA1}), $ r' = a' + b' \thet $, $ m' = c' + d' \thet $, $ a' $, $ b' $, $ c' $, $ d' \in k $, $ b' $, $ c' \neq 0 $ and $ - r' \thet / \ol{m'} \notin k^{\times} $,
	\[ \Sigma( r', m' ) = \hilk{\frac{a' c' + b' d' \thet^{2}}{c'}}{- \frac{\norm{r'}}{\norm{m'} \thet^{2}}} \cdot \hilk{\norm{r'}}{- \frac{b' \thet^{2}}{c'}}. \]
Since $ c' = - \norm{r'} / 2 $, this implies that
	\[ \Sigma( r', m' ) = \hilk{- \frac{2 ( a' c' + b' d' \thet^{2} )}{\norm{r'}}}{- \frac{\norm{r'}}{\norm{m'} \thet^{2}}} \cdot \hilk{\norm{r'}}{\frac{2 b' \thet^{2}}{\norm{r'}}}. \]
Also, we have
	\[ \del{2}{r'} = - \frac{1}{2 b' \thet}, \quad \del{2}{- \frac{r' \thet}{\ol{m'}}} = \frac{\norm{m'}}{2 ( a' c' + b' d' \thet^{2} ) \thet}.  \]
Thus,
	\[ \Sigma( r', m' )
	=	\begin{aligned}[t]
			&\hilk{- \frac{\norm{m'}}{\norm{r'} \del{2}{- r' \thet / \ol{m'}} \thet}}{- \frac{\norm{r'}}{\norm{m'} \thet^{2}}} \\
			&\phantom{\ } \cdot \hilk{\norm{r'}}{- \frac{\thet}{\norm{r'} \del{2}{r'}}},
		\end{aligned} \]
and simplifying the above, we get
	\[ \Sigma( r', m' ) = \hilk{- \frac{\del{2}{- r' \thet / \ol{m'}}}{\thet}}{\norm{- \frac{r' \thet}{\ol{m'}}}} \cdot \hilk{\norm{r'}}{\frac{\del{2}{r'}}{\thet}}. \]
In fact, if we let $ \spl{s_{1}}{n_{1}} \in A $ with $ s_{1} \neq 0 $, and
	\[ F_{5}( s_{1}, n_{1} ) = \hilk{- \frac{\del{2}{- s_{1} \thet / \ol{n_{1}}}}{\thet}}{\norm{- \frac{s_{1} \thet}{\ol{n_{1}}}}} \cdot \hilk{\norm{s_{1}}}{\frac{\del{2}{s_{1}}}{\thet}}, \]
then we can check that
	\[ \Sigma( r, m ) = F_{5}( r, m ) \]
for any $ \spl{r}{m} \in A $, $ r \neq 0 $.

So this implies that by Proposition~\ref{p:sigma} and \eqref{e:sigk},
\begin{align}
	&\sig{\ha{\lambda} \cdot \wa{0}{\thet}}{\xa{r}{m} \cdot \ha{\mu} \cdot \wa{0}{\thet}} \label{e:siglwrmw} \\
	&= \Sigma( r, m ) \cdot \sig{\ha{\lambda}}{\ha{\frac{\thet}{\ol{m}}}} \cdot \sig{\ha{\frac{\lambda \thet}{\ol{m}}}}{\ha{\mu}} \notag \\
	&=	\begin{aligned}[t]
			&\hilk{- \frac{\del{2}{- r \thet / \ol{m}}}{\thet}}{\norm{- \frac{r \thet}{\ol{m}}}} \cdot \hilk{\norm{r}}{\frac{\del{2}{r}}{\thet}} \cdot \Bigg[ \ug{\lambda}{\frac{\lambda \thet}{\ol{m}}} \\
			&\phantom{\ } \cdot \hilk{\frac{\del{2}{\thet / \ol{m}}}{\del{2}{\lambda \thet / \ol{m}}}}{- \frac{\norm{\thet / \ol{m}} \del{2}{\thet / \ol{m}}}{\del{2}{\lambda}}} \Bigg] \cdot \Bigg[ \ug{\frac{\lambda \thet}{\ol{m}}}{\frac{\lambda \mu \thet}{\ol{m}}} \\
			&\phantom{\ } \cdot \hilk{\frac{\del{2}{\mu}}{\del{2}{\lambda \mu \thet / \ol{m}}}}{- \frac{\norm{\mu} \del{2}{\mu}}{\del{2}{\lambda \thet / \ol{m}}}} \Bigg]. \notag
		\end{aligned}
\end{align}

We want to show, for every $ \gamma_{1} $, $ \gamma_{2} \in G( k ) $ such that $ \Xg{3} / ( \Xg{1} \Xg{2} ) \notin k^{\times} $, that \sig{\gamma_{1}}{\gamma_{2}} can be calculated using \eqref{e:siglwrmw}.  We will also find an expression for \sig{\gamma_{1}}{\gamma_{2}} only in terms of the bottom row entries of $ \gamma_{1} $, $ \gamma_{2} $ and $ \gamma_{3} = \gamma_{1} \gamma_{2} $.

Assume that $ \Xg{3} / ( \Xg{1} \Xg{2} ) \notin k^{\times} $.  It is obvious that $ \Xg{i} $ for each $ i = 1 $, $ 2 $, $ 3 $ is never zero.  Let
	\[ \gamma_{i} = \begin{pmatrix} a_{i} & * & * \\ d_{i} & * & * \\ g_{i} & h_{i} & j_{i} \end{pmatrix}. \]
Then by the Bruhat decomposition \eqref{e:bruc},
	\[ \gamma_{i} = \xa{- \frac{\ol{d_{i}}}{\ol{g_{i}}}}{\frac{a_{i}}{g_{i}}} \cdot \ha{\frac{1}{\ol{g_{i}} \thet}} \cdot \wa{0}{\thet} \cdot \xa{\frac{h_{i}}{g_{i}}}{\frac{j_{i}}{g_{i}}}. \]
By Proposition~\ref{p:sigma},
\begin{align*}
	\sig{\gamma_{1}}{\gamma_{2}}
	&=	\begin{aligned}[t]
			&\sigma \bigg( \xa{- \frac{\ol{d_{1}}}{\ol{g_{1}}}}{\frac{a_{1}}{g_{1}}} \cdot \ha{\frac{1}{\ol{g_{1}} \thet}} \cdot \wa{0}{\thet} \cdot \xa{\frac{h_{1}}{g_{1}}}{\frac{j_{1}}{g_{1}}}, \\
			&\phantom{\sigma \bigg( \xa{r}{m} \bigg)} \xa{- \frac{\ol{d_{2}}}{\ol{g_{2}}}}{\frac{a_{2}}{g_{2}}} \cdot \ha{\frac{1}{\ol{g_{2}} \thet}} \cdot \wa{0}{\thet} \cdot \xa{\frac{h_{2}}{g_{2}}}{\frac{j_{2}}{g_{2}}} \bigg)
		\end{aligned} \\
	&=	\begin{aligned}[t]
			&\sigma \bigg( \ha{\frac{1}{\ol{g_{1}} \thet}} \cdot \wa{0}{\thet}, \\
			&\phantom{\sigma \bigg( \xa{r}{m} \bigg)} \xa{\frac{h_{1}}{g_{1}}}{\frac{j_{1}}{g_{1}}} \cdot \xa{- \frac{\ol{d_{2}}}{\ol{g_{2}}}}{\frac{a_{2}}{g_{2}}} \cdot \ha{\frac{1}{\ol{g_{2}} \thet}} \cdot \wa{0}{\thet} \bigg).
		\end{aligned}
\end{align*}
Let
	\[ \xa{r}{m} =  \xa{\frac{h_{1}}{g_{1}}}{\frac{j_{1}}{g_{1}}} \cdot \xa{- \frac{\ol{d_{2}}}{\ol{g_{2}}}}{\frac{a_{2}}{g_{2}}} = \xa{\frac{h_{1}}{g_{1}} - \frac{\ol{d_{2}}}{\ol{g_{2}}}}{\frac{j_{1}}{g_{1}} + \frac{h_{1} d_{2}}{g_{1} g_{2}} + \frac{a_{2}}{g_{2}}}. \]
Then we have
\begin{gather*}
	\gamma_{3} = \begin{pmatrix} * & * & * \\ * & * & * \\ m g_{1} g_{2} & - r g_{1} \ol{g_{2}} / g_{2} + m g_{1} h_{2} & g_{1} / \ol{g_{2}} + r g_{1} \ol{h_{2}} / g_{2} + m g_{1} j_{2} \end{pmatrix}, \\
	\Xg{3} = \frac{1}{\ol{m} \ol{g_{1}} \ol{g_{2}} \thet} = \frac{\Xg{1} \Xg{2} \thet}{\ol{m}}.
\end{gather*}

Since $ \Xg{3} / ( \Xg{1} \Xg{2} ) \notin k^{\times} $, this implies that $ \thet / \ol{m} \notin k^{\times} $, i.e. $ r \neq 0 $.   Hence we use \eqref{e:siglwrmw} so that
\begin{equation*} 
	\sig{\gamma_{1}}{\gamma_{2}}
	=	\begin{aligned}[t]
			&\hilk{- \frac{\del{2}{- r \Xg{3} / ( \Xg{1} \Xg{2} )}}{\thet}}{\norm{- \frac{r \Xg{3}}{\Xg{1} \Xg{2}}}} \\
			&\phantom{\ } \cdot \hilk{\norm{r}}{\frac{\del{2}{r}}{\thet}} \cdot \ug{\Xg{1}}{\frac{\Xg{3}}{\Xg{2}}} \cdot \ug{\frac{\Xg{3}}{\Xg{2}}}{\Xg{3}} \\
			&\phantom{\ } \cdot \bigg( \frac{\del{2}{\Xg{3} / ( \Xg{1} \Xg{2} )}}{\del{2}{\Xg{3} / \Xg{2}}}, \\
			&\phantom{\del{2}{\Xg{3}}} - \frac{\norm{\Xg{3} / ( \Xg{1} \Xg{2} )} \del{2}{\Xg{3} / ( \Xg{1} \Xg{2} )}}{\del{2}{\Xg{1}}} \bigg)_{k, 2} \\
			&\phantom{\ } \cdot \hilk{\frac{\del{2}{\Xg{2}}}{\del{2}{\Xg{3}}}}{- \frac{\norm{\Xg{2}} \del{2}{\Xg{2}}}{\del{2}{\Xg{3} / \Xg{2}}}}.
		\end{aligned}
\end{equation*}

We could just let $ r = h_{1} / g_{1} - \ol{d_{2}} / \ol{g_{2}} $, but there is another way to calculate $ r $ from just the bottom rows of the $ \gamma_{i} $'s.  We have
	\[ g_{3} = m g_{1} g_{2}, \quad h_{3} = - \frac{r g_{1} \ol{g_{2}}}{g_{2}} + m g_{1} h_{2}, \]
and rearranging the above,
	\[ r = \frac{h_{2} g_{3} - h_{3} g_{2}}{g_{1} \ol{g_{2}}}. \]
Thus our result is proved.
\end{proof}

\begin{rem} \label{r:comm2}
Recall that in Lemma~\ref{l:comm}, we established that the commutator of the 2-cocycle $ \sigma_{u} $ on the torus $ T( k ) $ was, for $ \lambda $, $ \mu \in K^{\times} $,
	\[ [ \lambda, \mu ]_{\sigma_{u}} = \frac{\sigu{\ha{\lambda}}{\ha{\mu}}}{\sigu{\ha{\mu}}{\ha{\lambda}}} = \hilbkn{\lambda}{\mu}^{2} \cdot \hilbkn{\lambda}{\ol{\mu}}^{- 1}. \]
Since $ \sigma_{u}^{n / 2} = \sigma $, we can calculate the commutator of the 2-cocycle $ \sigma $ on $ T( k ) $.  Since $ \hilbkn{s}{t}^{n / 2} = \hilbk{s}{t} $ for all $ s $, $ t \in K^{\times} $, we have
	\[ [ \lambda, \mu ]_{\sigma} = [ \lambda, \mu ]_{\sigma_{u}}^{n / 2} = \hilbk{\lambda}{\mu}^{2} \cdot \hilbk{\lambda}{\ol{\mu}}^{- 1}. \]
But using \eqref{e:prop2i} and \eqref{e:prop5ii}, we get
	\[ [ \lambda, \mu ]_{\sigma} = \hilbk{\lambda}{\ol{\mu}}. \]
\end{rem}


\part{The local Kubota symbol} \label{p:kubota}

\chapter{The compact open subgroup on which the quadratic 2-cocycle splits} \label{c:splsubgrp}

We have a 2-cocycle $ \sigma $ on $ G( k ) $, where $ k $ is a local field.  In the case that $ k $ is non-archimedean, there is a compact open subgroup $ \hat{\Gamma}_{\mathfrak{p}} $ on which $ \sigma $ splits, i.e.
	\[ \sigma|_{\hat{\Gamma}_{\mathfrak{p}}} = \partial \kappa, \]
where $ \kappa \colon \hat{\Gamma}_{\mathfrak{p}} \to \mu_{2} $ is a 1-cochain.  Note that $ \kappa $ is not quite unique, since it may be multiplied by a homomorphism $ \hat{\Gamma}_{\mathfrak{p}} \to \mu_{2} $.  The function $ \kappa $ is called a local Kubota symbol.  In this chapter we shall determine the compact open subgroup $ \hat{\Gamma}_{\mathfrak{p}} $ on which $ \sigma $ splits.

For the rest of this chapter, let $ k $ be a local field, and $ K = k( \thet ) $ be the quadratic extension of $ k $.  Also, we will let $ \mathfrak{p} $ denote the maximal ideal of $ k $.  $ \mathfrak{p} $ may be odd or even, depending on $ k $.  When $ K / k $ is ramified, we will assume that $ \thet $ is a prime element of $ K $.

In addition to unramified and ramified extensions, we will also need to consider split extensions, i.e. $ K \cong k \oplus k $, when establishing the compact open subgroup on which $ \sigma $ splits for a given extension $ K / k $.  This is so that we can consider the ad\`{e}le group in Chapter~\ref{c:general}.  We will obtain the following theorem:

\begin{thm} \label{t:compact}
Define a compact open subgroup $\hat{\Gamma}_{\mathfrak{p}}$ of $G(k)$ as follows:
	\[ \hat{\Gamma}_{\mathfrak{p}} =	\begin{cases}
				\displaystyle G( \Ok ), &\text{if $ \mathfrak{p} $ is odd and unramified (either inert or split) in $ K $;} \\
				\displaystyle G( \Ok, \thet ), &\text{if $ \mathfrak{p} $ is odd and ramified in $ K $;} \\
				\displaystyle G( \Ok, 4 ), &\text{if $ \mathfrak{p} $ is even and split in $ K $,}
			\end{cases} \]
where $ G( \Ok, \thet ) $ is defined as in \eqref{e:gthet} and $ G( \Ok, 4 ) $ is defined as in \eqref{e:geven}.
Then $\sigma$ splits on $\hat{\Gamma}_{\mathfrak{p}}$.
\end{thm}

The rest of this chapter is a proof of the above theorem.

\section{The odd primes} \label{s:odd}

$ K / k $ is either unramified, ramified or split.  We will look at each type extension in turn.




\subsection{The unramified extension} \label{ss:oddunram}

In \cite{deligne96}, Deligne constructed for any reductive group $ G $ over $ k $ a canonical central extension
	\[ 0 \lra H^{2}( k, \mathbb{Z} / n( 2 ) ) \lra E( c_{k} ) \lra G( k ) \lra 1. \]
In the above, $ \mathbb{Z} / n( 2 ) = \mu_{n}^{\otimes -2} $, and in the case that $ k $ contains an $ n $-th root of unity, $ H^{2}( k, \mathbb{Z} / n( 2 ) ) $ is canonically isomorphic to $ \mu_{n} $ (see 5.4 of \cite{deligne96}).

Suppose now that $ G $ is defined over the valuation ring $ \Ok $ in $ k $.  Deligne shows that when $ G $ is semi-simple and simply connected over $ \Spec( \Ok ) $
and $n$ is not a multiple of $\mathfrak{p}$, the functoriality for the map $ \Spec( k ) \to \Spec( \Ok ) $ reduces to a splitting
\begin{displaymath}
	\xymatrix{
	& & & G( \Ok ) \ar[dl] \ar@{^{(}->}[d] & \\
	0 \ar[r] & {H^{2}( k, \mathbb{Z} / n ( 2 ) )} \ar[r] & E( c_{k} ) \ar[r] & G( k ) \ar[r] & 0. }
\end{displaymath}

In the case that $ K / k $ is unramified, we shall show that $ G $ is semi-simple and simply connected over $ \Spec ( \Ok ) $, and that Deligne's extension is the same as Deodhar's when $ n $ is the number of roots of unity in $ k $.  Hence Deligne's splitting shows that we may take $ \hat{\Gamma}_{\mathfrak{p}} = G( \Ok ) $ in Theorem~\ref{t:compact}.

\begin{lem} \label{l:speco}
$ G $ is semi-simple and simply connected over $ \Spec( \Ok ) $.
\end{lem}

\begin{proof}
We recall that this means that $ G $ is semi-simple and simply connected both over $ k $ and over the residue field $ \Ok / \mathfrak{p} $.  The conditions of being semi-simple and simply connected over a field are unchanged when one passes to a field extension.  It is therefore sufficient to show that $ G $ is semi-simple and simply connected over the algebraic closures $ \ol{k} $ and $ \ol{\Ok / \mathfrak{p}} $.

Over $ \ol{k} $, we have an isomorphism of algebraic groups:
	\[ G \cong_{\ol{k}} \SL_{3}. \]
The same isomorphism holds over $ \ol{\Ok / \mathfrak{p}} $.  To see this, note that for any $ ( \Ok / \mathfrak{p} ) $-algebra $ A $, we have
	\[ G( A ) = \{ \nu \in \SL_{3}( A \otimes_{\mathbb{F}_{p}} \mathbb{F}_{p^{2}} ) \colon \nu^{t} J' \ol{\nu} = J' \}, \]
where $ J' $ is defined as in Section~\ref{s:sustruc}.  When $ A $ is an algebra over the quadratic extension $ \OK / ( \mathfrak{p} \OK ) $, the tensor product $ A \otimes_{\Ok / \mathfrak{p}} \OK / ( \mathfrak{p} \OK ) $ splits into a sum of two copies of $ A $, which are swapped by conjugation in $ K / k $.  Choosing one of these copies gives an isomorphism $ G( A ) \to \SL_{3}( A ) $.

Now since $ \SL_{3} $ is semi-simple and simply connected, it follows that $ G $ is also semi-simple and simply connected.
\end{proof}

Note that when $ \mathfrak{p} $ is ramified in $ K $, the group $ G $ over $ \Ok / \mathfrak{p} $ is not reductive, since the radical is
	\[ \{ \nu \in G \colon \nu \equiv I_{3} \pod{\mathfrak{P}} \} \subset G / ( \Ok / \mathfrak{p} ), \]
where $ \mathfrak{P} $ is the prime ideal of $ K $.

\begin{lem} \label{l:utce}
If $ n $ is the number of roots of unity in $ k $, then Deligne's extension $ E( c_{k} ) $ is the universal topological central extension.
\end{lem}

\begin{proof}
There is a $ k $-subgroup of $ G $ isomorphic to $ \SL_{2} $, and so we have a restriction map in continuous cohomology:
	\[ H^{2}( \SU( k ), \mu_{n} ) \to H^{2}( \SL_{2}( k ), \mu_{n} ). \]
Given a map from $ K_{2}( k ) $ to $ \mu_{n} $, we obtain elements of $ H^{2}( \SU( k ), \mu_{n} ) $ and $ H^{2}( \SL_{2}( k ), \mu_{n} ) $ constructed by Deodhar and Kubota (as well as Matsumoto, see \cite{matsumoto69}) respectively.  It is clear by inspection that the restriction of Deodhar's element to $\SL_{2}(k)$ is Kubota's element.  In particular, this restriction map is injective.

We also have elements of $ H^{2}( \SU( k ), \mu_{n} ) $ and $ H^{2}( \SL_{2}( k ), \mu_{n} ) $ constructed by Deligne.  Deligne shows in the commutative diagram~3.9.2 of \cite{deligne96} that the restriction of his element
of $ H^{2}( \SU( k ), \mu_{n} ) $ is the other element.
Deligne also proves in Proposition~3.7 of the same paper that for a semi-simple, simply connected split group (such as $ \SL_{2} $), his element is the same as Matsumoto's.  This implies that Deligne's element of $ H^{2}( \SU( k ), \mu_{n} ) $ is the same as Deodhar's.  It follows that $ E( c_{k} ) $ in the above diagram is a universal central extension.
\end{proof}

\subsection{The ramified extensions} \label{ss:oddram}

Let $ \mathfrak{p} $ be an odd ramified prime, and hence $ \mathfrak{p} \mid \thet^{2} $.
This implies that there exists a maximal ideal $ \mathfrak{P} \subset \OK $ such that $ \mathfrak{p} \OK = \mathfrak{P}^{2} $.  In fact, since $ K = k( \thet ) $, we have $ \mathfrak{P} = ( \thet ) $ by Theorem~\ref{t:ram} as $ \thet $ is a prime element of $K$.
We shall write $\mathbb{F}$ for the field $\mathcal{O}_{k}/\mathfrak{p}$, which is the same as
$\mathcal{O}_{K}/\mathfrak{P}$.

By Theorem~\ref{t:ram}, we have $ \OK = \Ok [ \thet ] $.  This implies that we can take $ \{ 1, \thet \} $ as an integral basis.

So we have $ \mathfrak{p} \OK = \mathfrak{P}^{2} $.
  For $ m \in \mathbb{N} $, let
\begin{equation} \label{e:gthet}
	 G\left( \Ok, \mathfrak{P}^{m} \right) = \{ \nu \in G( \Ok ) \colon \nu \equiv I_{3} \pod{\mathfrak{P}^{m}} \}.
\end{equation}

We know that our 2-cocycle $ \sigma $ splits on $ G( \Ok, \thet^{N} ) $ for sufficiently large $ N $, and we will show that we may take $N=1$.
We have
	\[ G( \Ok ) \supset G( \Ok, \mathfrak{P} ) \supset G( \Ok, \mathfrak{P}^{2} ) \supset \dots, \]
and the quotients are:
\begin{gather*}
	G( \Ok ) / G( \Ok, \mathfrak{P} ) \cong G( \OK / \mathfrak{P} ) = G( \mathbb{F} ), \\
	G\left( \Ok, \mathfrak{P}^{m} \right) / G\left( \Ok, \mathfrak{P}^{m + 1} \right) \cong \mathfrak{g}( \mathbb{F} ),
\end{gather*}
for $ m \geq 1 $ and where $ \mathfrak{g} $ is the Lie algebra of $ G $ (see Section~\ref{s:sustruc}).  This implies that $ \left| G\left( \Ok, \mathfrak{P}^{m} \right) / G\left( \Ok, \mathfrak{P}^{m + 1} \right) \right| $ is odd since $ \mathfrak{g}(\mathbb{F} ) $ is a vector space over $ \mathbb{F}$.

Proposition~9 of Chapter~I of \cite{serre02} states the following:

\begin{prop} \label{p:serre}
Let $ G $ be a profinite group and $ H $ be a closed subgroup of $ G $, with $ A $ an abelian group on which $ G $ acts continuously.  Then if $ ( G : H ) = n $, the kernel of $ \Res \colon H^{q}( G, A ) \to H^{q}( H, A ) $ is killed by $ n $.
\end{prop}

If we put $ G = G\left( \Ok, \mathfrak{P}^{m} \right) $, $ H = G\left( \Ok, \mathfrak{P}^{m + 1} \right) $ and $ A = \mu_{2} $ in the above proposition, then let $ \sigma_{1} $ be in the kernel of the restriction homomorphism 
	\[ \Res \colon H^{2}( G\left( \Ok, \mathfrak{P}^{m} \right), \mu_{2} ) \to H^{2}( G\left( \Ok, \mathfrak{P}^{m + 1} \right), \mu_{2} ).  \]
Writing the group operations in $ H^{2}( G\left( \Ok, \mathfrak{P}^{m} \right), \mu_{2} ) $ additively, we have
	\[ \left| G\left( \Ok, \mathfrak{P}^{m} \right) / G\left( \Ok, \mathfrak{P}^{m + 1} \right) \right| \cdot \sigma_{1} = 0. \]
But since $ 2 \cdot \sigma_{1} = 0 $, and $ \left| G\left( \Ok, \mathfrak{P}^{m} \right) / G\left( \Ok, \mathfrak{P}^{m + 1} \right) \right| $ is odd, this implies that $ \sigma_{1} = 0 $.  But if our 2-cocycle $ \sigma $ splits on $ G( \Ok, \mathfrak{P}^{N+1} ) $ for some $ N\ge 1 $, our result shows that it must split on $ G( \Ok, \mathfrak{P}^{N} ) $, hence our 2-cocycle $ \sigma $ splits on $ G( \Ok, \mathfrak{P} ) $.

\subsection{The split extensions} \label{ss:oddspl}

In the split case, we have $ K \cong k \oplus k $.  This implies that $ G( k ) \cong \SL_{3}( k ) $.
The $n$-fold cover of $\SL_{r}(k)$ was studied by Kazhdan and Patterson (see \cite{kazhpat84}).
They proved (Proposition 0.1.2) that if $n$ is not a multiple of $\mathfrak{p}$ then the
extension splits on the compact open subgroup $ \SL_{3}( \Ok ) $.
Since in our case $n=2$, this holds for all odd split primes.
Alternatively, one could get the same result from Deligne's paper as above.

\section{The even split primes} \label{s:even}



Now assume that $ \mathfrak{p} $ divides 2 and assume that $\mathfrak{p}$ splits in $K$.
As in the other split cases we have $ G( k ) \cong \SL_{3}( k ) $, and we may use results of
other authors on $\SL_{3}$.
For this purpose, choose another number field $l'$, which is totally complex, and which has
a local completion isomorphic to $k$.
The Kubota symbol on $\SL_{3}(\mathcal{O}_{l'})$ has been studied in \cite{bms67}
in connection with the congruence subgroup problem.
The level at which the Kubota symbol is defined tells us the compact open subgroup
on which the cocycle splits.
This level is established in Theorem~4.1 of \cite{bms67}, which may be paraphrased as follows:

\begin{prop} \label{p:bms}
Let $ m $ be the number of roots of unity in $\mathcal{O}_{l'}$ and let $\mu_{r}$
be a subgroup of $\mu_{m}$.
Then the Kubota symbol $\kappa:\SL_{3}(\mathcal{O}_{l'},\mathfrak{q}) \to \mu_{r}$
is defined at level $\mathfrak{q}$ as long as for each prime $p$ dividing $r$ we have
	\[ \ord_{p}( r ) \le \min_{\mathfrak{p} \mid \mathfrak{q}} \left[ \frac{\ord_{\mathfrak{p}}( \mathfrak{q} )}{\ord_{\mathfrak{p}}( p )} - \frac{1}{p - 1} \right].
	 \]
\end{prop}

In particular, in the case $r=2$ we may take $\mathfrak{q}=4$.
This shows that our cocycle splits at level $4$, i.e. our cocycle splits on
\begin{equation} \label{e:geven}
	G( \Ok, 4 ) = \{ \nu \in G( \Ok ) \colon \nu \equiv I_{3} \pod{4} \}.
\end{equation}

If we have a Kubota symbol $ \kappa \colon G( \Ok, 4 ) \to \mu_{2} $, then $ \sigma|_{G( \Ok, 4 )} = \partial \kappa $.  But we should note that this Kubota symbol is not unique on $ G( \Ok, 4 ) $; it is only unique on
	\[ G( \Ok, 4 \mathfrak{p} ) = \{ \nu \in G( \Ok ) \colon \nu \equiv I_{3} \pod{4 \mathfrak{p}} \}. \]
This is because $ G( \Ok, 4 ) / G( \Ok, 4 \mathfrak{p} ) \cong \mathfrak{g}( \Ok / \mathfrak{p} ) $, and
	\[ \homo( G( \Ok, 4 ) / G( \Ok, 4 \mathfrak{p} ), \mu_{2} ) \neq 0, \]
hence
	\[ \homo( G( \Ok, 4 ), \mu_{2} ) \neq 0, \]
and any $ \chi \in \homo( G( \Ok, 4 ), \mu_{2} ) $ would make $ \kappa \chi $ another choice for the Kubota symbol on $ G( \Ok, 4 ) $.
We will only calculate our local Kubota symbol on $ G( \Ok, 4 \mathfrak{p} ) $ in the non-split case.

\chapter{Calculation of the Kubota symbol} \label{c:general}

We will outline the method of calculating this local Kubota symbol in this chapter.  Let $ L / l $ be a global quadratic extension.  For consistency of notation, we will be using the same notation as in Part~\ref{p:su2coc}, i.e. we have $ k $ a local field with $ K $ a quadratic extension of $ k $.  Then $ k = \lp $ and $ K = \Lp $ (with notation as in Section~\ref{s:adele}).  Also, note that in the split case we will use \eqref{e:splmap} to identify an element of $ \lp( \thet ) $ with an element of $ \Lp $.

We have $ \mathfrak{p} = \mathfrak{p}_{k} $ as the maximal ideal of \Ok.  Let the local Kubota symbol be denoted by $ \kappa_{\mathfrak{p}} $ and using the notation from the previous chapter, let the subgroup of $ G( \Ok ) $ on which the quadratic 2-cocycle splits be called $ \hat{\Gamma}_{\mathfrak{p}} $.  This implies that $ \kappa_{\mathfrak{p}} $ is a map
	\[ \kappa_{\mathfrak{p}} \colon \hat{\Gamma}_{\mathfrak{p}} \to \mu_{2}, \]
where $ \mu_{2} = \{ 1, - 1 \} $, and for any $ g $, $ h \in \hat{\Gamma}_{\mathfrak{p}} $,
	\[ \sig{g}{h} = \frac{\kup{g} \kup{h}}{\kup{g h}}. \]
This implies that
\begin{equation} \label{e:kupsig}
	\kup{g h} = \kup{g} \kup{h} \sig{g}{h}.
\end{equation}
Thus, for any $ g \in \hat{\Gamma}_{\mathfrak{p}} $, since $ \kup{g}^{2} = 1 $, we have
\begin{equation} \label{e:kupsq}
	\kup{g^{2}} = \sig{g}{g}.
\end{equation}

As noted in Theorem~\ref{t:compact}, $ \hat{\Gamma}_{\mathfrak{p}} $ depends on $ K / k $.  Also, in the case where $ k $ is of even residue characteristic, we will assume that the 2-cocycle $ \sigma $ splits on the compact open subgroup $ \hat{\Gamma}_{\mathfrak{p}} = G( \Ok, 4 ) $ for every extension $ K / k $.  We have already noted that in Section~\ref{s:even} that the Kubota symbol on $ \hat{\Gamma}_{\mathfrak{p}} = G( \Ok, 4 ) $ is not unique, but it is unique on $ G( \Ok, 8 ) $.  What we do observe later is that the unipotent elements of $ G( \Ok, 8 ) $ are squares of unipotent elements of $ G( \Ok, 4 ) $, and it will be shown that the elements of $ T( k ) \cap G( \Ok, 8 ) $ are squares of elements of $ T( k ) \cap G( \Ok, 4 ) $.  This implies that we should use \eqref{e:kupsq} on the elements of $ G( \Ok, 4 ) $ in order to find the unique Kubota symbol on $ G( \Ok, 8 ) $.

In addition, we have already stated in the Introduction that in the case where $ K = k \oplus k $, Theorem~\ref{t:sigma} does not completely describe $ \sigma $.  Even so, since we are only interested in elements of $ G( l ) $ contained in the group $ G( k ) $, our formula for $ \sigma $ is sufficient to calculate the local Kubota symbol in the split case.

Hence, let us define $ \Gamma_{\mathfrak{p}} $ to be the subgroup of $ G( \Ok ) $ which we will be choosing to calculate the unique Kubota symbol on, i.e.
	\[ \Gamma_{\mathfrak{p}}
	=	\begin{cases}
			\displaystyle G( \Ok ), &\text{if $ \mathfrak{p} $ is odd and unramified in $ K $;} \\
			\displaystyle G( \Ok, \thet ), &\text{if $ \mathfrak{p} $ is odd and ramified in $ K $;} \\
			\displaystyle G( \Ok ) \cap G( l ), &\text{if $ \mathfrak{p} $ is odd and split in $ K $;} \\
			\displaystyle G( \Ok, 8 ), &\text{if $ \mathfrak{p} $ is even and not split in $ K $;} \\
			\displaystyle G( \Ok, 8 ) \cap G( l ), &\text{if $ \mathfrak{p} $ is even and split in $ K $.}
		\end{cases} \]

\section{The unipotent matrices of the compact open subgroup} \label{s:uni}

We now calculate the Kubota symbol on the elements of $ N( k ) \cap \Gamma_{\mathfrak{p}} $ and $ \ol{N}( k ) \cap \Gamma_{\mathfrak{p}} $.

\begin{prop} \label{p:kupxa1}
Let $ \xa{s_{1}}{n_{1}} \in \Gamma_{\mathfrak{p}} $.  Then
	\[ \kup{\xa{s_{1}}{n_{1}}} = 1. \]
More generally, for any $ g \in \Gamma_{\mathfrak{p}} $ we have
	\[ \kup{\xa{s_{1}}{n_{1}} \cdot g} = \kup{g \cdot \xa{s_{1}}{n_{1}}} = \kup{g}. \]
\end{prop}

\begin{proof}
Since $ \spl{s_{1}}{n_{1}} \in A $ (see \eqref{e:defA2}) where
	\[ A = \{ ( z, - \norm{z} / 2 + t \thet ) \in K \times K \colon t \in k, ( z, - \norm{z} / 2 + t \thet ) \neq ( 0, 0 ) \}, \]
for ease of use, let $ s_{1} = z $ and $ n_{1} = - \norm{z} / 2 + t \thet $.  We first note that
	\[ \xa{z}{- \frac{\norm{z}}{2} + t \thet} = \xa{\frac{z}{2}}{- \frac{\norm{z}}{8} + \frac{t \thet}{2}}^{2}, \]
where $ \xa{z / 2}{- \norm{z} / 8 + t \thet / 2} \in \hat{\Gamma}_{\mathfrak{p}} $, i.e. $ \kup{\xa{z / 2}{- \norm{z} / 8 + t \thet / 2}} $ exists.

By \eqref{e:kupsq},
\begin{multline*}
	\kup{\xa{z}{- \frac{\norm{z}}{2} + t \thet}} \\
	= \sig{\xa{\frac{z}{2}}{- \frac{\norm{z}}{8} + \frac{t \thet}{2}}}{\xa{\frac{z}{2}}{- \frac{\norm{z}}{8} + \frac{t \thet}{2}}}.
\end{multline*}
But by Theorem~\ref{t:sigma}, we have
	\[ \kup{\xa{z}{- \frac{\norm{z}}{2} + t \thet}} = 1. \]

Now let $ g \in \Gamma_{\mathfrak{p}} $.  By \eqref{e:kupsig},
	\[ \kup{\xa{s_{1}}{n_{1}} \cdot g} = \kup{\xa{s_{1}}{n_{1}}} \cdot \kup{g} \cdot \sig{\xa{s_{1}}{n_{1}}}{g}. \]
Hence using the above and Theorem~\ref{t:sigma}, we have
	\[ \kup{\xa{s_{1}}{n_{1}} \cdot g} = \kup{g}. \]
We can similarly get $ \kup{g \cdot \xa{s_{1}}{n_{1}}} = \kup{g} $.
\end{proof}

\begin{prop} \label{p:kupxa}
Let $ \xam{s_{1}}{n_{1}} \in \Gamma_{\mathfrak{p}} $.  Then
	\[ \kup{\xam{s_{1}}{n_{1}}} = \kxam{s_{1}} \cdot \kxam{- \frac{s_{1} \thet}{n_{1}}}, \]
where
	\[ \kxam{s_{1}}
	=	\begin{cases}
			\displaystyle \hilk{- \trace{s_{1}}}{\norm{s_{1} \thet}}, &\text{if $ \trace{s_{1}} \neq 0 $;} \\
			\displaystyle 1, &\text{otherwise.}
		\end{cases} \]

\end{prop}

\begin{proof}
Again, since $ \spl{s_{1}}{n_{1}} \in A $ (see \eqref{e:defA2}), for ease of use, let $ s_{1} = z $ and $ n_{1} = - \norm{z} / 2 + t \thet $.  We have
	\[ \xam{z}{- \frac{\norm{z}}{2} + t \thet} = \xam{\frac{z}{2}}{- \frac{\norm{z}}{8} + \frac{t \thet}{2}}^{2}, \]
where $ \xam{z / 2}{- \norm{z} / 8 + t \thet / 2} \in \hat{\Gamma}_{\mathfrak{p}} $, i.e. $ \kup{\xam{z / 2}{- \norm{z} / 8 + t \thet / 2}} $ exists.

We have three cases to consider: $ z = 0 $, $ t \neq 0 $; $ z \neq 0 $, $ t = 0 $; and $ z $, $ t \neq 0 $.

We first note that $ z \in \OK $ and $ t \in \Ok $.  If $ z = 0 $ and $ t \neq 0 $, then by \eqref{e:kupsq}, we have
	\[ \kup{\xam{0}{t \thet}} = \sig{\xam{0}{\frac{t \thet}{2}}}{\xam{0}{\frac{t \thet}{2}}}. \]
By Theorem~\ref{t:sigma}, since
\begin{gather*}
	X( \xam{0}{t \thet}) = - \frac{1}{t \thet^{2}}, \quad X\left( \xam{0}{\frac{t \thet}{2}} \right) = - \frac{2}{t \thet^{2}}, \\
	\frac{X( \xam{0}{t \thet})}{X( \xam{0}{t \thet / 2})^{2}} = - \frac{t \thet^{2}}{4} \in k^{\times},
\end{gather*}
we have
\begin{align*}
	\kup{\xam{0}{t \thet}}
	&= \ug{\frac{X( \xam{0}{t \thet})}{X( \xam{0}{t \thet / 2})}}{X\left( \xam{0}{\frac{t \thet}{2}} \right)^{2}} \\
	&= \hilk{\frac{1}{2}}{- \left( - \frac{2}{t \thet^{2}} \right)^{2}};
\end{align*}
and using the properties of Hilbert symbols (Section~\ref{s:hilbert}), by \eqref{e:bastri} and \eqref{e:prop5ii},
	\[ \kup{\xam{0}{t \thet}} = \hilk{\frac{1}{2}}{- 1} \cdot \hilk{\frac{1}{2}}{\left( - \frac{2}{t \thet^{2}} \right)^{2}} = \hilk{\frac{1}{2}}{- 1}. \]
Hence by \eqref{e:prop4i},
\begin{equation} \label{e:xamt}
	\kup{\xam{0}{t \thet}} = \hilk{\frac{1 - ( - 1 )}{2}}{- 1} = 1.
\end{equation}

If $ z \neq 0 $, $ t = 0 $, then by \eqref{e:kupsq},
	\[ \kup{\xam{z}{- \frac{\norm{z}}{2}}} = \sig{\xam{\frac{z}{2}}{- \frac{\norm{z}}{8}}}{\xam{\frac{z}{2}}{- \frac{\norm{z}}{8}}}. \]
Thus by Theorem~\ref{t:sigma}, we have
\begin{gather*}
	X\left( \xam{z}{- \frac{\norm{z}}{2}} \right) = - \frac{2}{\norm{z} \thet}, \quad X\left( \xam{\frac{z}{2}}{- \frac{\norm{z}}{8}} \right) = - \frac{8}{\norm{z} \thet}, \\
	\frac{X( \xam{z}{- \norm{z} / 2} )}{X( \xam{z / 2}{- \norm{z} / 8} )} = \frac{1}{4}, \\
	\frac{X( \xam{z}{- \norm{z} / 2} )}{X( \xam{z / 2}{- \norm{z} / 8} )^{2}} = - \frac{\norm{z} \thet}{32} \notin k^{\times}.
\end{gather*}
Hence,
\begin{multline*}
	r\left( \xam{\frac{z}{2}}{- \frac{\norm{z}}{8}}, \xam{\frac{z}{2}}{- \frac{\norm{z}}{8}} \right) \\
	= \frac{( - \ol{z} / 2 ) ( - \norm{z} / 2 ) - ( - \ol{z} ) ( - \norm{z} / 8 )}{( - \norm{z} / 8 )^{2}} = \frac{8}{z},
\end{multline*}
and therefore
\begin{align*}
	&\kup{\xam{z}{- \frac{\norm{z}}{2}}} \\
	&= \sig{\xam{\frac{z}{2}}{- \frac{\norm{z}}{8}}}{\xam{\frac{z}{2}}{- \frac{\norm{z}}{8}}} \\
	&=	\begin{aligned}[t]
			&\hilk{- \del{2}{- \frac{8}{z} \left( - \frac{\norm{z} \thet}{32} \right)} \thet^{- 1}}{\norm{- \frac{8}{z} \left( - \frac{\norm{z} \thet}{32} \right)}} \\
			&\phantom{\ } \cdot \hilk{\norm{\frac{8}{z}}}{\frac{\del{2}{8 / z}}{\thet}} \cdot \ug{- \frac{8}{\norm{z} \thet}}{\frac{1}{4}} \cdot \ug{\frac{1}{4}}{- \frac{2}{\norm{z} \thet}} \\
			&\phantom{\ } \cdot \hilk{\frac{\del{2}{- \norm{z} \thet / 32}}{\del{2}{1 / 4}}}{- \frac{\norm{- \norm{z} \thet / 32} \del{2}{- \norm{z} \thet / 32}}{\del{2}{- 8 / ( \norm{z} \thet )}}} \\
			&\phantom{\ } \cdot \hilk{\frac{\del{2}{- 8 / ( \norm{z} \thet )}}{\del{2}{- 2 / ( \norm{z} \thet )}}}{- \frac{\norm{- 8 / ( \norm{z} \thet )} \del{2}{- 8 / ( \norm{z} \thet )}}{\del{2}{1 / 4}}}
		\end{aligned} \\
	&=	\begin{aligned}[t]
			&\hilk{- \frac{\del{2}{\ol{z} \thet / 4}}{\thet}}{- \frac{\norm{z} \thet^{2}}{16}} \cdot \hilk{\frac{64}{\norm{z}}}{\frac{\del{2}{8 / z}}{\thet}} \cdot \hilk{- \frac{64}{\norm{z}^{2} \thet^{2}}}{- \frac{1}{16}} \\
			&\phantom{\ } \cdot \hilk{\frac{1}{16}}{- \frac{4}{\norm{z}^{2} \thet^{2}}} \cdot \hilk{\frac{16}{\norm{z} \thet^{2}}}{\frac{1}{4}} \cdot \hilk{\frac{1}{4}}{- \frac{4}{\norm{z} \thet^{2}}}.
		\end{aligned}
\end{align*}
By the properties of Hilbert symbols, we have by \eqref{e:bastri},
\begin{multline*}
	\kup{\xam{z}{- \frac{\norm{z}}{2}}} \\
	=	\begin{aligned}[t]
			&\hilk{\frac{\del{2}{\ol{z} \thet / 4}}{\thet}}{- \norm{z} \thet^{2}} \cdot \hilk{- 1}{- \norm{z} \thet^{2}} \cdot \hilk{- \frac{\del{2}{\ol{z} \thet / 4}}{\thet}}{\frac{1}{16}} \\
			&\phantom{\ } \cdot \hilk{\frac{1}{\norm{z}}}{\frac{\del{2}{8 / z}}{\thet}} \cdot \hilk{64}{\frac{\del{2}{8 / z}}{\thet}} \cdot \hilk{- \frac{1}{\thet^{2}}}{- 1} \\
			&\phantom{\ } \cdot \hilk{\frac{64}{\norm{z}^{2}}}{- 1} \cdot \hilk{- \frac{64}{\norm{z}^{2} \thet^{2}}}{\frac{1}{16}} \cdot \hilk{\frac{1}{16}}{- \frac{4}{\norm{z}^{2} \thet^{2}}} \\
			&\phantom{\ } \cdot \hilk{\frac{16}{\norm{z} \thet^{2}}}{\frac{1}{4}} \cdot \hilk{\frac{1}{4}}{- \frac{4}{\norm{z} \thet^{2}}}.
		\end{aligned}
\end{multline*}
By \eqref{e:prop5ii},
	\[ \kup{\xam{z}{- \frac{\norm{z}}{2}}}
	=	\begin{aligned}[t]
			&\hilk{\frac{\del{2}{\ol{z} \thet / 4}}{\thet}}{- \norm{z} \thet^{2}} \cdot \hilk{- 1}{- \norm{z} \thet^{2}} \\
			&\phantom{\ } \cdot \hilk{\frac{1}{\norm{z}}}{\frac{\del{2}{8 / z}}{\thet}} \cdot \hilk{- \frac{1}{\thet^{2}}}{- 1},
		\end{aligned} \]
and by \eqref{e:prop2i},
	\[ \kup{\xam{z}{- \frac{\norm{z}}{2}}}
	=	\begin{aligned}[t]
			&\hilk{\frac{\del{2}{\ol{z} \thet / 4}}{\thet}}{- \norm{z} \thet^{2}} \cdot \hilk{- 1}{- \norm{z} \thet^{2}} \\
			&\phantom{\ } \cdot \hilk{\frac{\del{2}{8 / z}}{\thet}}{\norm{z}} \cdot \hilk{- 1}{- \frac{1}{\thet^{2}}}.
		\end{aligned} \]
This implies that by \eqref{e:bastri},
\begin{align*}
	&\kup{\xam{z}{- \frac{\norm{z}}{2}}} \\
	&= \hilk{\frac{\del{2}{\ol{z} \thet / 4}}{\thet}}{- \norm{z} \thet^{2}} \cdot \hilk{- 1}{\norm{z}} \cdot \hilk{\frac{\del{2}{8 / z}}{\thet}}{\norm{z}} \\
	&= \hilk{\frac{\del{2}{\ol{z} \thet / 4}}{\thet}}{- \norm{z} \thet^{2}} \cdot \hilk{- \frac{\del{2}{8 / z}}{\thet}}{\norm{z}}.
\end{align*}

If $ a $, $ b $, $ c \in k^{\times} $, $ d \in K^{\times} $, we have by \eqref{e:bastri}, \eqref{e:prop5ii} and \eqref{e:del2q} that
	\[ \hilk{\frac{b \del{2}{d / c^{2}}}{\thet}}{a} = \hilk{\frac{b \del{2}{d / c^{2}}}{c^{2} \thet}}{a} = \begin{cases} \hilk{b}{a}, &\text{if $ d \in k^{\times} $;} \\ \displaystyle \hilk{\frac{b \del{2}{d}}{\thet}}{a}, &\text{if $ d \notin k^{\times} $.} \end{cases} \]
But since $ \del{2}{d} = \thet $ when $ d \in k^{\times} $, the above implies that
\begin{equation} \label{e:delsq}
	\hilk{\frac{b \del{2}{d / c^{2}}}{\thet}}{a} = \hilk{\frac{b \del{2}{d}}{\thet}}{a},
\end{equation}
for any $ a $, $ b $, $ c \in k^{\times} $ and $ d \in K^{\times} $.  We should also note that if instead we have $ a = \norm{e} $ for some $ e \in K^{\times} $, and $ c = f \thet $ for some $ f \in k^{\times} $, with $ b $, $ d $ remaining the same, we have by \eqref{e:delsq} and \eqref{e:del2q},
\begin{align*}
	\hilk{\frac{b \del{2}{d / c^{2}}}{\thet}}{a}
	&= \hilk{\frac{b \del{2}{d / ( f \thet )^{2}}}{\thet}}{\norm{e}} \\
	&= \begin{cases} \displaystyle \hilk{b}{\norm{e}}, &\text{if $ d \in k^{\times} $;} \\ \displaystyle \hilk{b \del{2}{d} \thet}{\norm{e}}, &\text{if $ d \notin k^{\times} $.} \end{cases}
\end{align*}
But since by \eqref{e:norm} and \eqref{e:prop5ii},
	\[ \hilk{\thet^{- 2}}{\norm{e}} = \hilbk{\thet^{- 2}}{e} = 1, \]
this implies, for $ d \notin k^{\times} $, that by \eqref{e:bastri},
	\[ \hilk{b \del{2}{d} \thet}{\norm{e}} = \hilk{b \del{2}{d} \thet}{\norm{e}} \cdot \hilk{\thet^{- 2}}{\norm{e}} = \hilk{\frac{b \del{2}{d}}{\thet}}{\norm{e}}. \]
Thus we have
\begin{equation} \label{e:delsq2}
	\hilk{\frac{b \del{2}{d / ( f \thet )^{2}}}{\thet}}{\norm{e}} = \hilk{\frac{b \del{2}{d}}{\thet}}{\norm{e}}
\end{equation}
for any $ d $, $ e \in K^{\times} $, $ b $, $ f \in k^{\times} $.

Using \eqref{e:delsq} in our equation for $ \kup{\xam{z}{- \norm{z} / 2}} $, we have
\begin{align}
	\kup{\xam{z}{- \frac{\norm{z}}{2}}}
	&= \hilk{\frac{\del{2}{\ol{z} \thet / 4}}{\thet}}{- \norm{z} \thet^{2}} \cdot \hilk{- \frac{\del{2}{8 / z}}{\thet}}{\norm{z}} \label{e:xamz} \\
	&= \hilk{\frac{\del{2}{\ol{z} \thet}}{\thet}}{\norm{\ol{z} \thet}} \cdot \hilk{- \frac{\del{2}{2 / z}}{\thet}}{\norm{z}}. \notag
\end{align}

We now consider the case $ z $, $ t \neq 0 $.  Since
	\[ \xam{z}{- \frac{\norm{z}}{2} + t \thet} = \xam{z}{- \frac{\norm{z}}{2}} \cdot \xam{0}{t \thet}, \]
and $ \xam{z}{- \norm{z} / 2} $, $ \xam{0}{t \thet} \in \Gamma_{\mathfrak{p}} $, we have, by \eqref{e:kupsig},
\begin{multline*}
	\kup{\xam{z}{- \frac{\norm{z}}{2} + t \thet}} \\
	= \kup{\xam{z}{- \frac{\norm{z}}{2}}} \cdot \kup{\xam{0}{t \thet}} \cdot \sig{\xam{z}{- \frac{\norm{z}}{2}}}{\xam{0}{t \thet}}.
\end{multline*}
Consider $ \sig{\xam{z}{- \norm{z} / 2}}{\xam{0}{t \thet}} $.  By Theorem~\ref{t:sigma},
\begin{gather*}
	X\left( \xam{z}{- \frac{\norm{z}}{2} + t \thet} \right) = - \frac{2}{( \norm{z} + 2 t \thet ) \thet}, \\
	X\left( \xam{z}{- \frac{\norm{z}}{2}} \right) = - \frac{2}{\norm{z} \thet},\quad X\left( \xam{0}{t \thet} \right) = - \frac{1}{t \thet^{2}}, \\
	\frac{X( \xam{z}{- \norm{z} / 2 + t \thet} )}{X( \xam{z}{- \norm{z} / 2} ) \cdot X ( \xam{0}{t \thet} )} = - \frac{\norm{z} t \thet^{2}}{\norm{z} + 2 t \thet} \notin k^{\times}.
\end{gather*}
This implies that we have
	\[ r\left( \xam{z}{- \frac{\norm{z}}{2}}, \xam{0}{t \thet} \right) = \frac{0 \cdot ( - \norm{z} / 2 + t \thet ) - ( - \ol{z} ) \cdot t \thet}{( - \norm{z} / 2 ) ( - t \thet )} = \frac{2}{z}, \]
so that
\begin{align*}
	&\sig{\xam{z}{- \frac{\norm{z}}{2}}}{\xam{0}{t \thet}} \\
	&=	\begin{aligned}[t]
			&\hilk{- \del{2}{- \frac{2}{z} \cdot \left( - \frac{\norm{z} t \thet^{2}}{\norm{z} + 2 t \thet} \right)} \thet^{- 1}}{\norm{- \frac{2}{z} \cdot \left( - \frac{\norm{z} t \thet^{2}}{\norm{z} + 2 t \thet} \right)}} \\
			&\phantom{\ } \cdot \hilk{\norm{\frac{2}{z}}}{\frac{\del{2}{2 / z}}{\thet}} \cdot \ug{- \frac{2}{\norm{z} \thet}}{\frac{2 t \thet}{\norm{z} + 2 \thet}} \\
			&\phantom{\ } \cdot \ug{\frac{2 t \thet}{\norm{z} + 2 \thet}}{- \frac{2}{( \norm{z} + 2 t \thet ) \thet}} \\
			&\phantom{\ } \cdot \Bigg( \frac{\del{2}{- \norm{z} t \thet^{2} / ( \norm{z} + 2 t \thet )}}{\del{2}{2 t \thet / ( \norm{z} + 2 t \thet )}}, \\
			&\phantom{\qquad\qquad} - \frac{\norm{- \norm{z} t \thet^{2} / ( \norm{z} + 2 t \thet )} \del{2}{- \norm{z} t \thet^{2} / ( \norm{z} + 2 t \thet )}}{\del{2}{- 2 / ( \norm{z} \thet )}} \Bigg)_{k, 2} \\
			&\phantom{\ } \cdot \hilk{\frac{\del{2}{- 1 / ( t \thet^{2} )}}{\del{2}{- 2 / ( ( \norm{z} + 2 t \thet ) \thet )}}}{- \frac{\norm{- 1 / ( t \thet^{2} )} \del{2}{- 1 / ( t \thet^{2} )}}{\del{2}{2 t \thet / ( \norm{z} + 2 t \thet )}}}
		\end{aligned} \\
	&=	\begin{aligned}[t]
			&\hilk{- \del{2}{\frac{2 \ol{z} t \thet^{2}}{\norm{z} + 2 t \thet}} \thet^{- 1}}{\frac{4 \norm{z} t^{2} \thet^{4}}{\norm{\norm{z} + 2 t \thet}}} \cdot \hilk{\frac{4}{\norm{z}}}{\frac{\del{2}{2 / z}}{\thet}} \\
			&\phantom{\ } \cdot \hilk{- \frac{4}{\norm{z}^{2} \thet^{2}}}{\frac{4 t^{2} \thet^{2}}{\norm{\norm{z} + 2 t \thet}}} \\
			&\phantom{\ } \cdot \hilk{- \frac{4 t^{2} \thet^{2}}{\norm{\norm{z} + 2 t \thet}}}{\frac{4}{\norm{\norm{z} + 2 t \thet} \thet^{2}}} \cdot \hilk{\frac{1}{t \thet^{2}}}{1} \\
			&\phantom{\ } \cdot \hilk{\frac{\norm{z}}{\norm{\norm{z} + 2 t \thet}}}{\frac{4 \norm{z}}{\norm{\norm{z} + 2 t \thet} t \thet^{2}}}.
		\end{aligned}
\end{align*}
Similar to the calculation of $ \kup{\xam{z}{- \norm{z} / 2}} $, we can simplify the above using the properties of Hilbert symbols.  After simplifying, we get
\begin{multline*}
	\sig{\xam{z}{- \frac{\norm{z}}{2}}}{\xam{0}{t \thet}} \\
	= \hilk{- \del{2}{\frac{2 \ol{z} t \thet^{2}}{\norm{z} + 2 t \thet}} \frac{t}{\thet}}{\frac{\norm{z}}{\norm{\norm{z} + 2 t \thet}}} \cdot \hilk{- \frac{\del{2}{2 / z}}{\thet}}{\norm{z}}.
\end{multline*}
Inserting this result into our equation for $ \kup{\xam{z}{- \norm{z} / 2 + t \thet}} $ along with \eqref{e:xamt} and \eqref{e:xamz}, we have
\begin{multline*}
	\kup{\xam{z}{- \frac{\norm{z}}{2} + t \thet}} \\
	=	\begin{aligned}[t]
			&\left[ \hilk{\frac{\del{2}{\ol{z} \thet}}{\thet}}{\norm{\ol{z} \thet}} \cdot \hilk{- \frac{\del{2}{2 / z}}{\thet}}{\norm{z}} \right] \cdot 1 \\
			&\phantom{\ } \cdot \left[ \hilk{- \del{2}{\frac{2 \ol{z} t \thet^{2}}{\norm{z} + 2 t \thet}} \frac{t}{\thet}}{\frac{\norm{z}}{\norm{\norm{z} + 2 t \thet}}} \cdot \hilk{- \frac{\del{2}{2 / z}}{\thet}}{\norm{z}} \right].
		\end{aligned}
\end{multline*}
Thus simplifying the above using \eqref{e:prop5ii}, we have
\begin{multline} \label{e:xamzt}
	\kup{\xam{z}{- \frac{\norm{z}}{2} + t \thet}} \\
	= \hilk{\frac{\del{2}{\ol{z} \thet}}{\thet}}{\norm{\ol{z} \thet}} \cdot \hilk{- \del{2}{\frac{2 \ol{z} t \thet^{2}}{\norm{z} + 2 t \thet}} \frac{t}{\thet}}{\frac{\norm{z}}{\norm{\norm{z} + 2 t \thet}}}.
\end{multline}

Recall that we had put $ s_{1} = z $, $ n_{1} = - \norm{z} / 2 + t \thet $.  
If $ t \neq 0 $, i.e. $ n_{1} \neq - \norm{s_{1}} / 2 $, by \eqref{e:xamzt} we have
	\[ \kup{\xam{s_{1}}{n_{1}}} = \hilk{\frac{\del{2}{\ol{s_{1}} \thet}}{\thet}}{\norm{\ol{s_{1}} \thet}} \cdot \hilk{- \del{2}{\frac{- \ol{s_{1}} t \thet^{2}}{\ol{n_{1}}}} \frac{t}{\thet}}{\frac{\norm{s_{1}}}{4 \norm{n_{1}}}}. \]
But by \eqref{e:delsq2} and the properties of Hilbert symbols, the above becomes
	\[ \kup{\xam{s_{1}}{n_{1}}} = \hilk{\frac{\del{2}{\ol{s_{1}} \thet}}{\thet}}{\norm{\ol{s_{1}} \thet}} \cdot \hilk{- \del{2}{\frac{- \ol{s_{1}} t}{\ol{n_{1}}}} \frac{t}{\thet}}{\frac{\norm{s_{1}}}{\norm{n_{1}}}}. \]
Let $ \spl{s'}{n'} \in A $ (see \eqref{e:defA2}) such that $ s' \neq 0 $.  Also, let $ t' \in k^{\times} $ and
	\[ F( s', n', t' ) = \hilk{\frac{\del{2}{\ol{s'} \thet}}{\thet}}{\norm{\ol{s'} \thet}} \cdot \hilk{- \del{2}{\frac{- \ol{s'} t'}{\ol{n'}}} \frac{t'}{\thet}}{\frac{\norm{s'}}{\norm{n'}}}. \]
This implies that
	\[ \kup{\xam{z}{- \frac{\norm{z}}{2} + t \thet}} = F\left( z, - \frac{\norm{z}}{2} + t \thet, t \right), \]
and in fact, it can also be checked that
	\[ \kup{\xam{z}{- \frac{\norm{z}}{2}}} = F\left( z, - \frac{\norm{z}}{2}, 1 \right). \]

We will now look at the function $ F $ more closely.  Let $ \spl{s_{1}}{n_{1}} \in A $, where $ s_{1} \neq 0 $, and $ \xam{s_{1}}{n_{1}} \in \Gamma_{\mathfrak{p}} $.  We have
\begin{align*}
	\frac{\del{2}{\ol{s_{1}} \thet}}{\thet}
	&=	\begin{cases}
			\displaystyle 1, &\text{if $ \ol{s_{1}} \thet \in k^{\times} $;} \\
			\displaystyle ( - s_{1} \thet - \ol{s_{1}} \thet )^{- 1} / \thet, &\text{if $ \ol{s_{1}} \thet \notin k^{\times} $}
		\end{cases} \\
	&=	\begin{cases}
			\displaystyle 1, &\text{if $ \trace{s_{1}} = 0 $;} \\
			\displaystyle - \left[ \trace{s_{1}} \thet^{2} \right]^{- 1}, &\text{if $ \trace{s_{1}} \neq 0 $.}
		\end{cases}
\end{align*}
By using \eqref{e:prop2i} and \eqref{e:norm}, this implies that
	\[ \hilk{\frac{\del{2}{\ol{s_{1}} \thet}}{\thet}}{\norm{\ol{s_{1}} \thet}}
	=	\begin{cases}
			\displaystyle 1, &\text{if $ \trace{s_{1}} = 0 $;} \\
			\displaystyle \hilk{- \trace{s_{1}}}{\norm{s_{1} \thet}}, &\text{if $ \trace{s_{1}} \neq 0 $.}
		\end{cases} \]
Let $ t \in k^{\times} $.  Then
\begin{align*}
	- \del{2}{- \frac{\ol{s_{1}} t}{\ol{n}}} \frac{t}{\thet}
	&=	\begin{cases}
			\displaystyle - t, &\text{if $ \ol{s_{1}} t / \ol{n_{1}} \in k^{\times} $;} \\
			\displaystyle - \left[ - \frac{s_{1} t}{n_{1}} + \frac{\ol{s_{1}} t}{\ol{n}} \right]^{- 1} \frac{t}{\thet}, &\text{if $ \ol{s_{1}} t / \ol{n_{1}} \notin k^{\times} $}
		\end{cases} \\
	&=	\begin{cases}
			\displaystyle - t, &\text{if $ \trace{- s_{1} \thet / n_{1}} = 0 $;} \\
			\displaystyle - \left[ \trace{\frac{- s_{1} \thet}{n_{1}}} \right]^{- 1}, &\text{if $ \trace{- s_{1} \thet / n_{1}} \neq 0 $.}
		\end{cases}
\end{align*}
This implies that
\begin{align*}
	&\hilk{- \del{2}{\frac{- \ol{s_{1}} t}{\ol{n_{1}}}} \frac{t}{\thet}}{\frac{\norm{s_{1}}}{\norm{n_{1}}}} \\
	&=	\begin{cases}
			\displaystyle \hilk{- t}{\left( \frac{s_{1}}{n_{1}} \right)^{2}}, &\text{if $ \trace{- s_{1} \thet / n_{1}} = 0 $;} \\
			\displaystyle \hilk{- \left[ \trace{- \frac{s_{1} \thet}{n_{1}}} \right]^{- 1}}{\norm{\frac{s_{1}}{n_{1}}}}, &\text{if $ \trace{- s_{1} \thet / n_{1}} \neq 0 $.}
		\end{cases}
\end{align*}
Thus by \eqref{e:bastri}, \eqref{e:prop2i} and \eqref{e:prop5ii}, we have
\begin{align*}
	&\hilk{- \del{2}{\frac{- \ol{s_{1}} t}{\ol{n_{1}}}} \frac{t}{\thet}}{\frac{\norm{s_{1}}}{\norm{n_{1}}}} \\
	&=	\begin{cases}
			\displaystyle 1, &\text{if $ \trace{- s_{1} \thet / n_{1}} = 0 $;} \\
			\displaystyle \hilk{- \trace{- \frac{s_{1} \thet}{n_{1}}}}{\norm{- \frac{s_{1} \thet^{2}}{n_{1}}}}, &\text{if $ \trace{- s_{1} \thet / n_{1}} \neq 0 $.}
		\end{cases}
\end{align*}
Hence, we have
	\[ F( s_{1}, n_{1}, t ) = \kxam{s_{1}} \cdot \kxam{- \frac{s_{1} \thet}{n_{1}}}, \]
where
	\[ \kxam{s_{1}}
	=	\begin{cases}
			\displaystyle \hilk{- \trace{s_{1}}}{\norm{s_{1} \thet}}, &\text{if $ \trace{s_{1}} \neq 0 $;} \\
			\displaystyle 1, &\text{otherwise.}
		\end{cases} \]
Note that $ \rho $ and hence $ F $ is independent of $ t $; therefore, since
\begin{gather*}
	\kup{\xam{z}{- \frac{\norm{z}}{2}}} = F\left( z, - \frac{\norm{z}}{2}, 1 \right), \\
	\kup{\xam{z}{- \frac{\norm{z}}{2} + t \thet}} = F\left( z, - \frac{\norm{z}}{2} + t \thet, t \right)
\end{gather*}
for $ z \neq 0 $, $ t \in k^{\times} $, we may combine the above result so that for $ s_{1} \neq 0 $,
	\[ \kup{\xam{s_{1}}{n_{1}}} = \kxam{s_{1}} \cdot \kxam{- \frac{s_{1} \thet}{n}}. \]

We also note that when $ s_{1} = 0 $, the above formula for $ s_{1} \neq 0 $ is still applicable.  Thus our result is proved.
\end{proof} 

\begin{rem} \label{r:kupxa}
Note that we also have
	\[ \xam{z}{- \frac{\norm{z}}{2} + t \thet} = \xam{0}{t \thet} \cdot \xam{z}{- \frac{\norm{z}}{2}}, \]
since $ \xam{0}{t \thet} $ is a central element of $ \ol{N}( k ) $.  We can use Theorem~\ref{t:sigma} and the properties of Hilbert symbols to show that
	\[ \frac{\sig{\xam{z}{- \norm{z} / 2}}{\xam{0}{t \thet}}}{\sig{\xam{0}{t \thet}}{\xam{z}{- \norm{z} / 2}}} = 1. \]
This implies that we will ultimately get the same result if we had chosen to calculate $ \kup{\xam{z}{- \norm{z} / 2 + t \thet}} $ by using \eqref{e:kupsig}, so that
\begin{multline*}
	\kup{\xam{z}{- \frac{\norm{z}}{2} + t \thet}} \\
	= \kup{\xam{0}{t \thet}} \cdot \kup{\xam{z}{- \frac{\norm{z}}{2}}} \cdot \sig{\xam{0}{t \thet}}{\xam{z}{- \frac{\norm{z}}{2}}}.
\end{multline*}
\end{rem}

\section{The elements of the torus} \label{s:torus}

In this section, we want to find the local Kubota symbol on the elements of $ T( k ) \cap \Gamma_{\mathfrak{p}} $, which we will denote by $ \hat{T} $.

We have two cases to consider, depending on whether there are matrices in
$ \Gamma_{\mathfrak{p}} $ whose $ ( 3, 1 ) $-entry is a unit.

We know that when $ \mathfrak{p} $ is odd and ramified in $ K $, or if $ \mathfrak{p} $ is even in $ K $, then the $ ( 3, 1 ) $-entry is never a unit in $ \OK $.  This implies that the Bruhat decomposition (see Section~\ref{s:bruh}) of an element of $ \Gamma_{\mathfrak{p}} $ when the $ ( 3, 1 ) $-entry is non-zero in these cases will be a product of elements which are not in $ \Gamma_{\mathfrak{p}} $ (see \eqref{e:bruc}).

When $ \mathfrak{p} $ is odd and either unramified or split in $ K $, this does not pose a problem, and we can use the Bruhat decomposition in these cases.  We also note that in $ G( k ) $, we have from Section~\ref{s:suprop} that for any $ \lambda \in K^{\times} $,
\begin{equation} \label{e:halam}
	\ha{\lambda} = \wa{y( \lambda )}{\del{1}{\lambda}} \cdot \wa{0}{\del{2}{\lambda}}^{- 1},
\end{equation}
where
	\[ y( \lambda ) =	\begin{cases}
						0, &\text{if $ \lambda \in k^{\times} $;} \\
						1, &\text{otherwise.}
					\end{cases} \]
As $ \mathfrak{p} $ is odd and either unramified or split in $ K $, this implies that $ 2 $, $ \thet \in \OK^{\times} $, hence there exists $ \lambda \in K^{\times} $ such that $ \ha{\lambda} \in \hat{T} $ where $ \del{2}{\lambda} \in \OK^{\times} $, since $ \del{2}{\lambda} \in k^{\times} \thet $.  Therefore $ \wa{y( \lambda )}{\del{1}{\lambda}} $ and $ \wa{0}{\del{2}{\lambda}} \in \Gamma_{\mathfrak{p}} $.

\begin{prop} \label{p:kupwa}
If $ \mathfrak{p} $ is odd and either unramified or split in $ K $, then we have $ \wa{0}{b \thet} \in \Gamma_{\mathfrak{p}} $, where $ b \in \Ok^{\times} $, with
	\[ \kup{\wa{0}{b \thet}} = 1. \]
Also, if $ a \in \Ok $ such that $ - 1 / 2 + a \thet \in \OK^{\times} $, then $ \wa{1}{- 1 / 2 + a \thet} \in \Gamma_{\mathfrak{p}} $ and
	\[ \kup{\wa{1}{- 1 / 2 + a \thet}} = 1. \]
\end{prop}

\begin{proof}
For any $ \spl{r}{m} \in A $, we have, by \eqref{e:warm} that
	\[ \wa{r}{m} = \xa{r}{m} \cdot \xam{\frac{\ol{r}}{\ol{m}}}{\frac{1}{\ol{m}}} \cdot \xa{r \cdot \frac{\ol{m}}{m}}{m}. \]
This implies that if $ \xa{r}{m} $, $ \xam{\ol{r} / \ol{m}}{1 / \ol{m}} $ and $ \xa{r \ol{m} / m}{m} \in \Gamma_{\mathfrak{p}} $, then we have $ \wa{r}{m} \in \Gamma_{\mathfrak{p}} $ and by Proposition~\ref{p:kupxa1},
	\[ \kup{\wa{r}{m}} = \kup{\xam{\frac{\ol{r}}{\ol{m}}}{\frac{1}{\ol{m}}}}. \]
Since $ \xa{0}{b \thet} $, $ \xam{0}{- 1 / b \thet} \in \Gamma_{\mathfrak{p}} $ for $ b \in \Ok^{\times} $, by Proposition~\ref{p:kupxa},
	\[ \kup{\wa{0}{b \thet}} = \kup{\xam{0}{- \frac{1}{b \thet}}} = 1. \]
Also, for $ a \in \Ok $ with $ - 1 / 2 + a \thet \in \OK^{\times} $, we have the elements $ \xa{1}{- 1/ 2 + a \thet}$ , $ \xam{- 2 / ( 1 + 2 a \thet )}{- 2 / ( 1 + 2 a \thet )} $, $\xa{( 1 + 2 a \thet ) / ( 1 - 2 a \thet )}{- 1 / 2 + a \thet} \in \Gamma_{\mathfrak{p}} $, hence
	\[ \kup{\wa{1}{- \frac{1}{2} + a \thet}} = \kup{\xam{- \frac{2}{1 + 2 a \thet}}{- \frac{2}{1 + 2 a \thet}}}. \]
So by Proposition~\ref{p:kupxa}, we have
\begin{align*}
	\kup{\wa{1}{- \frac{1}{2} + a \thet}}
	&= \hilk{- \left( - \frac{2}{1 + 2 a \thet} - \frac{2}{1 - 2 a \thet} \right)}{\norm{- \frac{2 \thet}{1 + 2 a \thet}}} \cdot 1 \\
	&= \hilk{\frac{4}{\norm{1 + 2 a \thet}}}{- \frac{4 \thet^{2}}{\norm{1 + 2 a \thet}}}.
\end{align*} 
By \eqref{e:bastri}, \eqref{e:norm} and \eqref{e:prop5ii},
	\[ \kup{\wa{1}{- \frac{1}{2} + a \thet}} = \hilk{\frac{4}{\norm{1 + 2 a \thet}}}{- \frac{4}{\norm{1 + 2 a \thet}}}, \]
and thus by \eqref{e:prop3i},
	\[ \kup{\wa{1}{- \frac{1}{2} + a \thet}} = 1. \qedhere \]
\end{proof}

\begin{prop} \label{p:kupha}
If $ \mathfrak{p} $ is odd and either unramified or split in $ K $, then for $ \lambda \in \OK^{\times} $ such that $ \ha{\lambda} \in \hat{T} $, we have
	\[ \kup{\ha{\lambda}} = \begin{cases} \displaystyle \hilk{a}{b}, &\text{if $ \lambda = a + b \thet $, $ a $, $ b \neq 0 $ and $ b \notin \Ok^{\times} $;} \\
	\displaystyle 1, & \text{otherwise.} \end{cases} \]
\end{prop}

\begin{proof}
We first assume that for $ \lambda \in \OK^{\times} $, $ \ha{\lambda} \in \hat{T} $, we have $ \del{2}{\lambda} \in \OK^{\times} $.

If $ \lambda \in \Ok^{\times} $, then this implies that $ \del{2}{\lambda} = \thet \in \OK^{\times} $, hence we have, by \eqref{e:halam} and \eqref{e:kupsig}, that
	\[ \kup{\ha{\lambda}} = \kup{\wa{0}{\lambda \thet}} \cdot \kup{\wa{0}{\thet}^{- 1}} \cdot \sig{\wa{0}{\lambda \thet}}{\wa{0}{\thet}^{- 1}}. \]
By Theorem~\ref{t:sigma}, we have
	\[ \sig{\wa{0}{\lambda \thet}}{\wa{0}{\thet}^{- 1}} = \ug{\frac{\lambda}{- 1}}{\lambda ( - 1 )} = \hilk{- \lambda}{\lambda}; \]
hence by \eqref{e:prop3i}, we have
	\[ \sig{\wa{0}{\lambda \thet}}{\wa{0}{\thet}^{- 1}} = 1. \]
Thus, by the above and Proposition~\ref{p:kupwa}, we have
	\[ \kup{\ha{\lambda}} = 1. \]

If $ \lambda \in \Ok^{\times} \thet $, then $ \del{2}{\lambda} = - 1 / ( 2 \lambda ) \in \OK^{\times} $.  By \eqref{e:halam} and \eqref{e:kupsig},
	\[ \kup{\ha{\lambda}}
	=	\begin{aligned}[t]
			&\kup{\wa{1}{- \frac{1}{2}}} \cdot \kup{\wa{0}{- \frac{1}{2 \lambda}}^{- 1}} \\
			&\phantom{\ } \cdot \sig{\wa{1}{- \frac{1}{2}}}{\wa{0}{- \frac{1}{2 \lambda}}^{- 1}}.
		\end{aligned} \]
By Theorem~\ref{t:sigma}, we have
\begin{align*}
	&\sig{\wa{1}{- \frac{1}{2}}}{\wa{0}{- \frac{1}{2 \lambda}}^{- 1}} \\
	&= \sig{\wa{1}{- \frac{1}{2}}}{\wa{0}{\frac{1}{2 \lambda}}} \\
	&=	\begin{aligned}[t]
			&\ug{\frac{\lambda}{1 / ( 2 \lambda \thet )}}{\left( - \frac{1}{2 \thet} \right) \left( \frac{1}{2 \lambda \thet} \right)} \\
			&\phantom{\ } \cdot \hilk{\frac{\del{2}{\lambda}}{\del{2}{1 / ( 2 \lambda \thet )}}}{- \frac{\norm{1 / ( 2 \lambda \thet )} \del{2}{1 / ( 2 \lambda \thet )}}{\del{2}{- 1 / ( 2 \thet )}}} \\
			&\phantom{\ } \cdot \hilk{\frac{\lambda}{( - 1 / ( 2 \thet ) ) ( 1 / ( 2 \lambda \thet ) )}}{\frac{\del{2}{- 1 / ( 2 \thet )} \del{2}{1 / ( 2 \lambda \thet )}}{\del{2}{\lambda} \thet}}
		\end{aligned} \\
	&=	\begin{aligned}[t]
			&\hilk{- 4 \norm{\lambda}^{2} \thet^{2}}{- \frac{1}{16 \norm{\lambda} \thet^{4}}} \cdot \hilk{- \frac{1}{2 \lambda \thet}}{- \left( \frac{1}{2 \lambda \thet} \right)^{2}} \\
			&\phantom{\ } \cdot \hilk{- 4 \lambda^{2} \thet^{2}}{- 2 \lambda \thet}.
		\end{aligned}
\end{align*}
Simplifying the above using the properties of Hilbert symbols \eqref{e:bastri}, \eqref{e:prop3i} and \eqref{e:prop5ii}, we get
	\[ \sig{\wa{1}{- \frac{1}{2}}}{\wa{0}{- \frac{1}{2 \lambda}}^{- 1}} = 1. \]
Thus by the above and Proposition~\ref{p:kupwa},
	\[ \kup{\ha{\lambda}} = 1. \]

If $ \lambda = a + b \thet \in \OK^{\times} $ such that $ a $, $ b \in k^{\times} $ and $ \ha{\lambda} \in \hat{T} $, then $ \del{2}{\lambda} \in \OK^{\times} $ when $ b \in \Ok^{\times} $.  So assume that $ b \in \Ok^{\times} $.  By \eqref{e:halam} and \eqref{e:kupsig},
	\[ \kup{\ha{\lambda}}
	=	\begin{aligned}[t]
			&\kup{\wa{1}{- \frac{1}{2} - \frac{a}{2 b \thet}}} \cdot \kup{\wa{0}{- \frac{1}{2 b \thet}}^{- 1}} \\
			&\phantom{\ } \cdot \sig{\wa{1}{- \frac{1}{2} - \frac{a}{2 b \thet}}}{\wa{0}{- \frac{1}{2 b \thet}}^{- 1}}.
		\end{aligned} \]
Then by Theorem~\ref{t:sigma},
\begin{align*}
	&\sig{\wa{1}{- \frac{1}{2} - \frac{a}{2 b \thet}}}{\wa{0}{- \frac{1}{2 b \thet}}^{- 1}} \\
	&= \sig{\wa{1}{- \frac{\lambda}{2 b \thet}}}{\wa{0}{\frac{1}{2 b \thet}}} \\
	&=	\begin{aligned}[t]
			&\ug{\frac{\lambda}{1 / ( 2 b \thet^{2} )}}{\left( - \frac{\lambda}{2 b \thet^{2}} \right) \left( \frac{1}{2 b \thet^{2}} \right)} \\
			&\phantom{\ } \cdot \hilk{\frac{\del{2}{\lambda}}{\del{2}{1 / ( 2 b \thet^{2} )}}}{- \frac{\norm{1 / ( 2 b \thet^{2} )} \del{2}{1 / ( 2 b \thet^{2} )}}{\del{2}{- \lambda / ( 2 b \thet^{2} )}}} \\
			&\phantom{\ } \cdot \hilk{\frac{\lambda}{( - \lambda / ( 2 b \thet^{2} ) ) ( 1 / ( 2 b \thet^{2} ) )}}{\frac{\del{2}{- \lambda / ( 2 b \thet^{2} )} \del{2}{1 / ( 2 b \thet^{2} )}}{\del{2}{\lambda} \thet}}
		\end{aligned} \\
	&= \hilk{4 \norm{\lambda} b^{2} \thet^{4}}{- \frac{\norm{\lambda}}{16 b^{4} \thet^{8}}} \cdot \hilk{- \frac{1}{2 b \thet^{2}}}{- \left( \frac{1}{2 b \thet^{2}} \right)} \cdot \hilk{- 4 b^{2} \thet^{4}}{- 2 b \thet^{2}}.
\end{align*}
Simplifying the above using the properties of Hilbert symbols \eqref{e:prop5ii} and \eqref{e:bastri}, we get
	\[ \sig{\wa{1}{- \frac{1}{2} - \frac{a}{2 b \thet}}}{\wa{0}{- \frac{1}{2 b \thet}}^{- 1}} = 1. \]
Hence, by the above and Proposition~\ref{p:kupwa},
	\[ \kup{\ha{\lambda}} = 1. \]

If $ b \notin \Ok^{\times} $, then $ a \in \Ok^{\times} $ since $ \lambda \in \OK^{\times} $.  But we already know that
	\[ \kup{\ha{b + a / \thet}} = 1. \]
Since
	\[ \ha{\lambda} = \ha{b + a / \thet} \cdot \ha{\thet}, \]
by \eqref{e:kupsig},
	\[ \kup{\ha{\lambda}} = \kup{\ha{b + a / \thet}} \cdot \kup{\ha{\thet}} \cdot \sig{\ha{b + a / \thet}}{\ha{\thet}}. \]
By Theorem~\ref{t:sigma},
\begin{align*}
	\sig{\ha{b + a / \thet}}{\ha{\thet}}
	&= \sig{\ha{\frac{\lambda}{\thet}}}{\ha{\thet}} \\
	&= \ug{\frac{\lambda}{\thet}}{\lambda} \cdot \hilk{\frac{\del{2}{\lambda}}{\del{2}{\thet}}}{- \frac{\norm{\thet} \del{2}{\thet}}{\del{2}{\lambda / \thet}}} \\
	&= \hilk{- \frac{\norm{\lambda}}{\thet^{2}}}{- \norm{\lambda}} \cdot \hilk{\frac{1}{b}}{a}.
\end{align*}
By the properties of Hilbert symbols \eqref{e:prop2i} and \eqref{e:prop3i}, we have
	\[ \sig{\ha{b + a / \thet}}{\ha{\thet}} = \hilk{a}{b} \cdot \hilk{- \thet^{2}}{- \norm{\lambda}}. \]
But the characteristic of the residue field $ \Ok / \mathfrak{p}_{k} $ is not $ 2 $.  Thus we may use Proposition~\ref{p:tame}, as $ - \thet^{2} $, $ \norm{\lambda} \in \Ok^{\times} $, so that we have
	\[ \hilk{- \thet^{2}}{- \norm{\lambda}} = 1. \]
Applying the above to our equation, we get
	\[ \kup{\ha{\lambda}} = \hilk{a}{b}. \]
This completes our proof.
\end{proof}

\begin{rem} \label{r:kupha}
We should note that if $ \mathfrak{p} $ is odd and either unramified or split in $ K $, then from Propositions~\ref{p:kupwa} and~\ref{p:kupha}, we have that $ \hat{T} $ is generated by elements of $ N( k ) \cap G( k ) $ and $ \ol{N}( k ) \cap G( k ) $.
\end{rem}

We now look at the cases where either $ \mathfrak{p} $ is odd and ramified in $ K $, or $ \mathfrak{p} $ is even in $ K $.

\begin{prop} \label{p:kupha2}
If $ \mathfrak{p} $ is odd and ramified in $ K $, or $ \mathfrak{p} $ is even in $ K $, then if $ \ha{\lambda} \in \hat{T} $, there exists $ \ha{\mu} \in T( k ) \cap \hat{\Gamma}_{\mathfrak{p}} $ such that $ \ha{\lambda} = \ha{\mu}^{2} $, and
	\[ \kup{\ha{\lambda}} = 1. \]
\end{prop}

\begin{proof}
We first consider the case where $ \mathfrak{p} $ is odd and ramified in $ K $.  In this case, $ \Gamma_{\mathfrak{p}} = \hat{\Gamma}_{\mathfrak{p}} = G( \Ok, \thet ) $.  Thus elements of $ \hat{T} $ may be described by
	\[ \ha{1 + \lambda' \thet} = \begin{pmatrix} 1 + \lambda' \thet & 0 & 0 \\ 0 & ( 1 - \ol{\lambda'} \thet ) / ( 1 + \lambda' \thet ) & 0 \\ 0 & 0 & ( 1 - \ol{\lambda'} \thet )^{- 1} \end{pmatrix}, \]
where $ \lambda' \in \OK $.  We want to show that for every $ \lambda' \in \OK $, $ \ha{1 + \lambda' \thet} $ is a square of an element in $ \hat{T} $.  We will use Hensel's Lemma (Theorem~\ref{t:hensel}).  We have
	\[ f( X ) = X^{2} - ( 1 + \lambda' \thet ) \in \OK[ X ], \]
and hence the formal derivative $ f'( X ) $ is
	\[ f'( X ) = 2 X. \]
Let the prime above $ \mathfrak{p} $ in $ K $ be denoted by $ \mathfrak{P} $.  Then by Lemma~\ref{l:extval},
\begin{gather*}
	| f( 1 ) |_{\mathfrak{P}} = \left| - \norm{\lambda'} \thet^{2} \right|_{\mathfrak{p}} < 1, \\
	| f'( 1 ) |_{\mathfrak{P}} = \left| 2^{2} \right|_{\mathfrak{p}} = 1.
\end{gather*}
This implies by Hensel's Lemma that since $ | f( 1 ) |_{\mathfrak{P}} < | f'( 1 ) |_{\mathfrak{P}}^{2} $, there exists a solution for $ f( X ) $, i.e. there exists $ a \in \OK $ such that
	\[ f( a ) = 0, \quad | a - 1 |_{\mathfrak{P}} \leq \frac{| f( 1 ) |_{\mathfrak{P}}}{| f'( 1 ) |_{\mathfrak{P}}}. \]

We have by the binomial theorem that for any $ \lambda' \in \OK $,
	\[ ( 1 + \lambda' \thet )^{1 / 2} = 1 + \frac{\lambda' \thet}{2} + \frac{( 1 / 2 ) ( 1 / 2 - 1 )}{2!} \cdot ( \lambda' \thet )^{2} + \cdots \]
and
	\[ \left| 1 - ( 1 + \lambda' \thet )^{1 / 2} \right|_{\mathfrak{P}} = \left| \frac{\lambda' \thet}{2} \right|_{\mathfrak{P}} \leq \frac{| f( 1 ) |_{\mathfrak{P}}}{| f'( 1 ) |_{\mathfrak{P}}}, \]
thus by Hensel's Lemma, this expansion converges.  Hence, there is a canonical choice for $ 1 + \mu' \thet $, $ \mu' \in \OK $ for every $ 1 + \lambda' \thet $, $ \lambda' \in \OK $ such that $ 1 + \lambda' \thet = ( 1 + \mu' \thet )^{2} $, and hence
	\[ \ha{1 + \lambda' \thet} = \ha{1 + \mu' \thet}^{2}, \]
where $ \ha{1 + \mu' \thet} \in \hat{T} $.

As for the case where $ \mathfrak{p} $ is even, $ \Gamma_{\mathfrak{p}} = G( \Ok, 8 ) $, $ \hat{\Gamma}_{\mathfrak{p}} = G( \Ok, 4 ) $ and the elements of $ \hat{T} $ are of the form
	\[ \ha{1 + 8 \lambda''} = \begin{pmatrix} 1 + 8 \lambda'' & 0 & 0 \\ 0 & \left( 1 + 8 \ol{\lambda''} \right) / ( 1 + 8 \lambda'' ) & 0 \\ 0 & 0 & \left( 1 + 8 \ol{\lambda''} \right)^{- 1} \end{pmatrix}, \]
where $ \lambda'' \in \OK $.  Similar to the case where $ \mathfrak{p} $ is odd and ramified in $ K $, we will use Hensel's Lemma to show that every element in $ \hat{T} $ is a square of an element in $ T( k ) \cap \hat{\Gamma}_{\mathfrak{p}} $.  We have
	\[ f_{1}( X ) = X^{2} - ( 1 + 8 \lambda'' ), \]
with formal derivative
	\[ f_{1}'( X ) = 2 X. \]
Let $ \mathfrak{P} $ be a prime ideal in $ K $ such that $ \mathfrak{P} \mid \mathfrak{p} $.  By Lemma~\ref{l:extval},
\begin{gather*}
	| f_{1}( 1 ) |_{\mathfrak{P}} = \left| 8^{2} \norm{\lambda''} \right|_{\mathfrak{p}} = \left| 2^{6} \norm{\lambda''} \right|_{\mathfrak{p}}, \\
	| f_{1}'( 1 ) |_{\mathfrak{P}} = \left| 2^{2} \right|_{\mathfrak{p}},
\end{gather*}
thus we have $ | f_{1}( 1 ) |_{\mathfrak{P}} <  | f_{1}'( 1 ) |_{\mathfrak{P}}^{2} $.  (Note that for the split case we will use the isomorphism \eqref{e:splmap}.)  This implies by Hensel's Lemma that there exists a solution for $ f_{1}( X ) $, i.e. there exists $ b \in \OK $ such that
	\[ f_{1}( b ) = 0, \quad | b - 1 |_{\mathfrak{P}} \leq \frac{| f_{1}( 1 ) |_{\mathfrak{P}}}{| f_{1}'( 1 ) |_{\mathfrak{P}}}. \]
	
By the binomial theorem, for any $ \lambda'' \in \OK $,
	\[ ( 1 + 8 \lambda'' )^{1 / 2} = 1 + 4 \lambda'' + \frac{( 1 / 2 ) ( 1 / 2 - 1 )}{2!} ( 8 \lambda'' )^{2} + \cdots \]
and
	\[ \left| 1 - ( 1 + 8 \lambda'' )^{1 / 2} \right|_{\mathfrak{P}} = \left| 4 \lambda'' \right|_{\mathfrak{P}} \leq \frac{| f_{1}( 1 ) |_{\mathfrak{P}}}{| f_{1}'( 1 ) |_{\mathfrak{P}}}, \]
so by Hensel's Lemma, the expansion converges.  So there is a canonical choice for $ 1 + 4 \mu'' $, $ \mu'' \in \OK $, for every $ 1 + 8 \lambda'' $, $ \lambda'' \in \OK $, such that $ 1 + 8 \lambda'' = ( 1 + 4 \mu'' )^{2} $ and
	\[ \ha{1 + 8 \lambda''} = \ha{1 + 4 \mu''}^{2}. \]
So this verifies that $ \ha{1 + 4 \mu''} \in T( k ) \cap \hat{\Gamma}_{\mathfrak{p}} $.

Now that we have established that every element of $ \hat{T} $ is a square of an element in $ T( k ) \cap \hat{\Gamma}_{\mathfrak{p}} $, let us now calculate the local Kubota symbol.  Whether $ \mathfrak{p} $ is odd and ramified in $ K $, or even, let $ \ha{\lambda} \in \hat{T} $ and $ \ha{\mu} \in T( k ) \cap \hat{\Gamma}_{\mathfrak{p}} $ such that $ \ha{\lambda} = \ha{\mu}^{2} $.  Then by \eqref{e:kupsq},
	\[ \kup{\ha{\lambda}} = \sig{\ha{\mu}}{\ha{\mu}}. \]
There are only two cases for $ \mu $ to consider: either $ \mu \in k^{\times} $, or $ \mu = c + d \thet $, $ c $, $ d \in k^{\times} $.  In the first case, we have
	\[ \kup{\ha{\lambda}} = \ug{\frac{\mu^{2}}{\mu}}{\mu^{2}} = \hilk{\mu}{- \mu^{2}} \]
by Theorem~\ref{t:sigma}.  Simplifying the above using the Hilbert symbol property \eqref{e:prop3i},
	\[ \kup{\ha{\lambda}} = \hilk{\mu}{- 1}. \]

This implies, in the case where $ \mathfrak{p} $ is odd and ramified in $ K $, that $ \mu = 1 + b \thet^{2} $, where $ b \in \Ok $.  By Hensel's Lemma, $ X^{2} - ( 1 + b \thet^{2} ) $ has a solution with approximate root $ 1 $.  Thus $ 1 + b \thet^{2} $ is a square in $ k^{\times} $, hence by the property of Hilbert symbols \eqref{e:prop5ii},
	\[ \kup{\ha{\lambda}} = 1. \]

In the case where $ \mathfrak{p} $ is even, if $ \mu \in k^{\times} $, then $ \mu = 1 + 4 a' $ for $ a' \in \Ok $.  Thus we have
	\[ \kup{\ha{\lambda}} = \hilk{1 + 4 a'}{- 1}. \]
If $ 1 + 4 a' \in 1 + 8 \Ok $, then $ 1 + 4 a' $ is a square in $ k^{\times} $ (we showed this is true by Hensel's Lemma earlier in the proof ) and hence by \eqref{e:prop5ii},
	\[ \hilk{1 + 4 a'}{- 1} = 1. \]
If $ 1 + 4 a' \notin 1 + 8 \Ok $, we can show that $ \hilk{1 + 4a'}{- 1} = 1 $ as well.  In Section~\ref{s:hilbert}, we established that
	\[ \hilk{a}{b} = 1 \iff aX^{2} + bY^{2} - Z^{2} = 0 \text{ has a non-trivial solution $ ( X, Y, Z ) $ in $ k^{3} $}. \]
By \eqref{e:prop4i},
	\[ \hilk{1 + 4 a'}{- 1} = \hilk{1 + 4 a'}{- ( 1 - ( 1 + 4 a' ) )} = \hilk{1 + 4 a'}{4 a'}. \]
Since we have a solution $ \left( 1, 1, \sqrt{1 + 8 a'} \right) $ for $ \hilk{1 + 4 a'}{4 a'} $, this implies that $ \hilk{1 + 4 a'}{- 1} = 1 $ for all $ a' \in \Ok $.  Therefore for all elements of the form $ \ha{1 + 4 a'} $, $ a' \in \Ok $, $ \mathfrak{p} $ even,
	\[ \kup{\ha{\lambda}} = 1. \]


If $ \mu \notin k^{\times} $, then $ \mu = c + d \thet $ with $ c $, $ d \in k^{\times} $ and by Theorem~\ref{t:sigma},
\begin{align*}
	\kup{\ha{\lambda}}
	&= \ug{\frac{\mu^{2}}{\mu}}{\mu^{2}} \cdot \hilk{\frac{\del{2}{\mu^{2}}}{\del{2}{\mu}}}{- \frac{\norm{\mu} \del{2}{\mu}}{\del{2}{\mu}}} \\
	&= \hilk{\norm{\mu}}{- \norm{\mu}^{2}} \cdot \hilk{\frac{1}{2 c}}{- \norm{\mu}}.
\end{align*}
Simplifying the above using the properties of Hilbert symbols \eqref{e:prop2i} -- \eqref{e:bastri},
	\[ \kup{\ha{\lambda}} = \hilk{\norm{\mu}}{- 2 c} \cdot \hilk{c}{- 1}. \]

When $ \mathfrak{p} $ is odd and ramified in $ K $, $ \mu = 1 + ( a + b \thet ) \thet $, for $ a $, $ b \in \Ok $ with $ a \neq 0 $.  Thus $ c = 1 + b \thet^{2} \in \Ok^{\times} $.  By Proposition~\ref{p:tame}, since $ \norm{\mu} $, $ - 1 $, $ 2 $, $ c \in \Ok^{\times} $, we have $ \hilk{\norm{\mu}}{- 2 c} $, $ \hilk{2 c}{- 1} = 1 $.  Therefore we have $ \kup{\ha{\lambda}} = 1 $.

When $ \mathfrak{p} $ is even, we first consider the case where $ \OK = \Ok[ \thet ] $.  Then $ \mu = 1 + 4 ( a' + b' \thet ) $, for $ a' $, $ b' \in \Ok $ with $ b' \neq 0 $.  This implies that $ c = 1 + 4 a' \in \OK^{\times} $.  We have already shown that $ \hilk{1 + 4 a'}{- 1} = 1 $ in the case where $ \mu \in k^{\times} $.  As for $ \hilk{\norm{\mu}}{- 2 c} $, we see that $ \norm{\mu} \in 1 + 8 \Ok $, i.e. $ \norm{\mu} $ is a square in $ k^{\times} $.  Hence by \eqref{e:prop5ii}, $ \hilk{\norm{\mu}}{- 2 c} = 1 $.  Therefore $ \kup{\ha{\lambda}} = 1 $.

If instead $ \OK = \Ok[ ( 1 + \thet ) / 2 ] $ when $ \mathfrak{p} $ is even (i.e. $ \mathfrak{p} $ is unramified in $ K $), then we have
	\[ \mu = 1 + 4 \left( a' + b' \left[ \frac{1 + \thet}{2} \right] \right) = 1 + 4 a' + 2 b' + 2 b' \thet, \]
where $ a' $, $ b' \in \Ok $ with $ b' \neq 0 $, with $ c = 1 + 4 a' + 2 b' $, and $ d = 2 b' $.  This implies that when $ b' \notin \Ok^{\times} $, i.e. $ 2 \mid b' $, we are in the same case as when $ \OK = \Ok[ \thet ] $, since $ c \equiv 1 \pod{4} $, which implies that $ \hilk{c}{- 1} = 1 $ by our previous result, and $ \norm{\mu} \in 1 + 8 \Ok $, so $ \hilk{\norm{\mu}}{- 2 c} = 1 $.  Hence when $ 2 \mid b' $, $ \kup{\ha{\lambda}} = 1 $.

We can use the properties of Hilbert symbols \eqref{e:prop4i} and \eqref{e:bastri} to show that
	\[ \kup{\ha{\lambda}} = \hilk{\norm{\mu}}{- 1} \cdot \hilk{2 c}{- \norm{\mu}}. \]
The above formulation makes it easier to calculate the Kubota symbol in the case where $ b' \in \Ok^{\times} $.  We have $ \norm{\mu} \equiv 1 \pod{4} $, hence $ \hilk{\norm{\mu}}{- 1} = 1 $.  Also, $ \norm{\mu} \equiv 5 \pod{8} $ and $ c \equiv - 1 \pod{4} $.  So let $ \norm{\mu} = 5 + 8 e $ and $ c = - 1 + 4 f $, where $ e $, $ f \in \Ok $.  Then by \eqref{e:prop4i},
	\[ \hilk{2 c}{- \norm{\mu}} = \hilk{2 c}{- ( 1 - 2 c ) \norm{\mu}}. \]
But
	\[ - ( 1 - 2 c ) \norm{\mu} = - ( 3 - 8 f ) ( 5 + 8 e ) \equiv 1 \pod 8, \]
i.e. $ - ( 1 - 2 c ) \norm{\mu} $ is a square in $ k^{\times} $, and hence by \eqref{e:prop5ii}, $ \hilk{2 c}{- \norm{\mu}} = 1 $.  Thus we have shown that
	\[ \kup{\ha{\lambda}} = 1 \]
in all cases.
\end{proof}

\section{Other elements of the compact open subgroup} \label{s:other}

We are now ready to show how to find the local Kubota symbol on any element of $ \Gamma_{\mathfrak{p}} $.

By \eqref{e:bruc}, it is clear that the Bruhat decomposition of a matrix of $ G $ with a non-zero $( 3, 1 )$-entry depends only on the first column and last row of the matrix.  The following proposition establishes an important property of this type of matrix when $ \mathfrak{p} $ is odd and unramified in $ K $:

\begin{prop} \label{p:unit}
Let $ \mathfrak{p} $ is odd and unramified (but not split) in $ K $.  Also, let $ a $, $ b $, $ c $, $ d $, $ e \in \OK $, such that
	\[ a \ol{c} + \ol{a} c = - \norm{b},\ e \ol{c} + \ol{e} c = - \norm{d}, \]
and
	\[ \frac{a d}{c} + \frac{\ol{b}}{c},\ \frac{b d}{c} - \frac{\ol{c}}{c},\ \frac{b e}{c} + \frac{\ol{d}}{c},\ \frac{a e}{c} - \frac{\ol{b d}}{\norm{c}} + \frac{1}{\ol{c}} \in \OK. \]
Then
	\[ \beta = \begin{pmatrix}
					a & a d / c + \ol{b} / c & a e / c - \ol{b d} / ( \norm{c} ) + \ol{c}^{- 1} \\ 
					b & b d / c - \ol{c} / c & b e / c + \ol{d} / c \\ 
					c & d & e
				\end{pmatrix} \in \Gamma_{\mathfrak{p}}, \]
and either $ c $ or $ e \in \OK^{\times} $, or possibly both.
\end{prop}

\begin{proof}
The first part is easily established by \eqref{e:bruc}.  As for the second part, consider the Hermitian form $ \langle - , - \rangle $ defined by
	\[ \langle u, v \rangle = u^{t} \begin{pmatrix} 0 & 0 & 1 \\ 0 & 1 & 0 \\ 1 & 0 & 0 \end{pmatrix} \ol{v}, \]
where $ u $, $ v \in K^{3} $.  This implies that for all $ \gamma \in \SU( k ) $,
	\[ \langle \gamma u, \gamma v \rangle = \langle u, v \rangle. \]
If
	\[ u = \begin{pmatrix} u_{1} \\ u_{2} \\ u_{3} \end{pmatrix},\ v = \begin{pmatrix} v_{1} \\ v_{2} \\ v_{3} \end{pmatrix}, \]
then
	\[ \langle u, v \rangle = u_{3} \ol{v_{1}} + u_{2} \ol{v_{2}} + u_{1} \ol{v_{3}}. \]

Since
	\[ \left\langle \begin{pmatrix} a \\ b \\ c \end{pmatrix}, \begin{pmatrix} a \\ b \\ c \end{pmatrix} \right\rangle = \ol{a} c + \norm{b} + a \ol{c} = 0, \]
if $ \mathfrak{p} \mid c $, then $ \mathfrak{p} \mid b $.  Similarly, since
	\[ \left\langle \begin{pmatrix} a e / c - \ol{b d} / ( \norm{c} ) + \ol{c}^{- 1} \\ b e / c + \ol{d} / c \\ e \end{pmatrix}, \begin{pmatrix} a e / c - \ol{b d} / ( \norm{c} ) + \ol{c}^{- 1} \\ b e / c + \ol{d} / c \\ e \end{pmatrix} \right\rangle = 0, \]
if $ \mathfrak{p} \mid e $, then $ \mathfrak{p} \mid ( b e / c + \ol{d} / c ) $, i.e. $ \mathfrak{p} \mid d $.

For any $ \lambda \in \OK^{\times} $, $ \mathfrak{p} \nmid \lambda $.  As
\begin{align*}
	\left\langle \begin{pmatrix} a \\ b \\ c \end{pmatrix}, \begin{pmatrix} a e / c - \ol{b d} / ( \norm{c} ) + \ol{c}^{- 1} \\ b e / c + \ol{d} / c \\ e \end{pmatrix} \right\rangle
	&= \left\langle \beta \begin{pmatrix} 1 \\ 0 \\ 0 \end{pmatrix}, \beta \begin{pmatrix} 0 \\ 0 \\ 1 \end{pmatrix} \right\rangle \\
	&= \left\langle \begin{pmatrix} 1 \\ 0 \\ 0 \end{pmatrix}, \begin{pmatrix} 0 \\ 0 \\ 1 \end{pmatrix} \right\rangle \\
	&= 1,
\end{align*}
if $ \mathfrak{p} \mid c $ and $ \mathfrak{p} \mid e $, then $ \mathfrak{p} \mid 1 $, which is a contradiction.  Hence either $ c $ or $ e $ is a unit in $ \OK $, or possibly both.
\end{proof}

In the split case, i.e. when $ \mathfrak{p} $ is split in $ K $, we have to consider the fact that sometimes, none of the entries in the bottom row of an element in $ \Gamma_{\mathfrak{p}} $ is a unit.  This does not occur in the other cases.  The proposition below shows the existence of a transformation to a matrix of this form so that the resulting matrix has a unit in the $ ( 3, 3 ) $-entry.

\begin{prop} \label{p:splunit}
Let $ \mathfrak{p} $ be a split prime in $ K $.  If
	\[ \begin{pmatrix} a & b & c \\ d & e & f \\ g & h & j \end{pmatrix} \in \Gamma_{\mathfrak{p}}, \]
with neither $ g $ nor $ j $ a unit in $ \OK $, then there is always an element $ \xa{s_{1}}{n_{1}} \in N( k ) \cap \Gamma_{\mathfrak{p}} $ such that
	\[ \begin{pmatrix} g & h & j \end{pmatrix} \cdot \xa{s_{1}}{n_{1}} = \begin{pmatrix} g & h' & j' \end{pmatrix} \]
with $ j' \in \OK^{\times} $.
\end{prop}

\begin{proof}
Let $ \mathfrak{p} \OK = \mathfrak{p}_{1} \mathfrak{p}_{2} $, where $ \mathfrak{p}_{1} $, $ \mathfrak{p}_{2} \subset \OK $.  We first note that $ g $, $ h $ and $ j $ are coprime to each other, by the same principle as in the proof of Proposition~\ref{p:unit}.

If $ \mathfrak{p} \mid g $, then without loss of generality, we may assume that $ \mathfrak{p}_{1} \mid h $, which in turn implies that $ \mathfrak{p}_{1} \nmid j $.  Using the Hermitian form from the proof of Proposition~\ref{p:unit}, since
	\[ \left\langle \begin{pmatrix} g \\ h \\ j \end{pmatrix}, \begin{pmatrix} g \\ h \\ j \end{pmatrix} \right\rangle = g \ol{j} + \norm{h} + \ol{g} j = 0, \]
and $ \mathfrak{p} \mid \norm{h} $, we have that $ \mathfrak{p}_{2} \nmid j $, i.e. $ j \in \OK^{\times} $.

Now let $ g $, $ j \notin \OK^{\times} $, with $ \mathfrak{p}_{1} \mid g $, $ \mathfrak{p}_{2} \nmid g $.  Assume that $ \mathfrak{p}_{1} \mid h $.  Then $ \mathfrak{p}_{1} \nmid j $, and this implies that $ \mathfrak{p}_{2} \nmid \ol{j} $.  But $ \mathfrak{p} \mid \norm{h} $, hence $ \mathfrak{p} \mid \trace{g \ol{j}} $, i.e. $ \mathfrak{p}_{2} \mid \trace{g \ol{j}} $.  Therefore, since $ \mathfrak{p}_{2} \nmid g \ol{j} $, we have $ \mathfrak{p}_{2} \nmid \ol{g} j $, and hence $ \mathfrak{p}_{2} \nmid j $ (since $ \mathfrak{p}_{2} \mid \ol{g} $).  But this means that $ j \in \OK^{\times} $ since $ j $ is a unit mod $ \mathfrak{p}_{1} $ and mod $ \mathfrak{p}_{2} $, which contradicts our assumption that $ j $ is not a unit in $ \OK $.  Hence if $ \mathfrak{p}_{1} \mid g $ and $ \mathfrak{p}_{2} \nmid g $, then $ \mathfrak{p}_{1} \nmid h $.

Since $ g $ is a unit mod $ \mathfrak{p}_{2} $ and $ h $ is a unit mod $ \mathfrak{p}_{1} $, we can always choose a $ z' \in K $ such that
	\[ \begin{pmatrix} g & h & j \end{pmatrix} \cdot \xa{z'}{- \frac{\norm{z'}}{2}} = \begin{pmatrix} g & h' & j_{1} \end{pmatrix}, \]
with $ h' = g z' + h $ a unit mod $ \mathfrak{p}_{2} $, i.e. $ h' \in \OK^{\times} $.  So now we have $ \mathfrak{p} \nmid \norm{h'} $, hence $ \mathfrak{p} \nmid \trace{g \ol{j_{1}}} $.  So $ \mathfrak{p}_{1} \nmid ( g \ol{j_{1}} + \ol{g} j_{1} ) $, i.e. $ \mathfrak{p}_{1} \nmid \ol{g} j_{1} $.  This implies that $ \mathfrak{p}_{1} \nmid j_{1} $, i.e. $ j_{1} $ is a unit mod $ \mathfrak{p}_{1} $.  So we can choose $ t' \in \Ok $ such that
	\[ \begin{pmatrix} g & h' & j_{1} \end{pmatrix} \cdot \xa{0}{t' \thet} = \begin{pmatrix} g & h' & j' \end{pmatrix}, \]
with $ j' = ( t' \thet ) g + j_{1} $ a unit mod $ \mathfrak{p}_{2} $, i.e. $ j' \in \OK^{\times} $.

Therefore we can always choose some $ \xa{z'}{- \norm{z'} / 2 + t' \thet} \in N( k ) \cap \Gamma_{\mathfrak{p}} $ and
	\[ \begin{pmatrix} g & h & j \end{pmatrix} \cdot \xa{z'}{- \frac{\norm{z}}{2} + t' \thet} = \begin{pmatrix} g & h' & j' \end{pmatrix}, \]
with $ j' \in \OK^{\times} $.
\end{proof}

With the above two propositions, we can now show how to obtain the local Kubota symbol on any element of $ \Gamma_{\mathfrak{p}} $.

\begin{thm} \label{t:kupgen}
Let
	\[ \gamma = \begin{pmatrix} a & b & c \\ d & e & f \\ g & h & j \end{pmatrix} \in \Gamma_{\mathfrak{p}}. \]
Then
	\[ \kup{\gamma} \\
	=	\begin{cases}
			\displaystyle \kup{\ha{\ol{j}^{-1}}}, &\text{if $ g = 0 $;} \\
			\displaystyle \kup{\ha{( \ol{g} \thet )^{- 1}}}, &\text{if $ g \in \OK^{\times} $;} \\
			\begin{aligned}[b]
				&\kup{\ha{\ol{j}^{- 1}}} \cdot \kup{\xam{- \frac{\ol{h}}{\ol{j}}}{\frac{g}{j}}} \\
				&\phantom{\ } \cdot \sig{\ha{\ol{j}^{- 1}}}{\ha{\frac{\ol{j}}{\ol{g} \thet}}},
			\end{aligned}
			&\text{if $ g \neq 0 $, $ g \notin \OK^{\times} $.}
		\end{cases} \]
If $ \mathfrak{p} $ is split, then when both $ g $ and $ j $ are not units in $ \OK $, we can always choose some $ \xa{s_{1}}{n_{1}} \in N( k ) \cap \Gamma_{\mathfrak{p}} $ such that
	\[ \begin{pmatrix} g & h & j \end{pmatrix} \cdot \xa{s_{1}}{n_{1}} = \begin{pmatrix} g & h' & j' \end{pmatrix}, \]
with $ j' \in \OK $.  Then
	\[ \kup{\gamma} = \kup{\gamma \cdot \xa{s_{1}}{n_{1}}}, \]
and we may apply the above to $ \kup{\gamma \cdot \xa{s_{1}}{n_{1}}} $ to calculate our result.
\end{thm}

\begin{proof}
Consider the case where $ g = 0 $.  By \eqref{e:bruf},
	\[ \gamma = \begin{pmatrix} a & b & c \\ d & e & f \\ g & h & j \end{pmatrix} = \ha{\ol{j}^{- 1}} \cdot \xa{b \ol{j}}{c \ol{j}}, \]
hence by Proposition~\ref{p:kupxa1},
	\[ \kup{\gamma} = \kup{\ha{\ol{j}^{- 1}}}. \]

As noted at the start of Section~\ref{s:torus}, we will only have $ g \in \OK^{\times} $ when $ \mathfrak{p} $ is odd and either unramified or split in $ K $.  In this case, $ \xa{- \ol{d} / \ol{g}}{a / g} $, $ \ha{( \ol{g} \thet )^{-1}} $, $ \xa{h / g}{j / g} \in \Gamma_{\mathfrak{p}} $.  Therefore, since by \eqref{e:bruc},
	\[ \gamma = \xa{- \frac{\ol{d}}{\ol{g}}}{\frac{a}{g}} \cdot \ha{( \ol{g} \thet )^{-1}} \cdot \wa{0}{\thet} \cdot \xa{\frac{h}{g}}{\frac{j}{g}}, \]
by Proposition~\ref{p:kupxa1},
	\[ \kup{\gamma} = \kup{\ha{( \ol{g} \thet )^{-1}} \cdot \wa{0}{\thet}}. \]
Hence by \eqref{e:kupsig},
	\[ \kup{\gamma} = \kup{\ha{( \ol{g} \thet )^{-1}}} \cdot \kup{\wa{0}{\thet}} \cdot \sig{\ha{( \ol{g} \thet )^{-1}}}{\wa{0}{\thet}}. \]
So by Proposition~\ref{p:kupwa} and Theorem~\ref{t:sigma},
	\[ \kup{\gamma} = \kup{\ha{( \ol{g} \thet )^{-1}}}. \]

If $ g \neq 0 $, $ g \notin \OK^{\times} $, then $ j \in \OK^{\times} $ in most cases by Proposition~\ref{p:unit} (we deal with the case $ g $, $ j \notin \OK^{\times} $ later in the proof).  Also, using the Hermitian form established in the proof of Proposition~\ref{p:unit}, we have that
\begin{gather*}
	\left\langle \begin{pmatrix} a \\ d \\ g \end{pmatrix}, \begin{pmatrix} a \\ d \\ g \end{pmatrix} \right\rangle = a \ol{g} + \norm{d} + \ol{a} g = 0, \\
	\left\langle \begin{pmatrix} g \\ h \\ j \end{pmatrix}, \begin{pmatrix} g \\ h \\ j \end{pmatrix} \right\rangle = g \ol{j} + \norm{h} + \ol{g} j = 0.
\end{gather*}
We now have a few cases to consider.  We will use the notation in Section~\ref{s:iwahori}.  If $ \mathfrak{p} $ is odd and unramified in $ K $, then since $ \mathfrak{p} \mid g $, this implies that $ \mathfrak{p} \mid d $ and $ \mathfrak{p} \mid h $.   Thus, we have $ \mathfrak{a} = \mathfrak{p} $ and
	\[ \gamma \in G( \Ok )_{0}( \mathfrak{p} ) = N( \Ok ) \cdot T( \Ok ) \cdot \ol{N}( \mathfrak{p} ). \]
If $ \mathfrak{p} $ is odd and ramified in $ K $, we already have that $ \mathfrak{a} = ( \thet ) $ in Proposition~\ref{p:iwahori}, and
	\[ \Gamma_{\mathfrak{p}} = G( \Ok )_{1}( ( \thet ) ), \]
which implies that there exists an Iwahori factorisation of $ \Gamma_{\mathfrak{p}} $, with
	\[ \Gamma_{\mathfrak{p}} = N( ( \thet ) ) \cdot T( ( \thet ) ) \cdot \ol{N}( ( \thet ) ). \]
If instead $ \mathfrak{p} $ is even, we have $ \mathfrak{a} = ( 8 ) $ in Proposition~\ref{p:iwahori} and
	\[ \Gamma_{\mathfrak{p}} = G( \Ok )_{1}( ( 8 ) ) =  N( ( 8 ) ) \cdot T( ( 8 ) ) \cdot \ol{N}( ( 8 ) ). \]

So we have an Iwahori factorisation for all elements of $ \Gamma_{\mathfrak{p}} $ with $ g \neq 0 $, $ g \notin \OK^{\times} $.  Hence, we have
	\[ \gamma = \xa{- \frac{\ol{f}}{\ol{j}}}{\frac{c}{j}} \cdot \ha{\ol{j}^{- 1}} \cdot \xam{- \frac{\ol{h}}{\ol{j}}}{\frac{g}{j}}, \]
with $ \xa{- \ol{f} / \ol{j}}{c / j} $, $ \ha{\ol{j}^{- 1}} $, $ \xam{- \ol{h} / \ol{j}}{g / j} \in \Gamma_{\mathfrak{p}} $.  This implies that by Proposition~\ref{p:kupxa1},
	\[ \kup{\gamma} = \kup{\ha{\ol{j}^{- 1}} \cdot \xam{- \frac{\ol{h}}{\ol{j}}}{\frac{g}{j}}}, \]
and by \eqref{e:kupsig} and Proposition~\ref{p:sigma},
	\[ \kup{\gamma} = \kup{\ha{\ol{j}^{- 1}}} \cdot \kup{\xam{- \frac{\ol{h}}{\ol{j}}}{\frac{g}{j}}} \cdot \sig{\ha{\ol{j}^{- 1}}}{\ha{\frac{\ol{j}}{\ol{g} \thet}}}. \]

We have one last case to consider.  We have already noted that there will be elements in the split case where all the bottom row entries are non-units.  By Proposition~\ref{p:splunit}, we can apply a transformation by multiplying the matrix on the right by a unipotent upper triangular matrix $ \xa{s_{1}}{n_{1}} \in \Gamma_{\mathfrak{p}} $, so that
	\[ \begin{pmatrix} g & h & j \end{pmatrix} \cdot \xa{s_{1}}{n_{1}} = \begin{pmatrix} g & h' & j' \end{pmatrix} \]
with $ j' \in \OK^{\times} $.  But by Proposition~\ref{p:kupxa1},
	\[ \kup{\gamma \cdot \xa{s_{1}}{n_{1}}} = \kup{\gamma}. \]
Hence we can use the calculation on $ \kup{\gamma \cdot \xa{s_{1}}{n_{1}}} $ to get $ \kup{\gamma} $.

This concludes our proof.
\end{proof}


\part{The half-integral weight multiplier system} \label{p:weight}


\chapter{The global Kubota symbol} \label{c:global}

As previously stated in the Introduction, in order to construct the half-integral weight multiplier system, we need to calculate the global Kubota symbol $ \kappa $.  In Section~3 of \cite{kubota69}, the global Kubota symbol on a chosen arithmetic subgroup of $ \GL_{2}(F) $ for some field $ F $ was given in terms of quadratic Legendre symbols.  We want to write down a similar formula in terms of quadratic Legendre symbols for our arithmetic subgroup.

Let $l$ be an arbitrary number field, and let $ \mathfrak{p} $ run through all the primes of $ l $.  Let $ L $ be a quadratic extension of $ l $, so that the matrix entries of an element in $ G( l ) $ lie in $ L $.  On the arithmetic subgroup $ \Gamma = \prod_{\mathfrak{p}} \Gamma_{\mathfrak{p}} \cap G( l ) $, we have
\begin{equation} \label{e:kappa}
	\kappa( \gamma ) = \prod_{\mathfrak{p} < \infty} \kup{\gamma},
\end{equation}
for $ \gamma \in \Gamma $.

We can prove that the above product converges -- it is just a matter of using the product formula (Theorem~\ref{t:product}).  Let $ \sigma_{\mathfrak{p}} \in H^{2}( G( \lp ), \mu_{2} ) $, where $ \sigma_{\mathfrak{p}} = \sigma $ when $ k = \lp $ and $ K = \Lp $, as described in Chapter~\ref{c:univ}.  We know that for $ g $, $ g' \in G( l ) $,
	\[ \prod_{\mathfrak{p}} \sigma_{\mathfrak{p}}( g, g' ) = 1, \]
where the product is over all primes $ \mathfrak{p} $ of $ l $, with only finitely many non-trivial terms.  We can extend the definition of $ \kappa_{\mathfrak{p}} $ in some arbitrary way to the whole of $ G( \lp ) $.  Define $ H $ to be the set of elements $ h \in G( l ) $ such that the product over all finite $ \mathfrak{p} $,
	\[ \prod_{\mathfrak{p}} \kup{h}, \]
converges.  We make the following claim:

\begin{claim} \label{cl:converge}
$ H = G( l ) $.
\end{claim}

\begin{proof}
For any $ g $, $ g' \in G( l ) $, we have $ g $, $ g' \in \Gamma_{\mathfrak{p}} $ for almost all $ \mathfrak{p} $.  Hence for almost all $ \mathfrak{p} $, we have
	\[ \kup{gg'} = \sigma_{\mathfrak{p}} ( g, g' ) \cdot \kup{g} \cdot \kup{g'}. \]
Hence if $ g $, $ g' \in H $, then $ g g' \in H $.

Similarly, we already have $ 1 \in H $, and
	\[ \kup{g \cdot g^{- 1}} = \kup{1} = \sigma_{\mathfrak{p}}\left( g, g^{- 1} \right) \cdot \kup{g} \cdot \kup{g^{- 1}} \]
for almost all $ \mathfrak{p} $.  This implies that if $ g \in H $, then $ g^{- 1} \in H $.

Lastly, by \eqref{e:kupsq},
	\[ \kup{g^{2}} = \sigma_{\mathfrak{p}} ( g, g ) \]
for almost all primes $ \mathfrak{p} $.  This shows that $ g^{2} \in H $ for all $ g \in G( l ) $, therefore H is a subgroup of $ G( l ) $ containing all the squares of $ G( l ) $.  However, $ G( l ) $ is generated by unipotent elements, and these are all squares, so we have $ H = G( l ) $.
\end{proof}

Unfortunately, with our formulae for the local Kubota symbol $ \kappa_{\mathfrak{p}} $ (see Theorem~\ref{t:kupgen}), it is difficult to find a ``nice'' formula for the global Kubota symbol for all elements of $ \Gamma $.  What we can do is find a formula on the Borel subgroup of $ \Gamma $ (i.e. the group of upper triangular matrices of $ \Gamma $).  Recall that for elements $a,b\in\mathcal{O}_{l}$ with $b$ coprime to $2a$, the quadratic Legendre symbol is defined by
	\[ \left(\frac{a}{b}\right)_{l,2} = \prod_{\mathfrak{p} | b} (a,b)_{\lp,2}. \]

We have the following proposition:

\begin{prop} \label{e:globkub}
Let
	\[ \begin{pmatrix} f & g & h \\ 0 & \ol{f} / f & - \ol{g} / f \\ 0 & 0 & \ol{f}^{-1} \end{pmatrix} \in \Gamma. \]
Then
	\[ \kappa\left( \begin{pmatrix} f & g & h \\ 0 & \ol{f} / f & - \ol{g} / f \\ 0 & 0 & \ol{f}^{-1} \end{pmatrix} \right)
	=	\begin{cases}
			\left( \dfrac{b}{a} \right)_{l, 2} \cdot\prod_{\mathfrak{p}|\infty} ( a, b )_{\lp, 2}, &\begin{aligned}[t]
				\text{if}\ &\text{$ f = a + b \thet $,} \\
				&\text{$ a $, $ b \in l $, $ b \neq 0 $;}
				\end{aligned} \\
			1, &\text{otherwise.}
		\end{cases} \]
\end{prop}

\begin{proof}
Using \eqref{e:kappa}, we have by Theorem~\ref{t:kupgen} that
	\[ \kappa\left( \begin{pmatrix} f & g & h \\ 0 & \ol{f} / f & - \ol{g} / f \\ 0 & 0 & \ol{f}^{-1} \end{pmatrix} \right) = \prod_{\mathfrak{p} < \infty} \kup{\ha{f}}. \]
But by Propositions~\ref{p:kupha} and~\ref{p:kupha2}, if $ f \in l $, then $ \kup{\ha{f}} = 1 $ for all $ \mathfrak{p} $.  In this case, we have
	\[ \kappa\left( \begin{pmatrix} f & g & h \\ 0 & \ol{f} / f & - \ol{g} / f \\ 0 & 0 & \ol{f}^{-1} \end{pmatrix} \right) = 1. \]

In the other case, $ f \notin k^{\times} $, i.e. $ f = a + b \thet $, where $ a $, $ b \in l $ and $ b \neq 0 $.  Then if $ \mathfrak{p} $ is odd, and either unramified or split, Proposition~\ref{p:kupha} states that if $ b \notin \Olp^{\times} $,
	\[ \kup{\ha{f}} = ( a, b )_{\lp, 2}. \]
Otherwise, $ \kup{\ha{f}} = 1 $ by Propositions~\ref{p:kupha} and~\ref{p:kupha2}.  Hence by \eqref{e:kappa}, we have
	\[ \kappa\left( \begin{pmatrix} f & g & h \\ 0 & \ol{f} / f & - \ol{g} / f \\ 0 & 0 & \ol{f}^{-1} \end{pmatrix} \right) = \prod_{\mathfrak{p} < \infty} \kup{\ha{f}} = \prod_{\mathfrak{p} \mid b} ( a, b )_{\lp, 2}, \]
where the product is over all $ \mathfrak{p} $ which are unramified or split.  Hence by the product formula (Theorem~\ref{t:product}) and \eqref{e:prop2i}, we get
	\[ \kappa\left( \begin{pmatrix} f & g & h \\ 0 & \ol{f} / f & - \ol{g} / f \\ 0 & 0 & \ol{f}^{-1} \end{pmatrix} \right) = \prod_{\mathfrak{p} \nmid b} ( a, b )_{\lp, 2}^{- 1} = \prod_{\mathfrak{p} \nmid b} ( a, b )_{\lp, 2}. \]

There are now a few cases to consider for each $ \mathfrak{p} \nmid b $:
\begin{itemize}
	\item If $ \mathfrak{p} $ is even, since $ a \equiv 1 \pod{8} $, this implies that $ a $ is a square in $ \lp $ by Hensel's Lemma (Theorem~\ref{t:hensel}), hence $ ( a, b )_{\lp, 2} $ is trivial by \eqref{e:prop5ii}.
	\item Suppose $ \mathfrak{p} $ is finite and odd.  Then the Hilbert symbol is the tame symbol (Proposition~\ref{p:tame}).  In this case if $ \mathfrak{p} \nmid a $, then we also have $ ( a, b )_{\lp, 2} = 1 $.
\end{itemize}

Thus we are left with the finite primes $ \mathfrak{p} $ which divide $ a $ and the infinite primes.
Our equation becomes
	\[ \kappa\left( \begin{pmatrix} f & g & h \\ 0 & \ol{f} / f & - \ol{g} / f \\ 0 & 0 & \ol{f}^{-1} \end{pmatrix} \right)
	=
	\prod_{\mathfrak{p} \mid a} ( a, b )_{\lp, 2} \cdot \prod_{\mathfrak{p}|\infty} ( a, b )_{\lp, 2}. \]
We can now use the quadratic Legendre symbol formula on the above, hence
	\[ \kappa\left( \begin{pmatrix} f & g & h \\ 0 & \ol{f} / f & - \ol{g} / f \\ 0 & 0 & \ol{f}^{-1} \end{pmatrix} \right) = \left( \frac{b}{a} \right)_{l, 2} \cdot\prod_{\mathfrak{p}|\infty} ( a, b )_{\lp, 2}. \]

Thus our result is proved.
\end{proof}



\chapter{A section for the 2-cocycle on $ \SU( \mathbb{R} ) $} \label{c:real}

Let $ k = \mathbb{R} $, $ \thet = \sqrt{- d} $ for some positive real number $d$.
Thus, $ K = k( \thet ) = \mathbb{C} $.
We shall determine a section for the 2-cocycle on $ \SU( \mathbb{R} ) = G( \mathbb{R} ) $
as described in the Introduction.
Recall that we have a Hermitian form $ \langle - , - \rangle $ on $\mathbb{C}^{3}$ defined by
	\[ \langle u, v \rangle = u^{t} \begin{pmatrix} 0 & 0 & 1 \\ 0 & 1 & 0 \\ 1 & 0 & 0 \end{pmatrix} \ol{v}, \]
and we let
\begin{align*}
	X^{-}
	&= \left\{ [ v ] \in \mathbb{P}^{2}( \mathbb{C} ) \colon \left\langle v , v \right\rangle < 0 \right\} \\
	&= \left\{ \begin{bmatrix} \tau_{1} \\ \tau_{2} \\ 1 \end{bmatrix} \in \mathbb{P}^{2}( \mathbb{C} ) \colon \norm{\tau_{2}} + \trace{\tau_{1}} < 0 \right\}.
\end{align*}
Here $ [ v ] $ means the image of a vector $ v \in \mathbb{C}^{3} $ in projective space.  We want our modular form to be defined on
	\[ \HC = \left\{ \begin{pmatrix} \tau_{1} \\ \tau_{2} \end{pmatrix} \in \mathbb{C}^{2} \colon \norm{\tau_{2}} + \trace{\tau_{1}} < 0 \right\}. \]
We will use the abbreviation $\tau = \begin{pmatrix} \tau_{1} \\ \tau_{2} \end{pmatrix}$
 for an element of \HC.
Let
	\[ g = ( g_{i j} )_{1 \leq i, j \leq 3} = \begin{pmatrix} g_{11} & g_{12} & g_{13} \\ g_{21} & g_{22} & g_{23} \\ g_{31} & g_{32} & g_{33} \end{pmatrix} \in G( \mathbb{R} ), \]
and let
	\[ A = \begin{pmatrix} g_{11} & g_{12} \\ g_{21} & g_{22} \end{pmatrix},\ B = \begin{pmatrix} g_{13} \\ g_{23} \end{pmatrix},\ C = \begin{pmatrix} g_{31} & g_{32} \end{pmatrix},\ D = g_{33}. \]
This implies that
\begin{equation} \label{e:abcd}
	g = \begin{pmatrix} A & B \\ C & D \end{pmatrix}.
\end{equation}

Note that
	\[ g \begin{bmatrix} \tau \\ 1 \end{bmatrix} \in X^{-}. \]
Hence, we have
	\[ g \begin{bmatrix} \tau \\ 1 \end{bmatrix} = \begin{bmatrix} A \tau + B \\ C \tau + D \end{bmatrix} = \begin{bmatrix} \frac{A \tau + B}{C \tau + D} \\ 1 \end{bmatrix}, \]
and in particular $ C \tau + D \neq 0 $.
So we can define an action of $ G( \mathbb{R} ) $ on \HC{} by
	\[ g ( \tau ) = \frac{A \tau + B}{C \tau + D}. \]
(Note that $ C \tau + D $ is a scalar.)

We will write $ \wtilde{G}( \mathbb{R} ) $ for the connected double cover of $G(\mathbb{R})$.
Note that the fundamental group $ \pi_{1}( G( \mathbb{R} ) )$ is isomorphic to $\mathbb{Z} $, so there is a unique connected $ n $-fold cover for any $ n $.

One way of constructing the group $ \wtilde{G}( \mathbb{R} ) $ is as follows:
the elements of $ \wtilde{G}( \mathbb{R} ) $ are pairs
	\[ \left( g, \phi ( \tau ) \right), \]
where $g\in G(\mathbb{R})$ and $\phi$ denotes a continuous function on \HC{} satisfying
	\[ \phi ( \tau )^{2} = C \tau + D. \]
For two elements $ \left( g_{1}, \phi_{1}( \tau ) \right) $,
 $\left( g_{2}, \phi_{2}( \tau ) \right) \in \wtilde{G}( \mathbb{R} ) $, we define their product by
	\[ \left( g_{1}, \phi_{1}( \tau ) \right) \cdot \left( g_{2}, \phi_{2}( \tau ) \right) = \left( g_{1} g_{2}, \phi_{1}( g_{2} ( \tau ) ) \phi_{2}( \tau ) \right). \]
(Later, we will show that $ \wtilde{G}( \mathbb{R} ) $ truly is the unique connected double cover of $ G( \mathbb{R} ) $.)
Recall that we also have another group which we have been referring to as $ \wtilde{G}( \mathbb{R} ) $, and which is defined by the 2-cocycle $\sigma$.
We will now relate these two constructions of $ \wtilde{G}( \mathbb{R} ) $.

Suppose for each $ g \in G( \mathbb{R} ) $ we have chosen a function $ \phi_{g} $ with $ \phi_{g} ( \tau )^{2} = C \tau + D $ (where $ C $, $ D $ are as in the notation above).  Then the map
	\[ g \mapsto \wtilde{g} = \left( g, \phi_{g}( \tau ) \right) \]
defines a section of $ G( \mathbb{R} ) $ to $ \wtilde{G}( \mathbb{R} ) $.  Corresponding to this section, we have a 2-cocycle on $ G( \mathbb{R} ) $ written as
	\[ \Sig{g_{1}}{g_{2}} = \left( g_{1}, \phi_{g_{1}}( \tau ) \right) \cdot \left( g_{2}, \phi_{g_{2}}( \tau ) \right) \cdot \left( g_{1} g_{2}, \phi_{g_{1} g_{2}}( \tau ) \right)^{- 1} \]
for $ g_{1} $, $ g_{2} \in G( \mathbb{R} ) $, i.e.
	\[ \Sig{g_{1}}{g_{2}} = \frac{\left( \phi_{g_{1}} \circ g_{2} \right) \phi_{g_{2}}}{\phi_{g_{1} g_{2}}}. \]
Let
	\[ \ess{g_{1}}{g_{2}} = \frac{\Sig{g_{1}}{g_{2}}}{\sig{g_{1}}{g_{2}}}, \]
where we assume that the 2-cocycle $ \sigma $ has the same formula as that in Theorem~\ref{t:sigma} with $ k = \mathbb{R} $.  Since $ \Sigma $, $ \sigma \in Z^{2}( G( \mathbb{R} ), \mu_{2} ) $ and they both represent the unique non-trivial double cover, this implies that $ S $ is a 2-coboundary.  Our choice of $ \phi_{g} $ will be made carefully so that $ S = 1 $, i.e. $ \sigma = \Sigma $.

\begin{rem} \label{r:unisec}
There is only one choice of section with this property: any two would differ by a continuous homomorphism from $ G( \mathbb{R} ) \to \mu_{2} $; however $ G( \mathbb{R} ) $ is connected.
\end{rem}

We will first determine $ \phi_{g} $ for $ g \in T( \mathbb{R} ) $.  For any $ \lambda \in \mathbb{C}^{\times} $, we have
	\[ \ha{\lambda} = \begin{pmatrix} \lambda & 0 & 0 \\ 0 & \ol{\lambda} / \lambda & 0 \\ 0 & 0 & \ol{\lambda}^{- 1} \end{pmatrix}. \]
Choose
	\[ \phi_{\ha{\lambda}}( \tau ) = \ol{\lambda}^{- 1 / 2}, \]
where $ \arg\left( \ol{\lambda}^{- 1 / 2} \right) \in ( - \pi / 2, \pi / 2 ] $.  (Note that $ \phi_{\ha{\lambda}} $ is a constant function.)  We may check for any $ \lambda $, $ \mu \in \mathbb{C}^{\times} $ that
	\[ \ess{\ha{\lambda}}{\ha{\mu}} = 1 \]
using Theorem~\ref{t:sigma}.  Since $ T( \mathbb{R} ) $ is connected, this is the correct choice of $ \phi_{g} $ on $ T( \mathbb{R} ) $.

\begin{claim} \label{cl:doubcov}
$ \wtilde{G}( \mathbb{R} ) $ is the unique connected double cover of $ G( \mathbb{R} ) $.
\end{claim}

\begin{proof}
Suppose that $ \wtilde{G}( \mathbb{R} ) $ is the trivial double cover.  Then we have
	\[ \Sigma = \partial \nu \]
for some 1-cochain $ \nu \in C^{1}\left( G( \mathbb{R} ), \mu_{2} \right) $.  By our choice for $ \phi_{\ha{- 1}} $, we have
	\[ \Sig{\ha{- 1}}{\ha{- 1}} = - 1. \]
Also,
	\[ \Sig{\ha{- 1}}{\ha{- 1}} = \frac{\nu\left( \ha{- 1} \right)^{2}}{\nu\left( \ha{- 1}^{2} \right)} = \nu\left( I_{3} \right)^{- 1}. \]
This implies that $ \nu\left( I_{3} \right) = - 1 $.  On the other hand,
	\[ \Sig{\ha{- 1}}{I_{3}} = 1 \]
and
	\[ \Sig{\ha{- 1}}{I_{3}} = \frac{\nu\left( \ha{- 1} \right) \cdot \nu\left( I_{3} \right)}{\nu\left( \ha{- 1} \cdot I_{3} \right)} = \nu\left( I_{3} \right), \]
which implies that $ \nu\left( I_{3} \right) = 1 $.  This is a contradiction, thus $ \wtilde{G}( \mathbb{R} ) $ is the unique connected double cover of $ G( \mathbb{R} ) $.
\end{proof}

Without loss in generality, let $ W = \{ 1, \wa{0}{i} \} $, be the set of Weyl group representatives of $ G( \mathbb{R} ) $.  For ease of notation, let 
	\[ w = \wa{0}{i} = \begin{pmatrix} 0 & 0 & i \\ 0 & 1 & 0 \\ i & 0 & 0 \end{pmatrix}. \]
Every element in $ W \cdot T( \mathbb{R} ) $ can be chosen to have the form $ h_{1} $ or $ w \cdot h_{1} $ for $ h_{1} \in T( \mathbb{R} ) $.  Let $ \wtilde{W \cdot T}( \mathbb{R} ) $ be the restriction of $ \wtilde{G}( \mathbb{R} ) $ to elements of the form $ ( w' \cdot h_{1}, \phi_{w' \cdot h_{1}} ) $ where $ w' \in W $, $ h_{1} \in T( \mathbb{R} ) $.  Note that there are two sections $ W \cdot T( \mathbb{R} ) \to \wtilde{W \cdot T}( \mathbb{R} ) $ whose 2-cocycle is $ \sigma $.  This is because there is a non-trivial homomorphism
	\[ W \cdot T( \mathbb{R} ) \to \mathbb{Z} / 2 \to \mu_{2}. \]
For the element $ w $, we have $ C \tau + D = i \tau_{1} $, and $ \trace{\tau_{1}} < 0 $.  This implies that $ C \tau + D $ is in the lower half-plane.  We choose $ \phi_{w} $ such that
$ \arg\left( \phi_{w}( \tau ) \right) \in ( - \pi / 2, 0 ) $.  We will later show that our choice for $ \phi_{w} $ is correct.

For $ h_{1} \in T( \mathbb{R} ) $, $ \phi_{w \cdot h_{1}} $ can be determined from $ \phi_{w} $ as follows.  We have
	\[ ( w \cdot h_{1}, \phi_{w \cdot h_{1}}( \tau ) ) = \Sig{w}{h_{1}} \cdot ( w, \phi_{w}( \tau ) ) \cdot ( h, \phi_{h}( \tau ) ). \]
Since we want $ \sigma = \Sigma $, and $ \sig{w}{h_{1}} = 1 $ for any $ h_{1} \in T( \mathbb{R} ) $ by Theorem~\ref{t:sigma}, we find that
	\[ ( w \cdot h_{1}, \phi_{w \cdot h_{1}}( \tau ) ) = ( w, \phi_{w}( \tau ) ) \cdot ( h, \phi_{h}( \tau ) ) = ( w \cdot h_{1}, \phi_{w}( h_{1} ( \tau ) ) \phi_{h_{1}}( \tau ) ). \]
(It may also be checked that $ \ess{h_{1}}{w} = 1 $.)

Now recall the Bruhat decomposition (see Section~\ref{s:bruh})
	\[ G( \mathbb{R} ) = T( \mathbb{R} ) \cdot N( \mathbb{R} ) \sqcup N( \mathbb{R} ) \cdot w \cdot T( \mathbb{R} ) \cdot N( \mathbb{R} ). \] 
Since $ \sigma $ is trivial on $ N( \mathbb{R} ) $ and $ N( \mathbb{R} ) $ is connected,
	\[ \phi_{n_{1}}( \tau ) = 1 \]
for any $ n_{1} \in N( \mathbb{R} ) $ (i.e. $ \phi_{n_{1}} $ is a constant function).  Similarly, since $ \sig{n_{1}}{g} = 1 $ for any $ g \in G( \mathbb{R} ) $, and we have
	\[ ( n_{1} \cdot g, \phi_{n_{1} \cdot g}( \tau ) ) = \Sig{n_{1}}{g} \cdot ( n_{1}, \phi_{n_{1}}( \tau ) ) \cdot ( g, \phi_{g}( \tau ) ), \]
with $ \sig{n_{1}}{g} = \Sig{n_{1}}{g} $, this implies that
	\[ \phi_{n_{1} \cdot g}( \tau ) = \phi_{g}( \tau ). \]
Also, we have
	\[ ( g \cdot n_{1}, \phi_{g \cdot n_{1}}( \tau ) ) = \Sig{g}{n_{1}} \cdot ( g, \phi_{g}( \tau ) ) \cdot ( n_{1}, \phi_{n_{1}}( \tau ) ), \]
and since $ \sig{g}{n_{1}} = 1 $ and $ \sig{g}{n_{1}} = \Sig{g}{n_{1}} $, 
	\[ \phi_{g \cdot n_{1}}( \tau ) = \phi_{g}( n_{1} ( \tau ) ). \]
Hence, once we establish the choice for $ \phi_{w} $, we can determine $ ( g, \phi_{g} ) $ for any $ g \in G( \mathbb{R} ) $.

\begin{claim} \label{cl:phiw}
Our choice for $ \phi_{w} $ is the correct choice.
\end{claim}

\begin{proof}
Let
	\[ g_{1} = \xam{0}{i} = \begin{pmatrix} 1 & 0 & 0 \\ 0 & 1 & 0 \\ i & 0 & 1 \end{pmatrix}, \quad g_{2} = \begin{pmatrix} 0 & 0 & i \\ 0 & 1 & 0 \\ i & 0 & - 1 \end{pmatrix}. \]
So we have $ g_{1} \cdot w = g_{2} $.  By \eqref{e:bruc} (using $ \thet = i $),
	\[ g_{1} = \xa{0}{- i} \cdot w \cdot \xa{0}{- i}, \quad g_{2} = w \cdot \xa{0}{i}. \]

We want to show that with our choice for $ \phi_{w} $,
	\[ ( g_{2}, \phi_{g_{2}}( \tau ) ) = ( g_{1} \cdot w, \phi_{g_{1} \cdot w}( \tau ) ) = \sig{g_{1}}{w} \cdot ( g_{1}, \phi_{g_{1}}( \tau ) ) \cdot ( w, \phi_{w}( \tau ) ). \]
Since by Theorem~\ref{t:sigma}, $ \sig{g_{1}}{w} = 1 $, the above equation can be simplified to	\[ \phi_{w}( \xa{0}{i} ( \tau ) ) = \phi_{w}( ( \xa{0}{- i} \cdot w ) ( \tau ) ) \cdot \phi_{w}( \tau ). \]
Hence, if our choice for $ \phi_{w} $ is wrong, this would imply that the left-hand side would not equal the right-hand side.

We evaluate both sides at the point
	\[ \tau' = \begin{pmatrix} - 1 \\ 0 \end{pmatrix} \in \HC. \]
For the left-hand side, we have
	\[ \phi_{w}( \xa{0}{i} ( \tau' ) ) = \phi_{w}\left( \begin{pmatrix} - 1 + i \\ 0 \end{pmatrix} \right) = e^{- 3 i \pi / 8} \]
with our choice for $ \phi_{w} $.  As for the right-hand side,
\begin{align*}
	\phi_{w}( ( \xa{0}{- i} \cdot w ) ( \tau' ) ) \cdot \phi_{w}( \tau' )
	&= \phi_{w}\left( \begin{pmatrix} - 1 - i \\ 0 \end{pmatrix} \right) \cdot e^{- i \pi / 4} \\
	&= e^{- i \pi / 8} \cdot e^{- i \pi / 4} \\
	&= e^{- 3 i \pi / 8}.
\end{align*}
Since both sides concur, our choice for $ \phi_{w} $ is correct.
\end{proof}

We have thus proved the following, as first mentioned in the Introduction:

\begin{thm} \label{t:wtmultsys}
Let $ \tau \in \HC $.  For each $ g \in G( \mathbb{R} ) $ as defined in \eqref{e:abcd}, choose a function $ \phi_{g} $ such that $ \phi_{g}( \tau )^{2} = C \tau + D $ as follows:
\begin{itemize}
	\item $ \phi_{h \cdot n_{1}}( \tau ) = \ol{\lambda}^{- 1 / 2} $, where $ \arg\left( \ol{\lambda}^{- 1 / 2} \right) \in ( - \pi / 2, \pi / 2 ] $;
	\item $ \arg( \phi_{w}( \tau ) ) \in ( - \pi / 2, 0 ) $;
	\item $ \phi_{n_{2} \cdot w \cdot h \cdot n_{1}}( \tau ) = \phi_{w}( ( h \cdot n_{1} ) ( \tau ) ) \phi_{h \cdot n_{1}}(\tau ) $;
\end{itemize}
where $ n_{1} $, $ n_{2} \in N( \mathbb{R} ) $, $ w = \wa{0}{i} $ and $ h = \ha{\lambda} \in T( \mathbb{R} ) $.  Then we have a half-integral weight multiplier system
	\[ j( \gamma, \tau ) = \kappa( \gamma ) \phi_{\gamma}( \tau ), \]
where $ \gamma \in \Gamma $, the congruence subgroup on which the global Kubota symbol $ \kappa $ is defined on.
\end{thm}


\bibliographystyle{amsplain}
\bibliography{ref}

\end{document}